\documentclass[10pt]{article}
\usepackage[utf8]{inputenc}

\usepackage[margin=1in]{geometry} 
\usepackage{graphicx}

\usepackage{booktabs} 
\usepackage{array} 
\usepackage{paralist} 
\usepackage{verbatim}
\usepackage{mathrsfs}
\usepackage{amssymb}
\usepackage{amsthm}
\usepackage{amsmath,amsfonts,amssymb}
\usepackage{esint}
\usepackage{graphics}
\usepackage{enumerate}
\usepackage{mathtools}
\usepackage{xfrac}
\usepackage{bbm}
\usepackage{caption}
\usepackage{subcaption}
\usepackage{tikz-cd}
\usepackage{adjustbox}

\usepackage[colorlinks=true, pdfstartview=FitV, linkcolor=blue, citecolor=blue, urlcolor=blue]{hyperref}

\numberwithin{equation}{section}
\numberwithin{figure}{section}

\newtheorem{theorem}{Theorem}[section]

\newtheorem{corollary}[theorem]{Corollary}
\newtheorem{proposition}[theorem]{Proposition}
\newtheorem{lemma}[theorem]{Lemma}
\theoremstyle{definition}
\newtheorem{definition}[theorem]{Definition}

\newtheorem{remark}[theorem]{Remark}

\newcommand*{\supp}{\ensuremath{\mathrm{supp\,}}}

\newcommand*{\N}{\ensuremath{\mathbb{N}}}

\newcommand*{\Z}{\ensuremath{\mathbb{Z}}}

\newcommand*{\R}{\ensuremath{\mathbb{R}}}
\newcommand*{\Zd}{\ensuremath{\mathbb{Z}^d}}
\newcommand*{\Rd}{\ensuremath{\mathbb{R}^d}}

\newcommand{\eps}{\varepsilon}

\renewcommand*{\tilde}{\widetilde}

\newcommand{\g}{\mathbf{g}}

\newcommand{\ep}{\eps}

\DeclareMathOperator{\dist}{dist}

\newcommand{\E}{\mathbb{E}}

\DeclareSymbolFont{boldoperators}{OT1}{cmr}{bx}{n}
\SetSymbolFont{boldoperators}{bold}{OT1}{cmr}{bx}{n}
\edef\bar{\unexpanded{\protect\mathaccentV{bar}}\number\symboldoperators16}
\renewcommand{\a}{\mathbf{a}}

\newcommand{\ahom}{\bar{\a}}

\definecolor{labelkey}{rgb}{0,0,1}

\newcommand{\indc}{{\boldsymbol{1}}}

\usepackage{titlesec}

\newcommand{\addperiod}[1]{#1.}
\titleformat{\section}
   {\centering\normalfont\Large}{\thesection.}{0.5em}{}
\titleformat*{\subsection}{\bfseries}
\titleformat{\subsubsection}[runin]
  {\normalfont\bfseries}
  {\thesubsubsection.}
  {0.5em}
  {\addperiod}
\titleformat*{\subsubsection}{\bfseries}
\titleformat*{\paragraph}{\bfseries}
\titleformat*{\subparagraph}{\large\bfseries}

\title{Quantitative hydrodynamic limits of the Langevin dynamics for gradient interface models}

\author{Scott Armstrong
\thanks{Courant Institute of Mathematical Sciences, New York University.
{\footnotesize scotta@cims.nyu.edu.}
}
\and 
Paul Dario
\thanks{LAMA, Universit\'e Paris-Est Cr\'eteil, Cr\'eteil, France.
{\footnotesize paul.dario@u-pec.fr.}
}
}
\date{ }

\usepackage[nottoc,notlot,notlof]{tocbibind}

\begin{document}

\maketitle

\begin{abstract}
We study the Langevin dynamics corresponding to the $\nabla\phi$ (or Ginzburg-Landau) interface model with a uniformly convex interaction potential. We interpret these Langevin dynamics as a nonlinear parabolic equation forced by white noise, which turns the problem into a nonlinear homogenization problem. Using quantitative homogenization methods, we prove a quantitative hydrodynamic limit, obtain the $C^2$ regularity of the surface tension, prove a large-scale $C^{1 , \alpha}$-type estimate for the trajectories of the dynamics, and show that the fluctuation-dissipation relation can be seen as a commutativity of homogenization and linearization. Finally, we explain why we believe our techniques can be adapted to the setting of degenerate (non-uniformly) convex interaction potentials.
\end{abstract}

\setcounter{tocdepth}{1}
\tableofcontents

\section{Introduction}

\subsection{Motivation and informal summary of main results}

Random surfaces in statistical mechanics are used to model the interface separating two pure thermodynamic phases. In classical effective interface models of this type, the interface is represented by a function $\phi : \Zd \to \R$ to which one associates an energy applied to the discrete gradient of the field and defined as follows. On a finite set $\Lambda \subset \Z^d$ and with a prescribed tilt $p \in \Rd$, each surface $\phi:\Lambda \to \R$ satisfying the Dirichlet boundary conditions $\phi = 0$ on $\partial \Lambda$ is assigned the energy
\begin{equation} \label{03031932}
    H_{\Lambda , p} \left( \phi \right) = \sum_{ \substack{x , y \in \Lambda^+ \\ |x - y| = 1} } V(p \cdot (y - x) +  \phi(y) - \phi(x) ),
\end{equation}
where $\partial\Lambda$ is the external vertex boundary of $\Lambda$, $\Lambda^+$ is the set $\Lambda\cup\partial\Lambda$, $|\cdot|$ denotes the Euclidean norm, $V:\R \to \R$ is an interaction potential.

In this paper, we assume that the interaction potential is symmetric, uniformly convex and~$C^{1,1}$, that is,~$V'$ is Lipschitz and there exist two constants $c_-, c_+\in (0, \infty),$ such that the second derivative $V''$ satisfies
\begin{equation} \label{eq.unifconvV}
     c_- \leq V''(x)\leq c_+ ~~~~\mbox{for almost every}~ x \in \R.
\end{equation}
The law of the random surface is then given by
\begin{equation} \label{eq:defmuL}
    \mu_{\Lambda , p} (d \phi) := \frac{1}{Z_{\Lambda , p}}\exp \left( - H_{{\Lambda , p}} (\phi) \right) \prod_{v \in \Lambda} d \phi(v),
\end{equation}
where $d \phi(v)$ denotes the Lebesgue measure on $\R$ and $Z_\Lambda$ is the constant which makes $\mu_\Lambda$ a probability measure. Under the assumption~\eqref{eq.unifconvV}, it is known that, in any dimension $d \geq 1$~\cite{FS}, the measures $\mu_{\Lambda , p}$ converges as $\Lambda \to \Zd$ to a unique infinite-volume translation-invariant and ergodic Gibbs measure on the space of gradient fields (or configurations on $\Zd$ modulo constant), which we denote by $\mu_{\infty , p}$.
In dimension $d \geq 3$, the infinite-volume measures can be defined on configurations of $\Zd$ and not only gradient fields; we will denote them by $\mu_{\infty , p}$.

The model~\eqref{eq:defmuL} is known as the (uniformly convex) 
\emph{$\nabla \phi$-model} or 
\emph{discrete Ginzburg-Landau model} and has been extensively studied under the assumption~\eqref{eq.unifconvV}. We refer to~\cite{F05,Sh,V06} or Section~\ref{section.background} below for an overview of its literature.

The Gibbs measure $\mu_{\infty , p}$ is naturally associated with the following Langevin dynamic
\begin{equation} \label{Langevin.dyn}
    d \phi(t, x) 
     = \sum_{\substack{y \in \Zd \\ |y - x| = 1}} V'( p \cdot (y - x) + \phi(t, y) - \phi(t, x)) \, dt + \sqrt{2} dB_t(x), 
\end{equation}
where $\left\{ B_t(x) \, : \, x \in \Zd, t \in \R \right\}$ is a family of independent Brownian motions, and the notation $x \in e$ means that the vertex $x$ is one of the endpoints of the edge $e$.
Specifically, the measure $\mu_{\infty , p}$ is stationary, reversible and ergodic with respect to the dynamic~\eqref{Langevin.dyn}. We denote by $\phi(\cdot , \cdot ; p) : \R \times \Zd \to \R$ a stationary solution of the Langevin dynamic~\eqref{Langevin.dyn} (defined modulo constant in dimension $d = 2$), and refer to~\cite{FS, GOS} for a proof of its existence.

One generally wishes to understand the large-scale, macroscopic, statistical behavior of the dynamic $\phi(\cdot , \cdot ; p)$ and its discrete gradient. In this direction, two questions are of particular importance:
\begin{itemize}
    \item \textit{The hydrodynamic limit.} The first one aims at establishing that the large-scale behavior of the stochastic dynamic is governed by the solution of a suitable, deterministic partial differential equation. In the case of the uniformly-convex $\nabla \phi$-model, the hydrodynamic limit was established by Funaki and Spohn~\cite{FS}. Their result asserts that the macroscopic behavior of the dynamic is governed by the deterministic, nonlinear parabolic equation
    \begin{equation*}
        \partial_t h - \nabla \cdot D_p \bar \sigma ( \nabla h) = 0,
    \end{equation*}
    where the function $\bar \sigma : \Rd \to \R$ is $C^{1,1}$ and uniformly convex, intrinsic to the model, called \emph{the surface tension}. It is defined in~\eqref{def.surfacetension} below and $D_p \bar \sigma$ denotes its gradient.

    \item \textit{The scaling limit.} A second important problem is to understand the fluctuations of the Langevin dynamic around the deterministic profile. The question was settled by Naddaf-Spencer~\cite{NS} and by Giacomin-Olla-Spohn~\cite{GOS}. Specifically, it is established in~\cite{GOS} that the fluctuations of the space time dynamic $\phi(\cdot, \cdot ; p)$ are described by an infinite-dimensional Ornstein-Uhlenbeck process of the form
    \begin{equation*}
        d \Phi_t = \nabla \cdot \bar \a(p) \nabla \Phi dt +  \sqrt{2} \dot W,
    \end{equation*}
    where $\bar \a(p)$ is a deterministic uniformly elliptic matrix and $\dot W$ is a normalized space-time white noise.
    The article~\cite{GOS} (see also~\cite[Problems 5.1 and 11.1]{F05}) additionally conjectures that the surface tension $\bar \sigma$ is twice continuously differentiable and that the coefficient $\bar \a(p)$ is related to the surface tension via to the formula:
    \begin{equation} \label{2eq.flucdiss}
        \bar \a(p) = D^2_p \bar \sigma (p).
    \end{equation}
    The identity~\eqref{2eq.flucdiss} is known as \emph{the fluctuation-dissipation relation}, and was recently established by Armstrong and Wu~\cite{AW} under the assumption that the second derivative of the potential $V$ is H\"{o}lder continuous. 
\end{itemize}

The strategies developed in the articles~\cite{FS} and~\cite{NS, GOS} rely on different sets of tools. The proof of Funaki and Spohn is based on the techniques developed in the setting of the Ginzburg-Landau equation with a conserved order of parameter (see~\cite{guo1988nonlinear}). They establish by an entropy argument that the local space-time averaging of the law of the dynamic is a mixture of infinite-volume shift-invariant and ergodic gradient Gibbs measures. They then classify these measures by proving that, for any prescribed tilt $p \in \Rd$, there exists a unique shift-invariant and ergodic gradient Gibbs measure (the measure $\mu_{\infty, p}$).

The proof of the scaling limit of Naddaf-Spencer~\cite{NS} relies on the observation that the correlation structure of the scaling limit can be identified by homogenizing an infinite-dimensional PDE, called the \emph{Helffer-Sj\"{o}strand PDE} based on the work of Helffer and Sj\"{o}strand~\cite{HS, Sj}. This analytic technique successfully identified the scaling limit of the $\nabla \phi$-model, and was reworked probabilistically and extended by Giacomin-Olla-Spohn~\cite{GOS} who reformulated the question of the homogenization of an infinite-dimensional PDE into the proof of an invariance principle for a random walk evolving in a dynamic, random environment. The question admits a third, equivalent, reformulation: the identification of the correlation structure of the scaling limit boils down to establishing homogenization for the discrete parabolic equation (see~\eqref{eq:discdynlaplacian})
\begin{equation} \label{parabolic.eqlinearized}
    \partial_t u - \nabla \cdot \a \nabla u = 0 ~\mbox{in}~\Zd,
\end{equation}
with the random environment $\a := V''(\nabla \phi(\cdot , \cdot ; p))$.

\smallskip

In this article, we propose to view the Langevin dynamic~\eqref{Langevin.dyn} as a nonlinear parabolic equation forced by white noise, allowing us to apply recently developed methods in quantitative stochastic homogenization~\cite{AS,AFK1, AFK2, FN19, CG21}. In this way we circumvent the need to analyze the infinite-dimensional Helffer-Sj\"ostrand equation, whose role is played instead by the linearized Langevin dynamics---which turns out to be the linear, uniformly parabolic equation~\eqref{parabolic.eqlinearized}. This interpretation of the model, which is explained in more detail in Section~\ref{sectionlangevinaspdewithnoise} below, allows to reformulate previous results for the~$\nabla\phi$-model in terms of homogenization. For instance:
\begin{itemize}
    \item The hydrodynamic limit of Funaki and Spohn~\cite{FS} is a homogenization theorem, and the limit can therefore be quantified using homogenization methods. This is the subject of our first main result, Theorem~\ref{Th.quantitativehydr} below.
    
    \item The parabolic equation~\eqref{parabolic.eqlinearized} corresponds to the \emph{linearized equation} in~\cite{AFK1, AFK2, FN19, CG21}, that is, the Langevin dynamics linearized around the trajectories. The identification of the correlation structure and proof of the scaling limit for the Gibbs measure turns out to be equivalent to a homogenization statement for this linearized equation. Since this equation is uniformly parabolic, quantitative homogenization estimates for it are essentially known (see~\cite{ABM}) but we will not present them here.  
    
    \item The surface tension $\bar \sigma$ corresponds to the effective Lagrangian (see~\eqref{eq:18270903}). Note that the regularity of the effective Lagrangian in the setting of stochastic homogenization has been studied in detail in~\cite{AFK2}, and so we should expect that similar methods are applicable here. The fluctuation-dissipation relation in this terminology is known as \emph{the commutativity of homogenization and linearization}, because it essentially says that the homogenized coefficient for the linearized equation is equal to the coefficient for the  linearization of the homogenized equation (see~\cite{AFK1,AFK2,FN19}). The identity~\eqref{2eq.flucdiss} was established in~\cite{AW} under the assumption that $V^{\prime\prime}$ is H\"older continuous, based on a different approach which relied on a quantification of the homogenization of the infinite-dimensional Helffer-Sj\"{o}strand PDE.
    
\end{itemize}
Our motivation to develop this approach is threefold:

\begin{itemize}
    \item[(i)] We use the connection between Langevin dynamics and stochastic homogenization to strengthen some results known in the field and establish new ones. In this direction, we establish two theorems: in Theorem~\ref{Th.quantitativehydr}, we obtain a quantitative version of the hydrodynamic limit of Funaki and Spohn~\cite{FS}, quantified both over the rate of convergence and the stochastic integrability. In Theorem~\ref{t.regsurfacetension}, we prove the $C^2$-regularity of the surface tension $\bar \sigma$ under the assumption that the potential is $C^{1,1}(\R)$, generalizing the result~\cite{AW} where the regularity is proved for potentials $V$ whose second derivative is H\"{o}lder continuous, and solving the conjecture of~\cite{GOS} and~\cite[Problem 5.1]{F05} in full generality.
    
    \item[(ii)] We show that a quantitative version of the hydrodynamic can be used to develop a \emph{large-scale regularity theory} for the model (following~\cite{AL1, AL2, AS, GNO14}). In this direction, we prove a Lipschitz and $C^{1, \alpha}$ regularity estimate for the dynamic in Theorem~\ref{theoremlargescale}. Large-scale regularity is fundamental in stochastic homogenization, and we expect it to play a similarly important role in this setting. 
    
     \item[(iii)] We believe that the approach developed in this article can be extended to a class of non-uniformly convex potentials; for instance, degenerate potentials of the form $V(x) = |x|^{p}$ for $p > 2$ (see~\cite{magazinov2020concentration} for recent progress on these models), or more generally to potentials satisfying $c_- |x|^{p-2} \leq V''(x) \leq  c_+ |x|^{p-2}$ for some $c_- , c_+ \in (0, \infty)$ and $p > 2$. These models correspond to homogenization problems involving equations with degenerate coefficients. However,  quantitative stochastic homogenization methods have been developed which are quite robust and able to handle such degeneracies (see for instance~\cite{AD}). 
\end{itemize}

We refer to Section~\ref{sec.discussionlargescale} for a more detailed discussion regarding the question of degenerate potentials and the prospective applications of the large-scale regularity theory in this framework.

\subsection{The main results} \label{section:mainresult}

\smallskip

The first main result of this paper is a quantification of the hydrodynamic limit proved in~\cite{FS}. Our approach is different from the one of~\cite{FS} and is based on the quantitative stochastic homogenization methods introduced in~\cite{GO1, GO2, AKMbook}, extended to the parabolic setting in~\cite{ABM}.

In order to state the result formally, we first introduce a few notation. We let $D \subseteq \Rd$ be a bounded $C^{1,1}$ domain, let $I = (-1 , 0)$ and define the parabolic cylinder $Q = I \times D$. For $\ep > 0$, we discretize the sets $D$ and $Q$ at scale $\ep$ by setting $D^\ep := \ep \Zd \cap D$, $Q^\ep := I \times D^\ep$, we also denote by $\partial D^\ep$ the external vertex boundary of the set $D^\ep$ and by $\partial_{\mathrm{par}} Q^\ep := (\{-1\} \times D^\ep) \cup (I \times \partial D^\ep)$ the parabolic boundary of the cylinder $Q^\ep$. Given a function $v : Q^\ep \to \R$, and a point $(t , x) \in Q^\ep$, we denote by
\begin{equation} \label{nonlinearellipticop}
    \nabla^\ep \cdot V'(\nabla^\ep v) (t , x)  = \frac{1}{\ep} \sum_{\substack{y \in D^\ep \\ |x - y| = \ep}} V'\left( \frac{v(t , y) - v(t , x)}{\ep} \right).
\end{equation}
We also define the $L^2(Q^\ep)$-norm of the function $v$ according to the formula
\begin{equation*}
    \left\|  v\right\|_{L^2(Q^\ep)} = \ep^d \int_I  \sum_{x \in D^\ep} \left| v(t,x)\right|^2 \, dt.
\end{equation*}
We denote by $H^2(Q)$ the standard Sobolev space on the space-time cylinder $Q$ and discretize a function $f \in H^2(Q)$ at scale $\ep$ by setting $\tilde f_\ep(t , x) := (2\ep)^d \int_{[-\ep , \ep]^d} f(t , x + y) \, dy$. We fix through the article a collection of Brownian motions $\left\{ B_t(x) \, : \, x \in \Zd, t\in \R \right\}$ (see~\eqref{eq:14350901}) and measure the stochastic integrability of a random variable $X$ as follows: given an exponent $s > 0$ and a constant $K > 0$, we denote by
\begin{equation*}
    X \leq \mathcal{O}_s(K) ~~\mbox{if and only if}~~  \mathbb{E} \left[ \exp \left( \biggl( \frac{X}{K} \biggr)^{\!\!s} \right) \right] \leq 2.
\end{equation*}
The first result of this article is a quantitative version of the hydrodynamic limit of Funaki and Spohn~\cite{FS}.

\begin{theorem}[Quantitative hydrodynamic limit] \label{Th.quantitativehydr}
Let $f \in H^{2}(Q)$. Fix $\ep \in (0,1) $ and let $u^\ep : Q^\ep \to \R$ be the solution of the system of stochastic differential equations
\begin{equation} \label{eq:defuLthmhydro}
    \left\{ \begin{aligned}
    d u^\ep(t , x) & = \nabla^\ep \cdot V'(\nabla^\ep u^\ep) (t , x) dt + \sqrt{2} \ep dB_{\frac{t}{\ep^2}}\left( \frac{x}{\ep} \right) &~\mbox{for} &~(t,x) \in  Q^\ep,   \\
    u^\ep & = \tilde f_\ep &~\mbox{on} &~ \partial_{\mathrm{par}} Q^\ep,
    \end{aligned} \right.
\end{equation}
and let $\bar u : Q \to \R$ be the solution of the continuous nonlinear parabolic equation
\begin{equation} \label{eq:defubarthmhydro}
    \left\{ \begin{aligned}
    \partial_t \bar u - \nabla \cdot D_p \bar \sigma(\nabla \bar u) & = 0 &~\mbox{in} &~ Q, \\
    \bar u &= f &~\mbox{on} &~ \partial_{\mathrm{par}} Q. \\
    \end{aligned} \right.
\end{equation}
Then, there exists a constant $C_f < \infty$ depending on $d , c_+ , c_-, \left\| f \right\|_{H^2(Q)}$ and $D$ such that
\begin{equation*}
    \left\| u - \bar u \right\|_{L^2 \left( Q^\ep \right)} \leq \mathcal{O}_2 \left(C_f \ep^{\frac12} \bigl(1 + | \log \ep |^{\frac12} \indc_{\{ d = 2\}}\bigr) \right).
\end{equation*}
\end{theorem}

\begin{remark} \label{remark1.2}
\begin{enumerate}
    \item The proof gives the following estimate for the gradient of the function $u^\ep$ (using the notation of Section~\ref{Sectionmicroscopic})
\begin{equation*}
    \left\| \nabla^\ep u - \nabla^\ep w^\ep \right\|_{L^2 \left( Q^\ep \right)} \leq \mathcal{O}_2 \left(C_f \ep^{\frac12} \bigl(1 + | \log \ep |^{\frac12} \indc_{\{ d = 2\}}\bigr) \right),
\end{equation*}
where $w^\ep$ is the two-scale expansion defined in~\eqref{def.wL}.
    \item In the notation of the previous theorem, the result of Funaki and Spohn~\cite{FS} reads as follows. If the domain $D$ is the torus $\mathbb{T}^d$ (i.e., the result is established with periodic boundary conditions instead of Dirichlet), the initial condition $f(0 , \cdot)$ is assumed to be $L^2(\mathbb{T}^d )$, then, for any time $t \geq 0$,
\begin{equation} \label{eq:14000103}
    \E \left[ \left\| u(t , \cdot) - \bar u(t , \cdot) \right\|_{L^2\left( \mathbb{T}^d \right)}^2 \right] \underset{\ep \to 0}{\longrightarrow} 0.
\end{equation}
The convergence~\eqref{eq:14000103} was extended from periodic to Dirichlet boundary conditions by Nishikawa~\cite{nishikawa2003hydrodynamic}. We mention that the result~\eqref{eq:14000103} obtained by Funaki and Spohn could be obtained using the tools developed in this article, and quantified if more regularity is assumed on the initial condition.
    \item The result is optimal regarding stochastic integrability. Regarding the rate of convergence, we obtain half of the exponent of the optimal rate. Optimality could be reached at the cost of more technicalities (see Section~\ref{sec1.5.2} where the question is discussed).
\end{enumerate}
\end{remark}

The second result of this article establishes the $C^2$-regularity of the surface tension $\bar \sigma$. It provides a second proof of the result of~\cite{AW} and removes the H\"older regularity assumption on~$V^{\prime\prime}$, thus fully resolving the conjecture of~\cite{GOS} and~\cite[Problem 5.1]{F05}. In fact, we show that $\sigma \in C^{2}(\Rd)$ even if $V$ is uniformly convex and~$C^{1,1}(\R)$---that is, the surface tension may have more regularity than the interaction potential. This effect is specific to the Langevin dynamics, and is due to the presence of the Brownian motions~$B_t(x)$ in~\eqref{Langevin.dyn}. It is in particular not observed in homogenization of nonlinear equations~\cite{AFK1, AFK2}, in general. 

\smallskip

The argument is based on Lusin's theorem applied to~$V''$ and provides an estimate on the modulus of continuity of the Hessian~$D^2\bar{\sigma}$ depending on measure-theoretic information about the second derivative~$V^{\prime\prime}$. In order to state the result and show on which quantity the modulus of continuity of $D^2\bar{\sigma}$ depends, we need to quantify Lusin's theorem. To this end, it is both natural and convenient to use a mollification, and we let $\{ \eta_\kappa\}_{\kappa>0}$ be the standard mollifier, that is, $\eta_\kappa := \kappa^{-1} \eta\left(\cdot/\kappa\right)$ where $\eta : \R \to \R$ is a smooth, nonnegative function satisfying $\supp \eta \subseteq [-1,1]$ and $\int_\R \eta = 1$. 
We next define
\begin{equation*}
    V_\kappa := V \star \eta_\kappa.
\end{equation*}
Since $V\in C^{1,1}(\R)$, the second derivative of $V_\kappa$ satisfies the following bound, for a constant $C < \infty$ depending on $c_+ , c_-,$ and $\eta$,
\begin{equation} \label{eq:VrgLIP}
     \left| V_\kappa''(x) - V_\kappa''(y) \right| \leq C\kappa^{-1} |x - y|,
     \qquad x,y\in\Rd\,.
\end{equation}
Moreover, by the Lebesgue differentiation theorem,
\begin{equation} \label{eq:14152109}
\sup_{S\geq 1}
\lim_{\kappa \to 0}
\int_{-S}^S | V_\kappa''(x) - V''(x) | \, dx = 0.
\end{equation}
Consequently, the set 
\begin{equation*}
    A_{S , \kappa}(\ep) := \Bigl\{ x \in [-S , S] \, : \, \left| V''(x) - V_\kappa''(x) \right| \geq \ep \Bigr\}
\end{equation*}
satisfies
\begin{equation} \label{conv.regularityV''}
\sup_{(S,\ep) \in [1,\infty) \times (0,1]}
  \limsup_{\kappa\to 0} \left| A_{S , \kappa}(\ep)  \right| =0.
\end{equation}
Our estimate for the modulus of continuity of $D^2\bar{\sigma}$ depends only on the rate of the limit in~\eqref{conv.regularityV''} over all choices of parameters $(S,\ep)\in [1,\infty) \times (0,1]$.

\begin{theorem}[$C^2$-regularity of the surface tension]
\label{t.regsurfacetension}
Under the assumption that the potential $V$ is $C^{1 , 1}(\R)$ and uniformly convex, the surface tension $\bar \sigma$ belongs to the space $C^2 (\Rd)$.
Moreover, for each $R\geq 1$, there exists a continuous function $\chi_R:[0,\infty) \to [0,\infty)$, which depends only on $c_-,c_+,R$ 
and the rate of the limit in~\eqref{conv.regularityV''} over all parameters $(S,\ep)\in [1,\infty]\times (0,1]$, such that
\begin{equation}
\label{e.modululs.estimate}
    \bigl| D^2_p \bar{\sigma} (p) - D^2_p\bar{\sigma}(q) \bigr| 
    \leq \chi_R(|p-q|), \quad \forall p,q\in B_R.
\end{equation}
\end{theorem}

The surface tension $\bar \sigma$ is classically defined as a limit, as $L$ tends to infinity, of the finite-volume surface tensions $\sigma_L$ (see~\cite{FS} or~\eqref{def.surfacetension}). Our proof also yields an optimal rate of convergence for the gradient $D_p \sigma_L$ to $D_p \bar \sigma$.

\begin{remark}
An investigation of the proof of Theorem~\ref{t.regsurfacetension} reveals that, under the stronger assumption that the potential $V$ belongs to $C^{2 , \alpha}(\R)$, for some $\alpha > 0$, then the surface tension belongs to $C^{2 , \beta}(\R^d)$ for some~$\beta < \alpha$, with an appropriate estimate. We therefore recover the statement proved in~\cite{AW}.
\end{remark}

Our third main result is a \emph{large-scale regularity estimate} for the Langevin dynamics. Specifically, we show that the random surface is nearly $C^{1 , \alpha}$, and behaves like a $C^{1 , \alpha}$ function on ``large scales'' (i.e., scales which are a large multiple of the discrete scale). 
Our proof follows the approach of~\cite{AS}: by Theorem~\ref{Th.quantitativehydr}, we know that a solution of the Langevin dynamics is well-approximated by a solution of the equation~\eqref{eq:defubarthmhydro} which possesses good regularity properties, we are then able to transfer the regularity of the solution of the equation~\eqref{eq:defubarthmhydro} to the solution of the Langevin dynamics over large scales where homogenization occurs, using the quantitative homogenization estimates. 

\smallskip

In the following statement, we will denote by $\Lambda_L := \{ - L, \cdots, L \}^d$ and $Q_L := (-L^2 , 0) \times \Lambda_L \subseteq \R \times \Zd$ the discrete box and parabolic cylinder of size $L$, and denote by $\left| Q_L\right| := L^2 (2L+1)^d$ the volume of the cylinder $Q_L$. Given a function $u : Q_L \to \R$, we will denote by
\begin{equation*}
    \left( u \right)_{Q_L} := \frac{1}{\left| Q_L \right|} \int_{-L^2}^0 \sum_{x \in \Lambda_L} u(t , x) \, dt \hspace{5mm} \mbox{and} \hspace{5mm}  \left\| u \right\|_{\underline{L}^2(Q_L)}^2 := \frac{1}{\left| Q_L \right|} \int_{-L^2}^0 \sum_{x \in \Lambda_L} |u(t , x)|^2 \, dt,
\end{equation*}
the average value and averaged $L^2$-norm of the function $u$ over the parabolic cylinder $Q_L$. We also denote by $\nabla \cdot V'(\nabla v)$ the discrete elliptic operator~\eqref{nonlinearellipticop} with $\ep = 1$, and by $\mathcal{P}_1$ the set of affine functions in $\Rd$, i.e.,
\begin{equation*}
    \mathcal{P}_1 := \left\{ \ell : \Rd \to \R \, : \, \exists p \in \Rd, \, c \in \R, ~ \ell(x) = p \cdot x + c \right\}.
\end{equation*}

\begin{theorem}[Large-scale $C^{0,1}$ and $C^{1, \alpha}$ estimates] \label{theoremlargescale}
Fix $s \in (0, d)$ and $M < \infty$. There exist a constant $C < \infty$ and an exponent $\beta >0$ depending on $s , c_+, c_-, d$, a constant $K < \infty$ and a nonnegative random variable $\mathcal{M}_{\mathrm{reg}}^s$ depending on $s, M , c_+, c_-, d$ such that
\begin{equation} \label{eq:thmlargescale1234}
    \mathcal{M}_{\mathrm{reg}}^s \leq \mathcal{O}_s (K),
\end{equation}
and such that the following holds. For any $L \geq 2 \mathcal{M}_{\mathrm{reg}}^s$, every $u : Q_L \to \R$ solution of the Langevin dynamics
\begin{equation} \label{def.ulargescale}
    \left\{ \begin{aligned}
    d u(t , x) = \nabla \cdot V'(\nabla u) &(t , x) dt + \sqrt{2} dB_{t}\left( x \right)~\mbox{for}~(t,x) \in  Q_L, \\
    \frac{1}{L} \left\| u - (u)_{Q_L} \right\|_{\underline{L}^2(Q_L)} &\leq M,
    \end{aligned} \right.
\end{equation}
and every $l \in [\mathcal{M}_{\mathrm{reg}}^s , L]$, one has the estimates
\begin{equation} \label{eq:10040203}
    \frac{1}{l} \left\| u -  (u)_{Q_l} \right\|_{\underline{L}^2(Q_l)} \leq \frac{C}{L} \left\| u - (u)_{Q_L} \right\|_{\underline{L}^2 (Q_L)} + C,
\end{equation}
and 
\begin{equation} \label{eq:10050203}
    \inf_{\ell \in \mathcal{P}_1}\left\| u - \ell \right\|_{\underline{L}^2(Q_l)} \leq C \left( \frac{l}{L} \right)^{1+\alpha} \inf_{\ell \in \mathcal{P}_1}\left\| u - \ell \right\|_{\underline{L}^2(Q_L)} + C l^{-\beta} (M+1).
\end{equation} 
\end{theorem}

\begin{remark}
\begin{enumerate}
    \item The inequality~\eqref{eq:10040203} implies the following bound on the discrete gradient of the map $u$: For every $l \in [\mathcal{M}_{\mathrm{reg}}^s , L]$,
    \begin{equation} \label{eq:10180303}
    \left\| \nabla u \right\|_{\underline{L}^2(Q_l)} \leq \frac{C}{L} \left\| u - (u)_{Q_L} \right\|_{\underline{L}^2 (Q_L)} + C.
    \end{equation}
    Section~\ref{sec.discussionlargescale} discusses why inequalities of the type of~\eqref{eq:10180303}, providing a pointwise control on the gradient of a solution of the Langevin dynamics, can be useful in the context of random surfaces with degenerate potentials.
    \item Using the large-scale regularity theory and an interpolation argument, the quantitative hydrodynamic limit can be upgraded from a convergence in $L^2(Q)$ to a convergence in the space $C^{0,1-}(Q)$.
    \item Based on comparisons with stochastic homogenization, the optimal stochastic integrability exponent for the minimal scale~\eqref{eq:thmlargescale1234} should be $s = d$, the result is thus nearly optimal.
    \item The $C^2$ regularity of the surface tension stated in Theorem~\ref{t.regsurfacetension} implies that the solutions of the equation~\eqref{eq:defubarthmhydro} are $C^{1 , 1-}$. Hence it would be possible to transfer more regularity from the solutions of~\eqref{eq:defubarthmhydro} to the solutions of the Langevin dynamics and upgrade the $C^{1 , \alpha}$-large scale regularity of~\eqref{eq:10050203} to a $C^{1 , 1-}$-large scale regularity, at the cost of more technicalities.
\end{enumerate}
\end{remark}

In the following section, we discuss further the parallel between stochastic homogenization and the Langevin dynamics as well as extensions to related models and degenerate potentials.

\subsection{The Langevin dynamics as a stochastic homogenization problem} \label{sectionlangevinaspdewithnoise}

The standard problem of stochastic homogenization of nonlinear elliptic equations, following~\cite{AS, AFK1, AFK2, FN19, CG21}, is defined as follows. We consider a Lagrangian $L : (x , p) \mapsto L(x , p)$ with $x , p \in \Rd$ and assume that for any $x \in \Rd$, the map $p \mapsto L(x , p)$ is uniformly convex. We additionally assume that the Lagrangian~$L$ is random and that its law is stationary and ergodic with respect to the spatial translations. One is then interested in studying the large-scale behavior of the solutions of the nonlinear elliptic equation
\begin{equation} \label{eq:20370303}
    \nabla \cdot D_p L(x , \nabla u) = 0 ~\mbox{in}~ \Rd,
\end{equation}
where the notation $D_p L$ refers to the gradient with respect to the variable $p$ of the Lagrangian.

The starting point of our analysis is the observation that the Langevin dynamics~\eqref{Langevin.dyn} can be viewed as a (discrete) nonlinear parabolic equation with noise. Indeed, we should mentally make the replacement 
\begin{align}
\label{e.Langevin.to.PDE}
\sum_{e \ni x} V^\prime(\nabla \phi_t(e))
\quad
\leftrightsquigarrow 
\quad
\nabla \cdot D_p L(\nabla u),
\end{align}
where $p \mapsto L(p)$ is a uniformly convex Lagrangian, and think of~\eqref{Langevin.dyn} as being similar to the equation
\begin{align}
\label{e.Langevin.cartoon}
\partial_t u - \nabla \cdot D_p L (\nabla u) = \mbox{``noise''},
\end{align}
where the noise corresponds to the term involving the Brownian motions~\eqref{Langevin.dyn} and can be thought of as a discretized version (with respect to the space variable) of a space-time white noise. The equation~\eqref{e.Langevin.cartoon} can then be interpreted as a discrete and parabolic version of the equation~\eqref{eq:20370303} where the randomness is not encoded in the Lagrangian but externally through a random noise. A similar observation was made by Cardaliaguet, Dirr and Souganidis~\cite{cardaliaguet2022scaling}, who proved a qualitative homogenization result for a continuum version of the Langevin dynamics. 

The three following sections discuss further this analogy and some of the consequences which can deduced from it. They are summarized in Figure~\ref{fig:figure1}.

\subsubsection{The hydrodynamic limit as a homogenization theorem}
\label{subsec12430303}

The standard homogenization theorem for nonlinear elliptic equations~\cite{DM1,DM2} asserts that there exists a deterministic uniformly convex map $p \mapsto \bar L(p)$, called the \emph{homogeneous or effective Lagrangian}, such that any solution of~\eqref{eq:20370303} is well approximated over large scales by a solution $\bar u$ of the equation
\begin{equation} \label{eq:18270903}
    \nabla \cdot D_p \bar L(\nabla \bar u) = 0 ~\mbox{in}~ \Rd.
\end{equation}
In comparison to the previous statement, the \emph{hydrodynamic limit} for the $\nabla \phi$ model~\cite{FS} states that the solutions of the Langevin dynamics~\eqref{Langevin.dyn} are well-approximated over large-scales by the solution of the deterministic equation
\begin{align}
\label{e.hydrodynamic}
\partial_t \bar u - \nabla \cdot D_p \bar \sigma (\nabla \bar u) = 0.
\end{align}
One can thus view the \emph{hydrodynamic limit} as constituting a \emph{homogenization theorem}, where the surface tension $\bar \sigma$ has the same role as the effective Lagrangian. It turns out that this comparison can be further extended and that the hydrodynamic limit can be proved with the same technique as the one used to prove homogenization theorems, namely the two-scale expansion (see Section~\ref{sec1.5.2}).

An important ingredient in the implementation of a two-scale expansion is the first-order corrector, defined, for the model~\eqref{eq:20370303} and for any prescribed slope $p \in \Rd$, to be the unique function $x \mapsto \chi_p(x)$ (defined up to additive constant) such that the gradient $\nabla  \chi_p$ is spatially stationary and solution of the equation
\begin{equation*}
    \nabla \cdot D_p L(\cdot , p + \nabla \chi_p) = 0 ~\mbox{in}~ \Rd.
\end{equation*}
For the Langevin dynamic~\eqref{Langevin.dyn}, the first-order corrector corresponds to the stationary process $\phi(\cdot, \cdot ; p)$. Two properties are important on the first-order corrector to implement a two-scale expansion: (quantitative) sublinearity estimates on its fluctuations and on the weak-norm of its flux. These estimates are obtained in this article through concentration inequalities and deterministic regularity estimates (essentially the Nash-Aronson estimate and the De Giorgi-Nash-Moser regularity). Obtaining similar estimates in a degenerate setting would be an important ingredient to establish a quantitative hydrodynamic limit for dynamics with a degenerate potential (recent progress in this direction have been obtained by Peled and Magazinov in~\cite{magazinov2020concentration}).

We mention that, in the analogy between random surfaces and homogenization, bounds on the fluctuations of the first-order corrector corresponds to bounds on the fluctuations of the dynamic $\phi(\cdot, \cdot; p)$ which are closely related to estimating the fluctuations of a random surface distributed according to the Gibbs measure~\eqref{eq:defmuL}; in particular establishing \emph{localization or delocalization} for the random surface is equivalent to proving that the first-order corrector has \emph{bounded or unbounded} fluctuations.

We complete this section by discussing the extension of the results to some degenerate potentials such as the ones satisfying $c_- |x|^{p-2} \leq V(x) \leq c_+ |x|^{p-2}$, for some $c_- , c_+ > 0$ and $p > 2$. A critical difficulty to apply homogenization to nonlinear elliptic equation~\eqref{eq:20370303} with degenerate Lagrangian is that the solutions of the homogenized operator~\eqref{eq:18270903} may not possess good regularity properties. This lack of regularity for solutions is a consequence of the degenerate structure of the homogenized Lagrangian, which is itself related to the critical set of the first-order-corrector (i.e., the set of $x \in \Rd$ such that $p + \nabla \chi_p(x) =0$). This latter quantity is difficult to control which is an obstruction to the extension of the theory. 

We believe that, in the setting of the Langevin dynamics~\eqref{Langevin.dyn}, the situation is different and this difficulty can be dealt with. Indeed, the fact that the randomness is encoded outside the elliptic operator through the Brownian motions has a smoothing effect on the system. This phenomenon has been previously observed by Cotar, Deuschel and M\"{u}ller~\cite{CD12} who established the strict convexity of the surface tension for some non-uniformly convex potentials, and is also visible in Theorem~\ref{t.regsurfacetension} where we prove that the surface tension can be more regular that the potential $V$. We should therefore expect that the surface tension is (locally) uniformly convex for a large class of degenerate potentials (which are convex but not strictly convex), including the ones mentioned above. This would then imply, through standard parabolic regularity estimates, that the solutions of the limiting operator~\eqref{e.hydrodynamic} have some regularity, allowing the quantitative homogenization program to proceed. 

\subsubsection{The scaling limit via the linearized dynamics}

To understand the covariance structure of the field $\phi$ (resp. its discrete gradient~$\nabla \phi$) under $\mu_{\infty,p}$ (the field $\phi$ can only be considered in infinite volume in dimension $d \geq 3$) and to prove for instance a scaling limit, it is necessary to analyze the fluctuations of linear observables, that is, random variables of the form
\begin{align}
\label{e.linear.statistic}
\sum_{x \in \Zd} F(x) \phi(x) \hspace{10mm} \left( \mbox{resp.} \sum_{e\in E(\Zd)} F(e)\cdot \nabla \phi(e) \right)
\end{align}
where $F$ is a function (resp. vector field) of finite support. The Helffer-Sj\"{o}strand representation formula~\cite{NS, GOS} states that the variance of the linear observable~\eqref{e.linear.statistic} can be represented as follows
\begin{equation} \label{e.representation}
    \mathrm{Var}_{\mu_{\infty,p}} \left[\sum_{x \in \Zd} F(x) \phi(x)  \right] = \E \left[ \int_0^\infty   \sum_{x , x' \in \Zd} F(x)  P_\a (t , x ; x') F(x') \, dt \right],
\end{equation}
where $\a = V''(\nabla \phi(\cdot, \cdot; p))$, the symbol $\E$ denotes the expectation with respect to the dynamic and $P_\a$ is the heat-kernel under the environment $\a$, i.e., the solution of the parabolic equation
\begin{equation} \label{eq:15260503}
   \left\{ \begin{aligned}
    \partial_t P_\a(\cdot ; x) - \nabla \cdot \a \nabla P_\a (\cdot ; x) & = 0 &~\mbox{in}&~ (0 , \infty) \times \Zd, \\
     P_\a (0 , \cdot ; x) & = \delta_x  &~\mbox{in}&~\Zd,
    \end{aligned} \right.
\end{equation}
where $\delta_x$ denotes the Dirac at the vertex $x \in \Zd$.
Naddaf and Spencer~\cite{NS} (based on Helffer and Sj\"{o}strand~\cite{HS, Sj}) were the first to introduce a representation for the variance of a linear observable which is similar to~\eqref{e.representation} but is written instead with an elliptic (rather than parabolic) equation in ``infinitely many'' dimensions involving the so-called Witten Laplacian and known as the \emph{Helffer-Sj\"ostrand representation}. The scaling limit for the~$\nabla \phi$ field was first obtained in~\cite{NS} as a consequence of the homogenization of this equation (which has a corresponding homogenized matrix~$\ahom(p)$ which is evidently the same as the one above for~\eqref{eq:15260503}). Their technique was then reinterpreted probabilistically and extended by Giacomin, Olla and Spohn~\cite{GOS} who established the scaling limit of the space-time dynamics.

A quantitative version of this result was recently proved in~\cite{AW}, which lead to the identification of~$\ahom$ as the Hessian of the surface tension:
\begin{align}
\label{e.FDR}
\ahom(p) = D^2_p \bar \sigma(p)\,,
\end{align}
as well as the result on the~$C^2$ regularity of the surface tension~$\bar \sigma$ already mentioned above, under the assumption that $V \in C^{2 , \alpha}(\R)$. 
The identity~\eqref{e.FDR} is called the \emph{fluctuation-dissipation relation} for the $\nabla\phi$ model. It was conjectured in~\cite{GOS}, the main obstacle being the question of regularity of the surface tension. 

From the perspective of stochastic homogenization, the parabolic equation~\eqref{eq:15260503} corresponds to the \emph{linearized equation} associated with the Langevin dynamics, and can be homogenized quantitatively once a large-scale regularity has been established: indeed, it has been observed in~\cite{AFK1, AFK2} that a $C^{1 , \alpha}$ large-scale regularity, can be used to \emph{localize} the coefficient field $\a$, i.e., show that it is well-approximated by an environment with strong decorrelation properties, so that it fits in the framework of quantitative homogenization of parabolic equations with finite-range dependence~\cite{ABM}. The identity~\eqref{e.FDR} then states that the homogenized matrix for the linearized equation is the matrix for the linearization of the homogenized equation. In other words, \emph{homogenization and linearization commute}---a fact which has been recently observed in a very closely related context in~\cite{AFK1,AFK2}.

\subsubsection{A large-scale regularity theory for the Langevin dynamics and applications} \label{sec.discussionlargescale}
Establishing a quantitative version of the hydrodynamic limit has an important consequence (at least from the perspective of stochastic homogenization) as it allows to develop a large-scale regularity for the model. As a prospective application, we discuss how a large-scale regularity theory (which would itself be a consequence of the establishment of a quantitative hydrodynamic limit) in the setting of non-uniformly convex potentials can lead to localization estimates on random surfaces (in which case deterministic regularity does not apply).

To simplify the exposition, we will only explain the strategy on a formal level, consider a twice continuously differentiable convex potential of the form $V(x) = |x|^p$, for some $p > 2$, and let $\mu_{\Lambda_L, \mathbf{0}}$ be the finite-volume Gibbs measure~\eqref{eq:defmuL} in the box $\Lambda_L$ with slope $p = 0$. By the Helffer-Sj\"{o}strand representation formula, one has the identity
\begin{equation} \label{eq:19500503}
    \mathrm{Var}_{\mu_{\Lambda_L, \mathbf{0}}} \left[\phi(0)  \right] = \E \left[ \int_0^\infty P_{\a} (t , 0 )  \, dt \right],
\end{equation}
where $P_{\a}$ is the solution of the degenerate finite-volume parabolic equation 
\begin{equation*}
   \left\{ \begin{aligned}
    \partial_t P_{\a} - \nabla \cdot \a \nabla P_{\a} & = 0 &~\mbox{in}&~ (0 , \infty) \times \Lambda_L, \\
      P_{\a} & = \delta_0  &~\mbox{on}&~ \{ 0 \} \times \Lambda_L, \\
      P_{\a} &= 0  &~\mbox{on}&~ (0 , \infty) \times \partial \Lambda_L .
    \end{aligned} \right.
\end{equation*}
under the environment $\a(t , e) = p(p-1) |\nabla \phi(t , e)|^{p-2}$, and $\phi$ is the solution of the Langevin dynamic started at the time $t = 0$ from an initial profile sampled according to the Gibbs measure $\mu_{\Lambda_L, \mathbf{0}}$ independently of the Brownian motions. Assuming that a large-scale regularity can be established on the model, one obtains pointwise bounds on the discrete gradients $\nabla \phi(t , e)$ taking the following form: for any edge $e \in E(\Lambda_L)$, and any interval $I \subseteq (0 , \infty)$,
\begin{equation} \label{eq:18550503}
    \frac{1}{|I|}\int_I \left|\nabla \phi(t , e) \right|^{p-2} \, dt \leq \mathcal{O}_{\frac{1}{p-2}} (C).
\end{equation}
The estimate~\eqref{eq:18550503} is sufficient, as is explained in Section~\ref{OutlineC2reg} below (see also Proposition~\ref{prop:diffOrnstein-Uhlenbeck}), to prove that the dynamic $t \mapsto \nabla \phi(t , e)$ cannot spend (with high probability) more that a fraction $\epsilon$ of its time in an interval of size $\epsilon$. The previous statement implies that the integral $\int_I \a(t , e) \, dt = \int_I  |\nabla \phi(t , e)|^{p-2} \, dt$ is bounded away from zero with high probability. Building upon the techniques developed by Mourrat and Otto~\cite{mourrat2016anchored}, one can prove that the previous property implies upper bounds on the heat-kernel, which would then imply bounds on the variance of the height of the surface by the identity~\eqref{eq:19500503}. In the same direction, the results of Biskup and Rodriguez~\cite{biskup2018limit} would identify the correlation structure of the scaling limit of the model.

We mention that another approach to localization of degenerate random surfaces has been recently developed by Magazinov and Peled~\cite{magazinov2020concentration} and is based on reflection positivity (following~\cite{FSS}); they obtain sharp localization and delocalization estimates for periodic systems with non-uniformly convex potentials, as well as super-Gaussian stochastic integrability for potentials of the form $V(x) = x^2 + |x|^p$ with $p > 2$.

\begin{figure}
\adjustbox{scale=0.83,center}{%
\begin{tikzcd}[column sep=small]
\begin{tabular}{c}
 \mbox{Sublinearity of the corrector} \\
 \mbox{and its flux}
 \end{tabular}
 \arrow[d, Rightarrow] \arrow[rr, leftrightarrow]
& & \begin{tabular}{c}
 \mbox{Sublinearity for the random surface} \\
 \mbox{and its flux}
 \end{tabular} \arrow[d, Rightarrow] \\
 \mbox{Quantitative Homogenization} \arrow[dr, Rightarrow] \arrow[rr, leftrightarrow]  & & \mbox{Quantitative Hydrodynamic limit} \arrow[ld, Rightarrow] \\
& \mbox{Large-scale regularity}  \arrow[d, Rightarrow] & \\
&
\begin{tabular}{c}
\mbox{Localization for the linearized environment} \\
$\a = V''(\nabla \phi)$
\end{tabular}
\arrow[ld, Rightarrow] \arrow[rd, Rightarrow] & \\
\begin{tabular}{c}
 \mbox{Quantitative homogenization} \\
 \mbox{for the linearized equation}
 \end{tabular}
 \arrow[rr, leftrightarrow] &  & \mbox{Quantitative scaling limit} .
\end{tikzcd}
}
\caption{This figure summarizes the analogy between the Langevin dynamics and stochastic homogenization as well as the various implications described in Section~\ref{sec.discussionlargescale}. The double arrow $\leftrightarrow$ refers to results playing the same role in the two theories.} 
\label{fig:figure1}
\end{figure}
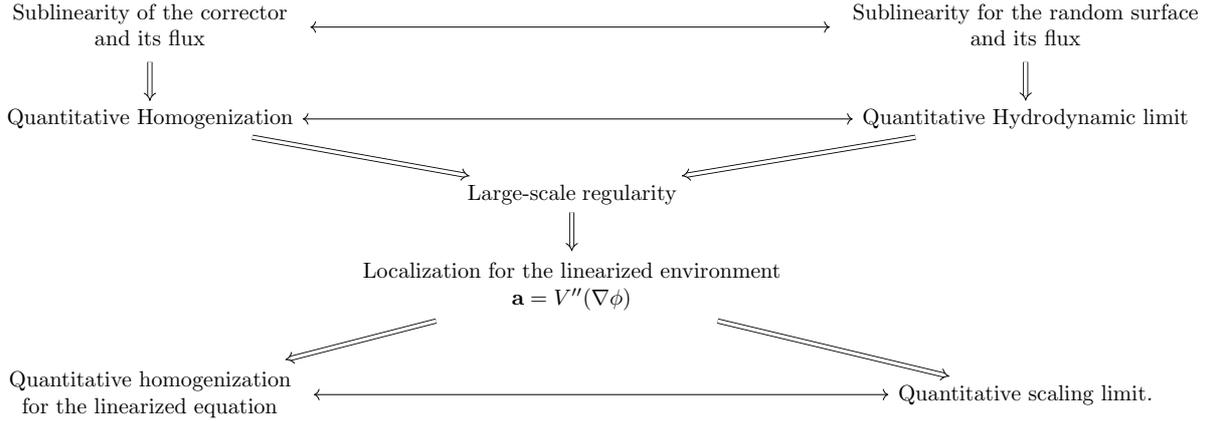

\subsubsection{Disordered random surfaces}

The analogy and techniques of proofs can be extended to other models of statistical mechanics; in particular to some models of random surfaces in the presence of a random disorder which have been studied in the literature~\cite{KO06, KO08, VK08, CK12, CK15}, and are described below:
\begin{itemize}
    \item \emph{The random conductance model}: In this model, we consider a collection of i.i.d. of uniformly convex potentials $(V_e)_{e \in E(\Zd)}$ indexed over the edges of the lattice. For each realization of the potentials, we consider the Gibbs measure
    \begin{equation*}
        \mu^{(V_e)}_\Lambda := \frac{1}{Z} \exp \Biggl( - \sum_{e \in E(\Lambda^+)} V_e ( \nabla \phi (e) ) \Biggr) \prod_{x \in \Lambda_L} d\phi(x),
    \end{equation*}
    and the corresponding Langevin dynamic
    \begin{equation} \label{Langevin.dyndisord1}
    d \phi(t, x) = \nabla \cdot V'_e( \nabla \phi)(t,x) \, dt + \sqrt{2} dB_t(x).
\end{equation}
This model has been studied by Cotar and K\"{u}lske in~\cite{CK12, CK15} where they establish existence and uniqueness of translation-covariant Gibbs states.
The dynamic~\eqref{Langevin.dyndisord1} is a combination of the model~\eqref{eq:20370303} and~\eqref{e.Langevin.cartoon}, in the sense that the randomness is encoded both inside the potential and externally through the Brownian motion. In particular, a quantitative version of the hydrodynamic limit should be accessible by adapting techniques developed in this article with the tools of stochastic homogenization~\cite[Chapter 10]{AKMbook} (though the proofs should be more technical due to the fact that the randomness is partly encoded in the potential).
    \item \emph{The random-field random surface model}: We consider a collection of i.i.d. Gaussian random variables $(\eta(x))_{x \in \Zd}$ indexed over the vertices of the lattice, called the \emph{random field}. For each realization of the random field, we consider the Gibbs measure
    \begin{equation*}
        \mu^\eta_{\Lambda} := \frac{1}{Z} \exp \Biggl( - \sum_{e \in E(\Lambda^+)} V ( \nabla \phi (e) ) + \sum_{x \in \Lambda} \eta(x) \phi(x)\Biggr) \prod_{x \in \Lambda_L} d\phi(x),
    \end{equation*}
    and the corresponding Langevin dynamic
    \begin{equation*}
    d \phi(t,x) =\nabla \cdot V'( \nabla \phi)(t , x) \, dt + \eta(x) \, dt +  \sqrt{2} dB_t(x).
\end{equation*}
Various aspects of the models have been studied in the contributions~\cite{KO06, KO08, VK08, CK12, CK15, DHP1, dario2021convergence} regarding localization and delocalization of the surface as well as existence and uniqueness of infinite-volume translation-covariant gradient Gibbs measures. The model of random-field random surfaces is similar to the one considered in this article and fits in the category~\eqref{e.Langevin.cartoon}. It can in fact be treated with the same techniques as the ones used in this article and we believe that a quantitative version of the hydrodynamic limit as well as a large-scale regularity and the $C^2$-regularity of its surface tension could be obtained with a (mostly notational) modification of the arguments in dimension $d \geq 3$ (the result does not apply in dimension $d = 1,2$ as the random surface is known to have super-linear fluctuations in this case, see~\cite[Theorem 2]{DHP1}). We mention that the large-scale regularity can be combined with the results of~\cite[Theorem 2]{DHP1} to deduce the following pointwise estimate on the discrete gradient of the random surface: in dimension $d \geq 3$ there exists an exponent $s := s(d, c_+,c_-) >0$ such that, for any $L \in \N$ and any edge $e \in E \left(\Lambda_{L/2}\right),$
\begin{equation} \label{eq:10441703}
    \E_\eta \left[ \left\langle \exp \left( \left|\nabla \phi(e)\right|^s \right) \right\rangle_{\mu_{\Lambda_L}^\eta} \right] \leq 
    C 
\end{equation}
where the $\left\langle \cdot \right\rangle_{\mu_{\Lambda_L}^\eta}$ denotes the expectation with respect to $\mu_{\Lambda_L}^\eta$, and $\E_\eta$ refers to the expectation with respect to the random field. This result would improve the spatial $L^2$-estimate obtained in~\cite[Theorem 1]{DHP1}, as well as the results of~\cite{dario2021convergence}, where the bound~\eqref{eq:10441703} is essentially proved in dimension $d \geq 4$ using the De Giorgi-Nash-Moser regularity. In dimension $d = 3$, the estimate~\eqref{eq:10441703} cannot be deduced from deterministic regularity estimates, and requires to use the large-scale regularity.
\end{itemize}

\subsection{Background} \label{section.background}

\subsubsection{Random surfaces}

The study of random surfaces was initiated in the 1970s by Brascamp, Lieb and Lebowitz~\cite{BLL75} who obtained sharp localization and delocalization estimates for uniformly convex potentials. The question of the scaling limit of the model was first addressed by Brydges and Yau~\cite{BY} in a perturbative setting based on a renormalization group approach. After the groundbreaking works of Funaki, Spohn~\cite{FS}, Naddaf, Spencer~\cite{NS} and Giacomin, Olla, Spohn~\cite{GOS}, large deviation estimates and concentration inequalities were established by Deuschel, Giacomin and Ioffe~\cite{DGI00}, and sharp decorrelation estimates for the discrete gradient of the field were obtained by Delmotte and Deuschel~\cite{DD05}. The scaling limit of the field in finite-volume was established by Miller~\cite{Mi}. More recently, Armstrong and Wu applied quantitative homogenization to the Helffer-Sj\"{o}strand PDE of~\cite{NS} to prove the $C^2$ regularity of the surface tension and the fluctuation-dissipation relation (see also the recent subsequent work of Wu~\cite{wu2022local}), and Deuschel and Rodriguez~\cite{deuschel2022isomorphism} identified the scaling limit of the square of the gradient field (hereby extending the result of Naddaf-Spencer~\cite{NS}) and established (among other results) an isomorphism theorem for the model (see~\cite[Theorem 4.3]{deuschel2022isomorphism}).

The case of non-uniformly convex potentials was studied in the high temperature regime by Cotar, Deuschel and M\"{u}ller~\cite{CDM09}, who established the strict convexity of the surface tension, and by Cotar and Deuschel~\cite{CD12} who proved the uniqueness of ergodic Gibbs measures, obtained sharp estimates on the decay of covariance and identified the scaling limit of the model in this framework (see also~\cite{DNV19} where the hydrodynamic limit is established).
The strict convexity of the surface tension in the low temperature regime was established by Adams, Koteck{\'y} and M\"{u}ller~\cite{AKM16} through a renormalization group argument. This renormalization group approach was further developed in~\cite{adams2019cauchy} to obtain a (form of) verification of the Cauchy-Born rule for these models. In~\cite{BK07}, Biskup and Koteck{\'y} showed the possible non-uniqueness of infinite-volume, shift-ergodic gradient Gibbs measures for some nonconvex interaction potentials, and Biskup and Spohn~\cite{BS11} proved that, for a general category of nonconvex potentials, the scaling limit of the model is a Gaussian free field. We finally mention the contribution of Magazinov and Peled~\cite{magazinov2020concentration}, who established sharp localization and delocalization estimates for a class of convex but highly degenerate potential $V$, and the one of Andres and Taylor~\cite{AT21} who identified the scaling limit of the field for a class of convex, degenerate potentials satisfying the assumption $\inf V'' \geq \lambda > 0.$

\subsubsection{Quantitative stochastic homogenization}
The theory of stochastic homogenization was initially developed qualitatively in the 80's, in the works of Kozlov~\cite{K1}, Papanicolaou and Varadhan~\cite{PV1}, and Yurinski\u\i~\cite{Y1}. The first quantitative results are due to Yurinski\u\i~\cite{Y1} and Naddaf, Spencer~\cite{NS1}. Major progress was achieved by Gloria and Otto in~\cite{GO1, GO2}, where, building upon the ideas of~\cite{NS1}, they used spectral gap inequalities (or concentration inequalities) to develop a quantitative theory for the first time. These results were then further developed by Gloria, Marahrens, Neukamm and Otto~\cite{GO15, GO115, GNO, GNO14}. The technique used in this article to obtain bounds on the first-order corrector is based on correlation inequalities (see Section~\ref{sec.1.5.1}), and is closely related to their approach.

Another approach was initiated by Armstrong and Smart in~\cite{AS}, who extended the techniques of Avellaneda and Lin~\cite{AL1, AL2}, the ones of Dal Maso and Modica~\cite{DM1, DM2} and obtained an algebraic rate of convergence for the homogenization error the nonlinear equation~\eqref{eq:20370303} and a large-scale regularity for the model. These results were subsequently improved in the linear setting in~\cite{AKM1, AKM, AKMbook} to obtain optimal rates of convergence.

The homogenization of nonlinear monotone operator was first addressed qualitatively in~\cite{DM1, DM2}, and extended to a class of nonconvex Lagrangians with polynomial growth in~\cite{MM94} (see also~\cite[Chapter 15]{JKO}). 
The first quantitative results were obtained in~\cite{AS}, and the theory was substantially extended recently by Armstrong, Ferguson and Kuusi in~\cite{AFK1, AFK2} where quantitative homogenization and a large-scale regularity are proved for the nonlinear equation~\eqref{eq:20370303} and the corresponding linearized equation (see also the subsequent work of Fischer and Neukamm~\cite{FN19}).

In the parabolic framework, qualitative homogenization was established by Zhikov, Kozlov, and Ole\u{\i}nik in~\cite{ZKO82}. From a probabilistic perspective, quenched invariance principles were proved for random walks evolving in a dynamic, degenerate and ergodic environment by Andres, Chiarini, Deuschel and Slowik~\cite{ACDS18} and a quenched local limit theorem was established by Andres, Chiarini and Slowik in~\cite{andres2021quenched}, Biskup and Rodriguez~\cite{biskup2018limit}, and a local limit theorem was proved by Andres and Taylor~\cite{AT21}. Optimal quantitative estimates on the heat-kernel and the corrector were obtained by Gloria, Neukamm and Otto in~\cite{GNO}. Quantitative homogenization and large-scale regularity was established by Armstrong, Bordas and Mourrat~\cite{ABM}, following the techniques in the elliptic setting of~\cite{AKMbook}.

\subsection{Outline  of the proof and additional comments}

In this subsection, we present a sketch of the proofs of the results presented in this article. The arguments follow a standard strategy in stochastic homogenization: in Section~\ref{sec.section2}, we define and study the first-order corrector for the Langevin dynamic, and obtain optimal scaling estimates for the corrector and its flux (these results are related to the Brascamp-Lieb concentration inequality~\cite{BL75, BL76} as discussed in Section~\ref{sec.1.5.1} below). In Section~\ref{eq:quanthydrolim2sc}, we use the results established on the first-order corrector in Section~\ref{sec.section2} and perform a two-scale expansion on the nonlinear parabolic equation with a random noise to obtain the quantitative hydrodynamic limit stated in Theorem~\ref{Th.quantitativehydr}. Section~\ref{sec.section4} is devoted to the proof of the $C^2$-regularity of the surface tension and is (mostly) independent of Section~\ref{sec.section2} and Section~\ref{eq:quanthydrolim2sc}. Finally, Section~\ref{sec.Section5} is devoted to the proof of Theorem~\ref{theoremlargescale} based on the results of the four previous sections. In the next four subsections, we present a more detailed outline of the main arguments developed in Sections~\ref{sec.section2},~\ref{eq:quanthydrolim2sc},~\ref{sec.section4} and Section~\ref{sec.Section5} respectively.

\subsubsection{Scaling estimates on the first-order corrector of the Langevin dynamic} \label{sec.1.5.1}

The main tool used to prove Theorem~\ref{Th.quantitativehydr} is the first-order corrector for the Langevin dynamic defined in Definition~\ref{def.firstordercorrfinvol}. In Section~\ref{sec.section2}, we establish some properties of the first-order corrector based on tools of elliptic regularity and concentration inequalities (following the strategy of Gloria and Otto~\cite{GO1, GO2}): bounds on the fluctuations of the corrector and its flux, Lipschitz bounds on the gradient of the corrector, estimates on the $L^2$-norm of the difference of gradient of the correctors with different slopes etc.

The proofs rely on two main techniques:
\begin{enumerate}
    \item If we let $Q$ be a parabolic cylinder, and let $\phi$ and $\psi$ be two solutions of the Langevin dynamic coupled with the same Brownian motions
    \begin{equation*}
        d \phi(t , x )  = \nabla \cdot V'(\nabla \phi) (t ,x) dt + \sqrt{2} dB_t(x)~~\mbox{in}~Q,
    \end{equation*}
    and
    \begin{equation*}
        d \psi(t , x ) = \nabla \cdot V'(\nabla \psi) (t ,x) dt + \sqrt{2} dB_t(x) ~~\mbox{in}~Q,
    \end{equation*}
    then the difference $w = \phi - \psi$ solves the \emph{linear parabolic equation}
    \begin{equation} \label{eq:14001310}
        \partial_t w - \nabla \cdot \a \nabla w = 0 ~\mbox{in}~Q,
    \end{equation}
    where the environment is defined by the formula
    \begin{equation*}
        \a(t , e) := \int_0^1 V''(s \nabla \psi(t , e ) + (1-s) \nabla \phi(t , e )) \, ds.
    \end{equation*}
We may thus use the classical theory of regularity for solutions of parabolic equations to study the behavior of the difference of solutions; this observation was originally used by Funaki and Spohn~\cite{FS} to prove the uniqueness of infinite volume shift-invariant and ergodic gradient Gibbs measure with prescribed slope.
    \item Given an integer $L\in \N$, let us define $\phi_{L} : [0, \infty) \times \Lambda_L \to \R$ be the solution of the system of stochastic equations
\begin{equation} \label{Langevin.def.corr}
    \left\{ \begin{aligned}
   & d \phi_{L}(t , x ) = \nabla \cdot V'(\nabla \phi_{L}) (t ,x) dt + \sqrt{2} dB_t(x) & \mbox{for} & \ (t,x) \in [0, \infty) \times \Lambda_L,  \\
  &  \phi_{L}(0 , \cdot ) = 0 & \mbox{in} & \ \Lambda_{L},
    \end{aligned} \right.
\end{equation}
    with periodic boundary conditions. The function $\phi_L$ is the finite-volume first-order corrector with slope $p = 0$ used in the article, and one would like to estimate the size of its fluctuations. This is achieved based on techniques of elliptic regularity and concentration inequalities as explained below. We first approximate map $\phi_L$ using a discretization of the Brownian motion in the right-hand side of~\eqref{Langevin.def.corr}: for a large integer $n \in \N$ and $k \in \N$, denote by $X_k^n(x)$ the increments of the Brownian motion $B_\cdot(x)$ between the times $k/n$ and $(k+1)/n$, i.e.,
\begin{equation*}
    X_k^n(x) :=  \sqrt{n}\left(B_{\frac{k+1}{n}}(x) - B_{\frac{k}{n}}(x) \right) \hspace{5mm} \mbox{and} \hspace{5mm} X^n(t , x) = \sum_{k} X_k^n(x) \indc_{\{k/n \leq t \leq (k+1)/n\}},
\end{equation*}
and consider an approximation $\phi_L^n$ of the map $\phi_L$ defined to be the solution of the parabolic equation
\begin{equation*}
    \left\{ \begin{aligned}
    \partial_t \phi_{L}^n - \nabla \cdot V'(\nabla \phi_{L}) & =   \sqrt{2} n^{-\frac12} X^n &~\mbox{in} ~~&~  [0, \infty) \times \Lambda_L, \\
    \phi_{L}(0 , \cdot ) &= 0 &~\mbox{in}~~&~ \Lambda_{L},
    \end{aligned} \right.
\end{equation*}
with periodic boundary conditions in the box $\Lambda_L$.
The argument then relies on the two following observations: as the mesh size $n$ tends to infinity the map $\phi_{L}^n$ converges to the map $\phi_{L}$; and the function $\phi_{L}^n$ depends only on the increments $X_k^n(x)$. Using the Gaussian concentration inequality (see Proposition~\ref{propEfronstein}) and the Nash-Aronson estimate (see Proposition~\ref{prop.NashAronson}), one can estimate the fluctuations of $\phi_{L}^n$ by measuring the influence of each one of the increments $X_k^n(x)$.
\end{enumerate}
The two previous tools allow to obtain optimal scaling estimates on the corrector. The technique can be used to study other observables: we obtain, by adapting the argument, optimal scaling estimates on the flux of the corrector in Section~\ref{sec.fluxcorr} and Section~\ref{section4.4}. These properties are, as it is well-known in homogenization, crucial to implement a two-scale expansion (see Section~\ref{sec1.5.2} below) and obtain the quantitative hydrodynamic limit stated in Theorem~\ref{Th.quantitativehydr}.

We complete this section by making a few remarks and connections with results already known in the literature. First, it is known that the Langevin dynamic~\eqref{Langevin.def.corr} is ergodic and reversible with respect to the Gibbs measure
\begin{equation} \label{measure.randomsurface}
    \mu(d\phi) := \frac{1}{Z} \exp \Biggl( - \sum_{e \in \vec{E}_+(\Lambda^+)} V ( \nabla \phi (e) ) \Biggr) \prod_{x \in \Lambda_L} d\phi(x),
\end{equation}
where the measure is considered over the set of periodic functions $\phi : \Lambda_L \to \R$ with spatial average equal to $0$.
The Brascamp-Lieb concentration inequality~\cite{BL75, BL76} estimates the variance of general functionals of a random field distributed according to the measure $\mu$. In the present setting, it can be stated as follows: given a continuously differentiable function $f : \R^{\Lambda_L} \to \R$, one has
\begin{equation*}
\mathrm{var}_\mu \left[ f \right] \leq \sum_{x , y \in \Lambda_L} \left\langle \partial_x f(\phi) G_{\Lambda_L} (x , y) \partial_y f(\phi) \right\rangle_\mu,
\end{equation*}
where $G_{\Lambda_L}$ denotes the periodic Green's function in the box $\Lambda_L$.
Various refinements of this inequality exist in the literature: Deuschel, Giacomin and Ioffe~\cite{DGI00} and Delmotte Deuschel~\cite{DD05} obtained an exponential version of the result, a Gaussian version of it (similar to Proposition~\ref{prop.2.2BL} below) can be found in~\cite{fontaine1983non} (see~\cite[Section 4.2]{F05} for additional discussion and references). The result is of course closely related to the one presented in this section: by the ergodicity of the dynamic, we know that the distribution of the random field $\phi_L(t , \cdot)$ converges to $\mu$ as the time $t$ tends to infinity. Since the bound obtained in Proposition~\ref{prop.2.2BL} does not depend on the time $t$, we may take the limit $t \to \infty$ to obtain an alternative proof of the Brascamp-Lieb inequality, based on the Gaussian concentration inequality and elliptic regularity. Using this technique instead of relying directly on the Brascamp-Lieb inequality will be useful to us in different ways: it yields sharp concentration estimates for nonlinear observables of the random field with optimal stochastic integrability, allows to study correlations for the dynamics in both the space and time variables, and allows to define a corrector with a slope which depends on the time variable (see Definition~\ref{def.firstordercorrfinvol} and Remark~\ref{remark3.2225}).

\subsubsection{Quantitative hydrodynamic limit for the Langevin dynamic}
\label{sec1.5.2}
Once precise estimates have been obtained on the first-order-corrector and its flux, Theorem~\ref{Th.quantitativehydr} can be deduced by performing a two-scale expansion on the model. On a formal level and for any $\ep > 0$, we introduce the two-scale expansion $w^\ep$ according to the formula
\begin{equation} \label{eq:opttwoscaleexp}
    w^\ep := \bar u^\ep + \ep \phi_{1/\ep} \left( \frac{\cdot}{\ep^2} , \frac{\cdot}{\ep} ; \nabla   \bar u^\ep \right),
\end{equation}
where $\phi_{1/\ep} \left( \cdot , \cdot ; \nabla   \bar u^\ep \right)$ denotes the corrector with slope $p$ (see Definition~\ref{def.firstordercorrfinvol}) in a box of radius $1/\ep$. We then show, through an explicit computation using the properties of the corrector established in Section~\ref{sec.section2}, that \begin{equation*}
    \partial_t \left(  w^\ep  - B_t \right) + \nabla \cdot V'(\nabla w^\ep ) = \mathcal{E},
\end{equation*}
where $\mathcal{E}$ is an error term which is small in the suitable functional space (the norm $\underline{H}^{-1}_{\mathrm{par}}$ introduced in Section~\ref{sectionholdernorms}).

As stated, this strategy runs into a difficulty. First, as it has been recently observed in the setting of nonlinear elliptic homogenization in~\cite{AFK1, AFK2, FN19, CG21}, in order to implement the two-scale expansion~\eqref{eq:opttwoscaleexp}, and obtain the optimal rate of convergence, one would need to have good control over the corrector associated with the \emph{linearized equation}, defined, for a pair of slopes $p , q \in \Rd$, to be the map $\psi_{p , q} : (0 , \infty) \times \Zd \to \R$ solution of the parabolic equation
\begin{equation} \label{eq:15151310}
    \partial_t \psi_{p , q} + \nabla \cdot \a_p \left( q + \nabla \psi_{p , q} \right)  = 0 \hspace{3mm} \mbox{in} \hspace{3mm} (0 , \infty) \times \Zd,
\end{equation}
where the environment $\a_p$ is given by the formula
\begin{equation} \label{eq:17172109}
    \a_p(t, e) := V'' ( p + \nabla \phi (t , e ; p) ).
\end{equation}
Such estimates are not established in this article and, in order to bypass this difficulty, we introduce a mesoscopic scale of size $\ep^{\frac 12}(1 + \left| \ln \ep \right|^{\frac12} \indc_{\{ d = 2\}})$ in the argument. This is responsible for the loss of half of the exponent compared to the optimal result. We nevertheless mention that optimal scaling estimates on the corrector associated with the linearized equation~\eqref{eq:15151310} (and thus and optimal rate of convergence for the hydrodynamic limit) could be obtained by the following strategy (see~\cite{FN19, CG21}):
\begin{enumerate}
    \item Using the De Giorgi-Nash-Moser regularity, one can prove that the environment defined in~\eqref{eq:17172109} satisfies a quantitative ergodicity assumption, and deduce from it a quantitative homogenization theorem for the linearized equation;
    \item One then establishes a large-scale regularity theory for the solutions of the linearized equation;
    \item Applying concentration estimates to the linearized equation~\eqref{eq:15151310} (following a similar, though more involved, argument to the one outlined in Section~\ref{sec.1.5.1}) and combining them with the large-scale regularity for the linearized equation, one obtains optimal scaling estimates for the corrector $\psi_{p , q}$ and its flux. Once equipped with these estimates, the two-scale expansion~\eqref{eq:opttwoscaleexp} can be used to deduce the hydrodynamic limit with an optimal rate of convergence.
\end{enumerate}

\subsubsection{$C^2$-regularity of the surface tension} \label{OutlineC2reg}
The proof of the $C^2$-regularity of the surface tension is (essentially) independent of the results established in Sections~\ref{sec.section2} and~\ref{eq:quanthydrolim2sc}, and follows the insight of~\cite[Section 2.3]{AFK1}. The proof relies on the introduction of a finite-volume approximation of the gradient of the surface tension denoted by $\tau_L$ in the proof below and defined by the formula (using the notation~\eqref{def.QL} and~\eqref{eq:defavrageoverQ})
\begin{equation} \label{def.tauL0901}
   \tau_L (p) := \mathbb{E} \left[ (V'( p + \nabla \phi_{L}(\cdot ; p) ))_{Q_{L/2}} \right].
\end{equation}
We then establish that the map $p \mapsto \tau_L (p)$ converges pointwise to the function $p \mapsto D_p \bar \sigma(p)$ as $L$ tends to infinity, that it is continuously differentiable and that the collection of functions $\left\{ p \mapsto D_p\tau_L(p) \,:\, L\in\N \right\}$ is equicontinuous. Theorem~\ref{t.regsurfacetension} is then obtained by applying the Arzel\`a-Ascoli Theorem.

To present a few more details of the proof, the derivative of the function $\tau_L$ is explicit and is given by the formula, for any $p \in \Rd$ and any $i \in \{1 , \ldots, d \}$,
    \begin{equation*}
        \partial_{i} \tau_L(p) = \mathbb{E} \left[ (V''( p + \nabla \phi_{L}(\cdot ; p)) (e_i + \nabla w_{p, e_i}) )_{Q_{L/2}} \right],
    \end{equation*}
where $w_{p , e_i}$ is defined as the solution of a parabolic equation (see~\eqref{def.linearizedcorr}). The continuity of the map $\partial_{i} \tau_L$, uniform over the parameter $L$, can be established by proving that the quantities
\begin{equation*}
    \mathbb{E} \left[ \left\| V''( p + \nabla \phi_{L}(\cdot ; p) - V''( q + \nabla \phi_{L}(\cdot ; q)) \right\|_{\underline{L}^1(Q_L)} \right] \hspace{4mm} \mbox{and} \hspace{4mm} \mathbb{E} \left[ \left\| \nabla w_{p, e_i} - \nabla w_{q, e_i} \right\|_{\underline{L}^2(Q_L)} \right]
\end{equation*}
tend to $0$ as $q$ tends to $p$ (uniformly in $L$). By arguments of elliptic regularity, the treatment of the second term can be reduced to the first one, which essentially boils down to proving the following convergence
\begin{equation} \label{13071510}
    \E \Bigl[ \left\|  V'' ( p + \nabla \phi_L(\cdot ; p) + \ep ) - V'' ( p + \nabla \phi_L(\cdot ; p) ) \right\|_{\underline{L}^1(Q_L)}\Bigr] \underset{\ep \to 0}{\longrightarrow} 0 \quad \mbox{uniformly in} \ L \in\N.
\end{equation}
In the case when the potential $V$ is assumed to be H\"{o}lder continuous, the proof of~\eqref{13071510} is immediate, yielding and algebraic rate of convergence in the parameter $\ep$ which can be transferred back to the surface tension, establishing the H\"{o}lder continuity of its second derivative.

In the case of $C^{1,1}$ potentials, the strategy is to appeal to Lusin's theorem to approximate the map $V''$ by a Lipschitz function, which we denote by $V_\kappa''$, on a set of large measure. We then prove the uniform convergence~\eqref{13071510} with the function $V''_\kappa$ instead of $V''$, and show that this approximation procedure only generates a small error (uniformly in $L$) in the convergence~\eqref{13071510}.

To implement this strategy, we need to prove that the map $\nabla \phi_L $ does not spend a large amount of time in the set of small measure where the function $V''$ is poorly approximated by the function $V''_\kappa$. This pathological behavior is ruled out by noting that, for each $x \in \Lambda_L$, the map $t \mapsto \phi_L (t , x)$ solves the stochastic differential equation
\begin{equation} \label{eq:16121410}
    d \phi_L (t , x) := \underset{\mathrm{drift \, term}}{\underbrace{\nabla \cdot V' \left(\nabla \phi_L \right) (t , x) \, dt}} +  \underset{\mathrm{noise}}{\underbrace{\sqrt{2} dB_t(x)}}.
\end{equation}
Using Proposition~\ref{prop.2.2BL} and specifically the bound~\eqref{eq:20291910grad} on the gradient of the field established in Section~\ref{sec.section2} (which is the only, but crucial, input from Sections~\ref{sec.section2} and~\ref{eq:quanthydrolim2sc}), we obtain that the elliptic operator appearing in the drift term of~\eqref{eq:16121410} is (essentially) bounded. One may thus rewrite the identity~\eqref{eq:16121410} as follows: for any $x \in \Lambda_L$ and any pair of times $t_+ , t_- \geq 0,$
\begin{equation*}
    \phi_L (t_+ , x ) - \phi_L (t_- , x ) := \underset{\mathrm{Lipschitz \, function}}{\underbrace{\int_{t_-}^{t_+} h_s ds}} +  \underset{\mathrm{Brownian \, motion}}{\underbrace{ \sqrt 2 (B_{t_+} - B_{t_-})}},
\end{equation*}
where the map $h_s := \nabla \cdot V' \left(\nabla \phi_L \right)(t , x) $ is (essentially) bounded. Using that the Brownian motion is a rough process, and in particular much less regular than a Lipschitz function, we are able to deduce that, for any $\delta \in [0,1]$, the process $\phi_L$ cannot spend more than a fraction $\delta$ of its time in a set of Lebesgue measure less than $\delta$ (see Lemma~\ref{section5.1}). Applying this result to the set of small Lebesgue measure where the map $V''$ is poorly approximated by the map $V''_\kappa$, we are able to rule out the pathological behavior mentioned above, and thus establish the uniform convergence~\eqref{13071510}.


\subsubsection{Large-scale regularity theory for the Langevin dynamic} \label{outlinelargescale}
Section~\ref{sec.Section5} is devoted to the proof of Theorem~\ref{theoremlargescale}. Based on the ideas of~\cite{AS}, we can use the quantitative hydrodynamic limit to show that any solution of the Langevin dynamic is well-approximated, over large scales where homogenization occurs, by a solution of the equation~\eqref{eq:defubarthmhydro} which possesses good regularity properties. We are then able to transfer the regularity of the solutions of~\eqref{eq:defubarthmhydro} to the solution of the Langevin dynamic. 

We mention the following technical point in the argument: in order to optimize the stochastic integrability of the minimal scale~\eqref{eq:thmlargescale1234}, Theorem~\ref{Th.quantitativehydr} cannot be applied directly, and we have to write a second version of it, optimizing the stochastic error at the cost of a deterministic term. This step is the subject of Section~\ref{Optimizingstochint}.

\subsection{Convention for constants} Throughout this article, the symbols $C$ and $c$ denote positive constants which may vary from line to line, with $C$ increasing and $c$ decreasing. These constants and exponents may depend only on the dimension $d$, and the ellipticity constants $c_+$ and $c_-$. We specify the dependency of the constants and exponents by writing, for instance, $C := C(d , , c_- c_+)$ to mean that the constant $C$ depends on the parameters $d$, $c_-$ and $c_+$.

\medskip

\section{Notation and preliminary results} \label{Section2prelimres}

We must unfortunately introduce quite a bit of notation, particularly since we are working with parabolic equations which require us to define
various function spaces, and since we are working in both the discrete and continuous settings which require to introduce different definitions for the differential calculus. We also need to collect some preliminary results pertaining to elliptic regularity (e.g., Nash-Aronson estimate on the heat kernel and De Giorgi-Nash-Moser regularity), concentration inequality (the Gaussian concentration inequality), and stochastic homogenization (the multiscale Poincar\'e inequality); most of them are standard in the literature. We thus encourage the reader to skim and consult this section as a reference.

\subsection{Notation} \label{sectionnotation}

\subsubsection{General notation}
Consider the hypercubic lattice $\Zd$ and the real vector space $\Rd$ in dimension $d \geq 2$. We let $\vec{E}(\Zd)$ and $E(\Zd)$ be the sets of directed and undirected edges of the lattice. We denote by $(e_1 , \ldots , e_d)$ the canonical basis of $\Rd$, and, for $x , y \in \Rd$ (or $\R^{2d}$), we use the notation $x \cdot y$ to refer to the Euclidean scalar product on the spaces $\Rd$ (or $\R^{2d}$). We denote by $\left|\cdot \right|$ the standard Euclidean norm on $\Rd$ (we sometimes abuse notation and use it to refer to the Euclidean norm on $\R^{2d}$). We write $\left|\cdot \right|_+ := \max \left( \left|\cdot \right| , 1 \right)$. Given two real numbers $a , b$, we denote by $a \wedge b = \min(a , b)$ and by $a \vee b = \max(a , b)$, and by $\lfloor a \rfloor$ the floor of $a$. We denote by $\indc_A$ the indicator function of a set $A$ ,and by $\mathcal{B}(\R)$ the Borel $\sigma$-algebra of $\R$.

Throughout the article, we fix two collections of independent Brownian motions \newline $\left\{ B_t^0(x) \, : \, t \geq 0, \, x \in \Zd \right\}$ and $\{ B_t^1(x) \, : \, t \geq 0, \, x \in \Zd \}$ and denote by, for $t \in \R$ and $x \in \Zd$, 
\begin{equation} \label{eq:14350901}
    B_t(x) := B_t^0(x) \indc_{\{ t \geq 0 \}} + B_{-t}^1(x) \indc_{\{ t < 0 \}}.
\end{equation}
This definition gives a (simple) notion of Brownian motion defined on the full line $\R$.

\subsubsection{Sets} \label{subSectionsets}

Given a subset $U \subseteq \Zd$, we let ${\vec E}(U)$ (resp. $E(U)$) be the set of directed (resp. undirected) edges of $U$. We write $x\sim y$ to denote that $\{x,y\}\in E(\Zd)$. We let $\partial U $ be the \emph{external vertex boundary} of $U$,
\begin{equation*}
    \partial U := \left\{ x \in \Zd \setminus U \, \colon \, \exists y \in U, y \sim x  \right\}.
\end{equation*}
A box $\Lambda\subseteq \Zd$ is a subset of the form $x + \{-L, \ldots, L\}^d$ for $x \in \Zd$ and $L \in \N$. For any integer $L \in \N$, we let 
\begin{equation*}
    \Lambda_L := \{-L, \ldots, L\}^d
    \subseteq \Z^d
\end{equation*}
be the box centered at $0$ of side length $(2L+1)$. We extend the notation to real-valued $L \geq 0$ by setting $\Lambda_L := \Lambda_{\lfloor L \rfloor}$. For a box $\Lambda := x + \Lambda_L$ and an integer $n\in \N$, we denote by $n\Lambda := x + \Lambda_{nL}$ the rescaled box (the center of the box remains unchanged but its sidelength is multiplied by $n$). If $I := (s_- , s_+) \subseteq \R$ is a bounded interval of $\R$, then we denote by $n I := (s_- - (n-1)(s_+ - s_-), s_+)$ (the right end of the interval remains unchanged and its length is multiplied by $n$).

A parabolic cylinder is a set of the form $Q := I \times \Lambda \subseteq \R \times \Zd$ where $I := (s_- , s_+)$ is a bounded interval of $\R$ and $\Lambda$ is a subset of $\Zd$ (frequently chosen to be a box). For $L \geq 0$, we denote by $Q_L$ the parabolic cylinder
\begin{equation} \label{def.QL}
    Q_L := (-L^2 , 0) \times \Lambda_L.
\end{equation}
We denote by $\partial_{\mathrm{par}} Q$ the parabolic boundary of the cylinder $Q$ defined by the formula
\begin{equation*}
    \partial_{\mathrm{par}} Q := \left( \{ - s_- \} \times \Lambda \right) \cup \left( (-s_- , s_+) \times \partial \Lambda \right).                                        
\end{equation*}
Given a parabolic cylinder $Q = I \times \Lambda$ and an integer $n \in \N$, we denote by $nQ := nI \times n \Lambda$, using the convention above.
As it will be useful for us to partition parabolic cylinders, we introduce the following notation: for each $m \leq n$,
\begin{equation*}
    \mathcal{Z}_{m,n} := (3^{2m} \Z \times 3^m \Zd) \cap Q_{3^n}.
\end{equation*}
Note that the collection $\left( z + Q_{3^m} \right)_{z \in \mathcal{Z}_{m   ,n}}$ forms a partition of the parabolic cylinder $Q_{3^n}$, i.e.,
\begin{equation*}
    Q_{3^n} := \cup_{z \in \mathcal{Z}_{m ,n}} \left( z + Q_{3^n} \right)
\end{equation*}
and
\begin{equation*}
\forall z , z' \in \mathcal{Z}_{m,n},~ z \neq z', ~\left( z + Q_{3^m} \right) \cap \left( z' + Q_{3^m} \right) = \emptyset.
\end{equation*}
We denote by $|I|$ the Lebesgue measure of $I$ and by $|\Lambda|$ the cardinality of $\Lambda$, and by $\left| Q \right| := \left| I \right| \times \left| \Lambda \right|$.

\subsubsection{Functions} \label{subsecfunctions}
Given a box $\Lambda \subseteq \Zd$, we denote by $\Omega_{\Lambda , \mathrm{per}}$ the set of periodic functions defined on the box $\Lambda$ and valued in $\R$, and by $\Omega_{\Lambda , \mathrm{per}}^\circ$ the subset of functions of $\Omega_{\Lambda , \mathrm{per}}$ satisfying $\sum_{x \in \Lambda} \phi(x) = 0$

Fix a parabolic cylinder $Q = I \times \Lambda \subseteq \R \times \Zd$ and a function $u : Q  \to \R$. For each $t \in I$ and each directed edge $e = (x,y) \in {\vec E}(\Lambda)$, we define the discrete gradient 
\begin{equation}\label{eq:discrete gradient}
\nabla u(t,e) := u(t,y) - u(t,x).
\end{equation}
In expressions which do not depend on the orientation of the edge, such as $|\nabla u(t, e)|^2$ or $\nabla u(t, e) \nabla v(t, e)$, we allow the edge $e$ to be undirected. It will be useful to us to have a notion \emph{vector-valued} discrete gradient, mimicking the definition in the continuum. We thus define, for $i \in \{ 1 , \ldots , d \}$,
\begin{equation*}
    \nabla_i u (t , x) := u (t , x + e_i) - u(t , x),
\end{equation*}
and
\begin{equation*}
    \vec{\nabla} u (t , x) := \left( \nabla_1 u (t , x) , \ldots, \nabla_d u (t , x)  \right) \in \Rd.
\end{equation*}
We consider the second derivative of a map $u$, defined as the matrix-value function $\vec{\nabla}^{ 2} u(t , x) := ( \nabla_i \nabla_j u(t , x) )_{1 \leq i , j \leq d}$ (mimicking the definition of the continuum), and denote by $\partial_t u$ the derivative of the function~$u$.
We define the average value of a function $u : Q \to \R$ according to the formula
\begin{equation} \label{eq:defavrageoverQ}
    \left( u \right)_Q := \frac{1}{|Q|} \int_{I} \sum_{x \in \Lambda} u(t , x) \, dt.
\end{equation}

A \emph{vector field} is a function $\g :I \times {\vec E}(\Lambda) \to\R$ satisfying $\g(t , (x,y)) = -\g(t , (y,x))$. For a vector $p \in \Rd$, we abuse notation and denote by $p$ the constant vector field defined by
\begin{equation*}
    p(e) := p \cdot (y - x) ~~\mbox{for every}~ e = (x , y) \in {\vec E}(\Zd).
\end{equation*}
For a vector field $\mathbf{g} : I \times \vec{E}(\Lambda) \to \R$, we define the average value $\left( \mathbf{g} \right)_Q:= (\left( \mathbf{g} \right)_{Q,1}, \ldots, \left( \mathbf{g} \right)_{Q, d}) \in \Rd$ according to the formula, for each $i \in \{1 , \ldots, d \}$,
\begin{equation} \label{defaveragevector}
    \left( \mathbf{g} \right)_{Q,i} := \frac{1}{\left| Q\right|} \sum_{x \in \Lambda} \int_I \mathbf{g}(t, (x , x+e_i)) \, dt.
\end{equation}
For each $(t,x) \in Q$, we define the divergence of a vector field $\g : I \times {\vec E}(\Lambda) \to \R$ according to the formula
\begin{equation*}
    \nabla \cdot \g (t , x) := \sum_{y \sim x} \g (t , (x, y)).
\end{equation*}

\subsubsection{Functional spaces} \label{sectionholdernorms}
In this subsection, we introduce the various functional spaces used in the article. We let $Q := I \times \Lambda$ be a parabolic cylinder and denote by $L$ the sidelength of the box $\Lambda$.

\smallskip

\textbf{(i) $L^p$-norms.} For $p \in [1 , \infty)$, we define the $L^p$ and normalized $\underline{L}^p$-norms of a function $u : Q \to \R$ and a vector field $\g : I \times \vec{E} \left( \Lambda\right) \to \R$ by the formulae
\begin{equation*}
    \left\| u \right\|_{L^p\left( Q \right)}^p := \int_I \sum_{x \in \Lambda} \left| u(t,x)\right|^p \, dt, \quad \left\| u \right\|_{\underline{L}^p \left( Q \right)}^p :=  \frac{1}{\left| Q \right|} \int_I \sum_{x \in \Lambda} \left| u(t,x)\right|^p \, dt ,
\end{equation*}
and
\begin{equation*}
    \left\| \g \right\|_{L^p\left( Q \right)}^p :=  \int_{I} \sum_{e \in E(\Lambda)} \left| \g(t, e)\right|^p \, dt, \quad \left\| \g \right\|_{\underline{L}^p \left( Q \right)}^p :=  \frac{1}{\left| Q \right|}\int_I \sum_{e \in E(\Lambda)} \left| \g(t ,e)\right|^p \, dt.
\end{equation*}
We define the $L^\infty$-norm of the function $u$ (resp. the vector field $\g$) to be the essential supremum of $u$ (resp. $\g$) over the set $Q$ (resp. $I \times \vec{E} \left( \Lambda\right)$) and denote it by $\left\| u \right\|_{L^\infty(Q)}$ (resp. $\left\| \g \right\|_{L^\infty(Q)}$).
We similarly define the $L^p(\Lambda)$-norm and normalized $\underline{L}^p(\Lambda)$-norm of functions (resp. vector fields) defined on the box $\Lambda$ (resp. the edge set $\vec{E} \left( \Lambda \right)$) and valued in $\R$.

\medskip

\textbf{(ii) Sobolev spaces.}
For $p \in [1 , \infty)$ and a function $u : Q \to \R$ (resp. $v : \Lambda \to \R$), we introduce the Sobolev norms
\begin{equation} \label{eq:17081801}
    \left\| u \right\|_{W^{1 , p}(Q)} := \left\| u \right\|_{L^p (Q)} + \left\| \nabla u \right\|_{L^p (Q)} + \left\| \partial_t u \right\|_{L^p (Q)} 
\end{equation}
and
\begin{equation*}
    \left\| v \right\|_{W^{1 , p}(\Lambda)} :=  \left\| v \right\|_{L^p (\Lambda)} + \left\| \nabla v \right\|_{L^p (\Lambda)}.
\end{equation*}
We denote by $W^{1 , p}_0(Q)$ (resp. $W^{1 , p}_0(\Lambda)$) the completion of the space of smooth compactly supported functions defined on the cylinder $Q$ (resp. the set of functions equal to $0$ on the external boundary $\partial \Lambda$) with respect to the norm~\eqref{eq:17081801}. In the case $p = 2$, we denote by $H^1(Q) := W^{1 , 2}(Q)$ and $H^1_0(Q) := W^{1 , 2}_0(Q)$ (and use similar notation with the box $\Lambda$). We introduce the Sobolev space $H^{2}(\Lambda)$ according to the formula
\begin{equation*}
    \left\| u \right\|_{H^2(\Lambda)} :=  \left\| u \right\|_{L^2 (\Lambda)} + \left\| \nabla u \right\|_{L^2 (\Lambda)} +  \left\| \nabla^2 u \right\|_{L^2 (\Lambda)}.
\end{equation*}
We similarly introduce the space $H^2(Q)$ taking into account the derivatives with respect to the time variable. We additionally define the norm
\begin{equation*}
    \left\| u \right\|_{L^p \left( I , W^{1 , p}\left( \Lambda \right) \right)}^p := \int_I \left\| u (t , \cdot) \right\|_{W^{1 , p}(\Lambda)}^p \, dt,
\end{equation*}
and define normalized version of these norms according to the formulae
\begin{equation*}
     \left\| v \right\|_{\underline{W}^{1 , p}(\Lambda)} :=  \frac{1}{L}  \left\| v \right\|_{\underline{L}^p (\Lambda)} + \left\| \nabla v \right\|_{\underline{L}^p (\Lambda)}
\end{equation*}
and
\begin{equation*}
    \left\| u \right\|_{\underline{L}^p \left( I , \underline{W}^{1 , p}\left( \Lambda \right) \right)}^p := \frac{1}{|I|}\int_I \left\| u (t , \cdot) \right\|_{\underline{W}^{1 , p}(\Lambda)}^p \, dt.
\end{equation*}
Let us denote by $q := p/(p-1)$ the conjugate exponent of $p$. We introduce the dual Sobolev norm
\begin{equation*}
    \left\| u \right\|_{W^{-1,p}(\Lambda)} := \sup \left\{ \sum_{x \in \Lambda} u(x) v(x) \, : \, v \in W^{1 , q}_0(\Lambda), \, \left\| v \right\|_{W^{1 , q}  \left( \Lambda \right)} \leq 1 \right\},
\end{equation*}
as well as its normalized version
\begin{equation*}
   \left\| u \right\|_{\underline{W}^{-1,p}(\Lambda)} := \sup \left\{ \frac{1}{\left| \Lambda\right|}\sum_{x \in \Lambda} u(x) v(x) \, : \, v \in W^{1 , q}_0(\Lambda), \, \left\| v \right\|_{\underline{W}^{1 , q} \left( \Lambda \right)} \leq 1 \right\},
\end{equation*}
and denote by $H^{-1}(\Lambda) = W^{-1 , 2}(\Lambda)$ and $\underline{H}^{-1}(\Lambda) = \underline{W}^{-1 , 2}(\Lambda)$.
We define the $\underline{L}^2 \left( I , \underline{W}^{-1,p}(\Lambda)\right)$-norm according to the formula
\begin{equation*}
    \left\| u \right\|_{\underline{L}^p \left( I ,  \underline{W}^{-1,p}(\Lambda) \right)}^p = \frac{1}{\left| I \right|} \int_I \left\| u(t , \cdot) \right\|_{\underline{W}^{-1,p}(\Lambda)}^p \, dt.
\end{equation*}
The following norm is used to study solutions of parabolic equations: given a function $u : Q \to \R$, we set
\begin{equation*}
    \left\| u \right\|_{\underline{W}^{1 , p}_{\mathrm{par}} (Q)} :=  \frac{1}{L} \left\| u \right\|_{\underline{L}^p (Q)} +\left\| \nabla u \right\|_{\underline{L}^p (Q)} + \left\| \partial_t u \right\|_{\underline{L}^p \left( I , \underline{W}^{-1 , p}\left( \Lambda \right) \right)}.
\end{equation*}
We denote by $W^{1 , p}_{\mathrm{par} , \sqcup}$ the completion of the set of smooth compactly supported functions in $(s_- , s_+] \times \Lambda$ with respect to the norm $\underline{W}^{1 , p}_{\mathrm{par}} (Q)$.
We then define the parabolic $\underline{W}^{-1,p}_{\mathrm{par}}(Q)$-norm of a map $u$ by
\begin{equation*}
    \left\| u \right\|_{\underline{W}_{\mathrm{par}}^{-1,p}(Q)} := \sup \left\{ \frac{1}{\left| Q \right|} \int_I \sum_{x \in \Lambda} u(t , x) v(t , x) \, dt \, : \, v \in W^{1 , p}_{\mathrm{par} , \sqcup}(Q), \, \left\| v \right\|_{\underline{W}^{1 , p}_{\mathrm{par}}(Q)} \leq 1 \right\}
\end{equation*}
and set $\underline{H}^{1}_{\mathrm{par}}(Q) = \underline{W}_{\mathrm{par}}^{1,2}(Q)$ and $\underline{H}^{-1}_{\mathrm{par}}(Q) = \underline{W}_{\mathrm{par}}^{-1,2}(Q)$.

\bigskip

\textbf{(iii) H\"{o}lder spaces.} For any exponent $\alpha \in [0,1]$, any parabolic cylinder $Q \subseteq \R \times \Zd$, and any function $u : Q \to \R$, we define the $C^{0 , \alpha}$-seminorm of the map~$u$ by
\begin{equation*}
    \left[ u \right]_{C^{0, \alpha}(Q)} := \sup_{(t , x) , (s , y) \in Q} \frac{\left| u(t , x) - u(s , y)\right|}{|t - s|^{\alpha/2} + |x -y|^{\alpha}}.
\end{equation*}
The $C^{1 , \alpha}$-seminorm of the map $u$ is then defined by
\begin{equation*}
    \left[ u \right]_{C^{1, \alpha}(Q)} :=  [ \vec{\nabla} u ]_{C^{0, \alpha}(Q)}.
\end{equation*}

We note that, while all the definitions above are given in the discrete setting, we may use them in the continuum to refer to either the boundary conditions $f$ to the solution of the homogenized equation (i.e., the map $\bar u$ in~\eqref{eq:defubarthmhydro}) in the proofs of Theorem~\ref{Th.quantitativehydr} or Theorem~\ref{theoremlargescale}.

\subsection{Microscopic notation} \label{Sectionmicroscopic}
In the statement and proof of Theorem~\ref{Th.quantitativehydr}, we work on the rescaled lattice $\ep \Zd$; this convention is standard in homogenization and is also the one used in the proof of the hydrodynamic limit by Funaki and Spohn~\cite{FS}. Working in this framework requires to introduce notation which are adapted to the rescaled lattice $\ep \Zd$ but essentially matches the ones of the previous subsections.

\subsubsection{Sets and functions}
As mentioned in Section~\ref{section:mainresult}, we let $D \subseteq \Rd$ be a bounded $C^{1,1}$ domain, let $I := [-1 , 0]$ and set $Q := I \times D$ (as in the statement of Theorem~\ref{Th.quantitativehydr}). We denote by $D^\ep := \ep \Zd \cap D$, by $Q^\ep := I \times D^\ep$. We denote by $\partial D^\ep$ the external vertex boundary of $D^\ep$ and by $\partial_{\mathrm{par}} Q^\ep := \left( \{ -1 \} \times D^\ep \right) \cup \left( I \times \partial D^\ep \right)$.
For $\kappa \in (0 , 1 )$, we denote by $\Lambda^\ep_\kappa := [- \kappa, \kappa]^d \cap \ep \Zd$ and by $Q^\ep_\kappa := (- \kappa^2 , 0) \times \Lambda^\ep_\kappa$.
We denote by $\left| \Lambda_\kappa^\ep \right|$ (resp. $\left| D^\ep \right|$) the cardinality of the box $\Lambda_\kappa^\ep$ (resp. the set $D^\ep$), and by $\left| Q^\ep_\kappa \right| = \kappa^2 \left| \Lambda_\kappa^\ep \right|$ (resp. $\left| Q^\ep \right| = \left|I\right| \left| D^\ep \right|$).

Given a function $u^\ep : D^\ep \to \R$ and an edge $e = ( x , y) \in \vec{E} \left( D^\ep \right)$, we denote its discrete gradient by $\nabla^\ep u(t,e) := \ep^{-1} \left( u(t,y) - u(t,x) \right).$ Similarly, for $i \in \{ 1 , \ldots, d\}$, we define $$\nabla_i^\ep u (t , x) := \ep^{-1} \left( u(t,x+ e_i) - u(t,x) \right) ~~\mbox{and}~~ \vec{\nabla}^\ep u (t , x) := \left( \nabla_1^\ep u(t , x) , \ldots,  \nabla_d^\ep u(t , x) \right).$$ We will have to consider the second derivative of a map $u$, defined as the matrix-valued function $$\nabla^{\ep, 2} u(t , x) := ( \nabla_i^\ep \nabla_j^\ep u(t , x) )_{1 \leq i , j \leq d}.$$
We define the average value of a function $u$ over the set $Q^\ep_\kappa$ according to the formula
\begin{equation*}
    \left( u \right)_{Q^\ep_\kappa} := \frac{1}{\left| Q^\ep_\kappa \right|} \int_{-\kappa^2}^0  \sum_{x \in \Lambda^\ep_\kappa} u(t , x) \, dt,
\end{equation*}
and similarly define the average value of a vector field by rescaling the definition~\eqref{defaveragevector}.

\subsubsection{Functional spaces}

We also define the suitably rescaled version of the $L^p$-norms by the formulae (which will be used either on the set $Q^\ep$ or on cylinders of the form $z + Q^\ep_\kappa$)
\begin{equation} \label{rescaledLp1}
    \left\| u \right\|_{L^p (Q^\ep)}^p:= \ep^d \int_{I}  \sum_{x \in D^\ep} \left| u(t , x)\right|^p \, dt ~\mbox{and}~  \left\| u \right\|_{L^p (Q^\ep_\kappa)}^p :=  \ep^d \int_{-\kappa^2}^0  \sum_{x \in \Lambda^\ep_\kappa} \left| u(t , x)\right|^p \, dt,
\end{equation}
as well as the averaged $L^p$-norm
\begin{equation} \label{rescaledLp2}
\left\| u \right\|_{\underline{L}^p (Q^\ep_\kappa)}^p :=  \frac{1}{\left| Q^\ep_\kappa \right|}  \int_{-\kappa^2}^0  \sum_{x \in \Lambda^\ep_\kappa} \left| u(t , x)\right|^p \, dt.
\end{equation}
We similarly define the $L^\infty$-norm by considering the essential supremum.
Given a time interval $I \subseteq \R$ and a box $\Lambda ^\ep \subseteq \ep \Zd$ of sidelength $\kappa$ (i.e., of the form $x + \Lambda^\ep_\kappa$ for $x \in \ep \Zd$), we extend the definition of the norms $W^{1 , p}(I \times \Lambda^\ep), L^p(I , W^{1 , p}(\Lambda^\ep)), W^{-1 , p} (\Lambda^\ep)$, $W^{1 , p}_{\mathrm{par}}(I \times \Lambda^\ep)$ and $W^{-1 , p}_{\mathrm{par}}(I \times \Lambda^\ep)$ by replacing the discrete gradient $\nabla$ by $\nabla^\ep$, the parameter $L$ by $\kappa$ and the $L^p$-norms by the ones defined in~\eqref{rescaledLp1} and~\eqref{rescaledLp2}.

Finally, we extend the definition of the norm $\underline{H}^{-1}_{\mathrm{par}}$ to the set $Q^\ep := I \times D^\ep$ by setting 
\begin{equation} \label{eq:rescalingH-1par}
\left\| u \right\|_{\underline{H}_{\mathrm{par}}^{-1}(Q^\ep)} := \sup \left\{ \frac{1}{\left| I \right| \left| D^\ep \right|} \int_I \sum_{x \in D^\ep} u(t , x) v(t , x) \, dt \, : \, v \in H^{1}_{\mathrm{par} , \sqcup}(Q^\ep), \, \left\| v \right\|_{\underline{H}^{1 }_{\mathrm{par}}(Q^\ep)} \leq 1 \right\}.
\end{equation}
For later use, we record that the $\underline{H}^{-1}_{\mathrm{par}}$ satisfies the following scaling identity: for any $L \in \N$, any function $u : Q_L \to \R$, if we let $u^\ep : Q^{\ep}_{\ep L} \to \R$ be the rescaled map defined by the formula $u^\ep (t, x) = u (t/\ep^2 , x/\ep)$, then we have
\begin{equation} \label{eq:29.5}
    \left\| u^\ep \right\|_{\underline{H}^{-1}_{\mathrm{par}}\left(Q^{\ep}_{\ep L} \right)} = \ep \left\| u \right\|_{\underline{H}^{-1}_{\mathrm{par}}\left(Q_{L} \right)}.
\end{equation}

\subsection{Stochastic integrability} \label{section.stochint}
Fix an exponent $s \in (0 , \infty)$ and a constant $K > 0$. Given a random variable $X$, we write
\begin{equation*}
    X \leq \mathcal{O}_s (K) \Leftrightarrow \mathbb{E} \left[ \exp \left( \left( \frac{|X|}{K} \right)^{\!\!s} \, \right) \right] \leq 2.
\end{equation*}
We record from~\cite[Appendix A]{AKM} some properties satisfied by this notation:
\begin{itemize}
\item \textit{Summation:} For each exponent $s > 0$, there exists a constant $C_s$ such that, if let $X_1 , \ldots, X_n$ be nonnegative random variables and $K_1, \ldots, K_n$ be positive constants, then
\begin{equation} \label{sum.Onotation}
        X_1 \leq \mathcal{O}_s(K_1), \, \ldots \,, X_n \leq \mathcal{O}_s(K_n) \implies \sum_{i = 1}^n X_i \leq \mathcal{O}_s \left( C_s \sum_{i = 1}^n K_i \right);
\end{equation}
\item \textit{Powers:} For each pair of exponents $s , s_0 > 0$, and each constant $K > 0$, one has the estimate
\begin{equation} \label{Onotation.power}
    X \leq \mathcal{O}_s (K) \implies X^{s_0} \leq \mathcal{O}_{s/s_0} \left(K^{s_0} \right);
\end{equation}
\item \textit{Integration:} For each exponent $s > 0$, there exists a constant $C_s$ such that the following holds. Let $(X_t)_{t \geq 0}$ be a collection of random variables, $(K_t)_{t \geq 0}$ be a nonnegative function and $I \subseteq (0 , \infty)$ be an interval, then
\begin{equation}  \label{int.Onotation}
    \forall t \geq 0, X_t \leq \mathcal{O}_s(K_t) \implies \int_I X_t \leq \mathcal{O}_s \left( C_s \int_I K_t \, dt \right);
\end{equation}
\item \textit{Expectation:} For each exponent $s > 0$, there exists a constant $C_s$ such that
\begin{equation*}
    X \leq \mathcal{O}_s(K) \implies \mathbb{E} \left[ X \right] \leq C_s K;
\end{equation*}
\item \textit{Multiplication:} For each exponent $s > 0$, and each $\lambda > 0$,
\begin{equation*}
    X \leq \mathcal{O}_s(K) \implies \lambda X \leq \mathcal{O}_s(\lambda K);
\end{equation*}
\item \textit{Maximum:} For each exponent $s > 0$, there exists a constant $C_s$ such that, if let $X_1 , \ldots, X_n$ be nonnegative random variables and $K$ be a positive constant, then
\begin{equation} \label{sum.Omaximum}
        X_1 \leq \mathcal{O}_s(K), \, \ldots \,, X_n \leq \mathcal{O}_s(K) \implies \max_{i = 1, \ldots, n} X_i \leq \mathcal{O}_s \left( C_s \left( \log  n\right)^{\frac 1s} K \right).
\end{equation}
\end{itemize}

\subsection{Parabolic equations and regularity}
In this section, we introduce the definition of discrete parabolic operators and record some of the main properties of the solutions of parabolic equations used in the article: the Caccioppoli and Meyers inequalities, and the Nash-Aronson estimate for the heat kernel. We recall that we fix a parabolic cylinder $Q = I \times \Lambda \subseteq \R \times \Zd$.

\subsubsection{Environment and parabolic equations}
An environment $\a$ defined on the parabolic cylinder $Q$ is a measurable map $\a : I \times E(\Lambda) \to [0 , \infty]$. Given an environment $\a$, we denote by $\nabla \cdot \a \nabla$ the time-dependent elliptic operator defined by the formula: for any map $u : Q \to \R$ and any $(t , x) \in Q$,
\begin{equation} \label{eq:discdynlaplacian}
    \nabla \cdot \a \nabla u (t , x) = \sum_{y \sim x} \a (t , \{ x , y \}) \left( u(t , y) - u(t , x) \right).
\end{equation}
The discrete Laplacian on the lattice is the elliptic operator $\Delta = \nabla \cdot \a \nabla$ when $\a \equiv 1$.
Given a parabolic cylinder $Q \subseteq \R \times \Zd$ and an environment $\a$, a function $u : Q \to \R$ is called $\a$-caloric if it solves the parabolic equation
\begin{equation} \label{def.paraboliceqautionheat-kernel}
    \partial_t u - \nabla \cdot \a \nabla u = 0 ~\mbox{in}~ Q.
\end{equation}
In the rest of this section (and of the article), we will assume that all the environments satisfy the uniform ellipticity condition $0 < c_- \leq \a \leq c_+ < \infty$.

\subsubsection{Caccioppoli and Meyers inequalities}
In this section, we state the parabolic versions of the Caccioppoli and Meyers inequalities. For the Caccioppoli inequality, we refer to~\cite[Lemma B.3]{ABM} (among other possible sources)

\begin{proposition}[Parabolic Caccioppoli inequality] \label{prop.CAccioppolipara}
There exists a constant $C < \infty$  depending on $d , c_+ , c_-$such that, for every sidelength $L \in \N$, every uniformly elliptic environment $\a$ defined on $Q_{2L}$, and every solution of the parabolic equation
\begin{equation*}
    \partial_t u - \nabla \cdot \a \nabla u = \nabla \cdot F ~\mbox{in}~ Q_{2L},
\end{equation*}
one has the upper bound
\begin{equation*}
    \left\| \nabla u \right\|_{\underline{L}^2(Q_{L})} \leq \frac{C}{L} \left\|u \right\|_{\underline{L}^2(Q_{2L})} + C \left\| F \right\|_{\underline{L}^2(Q_{2L})}.
\end{equation*}
\end{proposition}

The second proposition of this section is the interior parabolic Meyers estimate which provides a $W^{1 , 2+\gamma_0}$ regularity estimate for solutions of parabolic equations. In the elliptic case, the Meyers estimate is due to~\cite{meyers1963p}, in the parabolic case they were first proved in~\cite{giaquinta1982partial}, and the global version was considered in~\cite{parviainen2009global} (see also~\cite[Appendix B]{ABM}).

\begin{proposition}[Interior parabolic Meyers inequality, Theorem 2.1 of~\cite{giaquinta1982partial}] \label{interiorparabolicMEyers}
There exist a constant $C < \infty$ and an exponent $\gamma_0 > 0$ depending on $d , c_+ , c_-$ such that, for every sidelength $L \in \N$, every uniformly elliptic environment $\a$ defined on $Q_{2L}$, and every solution of the parabolic equation
\begin{equation*}
    \partial_t u - \nabla \cdot \a \nabla u = \nabla \cdot F ~\mbox{in}~ Q_{2L},
\end{equation*}
one has the upper bound
\begin{equation*}
    \left\| \nabla u \right\|_{\underline{L}^{2 + \gamma_0}(Q_{L})} \leq C \left\| \nabla u \right\|_{\underline{L}^2(Q_{2L})} + C \left\| F \right\|_{\underline{L}^{2 + \gamma_0}(Q_{2L})}.
\end{equation*}
\end{proposition}

We finally collect a global version of the Meyers estimate.

\begin{proposition}[Global parabolic Meyers inequality, Proposition B.2 of~\cite{ABM}] \label{propMeyers}
There exist a constant $C <\infty$ and an exponent $\gamma_0 >0$ depending on the parameters $d , c_+ ,c_-$ such that for every $L \in \N$, every environment $\a : Q_L \to [c_- , c_+]$, and every $F : Q_L \to \R$, if we let $u$ be the solution of the parabolic equation
\begin{equation*}
    \left\{ \begin{aligned}
    \partial_t u - \nabla \cdot \a \nabla u & = \nabla \cdot F & ~\mbox{in} ~ Q_L, \\
    u & = 0 & ~\mbox{in}~  \Lambda_L, \\
   u (t , \cdot) &\in \Omega_{\Lambda_L , \mathrm{per}} &~\mbox{for}~t \in I_L,
    \end{aligned} \right.
\end{equation*}
then one has the estimate
\begin{equation*}
     \left\| \nabla u \right\|_{\underline{L}^{2+\gamma_0}(Q_{L})} \leq C \left\| F \right\|_{\underline{L}^{2 + \gamma_0} (Q_{L})}.
\end{equation*}
\end{proposition}

\subsubsection{De Giorgi-Nash-Moser regularity}

In this section, we record the $C^{0 , \alpha}$-regularity for solutions of parabolic equations with uniformly elliptic coefficient. The proof of this result in the continuous setting can be found in~\cite{nash1958continuity, moser1961harnack, DG1} (see also~\cite[Chapter 8]{GT01} and~\cite[Chapter 3]{han2011elliptic}). In the discrete setting considered here, we refer to~\cite[Appendix B]{GOS} and \cite[Section 3]{DD05}.

\begin{proposition}[$L^\infty$ and De Giorgi-Nash-Moser regularity] \label{paraNash}
Fix an integer $L \in \N$ and let $\a$ be a uniformly elliptic environment defined in $Q_{2L}$. There exist an exponent $\alpha > 0$ and a constant $C< \infty$ depending only on the parameters $d , c_+ , c_-$ such that, for any solution $u : Q_{2L} \to \R$ of the parabolic equation
\begin{equation*}
    \partial_t u - \nabla \cdot \a \nabla u = 0 \hspace{2mm} \mbox{in} \hspace{2mm}  Q_{2L},
\end{equation*}
one has the estimate
\begin{equation} \label{NashParaest}
     \left\|  u  \right\|_{L^\infty(Q_L)}  + L^\alpha \left[ u \right]_{C^{0, \alpha}(Q_{L})} \leq C \left\|  u \right\|_{\underline{L}^2 \left( Q_{2L}\right)}.
\end{equation}
\end{proposition}

\subsubsection{Heat kernel, Nash-Aronson estimate and Duhamel principle}
Let $Q := I \times \Lambda$ be a parabolic cylinder denote by $I = [s , s_+]$, and assume that $\Lambda$ is a box of sidelength $L$. Fix a point $y \in \Lambda$. Given a uniformly environment~$\a$ defined on $Q$, we consider the heat kernel in the box $\Lambda$ with periodic boundary conditions defined by
\begin{equation} \label{def.heatkernelperiodic}
    \left\{ \begin{aligned}
    \partial_t P_{\a}(\cdot , \cdot , s , y) -\nabla \cdot \a \nabla P_{\a} &=  0&~\mbox{in}~& Q, \\
    P_{\a} (s , \cdot, s , y) &=  \delta_y - \frac{1}{\left| \Lambda \right|}&~\mbox{in}~& \Lambda, \\
    P_{\a} (t , \cdot, s , y) &\in ~\Omega_{\Lambda , \mathrm{per}} & ~\mbox{for}~& t \in I,
    \end{aligned} \right.
\end{equation}
where $\delta_y$ denotes the discrete Dirac function taking the value $1$ at $y$ and $0$ everywhere else; the term $1/\left| \Lambda\right|$ ensures that, for any time $t \geq s$, the spatial average of the heat kernel over the box $\Lambda$ is equal to~$0$. We extend the definition to the times $t < s$ by setting $P_\a \left(t , x ; s, y \right) = 0$ in that case. For each constant $C > 0$, we denote by $\Phi_{C, L} : (0 , \infty) \times \Rd \to \R$ the function
\begin{equation} \label{def.phiCL}
    \Phi_{C, L}(t , x) := C (t \vee 1)^{-\frac d2} \exp \left( - \frac{|x - y|}{C t^{\frac12}}\right) \exp \left( - \frac{t}{C L^2} \right). 
\end{equation}
The Nash-Aronson estimate, originally established in the continuous setting in~\cite{Ar}, provides an upper bound on the heat-kernel $P_\a$ in terms of the function $\Phi_{C,L}$. The result stated below in the discrete setting can essentially be deduced from~\cite[Proposition B.3]{GOS} (see Remark~\ref{remark2.6}).

\begin{proposition}[Nash-Aronson estimate] \label{prop.NashAronson}
Let $\a$ be a uniformly elliptic environment defined on $Q$, and let $P_\a$ be the heat kernel with periodic boundary conditions defined in~\eqref{def.heatkernelperiodic}. There exists a constant $C := C(d , c_- , c_+) < \infty$ such that, if $t - s \leq L^2$,
\begin{equation} \label{eq:prop.NashAronson}
    0 \leq P_{\a} (t , x, s , y) + \frac{1}{\left| \Lambda \right|} \leq \Phi_{C, L}(t - s , x - y),
\end{equation}
and, if $t-s \geq L^2$,
\begin{equation*}
    \left| P_{\a} (t , x, s , y) \right| \leq \Phi_{C, L}(t - s , x - y).
\end{equation*}
Additionally, there exists an exponent $\alpha := \alpha (d , c_+, c_-) > 0$ such that
\begin{equation} \label{eq:prop.NashAronsongrad}
   | \vec{\nabla} P_{\a} (t , x, s , y)| \leq (t \vee 1)^{-\alpha} \Phi_{C, L}(t - s , x - y).
\end{equation}
\end{proposition}

\begin{remark}
The term $1/\left| \Lambda \right|$ in the left-hand side of~\eqref{eq:prop.NashAronson} takes into account the additive factor $1/\left| \Lambda \right|$ in the initial condition of~\eqref{eq:prop.NashAronson}, which itself is added to the definition to ensure that the spatial average of the heat kernel over the box $\Lambda$ is equal to $0$.
\end{remark}

\begin{remark} \label{remark2.6}
The result of~\cite[Proposition B.3]{GOS} is stated for the heat kernel defined in infinite volume (compared to the one considered here which is defined in a box $\Lambda$ and with periodic boundary conditions). The result of Proposition~\ref{prop.NashAronson} can be deduced from the one of~\cite[Proposition B.3]{GOS} through standard arguments.
\end{remark}

The heat-kernel plays a fundamental role in this article, as it can be used to solve parabolic equations with a non-zero right-hand side. The formula is known as Duhamel's principle and is stated below.

\begin{proposition}[Duhamel's principle] \label{propDuhamel}
    Let $f : Q \to \R$ be a piecewise continuous function such that, for any $t \in I$, $\sum_{x \in \Lambda} f(t , x) = 0$. Let $u : Q \to \R$ be the (periodic) solution of the parabolic equation
    \begin{equation} \label{def.functionuperiodic}
    \left\{ \begin{aligned}
    \partial_t u -\nabla \cdot \a \nabla u &=  f &~\mbox{in}~& Q, \\
    u(s , \cdot) &=  0 &~\mbox{in}~& \Lambda, \\
    u &\in ~\Omega_{\Lambda , \mathrm{per}} & ~\mbox{for}~& t \in I,
    \end{aligned} \right.
\end{equation}
then we have the identity, for any $(t , x) \in Q$,
\begin{equation*}
    u(t , x) = \sum_{y \in \Lambda} \int_s^t P_\a(t , x ; s' , y) f(s' , y) \, ds'.
\end{equation*}
\end{proposition}

\begin{proof}
    For $(t , x) \in Q$, we denote by $w(t , x) := \sum_{y \in \Lambda} \int_s^t P_\a(t , x ; s' , y) f(s' , y) \, ds'$. We will prove that the map $w$ solves the equation~\eqref{def.functionuperiodic}. Differentiating the function $w$ with respect to time, we obtain
    \begin{equation} \label{eq:14390510}
        \partial_t w(t , x) = \sum_{y \in \Lambda} P_\a(t , x ; t , y) f(t , y) +  \sum_{y \in \Lambda} \int_s^t \partial_t P_\a(t , x ; s' , y) f(s' , y) \, ds'.
    \end{equation}
    For the first term in the right-hand side, we use the definition~\eqref{def.heatkernelperiodic} of the heat kernel and write
    \begin{equation*}
       \sum_{y \in \Lambda} P_\a(t , x ; t , y) f(t , y) = f(t , x) - \sum_{y \in \Lambda} \frac{f(t, y)}{\left| \Lambda_L \right|} =  f(t , x).
    \end{equation*}
    For the second term in the right-hand side of~\eqref{eq:14390510}, we use the definition~\eqref{def.heatkernelperiodic} a second time and write
    \begin{align*}
         \sum_{y \in \Lambda} \int_s^t \partial_t P_\a(t , x ; s' , y) f(s' , y) \, ds' & = \sum_{y \in \Lambda} \int_s^t \nabla \cdot \a \nabla P_\a(t , x ; s' , y) f(s' , y) \, ds'  \\
         & =  \nabla \cdot \a \nabla \left( \sum_{y \in \Lambda}  \int_s^t  P_\a(t , x ; s' , y) f(s' , y) \, ds' \right) \\
         & = \nabla \cdot \a \nabla w(t , x).
    \end{align*}
    A combination of the three previous displays implies that the map $w$ solves the parabolic equation
    \begin{equation*}
        \partial_t w - \nabla \cdot \a \nabla w = f ~~\mbox{in}~~Q.
    \end{equation*}
    Additionally, the definition of $w$ implies that $w(s , \cdot) = 0$. These two observations imply that $u$ and $w$ are both solutions of the parabolic equation~\eqref{def.functionuperiodic}, and are thus equal.
\end{proof}

\subsection{Gaussian concentration inequality}

In this section, we record Gaussian concentration inequality~\cite{borell1975brunn, sudakov1978extremal} (see also~\cite[Theorem 5.6]{boucheron2013concentration}). The result is used to estimate the fluctuations of the first-order corrector in Section~\ref{section.BLEF} and to estimate the $\underline{H}^{-1}_{\mathrm{par}}$-norm of the flux of the corrector in Section~\ref{sec.fluxcorr}.

\begin{proposition}[Gaussian concentration inequality~\cite{borell1975brunn, sudakov1978extremal}] \label{propEfronstein}
Let $X_1 , \ldots, X_n$ be independent Gaussian random variables with expectation $0$ and variance $1$. Let $F : \R^n \to \R$ be a 1-Lipschitz function (i.e., $|F(x) - F(y)| \leq |x - y|$ for all $x , y \in \R^n$ where we used the Euclidian metric on $\R^n$), then there exists an absolute constant $C < \infty$ such that
\begin{equation*}
    \left|  F - \E \left[ F \right] \right| \leq \mathcal{O}_2(C).
\end{equation*}
\end{proposition}

\subsection{Multiscale Poincar\'e inequality}

The multiscale Poincar\'e inequality will be used to estimate the weak $\underline{H}^{-1}_{\mathrm{par}}$-norm of a function in terms of its space-time averages over different scales. We refer to~\cite{AKM} for the result in the elliptic setting. In the parabolic setting, the proof can be found in~\cite[Proposition 3.6]{ABM}.

\begin{proposition}[Multiscale Poincar\'e inequality, Proposition 3.6 of~\cite{ABM}] \label{prop:multscPoinc}
There exists a constant $C < \infty$ depending on $d$ such that, for every $m \geq 1$ and every function $f \in L^2 \left( Q_{3^m}\right)$,
\begin{equation*}
    \left\| f \right\|_{\underline{H}^{-1}_{\mathrm{par}} (Q_{3^m})} \leq C \left\| f \right\|_{\underline{L}^{2} (Q_{3^m})} + C \sum_{k = 0}^m 3^{k}  \left( \frac{1}{|\mathcal{Z}_{k , m}|} \sum_{z \in \mathcal{Z}_{k, m}} \left| (f)_{z + Q_{3^k}} \right|^2 \right)^{\frac 12}.
\end{equation*} 
\end{proposition}

The multiscale Poincar\'e inequality can be combined with the following statement, for which we refer to~\cite[Proposition 3.7]{ABM} (the periodic version stated below can be deduced from their statement).

\begin{proposition}[Proposition 3.7 of~\cite{ABM}] \label{multscalepoincrealpar}
There exists a constant $C < \infty$ depending on $d$ such that, for every $L \in \N$, every $u \in H^{1}(Q_{L})$ with $u(t , \cdot) \in \Omega_{\Lambda_L, \mathrm{per}}$ for any $t \in I_L$ and satisfying $(u)_{Q_L} = 0$, and for every $\g \in L^2(Q_{L})$ satisfying $\partial_t u = \nabla \cdot \g$ , one has the estimate
\begin{equation*}
    \left\| u \right\|_{\underline{L}^2 \left( Q_L \right)} \leq C \left(  \left\| \nabla u \right\|_{\underline{H}^{-1}_{\mathrm{par}} \left( Q_{L} \right)}  +  \left\| \g \right\|_{\underline{H}^{-1}_{\mathrm{par}}(Q_{L}) }  \right).
\end{equation*}
\end{proposition}

The final statement of this section describes more explicitly the structure of the set $\underline{H}^{-1}_{\mathrm{par}} \left( Q_L \right)$. The proof of the result can be found in~\cite[Lemma 3.11]{ABM}.

\begin{lemma}[Identification of $\underline{H}^{-1}_{\mathrm{par}}(Q_L)$] \label{lem.idenH-1par}
There exists a constant $C := C(d) < \infty$ such that, for any $f \in L^{2}(Q_L) $, there exists $(w , w^*) \in L^2(I_L ; H^1_0(\Lambda_L)) \times L^2(I_L ; H^{-1}(\Lambda_L))$ such that
\begin{equation*}
    f = \partial_t w  + w^*,
\end{equation*}
with the upper bound
\begin{equation*}
    \left\| w \right\|_{\underline{L}^2(I_L ; \underline{H}^1(\Lambda_L))} + \left\| w^* \right\|_{\underline{L}^2(I_L ; \underline{H}^{-1}(\Lambda_L))} \leq C \left\| f \right\|_{\underline{H}^{-1}_{\mathrm{par}}(Q_L)}.
\end{equation*}
We can additionally assume that, for any $x \in \Lambda_L$, the map $t \mapsto w(t , x)$ is continuous in the closed interval $[-L^2 , 0]$ and satisfies $w(-L^2, x) = w(0 , x) = 0$ for any $x \in \Lambda_L$. 
\end{lemma}

\section{First-order corrector for the Langevin dynamic} \label{sec.section2}

This section is devoted to the first-order corrector for the Langevin dynamic and is structured as follows. In Section~\ref{section3.1}, we introduce the corrector for the Langevin dynamic. In Section~\ref{section.BLEF}, we estimate the fluctuations of the corrector following the outline presented in Section~\ref{sec.1.5.1}. In Section~\ref{sec.fluxcorr}, we obtain estimates on the flux of the corrector. Finally in Section~\ref{sec.regslope}, we obtain some estimates on the $L^2$-norm of the difference of the first-order correctors with two different slopes.

\subsection{Definition} \label{section3.1}
In this section, we introduce the first-order corrector for the Langevin dynamic. We let $Q := I \times \Lambda$ be  a parabolic cylinder, denote by $I = [s_- , s_+]$, and let $L$ be the sidelength of the box $\Lambda$. We assume that $\left| I \right| = s_+ - s_- \geq L^2$.

\begin{definition}[First-order corrector for the Langevin dynamic] \label{def.firstordercorrfinvol}
Fix a time-dependent slope $q : I \to \Rd$. We define $\varphi_{Q}(\cdot ; q) : Q \to \R$ to be the solution of the Langevin dynamic with periodic boundary conditions
\begin{equation} \label{Langevin.def.corrQ}
    \left\{ \begin{aligned}
    d \varphi_{Q}(t , x ; q) & = \nabla \cdot V'(q + \nabla \varphi_{Q}(\cdot, \cdot ; q)) (t ,x) dt + \sqrt{2} dB_t(x) &~\mbox{for}&~(t,x) \in Q,  \\
    \varphi_{Q}(s , \cdot ; q) &= 0 &~\mbox{for}&~ x \in \Lambda, \\
    \varphi_{Q} (t , \cdot ; q) &\in \Omega_{\Lambda , \mathrm{per}} & \mbox{for}&~ t \in I.
    \end{aligned} \right.
\end{equation}
We then define the first-order corrector $\phi_{Q}(\cdot , \cdot ; q)$ by subtracting a spatially constant term from the map $\varphi$ and write, for any $t \in I$ and $x \in \Lambda$,
\begin{equation} \label{eq:14311610}
    \phi_{Q}(t , x ; q) := \varphi_{Q}(t , x ; q) - \frac{1}{\left| \Lambda \right|}\sum_{y \in \Lambda}\varphi_{Q}(t , y ; q) = \varphi_{Q}(t , x ; q) - \frac{\sqrt{2}}{\left| \Lambda \right|}\sum_{y \in \Lambda}B_t(y).
\end{equation}
The corrector can be equivalently defined as the solution of the system of stochastic differential equations
\begin{equation} \label{eq:14311610+1}
    \left\{ \begin{aligned}
    d \phi_{Q}(t , x ; q) & = \nabla \cdot V'(q + \nabla \phi_{Q}(\cdot, \cdot ; q)) (t ,x) dt + \sqrt{2} dB_t(x) - \frac{\sqrt{2}}{\left| \Lambda \right|} \sum_{y \in \Lambda}  dB_t(y) &~\mbox{for}&~(t,x) \in Q,  \\
    \phi_{Q}(s , \cdot ; q) &= 0 &~\mbox{for}&~ x \in \Lambda, \\
    \phi_{Q} (t , \cdot ; q) &\in \Omega_{\Lambda , \mathrm{per}} & \mbox{for}&~ t \in I,
    \end{aligned} \right.
\end{equation}
The \emph{flux of the corrector} is then the vector field $V'(q + \nabla \phi_{Q}(\cdot, \cdot ; q)).$
\end{definition}

\begin{remark} \label{remark3.2225} 
Contrary to the usual convention in stochastic homogenization, we allow the slope $q$ to vary in the time variable. The reason motivating this choice is twofold:
\begin{enumerate}
    \item First, the technique used to estimate the corrector and its flux covers this case without modifying the argument;
    \item Second, allowing the corrector to have a time-dependent slope is useful to optimize the rate of convergence in Theorem~\ref{Th.quantitativehydr}. As mentioned in Section~\ref{sec1.5.2}, the proof of Theorem~\ref{Th.quantitativehydr} relies on a two-scale expansion and requires to introduce a mesoscopic scale. A natural strategy would be to introduce a mesoscopic scale with respect to both the space and time variables, but it turns out that this causes a deterioration of the rate of convergence, and leads to a rate of the form $\ep^{\frac 13} \bigl(1 + | \log \ep |^{\frac12} \indc_{\{ d = 2\}}\bigr)$. Having access to a corrector with a time-dependent slope allows to define the mesoscopic scale only with respect to the spatial variable and yields the rate of convergence stated in Theorem~\ref{Th.quantitativehydr}.
\end{enumerate}
\end{remark}

\subsection{Scaling estimates for the first-order corrector} \label{section.BLEF}

 The following proposition estimates the size of the first-order corrector of the Langevin dynamic following the outline presented in Section~\ref{sec.1.5.1}. We fix a parabolic cylinder $Q := I \times \Lambda$, denote by $I = [s_- , s_+]$, assume that $\Lambda$ is a box whose sidelength is denoted by $L$ and that $s_+ - s_- \geq L^2$. We also fix a time-dependent slope $q : I \to \Rd$.

\begin{proposition}[Fluctuation estimates for the first-order corrector] \label{prop.2.2BL}
There exists a constant $C := C(d , c_+ , c_-) < \infty$ such that, for any $z \in Q$,
\begin{equation} \label{eq:20291910}
    \left| \phi_{Q} (z ; q) \right| \leq
    \mathcal{O}_2 \left(C \bigl( 1 + (\log L)^{\frac12} \indc_{\{ d = 2\}}\bigr)\right).
\end{equation}
Additionally, for any $z \in Q$,
\begin{equation} \label{eq:20291910grad}
    | \vec{\nabla}\phi_{Q} (z ; q) | \leq
    \mathcal{O}_2 \left(C \right).
\end{equation}
\end{proposition}

\begin{proof}
To ease the notation in the proof, we will assume that $q = 0$, $Q = [s_- , 0 ] \times \Lambda_L$ for some $s_- < -L^2$ and that $z = 0$. We denote the corrector by $\phi$ (i.e., we drop the slope and the cylinder from the notation since they are fixed in the argument). The proof in the general case follows from the same lines.
We prove the estimate
\begin{equation*}
    \left| \phi (0) \right| \leq
    \mathcal{O}_2\left(C \bigl( 1 + (\log L)^{\frac12} \indc_{\{ d = 2\}}\bigr)\right),
\end{equation*}
where we denoted by $0$ the zero of $\R \times \Zd$.
As a preliminary observation, we note that the law of the first-order corrector is invariant under spatial translations. Combining this observation with the identity~\eqref{eq:14311610}, we deduce that: for any $(t,x)\in Q$,
\begin{equation} \label{eq:23311101}
    \E\left[ \phi(t , x)\right] = 0.
\end{equation}
We then consider a discretization of the Brownian motions $\{ B_t(x) \, : \, t \in \R, \, x \in \Zd\}$ in piecewise affine functions defined as follows: for any integer $n \in \N$, any integer $l \in \Z$, and any time $t \in [ l/n , (l+1)/n]$, we set
\begin{equation*}
    B_t^n(x) := \left(n t - l \right)  B_{\frac{k+1}{n}}(x) + \left(l+1 - n t\right) B_{\frac{k}{n}}(x) .
\end{equation*}
We denote by $X_l^n(x)$ the normalized increments of the Brownian motion $B(x)$ between the times $l/n$ and $(l+1)/n$, that is
\begin{equation*}
    X_l^n(x) :=  \sqrt{n} \left( B_{\frac{l+1}{n}}(x) - B_{\frac{l}{n}}(x) \right) \hspace{5mm} \mbox{and} \hspace{5mm} X^n(t , x) = \sum_{k \in \Z} X_l^n(x) \indc_{\{l/n \leq t \leq (l+1)/n\}} ,
\end{equation*}
and note that the random variables $\left\{ X_k^n(x) \, : \, l \in \Z , \, x \in \Zd \right\}$ are i.i.d. Gaussian with expectation $0$ and variance $1$. We denote by $\bar X^n(t) := \frac{1}{\left| \Lambda_L\right|}\sum_{y \in \Lambda_L} X^n(t , y)$ (and assume that this term is spatially constant), and let $\phi^n : Q \to \R$ be the solution of the system of ordinary differential equations
\begin{equation} \label{eq:09272709}
\left\{ 
\begin{aligned}
  & \partial_t \phi^n  = \nabla \cdot V'( \nabla \phi^n)  + \sqrt{2n^{-1}} \left( X^n - \bar X^n \right) & \mbox{in} & \ Q,  \\
  &  \phi^n(s , \cdot )= 0 & \mbox{in} &\,  \Lambda_{L}, \\
  &  \phi^n (t , \cdot)  \in \Omega_{\Lambda_L , \mathrm{per}} & \mbox{for} & \ t \in I.
    \end{aligned} \right.
\end{equation}
We note that, sending $n$ to infinity, the functions $\phi^n$ converges to the solution of the stochastic differential equation~\eqref{eq:14311610+1}, and thus
\begin{equation} \label{eq:21502609}
    \phi^n(0) \underset{n \to \infty}{\longrightarrow} \phi(0) \hspace{3mm} \mathbb{P}-\mbox{almost-surely};
\end{equation}
and we then establish that there exists a constant $C(d, c_+ , c_-) < \infty$ such that, for any $n \in \N$,
\begin{equation} \label{eq:20152609}
    \left| \phi^n (0) \right| \leq
    \mathcal{O}_2\left(C \bigl( 1 + (\log L)^{\frac12} \indc_{\{ d = 2\}}\bigr)\right).
\end{equation}
A combination of~\eqref{eq:21502609} and~\eqref{eq:20152609} with the definition of the $\mathcal{O}$-notation and Fatou's Lemma completes the proof of Proposition~\ref{prop.2.2BL}.

We next prove the stochastic integrability estimate~\eqref{eq:20152609}. We first note that, for each $n \in \N$, the function $\phi^n$ depends only on the values of the increments 
$$\left\{  X_l^n(y) \, : \, l \in \{ \lfloor n s_- \rfloor , \ldots , 0\} , y \in \Lambda_L \right\}.$$
Using that the increments of the Brownian motions are independent Gaussian random variables and the observation $\E[\phi^n(0)] =0$ (which follows from the same argument as in~\eqref{eq:23311101}), we see that Proposition~\ref{prop.2.2BL} is a consequence of Proposition~\ref{propEfronstein} if we can prove that, for a large constant $C$ depending only on $d , c_+ , c_-$, the mapping
\begin{equation*}
    F : \left\{  X_l^n(y) \, : \, l \in \{ \lfloor n s \rfloor , \ldots , 0\} , y \in \Lambda_L \right\} \mapsto \frac{\phi^n (0)}{C( 1 + (\log L)^{\frac12} \indc_{\{d = 2\}} )}
\end{equation*}
is $1$-Lipschitz (in the sense of Proposition~\ref{propEfronstein}). To this end, let us fix $l \in \Z$ and $y \in \Zd$, and note that the derivative $w := \partial \phi^n_L/\partial  X_l^n(y)$ solves the linear parabolic equation
\begin{equation} \label{eq:10282709}
\left\{ \begin{aligned}
   \partial_t w(t , x)  - \nabla \cdot \a \nabla w(t , x) & =  \sqrt{2 n} \indc_{\left\{ t \in \left[\frac{l}{n} , \frac{l+1}{n} \right] \right\}} \left( \indc_{\{ x = y\}} - \frac{1}{\left| \Lambda_L \right|} \right) &~\mbox{in} &~Q,  \\
    w(s , \cdot) &= 0 &~\mbox{in} &~ \Lambda_{L}, \\
    w(t , \cdot) &\in \Omega_{\Lambda_L , \mathrm{per}} & \mbox{for} &~ t \in I,
    \end{aligned} \right.
\end{equation}
with the environment $\a(t,e) = V'' (\nabla  \phi_{L}^{n} (t , e) )$.
Applying Duhamel's principle (Proposition~\ref{propDuhamel}), we obtain the following identity
\begin{equation} \label{eq:15502709}
     w (0) = \sqrt{2 n} \int_{l/n}^{(l+1)/n} P_\a \left( t , x ; s , y \right) \, ds,
\end{equation}
where $P_\a$ is the heat-kernel defined in~\eqref{def.heatkernelperiodic} with respect to the environment $\a$. Using the Nash-Aronson estimate stated in Proposition~\ref{prop.NashAronson} and recalling the definition of the map $\Phi_{C,L}$ stated in~\eqref{def.phiCL}, we obtain   
\begin{equation} \label{eq:11482709}
     |w (0)| \leq C n^{\frac12} \int_{l/n}^{(l+1)/n} \Phi_{C, L } \left( - s , y \right) \, ds + \frac{C}{n^{\frac 12} L^d} \indc_{\{ \frac{l}{n}  \leq L^2 \}}.
\end{equation}
Summing the inequality~\eqref{eq:11482709} over the integers $l \in n I \cap \Zd$ and over the vertices $y \in \Lambda_L$, we deduce that
\begin{equation*}
    \sum_{\substack{l \in n I \cap \Z \\ y \in \Lambda_L } } \left( \frac{\partial \phi^n_L(0)}{\partial  X_l^n(y)} \right)^2  \leq  \sum_{\substack{l \in n I \cap \Z \\ y \in \Lambda_L } } C n \left( \int_{l/n}^{(l+1)/n} \Phi_{C,L} \left( -s , y\right) \, ds  \right)^2 + C L^{-d+2}.
\end{equation*}
There remains to estimate the term in the right-hand side. By the Cauchy-Schwarz inequality, we have
\begin{equation*}
    \sum_{l \in n I \cap \Z } C n \left( \int_{l/n}^{(l+1)/n} \Phi_{C,L} \left( - s , y\right) \, ds \right)^2 \leq C \int_{I} \Phi_{C, L} \left(  - s , y\right)^2 \, ds  \leq \frac{C}{|y|_+^{2d -2}}.
\end{equation*}
where we recall the notation $\left| \cdot \right|_+ := \max \left( \left| \cdot \right| , 1 \right)$.
Summing over the vertices $y \in \Lambda_L$, we see that
\begin{equation*}
    \sum_{y \in \Lambda_L} \frac{C}{|y|_+^{2d -2}} \leq C (1  +  (\log L )\indc_{\{ d = 2 \}}).
\end{equation*}
A combination of the three previous displays
shows that
\begin{equation*}
    \left| \nabla F  \right|^2 = \frac{1}{ C (1  +  (\log L) \indc_{\{ d = 2 \}})}\sum_{\substack{l \in n I \cap \Z \\ y \in \Lambda_L } } \left( \frac{\partial \phi^n_L(0)}{\partial  X_l^n(y)} \right)^2 \leq 1.
\end{equation*}
Applying Proposition~\ref{propEfronstein} completes the proof of~\eqref{eq:20291910}.

The proof of~\eqref{eq:20291910grad} follows a similar outline: using the identity $\partial \vec{\nabla} \phi^n_L/\partial  X_l^n(y) = \vec{\nabla} w$ (where $w$ is the solution of the parabolic equation~\eqref{eq:10282709}), we may use the same argument. The only difference is that we use the estimate~\eqref{eq:prop.NashAronsongrad} on the gradient of the heat kernel instead of~\eqref{eq:prop.NashAronson} to remove the logarithm in two dimensions.
\end{proof}

\subsection{Decay of the spatial average of the flux} \label{sec.fluxcorr}

In this section, we obtain optimal estimates on the fluctuations of the space-time averaged value of the gradient of the corrector and its flux. The proof follows a similar outline to the proof of Proposition~\ref{prop.2.2BL} (and is based on a discretization of the Brownian motions and the Gaussian concentration inequality). However, it is technically more involved due to the nature of the observable considered. In particular, it makes use of the Caccioppoli inequality, the De Giorgi-Nash-Moser regularity, the Nash-Aronson estimate and a decomposition over the scales. We recall the definition~\eqref{defaveragevector} for the average value of a vector field over a parabolic cylinder (in particular, we recall that this quantity is valued in $\Rd$). As in the previous section, we fix a parabolic cylinder $Q := I \times \Lambda$ and a time-dependent slope $q : I \to \Rd$.

\begin{proposition}[Decay of the space-times averages of the gradient of the corrector and the flux]\label{prop.spataveflux}
   There exists a constant $C := C(d , c_+ , c_-) < \infty$, such that, for any integer $\ell \in \N$ and any vertex $z \in Q$ satisfying $z + Q_\ell \subseteq Q$,
    \begin{equation} \label{eq:14590112}
    \left| \left( V'\left( q + \nabla \phi_Q(\cdot, \cdot ; q) \right) \right)_{z + Q_\ell} - \E \left[  \left(  V'\left(q + \nabla \phi_Q(\cdot, \cdot ; q) \right)  \right)_{z + Q_\ell} \right] \right| \leq \mathcal{O}_2 \bigl( C \ell^{-\frac d2} \bigr)\,
    \end{equation}
    and
    \begin{equation} \label{eq:15000112}
    \left| \left( \nabla \phi_Q(\cdot,\cdot ; q) \right)_{z + Q_\ell} \right| \leq \mathcal{O}_2 \bigl( C \ell^{-\frac d2} \bigr)\,.
    \end{equation}
\end{proposition}

\begin{proof}
As in the proof of Proposition~\ref{prop.2.2BL}, and to ease the notation in the argument, we assume that the cylinder $Q$ is of the form $Q := I \times \Lambda_L$ with $I := (s_- , 0)$ and $s_- < - L^2$, that $q = 0$ and $z = 0$. We recall the definitions of the discretized Brownian motions $B^n$ and the dynamic $\phi^n$. We will prove the estimate: for any $n \in \N$,
\begin{equation*}
    \Bigl| \left( V'\left(\nabla \phi^n  \right) \right)_{Q_\ell} - \E \bigl[ \left( V'(\nabla \phi^n\right)_{Q_\ell} \bigr] \Bigl| \leq \mathcal{O}_2 \bigl( C \ell^{-\frac d2} \bigr).
\end{equation*}
By the Gaussian concentration inequality (Proposition~\ref{propEfronstein}), it is sufficient to prove the upper bound
\begin{equation} \label{eq:16142809}
     \sum_{\substack{l \in n I \cap \Z \\ y \in \Lambda_L}} \biggl( \frac{\partial \left( V'\left(\nabla \phi^n  \right) \right)_{Q_\ell}}{\partial X^{n}_l(y)}  \biggr)^2 \leq C\ell^{-d}.
\end{equation}
We fix $l , y$ and recall the notation for the map $w := \partial \phi^n_L/\partial X^{n}_l(y)$, the environment $\a$ and the heat kernel~$P_\a$ introduced in the proof of Proposition~\ref{prop.2.2BL}. We first claim that the following identity holds
\begin{equation} \label{eq:3.155}
    \frac{\partial \left( V'\left(\nabla \phi^n  \right) \right)_{Q_\ell}}{\partial X^{n}_l(y)} =  \sqrt{2n}  \int_{l/n}^{(l+1)/n} \left( \a \nabla P_\a (\cdot ; s , y) \right)_{Q_\ell} \, ds.
\end{equation}
Recalling the definition of the average value of a vector field over a parabolic cylinder stated in~\eqref{defaveragevector}, we have, for any integer $i \in \{ 1 , \ldots, d\}$,
\begin{equation*}
    \left( V'(\nabla \phi^n) \right)_{Q_\ell,i} := \frac{1}{\left| Q_\ell\right|} \sum_{x \in \Lambda_\ell} \int_{-\ell^2}^0 V'(\nabla \phi^n(t, (x , x+e_i))) \, dt.
\end{equation*}
Differentiating both sides of the identity by $ X^{n}_l(y)$, we obtain
\begin{equation*}
    \frac{\partial \left( V'(\nabla \phi^n) \right)_{Q_\ell,i}}{\partial X^{n}_l(y)} = \frac{1}{\left| Q_\ell\right|} \sum_{x \in \Lambda_\ell} \int_{-\ell^2}^0 \frac{\partial V'(\nabla \phi^n(t, (x , x+e_i)))}{\partial X^{n}_l(y)} \, dt = \left( \frac{\partial  V'(\nabla \phi^n) }{\partial X^{n}_l(y)} \right)_{Q_\ell,i}.
\end{equation*}
Since the previous identity is valid for any $i \in \{1 , \ldots, d\}$, we obtain
\begin{equation*}
    \frac{\partial \left( V'(\nabla \phi^n) \right)_{Q_\ell}}{\partial X^{n}_l(y)} = \left( \frac{\partial  V'(\nabla \phi^n) }{\partial X^{n}_l(y)} \right)_{Q_\ell}.
\end{equation*}
We next note that, for any time $t \in I$ and any edge $e \in \vec E(\Lambda_L)$, the chain rule implies
\begin{equation*}
     \frac{\partial  V'\left(\nabla \phi^n (t , e)  \right)}{\partial X^{n}_l(y)} = V''(\nabla \phi^n (t , e)) \nabla w(t,e). 
\end{equation*}
Using the definition of the coefficient $\a$ (below~\eqref{eq:10282709}) and the identity~\eqref{eq:15502709} for the function $w$, we may rewrite the previous display as follows
\begin{align*}
        \frac{\partial  V'\left(\nabla \phi^n (t , e) \right)}{\partial X^{n}_l(y)} & = \a(t , e)  \sqrt{2 n} \int_{l/n}^{(l+1)/n} \nabla P_\a \left( t , e ; s , y \right) \, ds \\
        & = \sqrt{2 n} \int_{l/n}^{(l+1)/n} \a(t , e) \nabla P_\a \left( t , e ; s , y \right) \, ds. \notag
\end{align*}
Combining the previous identities, we see that
\begin{align*}
    \frac{\partial \left( V'(\nabla \phi^n) \right)_{Q_\ell}}{\partial X^{n}_l(y)} & = \left( \frac{\partial  V'(\nabla \phi^n) }{\partial X^{n}_l(y)} \right)_{Q_\ell} \\
    & =  \left( \sqrt{2 n} \int_{l/n}^{(l+1)/n} \a(t , e) \nabla P_\a \left( t , e ; s , y \right) \, ds \right)_{Q_\ell} \\
    & = \sqrt{2 n} \int_{l/n}^{(l+1)/n} \left( \a(t , e) \nabla P_\a \left( t , e ; s , y \right) \right)_{Q_\ell} \, ds,
\end{align*}
where, in the last identity, we commuted the integral over the parameter $s$ with the integral over $t$ and the sum over the edges $e$. This is exactly~\eqref{eq:3.155}.

From~\eqref{eq:3.155}, we deduce that
\begin{equation} \label{eq:14442809}
   \Biggl| \frac{\partial \bigl( V'(\nabla \phi^n  ) \bigr)_{Q_\ell}}{\partial X^{n}_l(y)} \Biggr| \\ \leq C n^{\frac12}  \int_{l/n}^{(l+1)/n} \left\|  \nabla P_\a \left(\cdot ; s , y \right)  \right\|_{\underline{L}^1\left( Q_\ell \right)} \, ds.
\end{equation}
We next estimate the term in the right-hand side. To this end, we denote by $\alpha$ the exponent of the De Giorgi-Nash-Moser regularity (Proposition~\ref{paraNash}), and, for $C < \infty$, introduce the map $\Psi_{C, \ell , L} : Q \to [0 , \infty)$ defined by the formula
\begin{equation} \label{def.psinL0112}
    \Psi_{C, \ell , L}(s  , y) :=  \left\{ \begin{aligned} 
   & \left( \left( \frac{\ell}{s^{\frac12}} \right)^\alpha \wedge 1 \right) \frac{\Phi_{C,L} (s, y)}{\ell} & \mbox{if} & \ s \notin I_{4\ell}, \\
  &  C \ell^{-d-1} \exp \left( -  \frac{|y|}{C \ell}\right) & \mbox{if} & \ s \in I_{4 \ell}.
    \end{aligned} \right.
\end{equation}
We will prove the following upper bound: there exists a constant $C:= C(d , c_+ , c_-) < \infty$ such that, for any $(s , y) \in Q$,
\begin{equation} \label{eq:09592809}
    \left\|  \nabla P_\a \left(\cdot ; s , y \right)  \right\|_{\underline{L}^1\left( Q_\ell \right)} \leq 
    \Psi_{C, \ell , L}(s, y).
\end{equation}
To prove the inequality~\eqref{eq:09592809}, the strategy is to apply the Caccioppoli inequality combined with the De Giorgi-Nash-Moser regularity and the Nash-Aronson estimate. 
The implementation of the argument faces the following technicality (see Case 2 below): to apply the Caccioppoli inequality in the parabolic cylinder $Q_\ell$, we need the heat kernel $P_\a$ to be a solution of the parabolic equation~\eqref{def.paraboliceqautionheat-kernel} in the cylinder $Q_{2 \ell}$, which is not the case if the point $s$ belongs to the interval $I_{2\ell}$. To overcome this issue, we distinguish two cases: whether the point $s$ belongs to the interval $I_{4\ell}$ or not (we use the multiplicative factor $4$ instead of $2$ for technical convenience). In the first case, the argument is rather straightforward, in the second case, we construct a sequence of scales, denoted by $\ell_0 , \ldots, \ell_K$ and defined in~\eqref{def.ellkk}, and define a covering of the parabolic cylinder $Q_\ell$ by a collection of smaller parabolic cylinders, denoted in the argument below by (see~\eqref{eq:15192010})
\begin{equation} \label{defpseudipartP}
\mathcal{P} := \left\{ z + Q_{\ell_k} \, : \,  k \in \{0 , \ldots, K\}, \, z \in \mathcal{Z}_k \right\},
\end{equation}
and satisfying the following property:
\begin{equation} \label{eq:11172809}
    \forall  \, \tilde Q \in \mathcal{P}, \hspace{5mm} \tilde Q = \tilde I \times \tilde \Lambda \implies s \notin 2 \tilde I .
\end{equation}
We are then able to apply the Caccioppoli inequality and the Nash-Aronson estimate on each cylinders of the partition $\mathcal{P}$ to prove the result.

\medskip

\textbf{Case 1: $s \notin I_{4 \ell}$.} In this setting, we apply the Cauchy-Schwarz inequality, then the Caccioppoli inequality and then the De Giorgi-Nash-Moser regularity estimate. We obtain
\begin{align*}
     \left\|  \nabla P_\a \left(\cdot ; s , y \right)  \right\|_{\underline{L}^1\left(Q_{\ell} \right)} & \leq  \left\|  \nabla P_\a \left(\cdot ; s , y \right)  \right\|_{\underline{L}^2\left(Q_{\ell} \right)} \\
     & \leq \frac{C}{\ell} \left\|  P_\a \left(\cdot ; s , y \right) - ( P_\a \left(\cdot ; s , y \right) )_{Q_{2\ell}} \right\|_{\underline{L}^2\left( Q_{2\ell} \right)} \\
     & \leq \frac{C}{\ell}  \ell^\alpha \left[ P_\a \left(\cdot ; s, y \right) \right]_{C^{0 , \alpha}(Q_{2\ell})} \\
     & \leq \frac{C}{\ell}  \left( \frac{\ell}{s^{\frac12} \wedge L} \right)^\alpha \left\| P_\a \left(\cdot ; s , y \right) -  ( P_\a \left(\cdot ; s , y \right) )_{Q_{\sqrt{s / 2} \wedge L}} \right\|_{\underline{L}^2\left(Q_{\sqrt{s / 2} \wedge L}\right)}.
\end{align*}
Applying the Nash-Aronson estimate, we deduce that
\begin{equation*}
     \left\|  \nabla P_\a \left(\cdot ; s , y \right)  \right\|_{\underline{L}^1\left(Q_{\ell} \right)} \leq  \frac{C}{\ell}  \left( \frac{\ell}{s^{\frac12} \wedge L} \right)^\alpha \left\|\Phi_{C,L} \left(\cdot - s ; \cdot - y \right) \right\|_{\underline{L}^2\left(Q_{\sqrt{s / 2} \wedge L}\right)}.
\end{equation*}
Using the assumption $s \notin I_{4\ell}$ together with an explicit computation based on the formula~\eqref{def.phiCL} for the map $\Phi_{C,L}$, we obtain the upper bound
\begin{equation*}
    \left\|  \Phi_{C,L} \left(\cdot  - s , \cdot - y\right)  \right\|_{\underline{L}^2\left( Q_{2\ell} \right)} \leq \Phi_{C,L} (s , y),
\end{equation*}
where we increased the value of the constant $C$ in the right-hand side. A combination of the two previous displays yields
\begin{equation} \label{eq:influenceflux1}
    \left\|  \nabla P_\a \left(\cdot ; s , y \right)  \right\|_{\underline{L}^1\left(Q_{\ell} \right)} \leq \frac{C}{\ell}  \left( \frac{\ell}{s^{\frac12} \wedge L} \right)^\alpha \Phi_{C,L} (s , y).
\end{equation}
Finally, if $s \geq L^2$, then the definition of the map $\Phi_{C,L}$ and the inequality $L^{-\alpha} \exp( - s / (C L^2)) \leq C' s^{-\alpha/2} \exp( - s / (C' L^2))$, valid for some $C' > C$, imply
\begin{equation*}
     \frac{1}{L^\alpha} \Phi_{C,L} (s , y) \leq \frac{1}{s^{\alpha/2}}  \Phi_{C,L} (s , y),
\end{equation*}
where we have increased the value of the constant $C$ in the right-hand side. Combining the two previous inequalities implies~\eqref{eq:09592809}.

\medskip

\textbf{Case 2: $s \in I_{4 \ell}$.} We first partition the time interval $I_\ell$ by constructing an increasing sequence of times $(t_k)_{k \geq 0}$ defined as follows
\begin{equation*}
    t_0 := \max\{ -\ell^2, s\} \hspace{5mm} \mbox{and} \hspace{5mm}
    t_{k+1} :=  t_{k} + \frac{(t_k - s)}{15} \vee 1.
\end{equation*}
We let $K \in \N$ be the smallest integer such that $t_K \geq 0$, and define, for each $k \in \{ 0 , \ldots, K-1 \}$,
\begin{equation} \label{def.ellkk}
    \ell_k := \sqrt{t_{k+1} - t_k}
\end{equation}
and the set
\begin{equation} \label{eq:15192010}
    \mathcal{Z}_k := \left\{ (t_{k+1}, x) \, : \, x \in \ell_k \Zd \cap \Lambda_\ell \right\}.
\end{equation}
For later use, we note that the scales $\ell_k$ grow exponentially in $k$. Consequently, there exists a numerical constant $C < \infty$ such that
\begin{equation*}
    \sum_{k=0}^K \ell_k \leq C \ell.
\end{equation*}
By construction, the set $\mathcal{P}$ defined in~\eqref{defpseudipartP} is a covering of the cylinder $Q_\ell$ satisfying~\eqref{eq:11172809}. We thus have
\begin{equation} \label{eq:12562809}
    \left\|  \nabla P_\a \left(\cdot ; s , y \right)  \right\|_{\underline{L}^1\left( Q_\ell \right)} \leq \frac{1}{\left| Q_\ell \right|} \sum_{k = 1}^K \sum_{z \in \mathcal{Z}_k} \left\|  \nabla P_\a \left(\cdot ; s , y \right)  \right\|_{L^1\left( z + Q_{\ell_k} \right)}.
\end{equation}
We next estimate the terms in the right-hand side by first applying the Cauchy-Schwarz inequality, then the Caccioppoli inequality, and then the Nash-Aronson estimate. We obtain, for any $k \in \{ 0 , \ldots , K\}$ and any $z \in \mathcal{Z}_k$,
\begin{align} \label{eq:28091232}
    \left\|  \nabla P_\a \left(\cdot ; s , y \right)  \right\|_{L^1\left(z + Q_{\ell_k} \right)} & \leq \left| Q_{\ell_k}\right| \left\| \nabla P_\a \left(\cdot ; s , y \right)  \right\|_{\underline{L}^2\left(z + Q_{\ell_k} \right)} \\
     & \leq \frac{C \left| Q_{\ell_k} \right|}{\ell_k} \left\|  P_\a \left(\cdot ; s , y \right) - \left( P_\a \left( \cdot ; s , y \right) \right)_{z + Q_{2\ell_k}} \right\|_{\underline{L}^2\left( z + Q_{2\ell_k} \right)} \notag \\
     & \leq \frac{C \left| Q_{\ell_k} \right|}{\ell_k} \left\|  \Phi_{C,L} \left(\cdot  - s , \cdot - y\right)  \right\|_{\underline{L}^2\left( z + Q_{2\ell_k} \right)}. \notag
\end{align}
Using the property~\eqref{eq:11172809} of the cylinders of the set $\mathcal{P}$ and an explicit computation based on the formula for the map $\Phi_{C , L}$, we deduce that, for any $k \in \{ 0 , \ldots , K\}$ and any $z = (t_{k+1} , x) \in \mathcal{Z}_k$,
\begin{equation} \label{eq:28091233}
    \left\|  \Phi_{C , L} \left(\cdot  - s , \cdot - y\right)  \right\|_{\underline{L}^2\left( z + Q_{2\ell_k} \right)} \leq \Phi_{C,L} \left(t_{k+1}  - s , x - y\right),
\end{equation}
where we have increased the value of the constant $C$ in the right-hand side. Summing over the points $x \in \ell_k \Zd \cap \Lambda_\ell$, we deduce that
\begin{align} 
\label{eq:28091234}
\sum_{x \in \ell_k \Zd \cap \Lambda_\ell} \Phi_{C,L} \left(t_{k+1}  - s , x - y\right) & \leq  \sum_{x \in \ell_k \Zd \cap \Lambda_\ell} \Phi_{C,L} \left(t_{k+1}  - s , x - y\right)  \\
& \leq C \ell_k^{-d} \sum_{x \in \Zd \cap \Lambda_\ell} \Phi_{C,L} \left(t_{k+1}  - s , x - y\right) \notag \\
& \leq C \ell_k^{-d} \exp \left( - \frac{|y|}{C \ell} \right). \notag
\end{align}
Using that $\left| Q_{\ell_k} \right| = \ell_k^{2} (2 \ell_k + 1 )^d$ and combining the inequalities~\eqref{eq:28091232},~\eqref{eq:28091233} and~\eqref{eq:28091234}, we obtain
\begin{equation*}
     \sum_{z \in \mathcal{Z}_k}\left\|  \nabla P_\a \left(\cdot ; s , y \right)  \right\|_{L^1\left(z + Q_{\ell_k} \right)} \leq C \ell_k \exp \left( - \frac{|y|}{C \ell} \right).
\end{equation*}
Summing over the scales $k \in \{0 , \ldots , K \}$ and using the inequality~\eqref{eq:12562809}, we deduce that
\begin{equation*}
    \sum_{k = 1}^K \sum_{z \in \mathcal{Z}_k} \left\|  \nabla P_\a \left(\cdot ; s , y \right)  \right\|_{L^1\left( z + Q_{\ell_k} \right)} \leq C \exp \left( - \frac{|y|}{C \ell} \right) \sum_{k = 1}^K \ell_k \leq C \ell \exp \left( - \frac{|y|}{C \ell} \right). 
\end{equation*}
Dividing both sides of the previous display by the volume factor $\left| Q_\ell \right|$ and using~\eqref{eq:12562809}, we obtain
\begin{equation} \label{eq:influenceflux2}
    \left\|  \nabla P_\a \left(\cdot ; s , y \right)  \right\|_{\underline{L}^1\left( Q_\ell \right)} \leq \frac{C}{\ell^{d+1}} \exp \left( - \frac{|y|}{C \ell} \right).
\end{equation}
In the setting where the point $s$ belongs to the time interval $I_{4 \ell}$, the inequality~\eqref{eq:influenceflux2} implies~\eqref{eq:09592809}. Combining~\eqref{eq:influenceflux1} and~\eqref{eq:influenceflux2} completes the proof of~\eqref{eq:09592809}.

We now complete the proof of Proposition~\ref{prop.spataveflux} using the inequality~\eqref{eq:09592809}. Combining~\eqref{eq:14442809} and~\eqref{eq:09592809} with the Cauchy-Schwarz inequality, we see that
\begin{equation*}
    \sum_{\substack{l \in n I \cap \Z \\ y \in \Lambda_L}} \biggl( \frac{\partial \left( V'\left(\nabla \phi^n  \right) \right)_{Q_\ell}}{\partial X^{n}_l(y)}  \biggr)^2
     \leq  C \int_{I} \sum_{y \in \Lambda_L}\Psi_{C, \ell , L}(s   , y)^2 \, ds .
\end{equation*}
The term in the right-hand side can then be explicitly computed and we have, for any $s \in I$,
\begin{equation*}
    \sum_{y \in \Lambda_L}\Psi_{C, \ell , L}(s   , y)^2 \leq C \biggl( \biggl( \frac{\ell}{s^{\frac12}}\biggr)^{d + 2\alpha} \wedge 1 \biggr) \frac{1}{\ell^{d+2}}.
\end{equation*}
Integrating over the time interval $I$, we obtain
\begin{equation*}
   \int_{I} \sum_{y \in \Lambda_L} \Psi_{C, \ell , L}(s   , y)^2 \, ds \leq C \int_{I} \biggl( \biggl( \frac{\ell}{s^{\frac12}}\biggr)^{d + \alpha} \wedge 1 \biggr) \frac{1}{\ell^{d+2}} \, ds \leq C\ell^{-d}.
\end{equation*}
Combining the three previous displays completes the proof of the inequality~\eqref{eq:16142809}, and thus the proof of Proposition~\ref{prop.spataveflux}. The proof of the inequality~\eqref{eq:15000112} follows from a notational modification of~\eqref{eq:14590112}, and we thus omit the details.
\end{proof}

\subsubsection{Spatial average of the flux corrector with a cutoff function} \label{sec.section3.3.1}

In this section, we state and prove a slightly more general version of Proposition~\ref{prop.spataveflux} which will be useful in order to obtain the rate of convergence in the proof of Theorem~\ref{Th.quantitativehydr}. To state the exact result, we consider a parabolic cylinder~$Q$ of the form $Q := I \times \Lambda_L$ where the interval $I$ is of the form $I = (s_- , 0)$ with $s_- \leq -L^2$. We denote by $\phi$ the Langevin dynamic defined in Definition~\ref{def.firstordercorrfinvol} in the cylinder $Q$ with slope $q = 0$, and fix a function $\chi : \Lambda_L \to \R$ satisfying
\begin{equation} \label{def.cutoffchi}
    0  \leq \chi \leq \indc_{\Lambda_L} \hspace{5mm} \mbox{and} \hspace{5mm} \left\| \nabla \chi \right\|_{L^\infty (\Lambda_L)} \leq CL^{-1}\,.
\end{equation}
We then consider the map, defined on the set $I \times \Lambda_L $, 
\begin{equation*}
    \nabla \chi \cdot V'(\nabla \phi) : (t , x) \mapsto \sum_{\substack{e \in \vec{E}\left( D^\ep\right) \\ e = (x , y')}}\nabla \chi(e) V'\left(\nabla \phi(t , e) \right).
\end{equation*}
The purpose of the following proposition is to study the fluctuations of the average value of the map $\nabla \chi \cdot V'(\nabla \phi)$ over parabolic cylinders of different sizes. Two cases will be of interest for us:
\begin{itemize}
    \item Parabolic cylinders of small sizes: typically cylinders of the form $Q_\ell$ with $\ell \leq L$, in that case the result is essentially a consequence of Proposition~\ref{prop.spataveflux} and the properties of the map $\chi$ (the pointwise bound on its gradient);
    \item Parabolic cylinders of large sizes: typically cylinders of the form $Q_\ell$ with $\ell \geq L$, in that case the result is obtained by making use of the definition of the flux combined with an integration by parts. We emphasize that this case cannot be deduced Proposition~\ref{prop.spataveflux} and contains additional information on the behavior of the flux of the corrector, which is then crucial in order to obtain the optimized rate of convergence in Theorem~\ref{Th.quantitativehydr}.
\end{itemize}

The strategy of the proof follows the one of Proposition~\ref{prop.2.2BL} and Proposition~\ref{prop.spataveflux}: we discretize the Brownian motions and compute the derivative of the dynamic with respect to the discrete increments $X_k^n(y)$.
As it will be useful to us to preserve some information, we state the result in terms of an upper bound of the derivative of the observable $\left( \nabla \chi \cdot V'\left(\nabla \phi^n  \right) \right)_{Q_\ell}$ with respect to the increment $X_k^n(y)$. Remark~\ref{remark3.6} below states the concentration estimate which would be obtained by combining the result of Proposition~\ref{prop.prop3.5techni} with the Gaussian concentration inequality. We recall the notation for the map $\Psi_{C,\ell, L}$ defined in~\eqref{def.psinL0112}, and that the dynamic is defined on a parabolic cylinder of the form $Q = I \times \Lambda_L = [s_- , 0] \times \Lambda_L$ with slope $q = 0$.

\begin{proposition} \label{prop.prop3.5techni}
There exists a constant $C := C(d , c_+ , c_-) < \infty$, such that:
\begin{itemize}
    \item On small cylinders: For any $\ell \leq L$, any $n \in \N$, any $l \in \{ \lfloor n s_- \rfloor , \ldots , 0\}$ and any $y \in \Lambda_L$,
    \begin{equation} \label{eq:16470112}
    \left| \frac{\partial \left( \nabla \chi \cdot V'\left(\nabla \phi^n  \right) \right)_{Q_\ell}}{\partial X^n_l(y)} \right| \leq  C n^{\frac12} L^{-1} \Psi_{C, \ell, L}(l/n, y).
    \end{equation}
    \item On large cylinders: for any $\ell \in [L , \sqrt{|s_-|}]$, any $n \in \N$, any $l \in \{ \lfloor n s_- \rfloor , \ldots , 0\}$ and any $y \in \Lambda_L$,
        \begin{equation} \label{eq:16480112}
     \biggl| \frac{\partial \left( \nabla \chi \cdot V'\left(\nabla \phi^n  \right) \right)_{Q_\ell}}{\partial X_l^n(y)} \biggr| \leq C n^{-\frac12} \ell^{-d-2} \exp \left( - \frac{(l/n)}{C \ell^2} \right).
\end{equation}
\end{itemize}
\end{proposition}

\begin{remark} \label{remark3.6}
As mentioned above, the reason we state the result at this level of precision is that it is used to obtain a sharper rate of convergence in Theorem~\ref{Th.quantitativehydr}. The proof follows the same strategy as the one of Proposition~\ref{prop.2.2BL} and Proposition~\ref{prop.spataveflux}, and the argument is in fact straightforward once Proposition~\ref{prop.spataveflux} has been established.
\end{remark}

\begin{remark}
The proof is written for the corrector defined on a specific class of parabolic cylinders and with a slope constant equal to $0$ for notational convenience. The same argument would apply to correctors defined on a general parabolic cylinder and with a general time-dependent slope $q : I \to \R$. The only constraint in the extension is the following: if we denote by $Q = I \times \Lambda$ the domain of the corrector, and by $\tilde Q = \tilde I \times \tilde \Lambda$ the large cylinder in~\eqref{eq:16480112}, then one must have $\Lambda \subseteq \tilde \Lambda$ (so that one can integrate by parts in the argument).
\end{remark}

\begin{remark}
Combining the result of Proposition~\ref{prop.prop3.5techni} with the Gaussian concentration inequality, one obtains the fluctuation estimates:
\begin{itemize}
\item On small cylinders ($\ell \leq L$): 
\begin{equation*}
\left| \left( \nabla \chi \cdot V'\left(\nabla \phi^n  \right) \right)_{Q_\ell} - \E \bigl[ \left( \nabla \chi \cdot V'(\nabla \phi^n\right)_{Q_\ell} \bigr] \right| \leq \mathcal{O}_2 \left( C L^{-1} \ell^{-d/2} \right);
\end{equation*}
\item On large cylinders ($\ell \geq L$): we first remark that since the map $\chi$ is supported in the box $\Lambda_L$, we have $\left( \nabla \chi \right)_Q = 0$. Since, by definition, the field $\phi^n$ is spatially stationnary (thus its expectation does not depend on the spatial variable), and since the map $\chi$ is deterministic, we have the identity
\begin{equation*}
	\E \left[ \left( \nabla \chi \cdot V'\left(\nabla \phi^n  \right) \right)_{Q_\ell} \right] =  \left( \nabla \chi \cdot \E \left[ V'\left(\nabla \phi^n  \right)  \right] \right)_{Q_\ell} = 0.
\end{equation*}
The Gaussian concentration inequality gives in this case
\begin{equation*}
\left| \left( \nabla \chi \cdot V'\left(\nabla \phi^n  \right) \right)_{Q_\ell} \right| \leq \mathcal{O}_2 \left( C \ell^{-d/2 +1} \right).
\end{equation*}
\end{itemize}
\end{remark}

\begin{proof}
We first treat the case $\ell \leq L$ (small cylinders). Using the same arguments as in the proof of Proposition~\ref{prop.spataveflux}, we obtain the identity
    \begin{equation} \label{eq:14252401}
    \frac{\partial \left(  \nabla \chi \cdot V'\left(\nabla \phi^n  \right) \right)_{Q_\ell}}{\partial X^{n}_l(y)} =  n^{\frac12} \int_{l/n}^{(l+1)/n} \left( \nabla \chi \cdot \a \nabla P_\a (\cdot ; s , y) \right)_{Q_\ell} \, ds.
\end{equation}
To prove the inequality~\eqref{eq:16470112}, we use the properties of the map $\chi$ and obtain
\begin{equation*}
   \biggl| \frac{\partial \left(  \nabla \chi \cdot V'\left(\nabla \phi^n  \right) \right)_{Q_\ell}}{\partial X^{n}_l(y)} \biggr| 
   \\ 
   \leq C n^{\frac12} L^{-1}  \int_{l/n}^{(l+1)/n} \left\|  \nabla P_\a \left(\cdot ; s , y \right)  \right\|_{\underline{L}^1\left( Q_\ell \right)} \, ds.
\end{equation*}
The rest of the argument is identical to the proof of Proposition~\ref{prop.spataveflux}.

We next treat the case $\ell \geq L$ (large cylinders). In this setting, the identity~\eqref{eq:14252401} is still valid. Using that $\ell \geq L$ and that the map $\chi$ is supported in $\Lambda_L$, we may perform a discrete integration by parts and write
\begin{equation*}
    \left( \nabla \chi \cdot \a \nabla P_\a (\cdot ; s , y) \right)_{Q_\ell}  = - \left( \chi \nabla \cdot \a \nabla P_\a (\cdot ; s , y) \right)_{Q_\ell}.
\end{equation*}
Using that the heat kernel $P_\a$ solves the parabolic equation $\partial_t P_\a = \nabla \cdot \a \nabla P_\a$, we have
\begin{align*}
    \left( \nabla \chi \cdot \a \nabla P_\a (\cdot ; s , y) \right)_{Q_\ell} & = \left( \chi \partial_t P_\a (\cdot ; s , y) \right)_{Q_\ell} \\
    &  =  \ell^{-2} \left[\left( \chi P_\a (0, \cdot ; s , y) \right)_{\Lambda_\ell} - \left( \chi P_\a ((- \ell^2) \vee s, \cdot ; s , y) \right)_{\Lambda_\ell}\right].
\end{align*}
Using the Nash-Aronson estimate on the heat kernel $P_\a$, we deduce that
\begin{align*}
    \lefteqn{\left| \frac{ \partial \left( \nabla \chi \cdot V'\left(\nabla \phi^n  \right) \right)_{Q_\ell}}{\partial X^{n}_l(y)} \right|} \qquad & \\ & \leq  Cn^{\frac12} \ell^{-d-2} \int_{l/n}^{(l+1)/n} \left\|  P_\a \left((- \ell^2) \vee s, \cdot ; s , y \right)  \right\|_{L^1\left( \Lambda_L \right)} + \left\|  P_\a \left(0, \cdot ; s , y \right)  \right\|_{L^1\left( \Lambda_L \right)} \, ds, \\
    & \leq C n^{-\frac12}\ell^{-d-2} \exp \left( - \frac{(l/n)}{C \ell^2} \right).
\end{align*}
This completes the proof of the proposition. 
\end{proof}

\subsection{Expectation of the average of the flux of the corrector and the surface tension}

In this section, we study quantitatively the expectation of the spatial average of the flux of the corrector and relate it to the surface tension. Specifically, in Section~\ref{subsecsurfacetension}, we introduce a finite-volume version of the surface tension (following~\cite{FS}), and establish its convergence quantitatively with an optimal rate to the (infinite-volume) surface tension of the model. In Section~\ref{subsec.expectationaveflux}, we estimate the difference between the expectation of the spatial average of the flux of the first-order corrector and the finite-volume surface tension.

\subsubsection{The surface tension} \label{subsecsurfacetension}

In this section, we introduce and study the finite-volume surface tension associated with the model. In~\eqref{def.surfacetension}, we recall the notation $\Omega_{\Lambda_L , \mathrm{per}}^\circ$ for the set of periodic real-valued functions $\phi : \Lambda_L \to \R$ satisfying $\sum_{x \in \Lambda_L} \phi(x) = 0$.

\begin{definition}[Finite-volume Gibbs measure and surface tension]
For each integer $L \in \N$, and each slope $p \in \Rd$, we define the Gibbs measure $\mu_{\Lambda_L, p}$ on the space of functions $\Omega_{\Lambda_L , \mathrm{per}}^\circ$ according to the formula
\begin{equation} \label{def.surfacetension}
    \sigma_L (p) := -\frac{1}{\left| \Lambda_L \right|} \log \int_{\Omega_{\Lambda_L , \mathrm{per}}^\circ} \exp \Biggl( - \sum_{e \in E({\Lambda_L})} V(p \cdot e + \nabla \phi(e)) \Biggr) \prod_{x \in {\Lambda_L}} d \phi(x).
\end{equation}
We additionally define the finite-volume Gibbs measure over the space of functions $ \Omega_{\Lambda_L}^\circ$,
\begin{equation} \label{def.Gibbs}
    \mu_{\Lambda_L , p} := \frac{1}{Z_{\Lambda_L,p}} \exp \Biggl( - \sum_{e \in E(\Lambda_L)} V(p \cdot e + \nabla \phi(e)) \Biggr)  \prod_{x \in \Lambda_L} d \phi(x),
\end{equation}
where $Z_{\Lambda_L,p}$ is the normalizing constant.
\end{definition}
In the rest of this section, we let $\phi_{\Lambda_L , p} : \Lambda_L \to \R$ be a random variable distributed according to $\mu_{\Lambda_L, p}$ independent of the Brownian motions. We record below three properties of the finite-volume surface tension and the Gibbs measure $\mu_{\Lambda_L , p}$:
\begin{itemize}
    \item By the Brascamp-Lieb inequality~\cite{BL75, BL76, fontaine1983non}, for every $x \in \Lambda_L$,
    \begin{equation*}
        \left| \phi_{\Lambda_L , p} (x)\right| \leq \mathcal{O}_2 (C (1 + (\log L)^{\frac12} \indc_{\{ d = 2\}}))
        \,.
    \end{equation*}
    \item The Langevin dynamic is ergodic and stationary with respect to~$\mu_{\Lambda_L, p}$. Consequently, if we consider the system of stochastic differential equations
    \begin{equation*}
    \left\{ \begin{aligned}
    d \psi_{{Q_L}, \mathrm{per}}(t , x ; p) & = \nabla \cdot V'(p + \nabla \psi_{Q_L, \mathrm{per}}(\cdot , \cdot ; p)) (t ,x) dt + \sqrt{2} dB_t(x) &\mbox{in}&~ Q_L, \\
    \psi_{{Q_L}, \mathrm{per}}(-L^2 , \cdot ; p) &= \psi_{\Lambda_L , p} &\mbox{in}&~ \Lambda_L, \\
    \psi_{{Q_L}, \mathrm{per}} (t , \cdot ; p) &\in \Omega_{{\Lambda_L} , \mathrm{per}} &\mbox{for}&~ t \in I_L,
    \end{aligned} \right.
    \end{equation*}
    and define, for any time $t \in I$,
    \begin{equation*}
    \phi_{Q_L, \mathrm{per}}(t , x;p) := \psi_{Q_L, \mathrm{per}}(t , x;p) - \frac{1}{\left| \Lambda_L \right|} \sum_{x \in \Lambda_L} \psi_{Q_L, \mathrm{per}}(t , x;p),
    \end{equation*}
    then, for any time $t \in I$, the random variable $\phi_{{Q}, \mathrm{per}}(t , \cdot;p)$ is distributed according to $\mu_{\Lambda_L, p}\,.$
    Consequently, one has the estimate, for any $x \in \Lambda_L$,
    \begin{equation} \label{eq.realBL}
        \left| \phi_{Q_L, \mathrm{per}}(t , x ;p)  \right| \leq \mathcal{O}_2 \bigl(C (1 + (\log L)^{\frac12} \indc_{\{ d = 2\}})\bigr)
        \,.
    \end{equation}
    \item Using an explicit computation, one has the identity
    \begin{equation} \label{eq:17311910}
        D_p \sigma_{L} (p) := \E \left[  \left( V'(p + \nabla \phi_{\Lambda_L, \mathrm{per}} (\cdot, \cdot ; p))  \right)_{\Lambda_L} \right].
    \end{equation}
    Consequently, using the stationarity of the dynamic in both the space and time variables, we have the identity, for any $\ell \leq L$,
    \begin{equation*}
         D_p \sigma_L (p) = \E \bigl[  \left( V'(p + \nabla \phi_{Q_L, \mathrm{per}} (\cdot, \cdot ; p))  \right)_{Q_\ell} \bigr].
    \end{equation*}
\end{itemize}

The main result of this section provides a quantitative rate of convergence for the gradient $D_p \sigma_L$ of the finite-volume surface tension.

\begin{proposition}[Quantitative convergence of the finite-volume surface tension] \label{prop.surfacetesnionapprox}
For each $ p \in \Rd$, the sequence $ \sigma_L(p)$ converges as $L$ tends to infinity to the surface tension $\bar \sigma(p)$. Additionally, there exists a constant $C < \infty$ depending on the parameters $d , c_+ , c_-$ and  $p$ such that, for any $L \in \N$,
\begin{equation*}
\left| D_p \sigma_L(p) -  D_p \bar \sigma(p) \right| 
\leq 
CL^{-1}\bigl(1 + (\log L)^{\frac12} \indc_{\{ d = 2\}}\bigr)
\,.
\end{equation*}
\end{proposition}

\begin{proof}
We first note that, it is sufficient to prove the following inequality: for every~$L \in \N$,
\begin{equation} \label{eq:17081810}
\left| D_p \sigma_L(p) - D_p \sigma_{2L}(p) \right|
\leq
CL^{-1}\bigl(1 + (\log L)^{\frac12} \indc_{\{ d = 2\}}\bigr)
\,.
\end{equation}
Proposition~\ref{prop.surfacetesnionapprox} can indeed be deduced from~\eqref{eq:17081810} by a summation over dyadic scales.
The inequality~\eqref{eq:17081810} can be proved by noting that the map $w  = \phi_{\Lambda_L, \mathrm{per}}(\cdot ; p) - \phi_{\Lambda_{2L}, \mathrm{per}}(\cdot ; p)$ solves a linear parabolic equation with uniformly elliptic coefficient. By the Caccioppoli inequality and the estimate~\eqref{eq.realBL}, we have
\begin{align*}
    \E \bigl[ \left\| \nabla w \right\|_{\underline{L}^2(Q_{L/2})} \bigr] 
    & \leq \frac{C}{L} \E \left[ \left\| \phi_{Q_L, \mathrm{per}}(\cdot , \cdot; p) \right\|_{\underline{L}^2(Q_{L})} \right] + \frac{C}{L} \E \left[ \left\| \phi_{Q_{2L}, \mathrm{per}}(\cdot, \cdot ; p) \right\|_{\underline{L}^2(Q_{L})} \right] \\
    &  \leq CL^{-1}\bigl(1 + (\log L)^{\frac12} \indc_{\{ d = 2\}}\bigr)\,.
\end{align*}
Using the identity~\eqref{eq:17311910}, together with the assumption that $V'$ is Lipschitz, we deduce that
\begin{align} \label{eq:20311910}
    \left| D_p \sigma_L(p) - D_p \sigma_{2L}(p) \right| 
    & = 
    \Bigl| \E \bigl[  \left( V'(p + \nabla \phi_{L , \mathrm{per}}(\cdot, \cdot ; p ))  \right)_{Q_{L/2}} \bigr] - \E \bigl[  \left( V'(p + \nabla \phi_{2L , \mathrm{per}}(\cdot , \cdot; p ))  \right)_{Q_{L/2}} \bigr] \Bigr| \notag \\
    & \leq \E \bigl[ \left\|  \nabla w \right\|_{\underline{L}^2(Q_{L/2})} \bigr].
\end{align}
A combination of the two previous displays completes the proof of the Proposition~\ref{prop.surfacetesnionapprox}.
\end{proof}

\subsubsection{Expectation of the average of the flux corrector} \label{subsec.expectationaveflux}

In this section, we estimate the difference between the expectation of the spatial average of the corrector and the gradient of the surface tension. In the following statement, we let $Q = I \times \Lambda$ and $Q_1 = I_1 \times \Lambda$ be two parabolic cylinders such that $2I_1 \subseteq I$. We additionally denote by $L$ the sidelength of $\Lambda$ and assume that $|I_1| = L^2$. Given a (time-dependent) slope $q : I \to \Rd$ and a (constant) slope $p \in \Rd$, we set
\begin{equation*}
    \left\| p - q \right\|_{\underline{L}^2 (2I_1)}^2  := \frac{1}{|2I_1|} \int_{2I_1} |q(t) - p|^2 \, dt.
\end{equation*}

\begin{proposition} \label{prop3.9}
There exists a constant $C:= C(d ,c_+ , c_-) < \infty$ such that for any constant slope $p \in \Rd$, and any time-dependent slope $q : I \to \Rd$,
\begin{equation*}
\left| \E \left[ \left( V'( q + \nabla \phi_{Q}(\cdot, \cdot ; q))\right)_{Q_1}\right] - D_p \bar \sigma(p) \right| \leq C \left\| p - q \right\|_{\underline{L}^2 (2I_1)} 
+ CL^{-1} \bigl( 1 + (\log L)^{\frac12} \indc_{\{ d = 2\}} \bigr)
\,.
\end{equation*}
\end{proposition}

\begin{proof}
Using the translation-invariance of the system, we assume without loss of generality that $\Lambda = \Lambda_L$ (i.e., the box $\Lambda$ is centered at $0$) and that $Q_1 = Q_L$. By Proposition~\ref{prop.surfacetesnionapprox}, it is sufficient to prove the estimate
\begin{equation*}
     \Bigl| \E \left[ \left( V'( q + \nabla \phi_{Q}(\cdot, \cdot ; q))\right)_{Q_L} \right] - D_p  \sigma_{2L}(p)  \Bigr| 
     \leq C \left\| p - q \right\|_{\underline{L}^2 (2I_L)} 
     + CL^{-1} \bigl( 1 + (\log L)^{\frac12} \indc_{\{ d = 2\}} \bigr)\,.
\end{equation*}
Using that the field $\phi_Q$ is spatially stationary, it is equivalent to prove
\begin{equation*}
    \Bigl| \E \left[ \left( V'( q + \nabla \phi_{Q}(\cdot, \cdot ; q))\right)_{I_L \times \Lambda_{L/2}} \right] - D_p  \sigma_{2L}(p)  \Bigr| 
     \leq C \left\| p - q \right\|_{\underline{L}^2 (2I_L)} 
     + CL^{-1} \bigl( 1 + (\log L)^{\frac12} \indc_{\{ d = 2\}} \bigr)\,.
\end{equation*}
The proof follows the same outline as the proof of Proposition~\ref{prop.surfacetesnionapprox}: we use that the map $w :=  \phi_{Q}(\cdot ; q) - \phi_{Q_{2L}, \mathrm{per}}(\cdot ; p)$ solves a linear parabolic equation with uniformly elliptic coefficient of the form
\begin{equation*}
    \partial_t w - \nabla \cdot \a \left( p - q + \nabla w\right) = 0 ~\mbox{in}~ 2I_{L} \times \Lambda_L.
\end{equation*}
Applying the Caccioppoli inequality together with the estimates~\eqref{eq:20291910} and~\eqref{eq.realBL}, we obtain
\begin{equation*}
    \E \left[ \left\| \nabla w \right\|_{\underline{L}^2(I_L \times \Lambda_{L/2})} \right] \leq C \left\| p - q \right\|_{\underline{L}^2 (2I_L)} 
+ CL^{-1} \bigl( 1 + (\log L)^{\frac12} \indc_{\{ d = 2\}} \bigr).
\end{equation*}
We complete the argument using that the map $V'$ is Lipschitz as in the computation~\eqref{eq:20311910}.
\end{proof}

\subsection{Regularity of the corrector with respect to the slope}
\label{sec.regslope}

In this section, we prove an upper bound on the $L^2$-norm of the difference of two correctors defined on different parabolic cylinders (with large intersection) and with different (time-dependent) slopes. The result will be useful in order to establish Theorem~\ref{Th.quantitativehydr}; its proof is elementary and relies on an application of the Caccioppoli inequality. To state the result, we fix two two parabolic cylinders $Q_1 := I_1 \times \Lambda_1$ and $Q_2 := I_2 \times \Lambda_2$, denote the sidelengths of $\Lambda_1$ and $\Lambda_2$ by $L_1$ and $L_2$ respectively. We additionally fix two time-dependent slopes $q_1 : I_1 \to \Rd$ and $q_2 : I_2 \to \Rd$.

\begin{proposition} \label{prop.linearizedcorrector}
There exists a constant $C:= C(d , c_+ , c_-) < \infty$ such that the following holds. If there exists an integer $L \in \N$ and a point $z = (t , x) \in (0 , \infty) \times \Zd$ such that $L_1 , L_2 \leq 4 L$ and $z + Q_{2L} \subseteq Q_1 \cap Q_2$, then
\begin{align} 
\label{eq:diff2corr}
\lefteqn{
\left\| \nabla \phi_{Q_1}(\cdot, \cdot; q_1) - \nabla \phi_{Q_2}(\cdot, \cdot; q_2) \right\|_{\underline{L}^2(z+ Q_L)} 
} \qquad & 
\notag \\ & 
\leq C \left\|q_1 - q_2 \right\|_{\underline{L}^2(t + I_{2L})} 
+ CL^{-1} \Bigl( \left\|  \phi_{Q_1}(\cdot , \cdot ; q_1)\right\|_{\underline{L}^2(Q_1)}  + \left\|  \phi_{Q_2}(\cdot , \cdot; q_2)\right\|_{\underline{L}^2(Q_2)} \Bigr)
\,.
\end{align}
\end{proposition}

\begin{proof}
As in the proof of Proposition~\ref{prop3.9}, we note that the difference $w := \phi_{Q_1}(\cdot ;  q_1) - \phi_{Q_2}(\cdot  ; q_2)$ solves a linear parabolic equation of the form
\begin{equation} \label{eq:10310908}
\partial_t w + \nabla \cdot \a  ( p - q + \nabla w)  = 0 \hspace{3mm} \mbox{in} ~ (z + Q_{2L})
\,.
\end{equation}
The bound~\eqref{eq:diff2corr} then follows from the Caccioppoli inequality.
\end{proof}

\section{The quantitative hydrodynamic limit} \label{eq:quanthydrolim2sc}

This section is devoted to the proof of Theorem~\ref{Th.quantitativehydr} based on the results established in Section~\ref{sec.section2}. As mentioned in Section~\ref{sec1.5.2}, the proof relies on a two-scale expansion and, in order to be implemented, requires a few preliminary results and notation listed below. In Section~\ref{sec4.1.1}, we collect, without proof, some standard regularity estimates on the solutions of the homogenized equation (i.e., the equation~\eqref{eq:defubarthmhydro}). In Section~\ref{sec4.1.2}, we introduce and study an approximation scheme for the homogenized equation which is used to pass from the discrete setting (where the Langevin dynamics are defined) to the continuous one (where the homogenized equation is defined). Such an approximation scheme has already been used qualitatively in the proof of the hydrodynamic limit in~\cite{FS}; the main input of this section is to obtain a quantitative version of their result (see Proposition~\ref{prop.approx}) with a sufficiently good rate (i.e., of the form $\ep^{\frac12}$) to match the error term in Theorem~\ref{Th.quantitativehydr}. Section~\ref{section4.2} is devoted to the construction of the two-scale expansion. As mentioned in Section~\ref{sec1.5.2}, this construction requires to introduce a mesoscopic scale of size $\ep^{\frac 12}$ with respect to the spatial variable. In Section~\ref{section4.3}, we define a first error term which appears when we implement the two-scale expansion and estimate it using the results of Section~\ref{sec.section2} and the regularity estimates of Section~\ref{section4.3}. In Section~\ref{section4.4}, we introduce and estimate a second error term: the one arising from the flux of the corrector in the two-scale expansion. Its estimation is a key point of the analysis: in order to obtain the rate of convergence stated in Theorem~\ref{Th.quantitativehydr}, one needs to obtain precise estimates on its $\underline{H}^{-1}_{\mathrm{par}} (Q^\ep)$-norm (see Proposition~\ref{prop4.3fluxest}) which requires a precise analysis. Finally, Section~\ref{secTh.quantitativehydr} implements the two-scale expansion and proves Theorem~\ref{Th.quantitativehydr}, following mostly standard techniques and making use of the estimates established in the previous sections.

\subsection{Preliminary results}

\subsubsection{Regularity estimates for the solution of the limiting equation} \label{sec4.1.1}

In this section, we record the regularity properties for solutions of the homogenized equation used in the proof of Theorem~\ref{Th.quantitativehydr}. These estimates are classical and we refer to~\cite[Section 8.4]{giusti2003direct} for a proof in the elliptic setting, and to~\cite[Chapter 6]{ladyvzenskaja1988linear} in the parabolic setting.
Briefly, the bounds~\eqref{eq:regH2eqprop} and~\eqref{eq:regdtgradeqprop} are obtained by differentiating the nonlinear homogenized equation in space and in time, respectively, and applying the usual global energy estimates. 
We recall the notation $I := (-1 , 0)$ used in this section. 

\begin{proposition}[Regularity for solutions of the homogenized equation] \label{prop:H2regdiscfct}
Let $D \subseteq \Rd$ be a bounded and $C^{1,1}$ domain. Then, the solution $\bar u$ of the parabolic equation~\eqref{eq:defubarthmhydro} over the set $Q = I \times D$ satisfies:
\begin{itemize}
    \item The $H^2$-regularity estimate: there exists a constant $C(d , c_+ , c_-)<\infty$ such that
\begin{equation} \label{eq:regH2eqprop}
   \left\|\bar u\right\|_{L^2 \left( I , H^2 \left( D \right) \right)} \leq C  (\left\| f \right\|_{
    H^{2}(Q)} + 1). 
\end{equation}
    \item The time regularity estimate: there exists a constant $C(d , c_+ , c_-) <\infty$ such that
\begin{equation} \label{eq:regdtgradeqprop}
        \left\| \partial_t \bar u \right\|_{L^2 \left( I , H^1 \left( D \right) \right)} \leq C  ( \left\| f \right\|_{H^2(Q)}  +1 ).
    \end{equation}
\end{itemize}
\end{proposition}

\subsubsection{An approximation scheme for nonlinear parabolic equation and regularity estimates} \label{sec4.1.2} 

In this section, we construct and study an approximation scheme for nonlinear parabolic equations. We recall the notation introduced in Section~\ref{Sectionmicroscopic}. For each point $x \in D^\ep$ and each $\ep \in (0 , 1)$, we recall the definition of the (vector-valued) discrete gradient
\begin{equation} \label{eq.notacalD}
    \vec{\nabla}^\ep u (t , x) := \left( \nabla_1^\ep u(t , x) , \ldots, \nabla_d^\ep u(t , x) \right) \in \Rd.
\end{equation}
For $i \in \{ 1 , \ldots,d \}$, we denote by $\nabla_i^{\ep , *} u(t , x) = \ep^{-1} \left( u(t , x - \ep e_i) - u(t , x) \right)$ the adjoint of the discrete derivative $\nabla_i^\ep$. We then denote by
\begin{equation*}
    \vec{\nabla}^{\ep , *} u (t , x) := \left( \nabla_1^{\ep , *} u(t , x) , \ldots, \nabla_d^{\ep , *} u(t , x) \right) \in \Rd.
\end{equation*}
We next define the discrete elliptic operator
\begin{equation} \label{eq:discellop}
     \vec{\nabla}^\ep \cdot D_p \bar \sigma ( \vec{\nabla}^\ep u) (t , x) = \ep^{-1} \sum_{i = 1}^d \left( D_p \bar \sigma (  \vec{\nabla}^\ep u (t , x)) +  D_p \bar \sigma (  \vec{\nabla}^{\ep , *} u (t , x)) \right) \cdot e_i.
\end{equation}
The operator~\eqref{eq:discellop} is defined so as to satisfy the following identity: for any pair of functions $ u , v : \ep \Zd \to \R$ with finite support
\begin{equation} \label{def.op3444}
    -\sum_{x \in \ep \Zd} \vec{\nabla}^\ep \cdot D_p \bar \sigma (\vec{\nabla}^\ep u) (x) v(x) =  \sum_{x \in \ep \Zd} D_p \bar \sigma (\vec{\nabla}^\ep u(x)) \cdot \vec{\nabla}^\ep v(x).
\end{equation}
Let us recall the notation $\tilde f_\ep(t , x) := (2\ep)^d \int_{[-\ep , \ep]^d} f(t , x + y) \, dy$ and let $\bar u^\ep: Q^\ep \to \R$ be the solution of the discrete parabolic equation
\begin{equation} \label{def.discequep}
    \left\{ \begin{aligned}
     \partial_t \bar u^\ep - \vec{\nabla}^\ep \cdot D_p \bar \sigma (\vec{\nabla}^\ep \bar u^\ep) & = 0 &~\mbox{in} &~  Q^\ep, \\
     \bar u^\ep & = \tilde f_\ep &~\mbox{on} &~ \partial_{\mathrm{par}} Q^\ep.
    \end{aligned} \right.
\end{equation}
The following proposition quantifies the difference of the $L^2$-norm between the solution $\bar u$ of the continuous parabolic equation~\eqref{eq:defubarthmhydro} and the solution $\bar u^\ep$ of the discretized equation~\eqref{def.discequep}. In order to state the result, we extend the map $\bar u^\ep$ and its gradient from the discrete setting to the continuous one into piecewise constant functions by setting
\begin{equation*}
    \bar u^\ep (t , x) := \sum_{y \in \ep \Zd} \bar u^\ep(t , y) \indc_{\{ y \in x +  [-\ep, \ep]^d \}} \hspace{5mm} \mbox{and} \hspace{5mm} \vec{\nabla}^\ep \bar u^\ep(t , x) := \sum_{y \in \ep \Zd} \vec{\nabla}^\ep \bar u^\ep (t , y) \indc_{\{ y \in x +  [-\ep, \ep]^d \}}.
\end{equation*}
The $L^2$-norms in Proposition~\ref{prop.approx} then denotes the continuous one on the space $Q = I \times D$. 
 
\begin{proposition} \label{prop.approx}
There exists a constant $C := C(d) < \infty$ such that, for any $\ep > 0$,
\begin{equation*}
    \left\| \bar u^\ep - \bar u  \right\|_{L^2(Q)} + \| \vec{\nabla}^\ep \bar u^\ep - \nabla \bar u  \|_{L^2(Q)} \leq C \ep^{\frac12}  \left\| f \right\|_{H^{2}(Q)}.
\end{equation*}
\end{proposition}
The proof of this proposition can be found in Appendix~\ref{app.appendixA}. It implies that, in order to prove Theorem~\ref{Th.quantitativehydr}, it is sufficient to establish it on the discrete space $Q^\ep$ with the function $\bar u^\ep$ instead of $\bar u$. All the analysis of this section will, from now on, be carried out in the discrete setting (either the lattice $\ep \Zd$ or $\Zd$).

\subsection{Construction of the two-scale expansion} \label{section4.2}

This section is devoted to the construction of the two-scale expansion and introduces the mesoscopic scale.

\subsubsection{Mesoscopic scale and partition of unity} Fix a parameter $\ep > 0$ and define the mesoscopic scale
\begin{equation*}
    \kappa := \ep^{\frac 12} \left( 1 + \left| \ln \ep \right|^{\frac 12} \indc_{\{ d = 2 \}} \right).
\end{equation*} 
We recall the definitions of the box $\Lambda_\kappa^\ep$ and the parabolic cylinder $Q^\ep$ introduced in Section~\ref{Sectionmicroscopic}. We next partition the set $D^\ep$ (resp. the cylinder $Q^\ep$) into boxes of the form $y + \Lambda_\kappa^\ep$ (resp. cylinders of the form $z + Q^\ep_\kappa$). To this end, we introduce the sets
\begin{equation*}
    \mathcal{Y}_\kappa :=  \kappa \Zd  \cap D ~\mbox{and}~ \mathcal{Z}_\kappa := \left(\kappa^2 \N_* \times \kappa \Zd \right) \cap Q.
\end{equation*}
Given a point $z = (t , y) \in \mathcal{Z}_\kappa$, we will abuse notation and denote by $z/\ep := (t/\ep^2 , y/\ep)$ the point rescaled diffusively.
We say that two points $y, y' \in \mathcal{Y}_\kappa$ are neighbors if $|y - y'| \leq \kappa$, and that two points $z = (t , y) \in \mathcal{Z}_\kappa$ and $z' = (t',z') \in \mathcal{Z}_\kappa$ are neighbors if $|t - t'| \leq \kappa^2$ and $|y - y'| \leq \kappa$. We note that, with this notation, we have $y \sim y$ and $z \sim z$. For any fixed point $y \in \mathcal{Y}_\kappa$, we will use the notation $\sum_{ y' \sim y} $ to refer to sum over the points $y' \in \mathcal{Y}_\kappa$ satisfying $y' \sim y$.

We next consider a smooth partition of unity $(\chi_y)_{y \in \mathcal{Y}_\kappa} : D^\ep \to \R$ satisfying the following properties: 
\begin{equation} \label{prop.partofunity2sc}
    0 \leq \chi_y \leq \indc_{\{y + \Lambda_{2 \kappa}^\ep\}}, \hspace{3mm} \sum_{y \in \mathcal{Y}_\kappa} \chi_y = 1 \hspace{3mm} \mbox{in} ~D^\ep,
\end{equation}
and
\begin{equation} \label{prop.partofunity2scbis}
     \kappa \| \nabla^{\ep} \chi_y \|_{L^\infty(D^\ep)} + \kappa^{2} \| \nabla^{\ep,2} \chi_y \|_{L^\infty(D^\ep)} \leq C.
\end{equation}

\subsubsection{Average value of the gradient of the homogenized solution} \label{sec.4.2.2} Let $\bar u^\ep$ be the solution of the equation~\eqref{def.discequep}. Given a point $z = (t , y) \in \mathcal{Z}_{\kappa}$, we denote by $\xi_z$ the average value of the gradient $\nabla \bar u^\ep$ over the parabolic cylinder $z + Q_{2\kappa}^\ep$, i.e.,
\begin{equation*}
   \xi_z := \left( \nabla^\ep  \bar u^\ep \right)_{z + Q_{2\kappa}^\ep}.
\end{equation*}
For each $y \in  \mathcal{Y}_\kappa$, we denote by $\xi_y : I \to \Rd$ the time-dependent slope defined by the formula
\begin{equation} \label{def.xi_y}
    \xi_y(t) := \sum_{z \in \mathcal{Z}_{2\kappa}} \xi_z \indc_{\{ (t , y) \in z + Q_{2\kappa}^\ep\}}.
\end{equation}
We note that, with this notation, we have
\begin{equation*}
    \frac{1}{\left| \mathcal{Y}_\kappa \right|}\sum_{y \in \mathcal{Y}_\kappa} \left\| \xi_y \right\|_{L^{2} ((-1,0))} \leq C \left\| \nabla^\ep \bar u^\ep\right\|_{L^2(Q^\ep)}.
\end{equation*}

\subsubsection{Definition of the first-order corrector} Let us set $L := \lfloor \kappa/\ep \rfloor$, and, for $y \in \mathcal{Y}_\kappa$, denote by 
$$Q_y := \ep^{-2}I\times (y/\ep + \Lambda_{2L}).$$
We define the first-order corrector $\phi^\ep$ according to the formula
\begin{equation} \label{def.firstordcorrector2sc}
    \phi^\ep(t,x) = \ep \sum_{y \in \mathcal{Y}_{\kappa}} \chi_{y}(x) \phi_{Q_{y}} \left( \frac{t}{\ep^2} , \frac{x}{\ep} ; \xi_y(t) \right),
\end{equation}
and the two-scale expansion by
\begin{equation} \label{def.wL}
    w^\ep := \bar u^\ep + \phi^\ep.
\end{equation}
For later use, we introduce the notation, for $y \in \mathcal{Y}_{\kappa}$,
\begin{equation*}
    v_{y}(t , x) := \xi_{y}(t) \cdot x + \ep  \phi_{Q_{y}} \left( \frac{t}{\ep^2} , \frac{x}{\ep} ; \xi_y(t) \right).
\end{equation*}

\begin{remark}
As mentioned in Remark~\ref{remark3.2225}, the mesoscopic scale is only defined with respect to the spatial variable as the variations of the slope of the homogenized solution with respect to the time variable are encoded directly in the corrector (the function $\xi_y$ defined in~\eqref{def.xi_y} depends on the time variable); this simplifies the implementation of the two-scale expansion in Section~\ref{secTh.quantitativehydr} and improves the rate of convergence.
\end{remark}

\subsection{Error terms and preliminary estimates} \label{section4.3}
In this section, we introduce a collection of error terms $(E_z)_{z \in \mathcal{Z}_\kappa}$ which appear frequently in the proof of Theorem~\ref{Th.quantitativehydr}. Formally, we define, for any $z = (t , y) \in \mathcal{Z}_\kappa$,
\begin{equation} \label{def.errortermEz}
    E_z :=  \sum_{z' \sim z}\left\| \vec{\nabla}^\ep \bar u^\ep - \xi_z  \right\|_{\underline{L}^2 \left( z + Q_\kappa^\ep \right)}  + \sum_{z' \sim z} \left| \xi_z - \xi_{z'}\right|  + \frac{\ep}{\kappa}  \sum_{ y' \sim y} \left\|  \phi_{Q_{y'}} \left(\cdot, \cdot ; \xi_{y'} \right)\right\|_{\underline{L}^2 (z/\ep + Q_{L})}. 
\end{equation}
The error terms $(E_z)_{z \in \mathcal{Z}_\kappa}$ are small in the following sense: using Proposition~\ref{prop.approx}, the regularity estimate~\eqref{prop:H2regdiscfct}, Proposition~\ref{prop.2.2BL} and the Poincar\'e inequality, we have
\begin{align} \label{eq:estL2normEz}
    \frac{1}{|\mathcal{Z}_\kappa|}\sum_{z \in \mathcal{Z}_\kappa} E_z^2 & \leq C \| \vec{\nabla}^\ep \bar u^\ep - \nabla \bar u  \|_{L^2(Q)}^2 + C \kappa^2 \| \nabla^2 \bar u  \|_{L^2(Q)}^2 + C \kappa^4 \| \partial_t \nabla \bar u  \|_{L^2(Q)}^2 \\
    & \qquad+  \frac{C \ep^2}{\kappa^2} \frac{1}{\left| \mathcal{Y}_\kappa \right|}  \sum_{y \in \mathcal{Y}_\kappa}  \left\|  \phi_{Q_{y'}} \left(\cdot, \cdot ; \xi_{y'} \right)\right\|_{\underline{L}^2 (Q_{y})}^2 \notag \\
    & \leq \mathcal{O}_1 \left( C \ep \bigl( 1 + | \log \ep | \indc_{\{ d = 2\}} \bigr) \right). \notag
\end{align}
We complete this section by stating an inequality involving the error terms $(E_z)_{z \in \mathcal{Z}_\kappa}$ which will be used in the proof of Theorem~\ref{Th.quantitativehydr} below: for any $z = (t , y) \in \mathcal{Z}_\kappa$,
\begin{equation} \label{eq:1805110888}
     \left\| \left( \sum_{y' \sim y} \chi_{y'} \nabla^\ep v_{y'} \right) - \nabla^\ep v_{y}  \right\|_{\underline{L}^2(z + Q_\kappa^\ep)} \leq C E_{z}.
\end{equation}
The upper bound~\eqref{eq:1805110888} is obtained by using the properties of the partition of unity $(\chi_z)_{z \in \mathcal{Z}_\kappa}$ listed in~\eqref{prop.partofunity2sc} and Proposition~\ref{prop.linearizedcorrector}.

\subsection{Estimating the weak norm of the flux of the corrector} \label{section4.4}
In this section, we estimate the $\underline{H}^{-1}_\mathrm{par}(Q^\ep)$-norm of the error term involving the flux of the corrector which appears in~\eqref{eq:16350512} when implementing the two-scale expansion. Obtaining a sharp estimate on this norm is crucial to optimize the rate of convergence in Theorem~\ref{Th.quantitativehydr}, and we dedicate the next proposition to its study. 

While the nature of the observable, and consequently the argument, are technical, the strategy of the proof is straightforward: using the multiscale Poincar\'e inequality (Proposition~\ref{prop:multscPoinc}), we first reduce the study of the $H^{-1}_{\mathrm{par}}(Q^\ep)$-norm of the error term to the study of its averages over parabolic cylinders of various sizes. We then estimate these spatial averages using the same techniques as in Proposition~\ref{prop.2.2BL}, Proposition~\ref{prop.spataveflux} and Proposition~\ref{prop.prop3.5techni}: we discretize the increments of the Brownian motions, compute the derivative of the observable with respect to the discrete increments using Proposition~\ref{prop.prop3.5techni} and apply the Gaussian concentration inequality.

\begin{proposition} \label{prop4.3fluxest}
There exists a constant $C := C(d , c_+ , c_-) < \infty$ such that
\begin{equation} \label{eq:21000212}
    \left\| \sum_{y \in \mathcal{Y}_\kappa} \nabla^\ep \chi_{y} \cdot \left( V' \left( \nabla^\ep v_{y} \right) - D_p \bar \sigma \left( \xi_{y} \right) \right) \right\|_{\underline{H}^{-1}_\mathrm{par} (Q^\ep)} \leq \mathcal{O}_2 \left( C \ep^{\frac12} \bigl( 1 + | \log \! \ep |^{\frac12} \indc_{\{ d = 2\}} \bigr) \right).
\end{equation}
\end{proposition}

 \begin{proof}
To ease the presentation of the argument and, in particular, the application of the multiscale Poincar\'e inequality (Proposition~\ref{prop:multscPoinc}), we make two additional assumptions in the proof: the set $D$ is included in the box $[-1 , 1]^d$ and there exists an even integer $m \in \N$ such that $\ep = 3^{-m}$, and consequently $\kappa := 3^{-m/2} (1 + m^{\frac12} \indc_{\{ d = 2\}})$ and $L := \kappa / \ep = 3^{m/2}(1 + m^{\frac12} \indc_{\{ d = 2\}}) .$

We split the proof of the inequality~\eqref{eq:21000212} into two inequalities:
 \begin{equation} \label{eq:21060212}
     \left\| \sum_{y \in \mathcal{Y}_\kappa} \nabla^\ep \chi_{y} \cdot \left( \E \left[ V' \left( \nabla^\ep v_{y} \right) \right] - D_p \bar \sigma \left( \xi_{y} \right) \right) \right\|_{\underline{H}^{-1}_\mathrm{par} (Q^\ep)} \leq C 3^{- \frac m2} \left( 1 + m^{\frac 12} \indc_{\{ d = 2\}} \right),
 \end{equation}
 and
  \begin{equation} \label{eq:21070212}
     \left\| \sum_{y \in \mathcal{Y}_\kappa} \nabla^\ep \chi_{y} \cdot \left( V' \left( \nabla^\ep v_{y} \right) - \E \left[ V' \left( \nabla^\ep v_{y} \right) \right] \right) \right\|_{\underline{H}^{-1}_\mathrm{par} (Q^\ep)} \leq \mathcal{O}_2 \left( C 3^{- \frac m2} \bigl( 1 + m^{\frac 12} \indc_{\{ d = 2\}} \bigr) \right).
 \end{equation}
 The inequality~\eqref{eq:21000212} is then a consequence of~\eqref{eq:21060212},~\eqref{eq:21070212} and the triangle inequality.

 \medskip
 
 \textit{Step 1. Proof of the inequality~\eqref{eq:21060212}.} We note that the maps $\E \left[ V' \left( \nabla^\ep v_{y} \right) \right]$ and $D_p \bar \sigma \left( \xi_{y} \right)$ depend only on the time variable (and not on the spatial variable due to the spatial stationarity of the first-order corrector). Combining this observation with the definition of the $\underline{H}^{-1}_\mathrm{par} (Q^\ep)$-norm and a discrete integration by parts, we obtain
 \begin{equation*}
     \left\| \sum_{y \in \mathcal{Y}_\kappa} \nabla^\ep \chi_{y} \cdot \left( \E \left[ V' \left( \nabla^\ep v_{y} \right) \right] - D_p \bar \sigma \left( \xi_{y} \right) \right) \right\|_{\underline{H}^{-1}_\mathrm{par} (Q^\ep)} \\ \leq \left\| \sum_{y \in \mathcal{Y}_\kappa} \chi_{y} \left( \E \left[ V' \left( \nabla^\ep v_{y} \right) \right] - D_p \bar \sigma \left( \xi_{y} \right) \right) \right\|_{L^{2} (Q^\ep)}.
 \end{equation*}
To estimate the term in the right-hand side, we fix a point $y \in \mathcal{Y}_{\kappa}$, a point $z = (t , y) \in \mathcal{Z}_{\kappa}$, and denote by $z' = (t - 4\kappa^2 , y) \in \mathcal{Z}_{\kappa}$. We then apply Proposition~\ref{prop3.9} with the parabolic cylinder $Q' := z/\ep + Q_{2L}$ and obtain
 \begin{align*}
     \left\| \E \left[ V' \left( \nabla^\ep v_{y} \right) \right] - D_p \bar \sigma \left( \xi_{y} \right) \right\|_{\underline{L}^{2} (z + Q_{2\kappa}^\ep)} & \leq C \left| \xi_{z'} - \xi_{z} \right| + \frac{C\left( 1 + \sqrt{\log \! L} \indc_{\{ d = 2\}} \right)}{L}.
 \end{align*}
Using that the map $\chi_{y}$ is bounded by $1$ and supported in the set $y + \Lambda_{2\kappa}^\ep$, we obtain
\begin{align*}
    \left\| \sum_{y \in \mathcal{Y}_\kappa} \chi_{y} \left( \E \left[ V' \left( \nabla^\ep v_{y} \right) \right] - D_p \bar \sigma \left( \xi_{y} \right) \right) \right\|_{L^{2} (Q^\ep)}^2 &
    \leq \frac{C}{\left| \mathcal{Z}_\kappa \right|} \sum_{z \in \mathcal{Z}_{\kappa}}  \left\| \E \left[ V' \left( \nabla^\ep v_{y} \right) \right] - D_p \bar \sigma \left( \xi_{y} \right) \right\|_{\underline{L}^{2} (z + Q_{2\kappa}^\ep)}^2 \\
    & \leq C 3^{-m} \left( 1 + m \indc_{\{ d = 2\}} \right).
\end{align*}

\textit{Step 2. Proof of the inequality~\eqref{eq:21070212}.} The proof of the inequality~\eqref{eq:21070212} is more involved than the proof of~\eqref{eq:21060212} and is based on an application of the multiscale Poincar\'e inequality. It is subdivided in two steps below. We first rescale the inequality~\eqref{eq:21070212} (using, for instance, the identity~\eqref{eq:29.5}), use the assumptions $D \subseteq [-1 ,1]^d$ and~$\ep = 3^{-m}$ (which imply $[- \ep^{-2} , 0] \times D^\ep/\ep \subseteq Q_{3^m}$) and the definition of the map $v_y$. These observations imply that~\eqref{eq:21070212} is equivalent to the inequality
\begin{multline*}
    \left\| \sum_{y \in \mathcal{Y}_\kappa} \nabla \chi_{y}^\ep \cdot \left( V' \left(  \xi_y + \nabla  \phi_{Q_y} \left( \cdot, \cdot ; \xi_y \right)  \right) - \E \left[ V' \left(  \xi_y + \nabla  \phi_{Q_y} \left(\cdot, \cdot ; \xi_y \right) \right) \right] \right) \right\|_{\underline{H}^{-1}_\mathrm{par} (Q_{3^m})} \\ \leq \mathcal{O}_2 \left( C 3^{-\frac m2} \left( 1 + m^{\frac12} \indc_{\{ d = 2\}} \right) \right),
\end{multline*}
where we used the notation $\chi_{y}^\ep :=  \chi_y (\ep \cdot)$.
By the multiscale Poincar\'e inequality (Proposition~\ref{prop:multscPoinc}) and the property~\eqref{sum.Onotation} of the $\mathcal{O}$-notation, we see that it is sufficient to prove the two following estimates:
\begin{equation} \label{eq:14200312}
     \left\| \sum_{y \in \mathcal{Y}_\kappa} \nabla \chi_{y}^\ep \cdot  V' \left( \xi_y + \nabla  \phi_{Q_y} \left(\cdot,\cdot ; \xi_y \right) \right) \right\|_{\underline{L}^2 (Q_{3^m})} \leq \mathcal{O}_2 \left(C L^{-1}\right),
\end{equation}
and, for every $k \in \{ 0 , \ldots , m\}$ and every $z \in \mathcal{Z}_{k,m}$,
\begin{multline} \label{eq:14210312}
    \left|  \sum_{y \in \mathcal{Y}_\kappa}  \left(\nabla \chi_{y}^\ep \cdot \left( V' \left( \xi_y + \nabla \phi_{Q_y} \left(\cdot ,\cdot ; \xi_y \right) \right) - \E \left[ V' \left( \xi_y + \nabla \phi_{Q_y} \left(\cdot, \cdot ; \xi_y \right) \right) \right] \right) \right)_{z + Q_{3^k}} \right| \\ \leq  \left\{ \begin{aligned}
    \mathcal{O}_2 \left( C L^{-1} 3^{- \frac{dk}{2}} \right) ~\mbox{if}~ 3^{k} \leq L, \\
    \mathcal{O}_2 \left( C 3^{- \frac{(d + 2) k}2} \right)~\mbox{if}~ 3^{k} \geq L.
    \end{aligned} \right.
\end{multline}
We split the proof of the inequalities~\eqref{eq:14200312} and~\eqref{eq:14210312} into three substeps.

\medskip

\textit{Substep 2.1. Proof of the inequality~\eqref{eq:14200312}.} To prove the estimate~\eqref{eq:14200312}, we use the bound~\eqref{prop.partofunity2scbis} on the cutoff functions $(\chi_y)_{y \in \mathcal{Y}_\kappa}$ and obtain
\begin{align*}
    \left\| \sum_{y \in \mathcal{Y}_\kappa} \nabla \chi_{y}^\ep \cdot  V' \left( \xi_y + \nabla \phi_{Q_y} \left(\cdot, \cdot ; \xi_y \right) \right) \right\|_{\underline{L}^{2} (Q_{3^m})}^2
    & \leq \frac{C L^{-2}}{|\mathcal{Y}_\kappa|}\sum_{y \in \mathcal{Y}_\kappa} \left( \left\| \xi_y \right\|_{\underline{L}^{2} (Q_y)}^2 + \left\| \nabla \phi_{Q_y} \left( \cdot, \cdot ; \xi_y \right)  \right\|_{\underline{L}^{2} (Q_y)}^2 \right) \\
    & \leq \mathcal{O}_1 \left( CL^{-2} \right).
\end{align*}
Taking the square root in the previous display completes the proof of~\eqref{eq:14200312}. 

\medskip

\textit{Substep 2.2. Proof of the inequality~\eqref{eq:14210312}, case $3^{k} \leq L$.} We note that, for each  $z \in \mathcal{Z}_{k,m}$, there are at most $C := C(d) < \infty$ vertices $y \in \mathcal{Y}_\kappa$ such that
\begin{equation*}
    \left(\nabla \chi_{y}^\ep \cdot \left( V' \left(  \xi_y + \nabla \phi_{Q_y} \left(\cdot, \cdot ; \xi_y \right)  \right) - \E \left[ V' \left(  \xi_y + \nabla \phi_{Q_y} \left(\cdot, \cdot ; \xi_y \right) \right) \right] \right) \right)_{z + Q_{3^k}} \neq 0,
\end{equation*}
and that, by Proposition~\ref{prop.prop3.5techni}, for any $y \in \mathcal{Y}_\kappa$,
\begin{equation*}
    \left| \left(\nabla \chi_{y}^\ep \cdot \left( V' \left(  \xi_y + \nabla \phi_{Q_y} \left(\cdot, \cdot ; \xi_y \right) \right) - \E \left[ V' \left(  \xi_y + \nabla \phi_{Q_y} \left(\cdot, \cdot ; \xi_y \right) \right) \right] \right) \right)_{z + Q_{3^k}} \right| \leq \mathcal{O}_2 \left( C L^{-1} 3^{-\frac{dk}{2}} \right).
\end{equation*}
A combination of the two previous displays gives the inequality~\eqref{eq:14210312} in the case $3^{k} \leq \kappa/\ep.$

\medskip

\textit{Substep 2.3. Proof of the inequality~\eqref{eq:14210312}, case $3^{k} \geq L$.} To ease the notation and without loss of generality, we assume that $z = 0$. The strategy of the argument follows the one of Proposition~\ref{prop.2.2BL} and Proposition~\ref{prop.spataveflux}. We consider the dynamic $\phi_{Q_y}^n \left(\cdot , \cdot ; \xi_y \right)$ run with the discretized Brownian motions introduced in these proofs (and let $n$ be the size of the mesh of the discretization). In order to prove~\eqref{eq:14210312}, it is sufficient to show, uniformly in the parameter $n$,
\begin{multline} \label{eq:19180312}
    \left|  \sum_{y \in \mathcal{Y}_\kappa}  \left(\nabla \chi_{y}^\ep \cdot \left( V' \left( \xi_y + \nabla \phi_{Q_y}^n \left(\cdot , \cdot ; \xi_y \right) \right) - \E \left[ V' \left( \xi_y + \nabla \phi_{Q_y}^n \left(\cdot , \cdot ; \xi_y \right) \right) \right] \right) \right)_{z + Q_{3^k}} \right| \\ \leq \mathcal{O}_2 \left( C 3^{- \frac{(d + 2) k}2} \right).
\end{multline}
By the Gaussian concentration inequality (Proposition~\ref{propEfronstein}), it is sufficient to prove the two following results: for any $x \in \Lambda_{3^k + 2 L}$ and any $l \in \left\{ - n 3^{2m}, \ldots, 0 \right\}$,
\begin{equation} \label{eq:1408sam}
    \left|   \sum_{y \in \mathcal{Y}_\kappa}  \frac{\partial \left(\nabla \chi_{y}^\ep \cdot  V' \left( \xi_y + \nabla \phi_{Q_y}^n \left(\cdot , \cdot ; \xi_y \right) \right) \right)_{Q_{3^k}}}{\partial X^n_l (x)} \right| \leq C \sqrt{n^{-1}} 3^{-k(d+2)} \exp \left( - \frac{(l/n)}{3^{2k}} \right),
\end{equation}
and, for any $x \notin \Lambda_{3^k + 4 L}$ and any $l \in \left\{ - n 3^{2m}, \ldots, 0 \right\}$,
\begin{equation} \label{eq:1409sam}
   \sum_{y \in \mathcal{Y}_\kappa}  \frac{\partial \left(\nabla \chi_{y}^\ep \cdot  V' \left( \xi_y + \nabla \phi_{Q_y}^n \left(\cdot , \cdot ; \xi_y \right) \right) \right)_{Q_{3^k}}}{\partial X^n_l (x)} = 0.
\end{equation}
The proof of the inequality~\eqref{eq:1408sam} and the identity~\eqref{eq:1409sam} relies the two following observations. First, for any vertex $x \in \Lambda_{3^m}$, there exist at most $C:= C(d) <\infty$ vertices $y \in \mathcal{Y}_\kappa$ such that $x \in y/\ep + \Lambda_{2 L}$. Second, for any $y \in \mathcal{Y}_\kappa$, the dynamic $\phi_{Q_y}^n \left(\cdot , \cdot ; \xi_y \right)$ depends only on the increments $X_l^n(x)$ inside the parabolic cylinder $Q_y$, that is, on the increments $X_l^n(x)$ with $x \in y/\ep + \Lambda_{2L}$.

Consequently, for any $x \in \Lambda_{3^m}$, there exist at most $C:= C(d) < \infty$ vertices $y \in \mathcal{Y}_\kappa$ such that
\begin{equation*}
    \frac{\partial \left(\nabla \chi_{y}^\ep \cdot  V' \left( \xi_y + \nabla \phi_{Q_y}^n \left(\cdot , \cdot ; \xi_y \right) \right) \right)_{ Q_{3^k}}}{\partial X^n_l (x)} \neq 0.
\end{equation*}
Additionally, if a vertex $y \in \mathcal{Y}_\kappa$ is such that $y/\ep\notin \Lambda_{3^{k} + 2 L}$, then the cylinder $Q_y$ does not intersect the cylinder $Q_{3^k}$, and consequently
\begin{equation*}
    \left(\nabla \chi_{y}^\ep \cdot  V' \left( \xi_y + \nabla \phi_{Q_y}^n \left(\cdot , \cdot ; \xi_y \right) \right) \right)_{ Q_{3^k}} = 0.
\end{equation*}
Finally, using Proposition~\ref{prop.prop3.5techni}, we have, for any $y \in \mathcal{Y}_\kappa$, any $x \in y/\ep + \Lambda_{2 L}$, and any $l \in \left\{ - n 3^{2m}, \ldots, 0 \right\}$,
\begin{equation*} 
     \left| \frac{\partial \left( \nabla \chi_{y}^\ep \cdot  V' \left( \xi_y + \nabla \phi_{Q_y}^n \left(\cdot , \cdot ; \xi_y \right) \right) \right)_{Q_{3^k}}}{\partial X_l^n(x)} \right| \leq C \sqrt{n^{-1}} 3^{-k(d+2)} \exp \left( - \frac{(l/n)}{3^{2k}} \right).
\end{equation*}
Combining the previous remarks and inequalities completes the proof of~\eqref{eq:1408sam} and~\eqref{eq:1409sam}.

\end{proof}

\subsection{Two-scale expansion and proof of Theorem~\ref{Th.quantitativehydr}} \label{secTh.quantitativehydr}

This section is devoted to the proof of Theorem~\ref{Th.quantitativehydr} making use of the previous results established in Section~\ref{sec.section2} and Section~\ref{eq:quanthydrolim2sc}.

\begin{proof}[Proof of Theorem~\ref{Th.quantitativehydr}]
We split the proof of the theorem into many steps. The main objective of the proof is to show that the two-scale expansion $w^\ep$ is almost a solution of the equation~\eqref{eq:defuLthmhydro}. Specifically, we will prove the identity~\eqref{eq:16012107} below and estimate the size of the two (small) error terms $\vec{\mathcal{E}}$ and $\mathcal{E}$. To achieve this, we compute the time derivative and gradient of the map $w^\ep$, and compute the value of the term $\nabla^\ep \cdot V'(\nabla^\ep w^\ep)$ in Steps 1, 2, 3 and 4 below. Once the identity~\eqref{eq:16012107} has been established, we show that the $L^2$-norm of the gradient of the difference $u^\ep - w^\ep$ is small. This is the subject of Step~5. In Step 6, we conclude the proof building upon the results established in Step 5.

\medskip

To ease the notation, we denote by $(B^\ep_\cdot(x))_{x \in \ep \Zd}$ the suitably rescaled Brownian motions: for any $(t , x) \in \R \times D^\ep$,
\begin{equation*}
    B^\ep_t (x) := \ep B_{\frac{t}{\ep^2}} \left( \frac{x}{\ep}\right).
\end{equation*}
We first establish the identity
\begin{equation} \label{eq:16012107}
    \partial_t \left(  w^\ep - \sqrt{2} B_{\cdot}^\ep - \mathcal{E}^{\mathrm{mean}} \right)- \nabla^\ep \cdot V'(\nabla^\ep w^\ep)  = \nabla^\ep \cdot \vec{\mathcal{E}} + \mathcal{E} ,
\end{equation}
where the functions $\vec{\mathcal{E}} : (-1 , 0) \times \vec{E}( D^\ep) \to \R$ and $\mathcal{E}: (-1 , 0) \times D^\ep \to \R$ are explicit error terms satisfying the estimate
\begin{align} \label{eq:10451909}
   \| \vec{\mathcal{E}} \|_{L^2(Q^\ep)} + \left\| \mathcal{E} \right\|_{\underline{H}^{-1}_{\mathrm{par}}(Q^\ep)} \leq \mathcal{O}_2 \left( C \ep^{\frac12} \bigl( 1 + |\log \ep|^{\frac12} \indc_{\{ d = 2\}} \bigr) \right),
\end{align}
and the error term $\mathcal{E}^{\mathrm{mean}} : (-1 , 0) \times D^\ep \to \R$ arises from the correction involving the averaged sums of the Brownian motions in the right-hand side of~\eqref{eq:14311610+1}, and is defined by the formula
\begin{equation} \label{def.Emean}
    \mathcal{E}^{\mathrm{mean}} (t , x) := \sum_{y \in \mathcal{Y}_\kappa} \chi_y(x) \left( \frac{\sqrt{2}}{\left| \Lambda^\ep_{2 \kappa}\right|}\sum_{x' \in y + \Lambda^\ep_{2\kappa}} B^\ep_t \left(x'\right) \right).
\end{equation}

\textit{Step 1. Computing the time derivative of the two-scale expansion $w^\ep$.} To prove the formula~\eqref{eq:16012107} and the estimate~\eqref{eq:10451909}, we first compute the time derivative of the map $w^\ep -  B^\ep$. Using the definition~\eqref{def.wL} of the two-scale expansion $w^\ep$, we obtain the identity
\begin{align} \label{eq:16002107.sec4}
\lefteqn{
\partial_t \left(  w^\ep -  \sqrt{2} B_\cdot^\ep - \mathcal{E}^{\mathrm{mean}}  \right)
} \qquad & \notag \\ & 
= \partial_t \bar u^\ep  +  \partial_t \left( \phi^\ep -  \sqrt{2} B_\cdot^\ep - \mathcal{E}^{\mathrm{mean}} \right) \notag \\
 & = \partial_t  \bar u^\ep   +  \ep \sum_{y \in \mathcal{Y}_\kappa} \chi_{y} \partial_t \left( \phi_{Q_{y}} \left( \frac{\cdot}{\ep^2} , \frac{\cdot}{\ep} ; \xi_{y} \right) - \sqrt{2} B_\cdot^\ep - \frac{\sqrt{2}}{\left| \Lambda^\ep_{2 \kappa}\right|}\sum_{x' \in y + \Lambda^\ep_{2\kappa}} B^\ep_\cdot \left(x'\right) \right). 
\end{align}

\textit{Step 2. Computing the gradient of the two-scale expansion $w^\ep$.} We next compute the gradient of the map $w^\ep$. Using the definition~\eqref{def.wL}, an explicit computation, the definition of the error term $E_z$ and the inequalities~\eqref{prop.partofunity2sc} and~\eqref{eq:estL2normEz}, we obtain, for each $(t , e) \in (-1 , 0) \times \vec{E} \left( D^\ep\right)$,
\begin{equation} \label{formulawL.2sc}
   \nabla^\ep w^\ep (t , e) = \sum_{y \in \mathcal{Y}_\kappa} \vec{\chi}_{y}(e) \nabla^\ep v_{y}(t , e) + \vec{\mathcal{E}}_1 (t , e),
\end{equation}
where we used the notation, for any edge $e = (x_0 , x_1) \in E \left( D^\ep\right)$,
\begin{equation*}
    \vec{\chi}_{y}(e) = \frac{\chi_{y}(x_0) + \chi_{y}(x_1)}{2},
\end{equation*}
and the error term $\vec{\mathcal{E}}_1$ satisfies the $L^2$-bound
\begin{equation} \label{eq:11471208}
    \| \vec{\mathcal{E}}_1 \|_{L^2(Q^\ep)} \leq \mathcal{O}_2 \left( C \ep^{\frac12} \bigl( 1 + | \log \ep |^{\frac12} \indc_{\{ d = 2\}} \bigr) \right).
\end{equation}

\textit{Step 3. Computing the value of $\nabla^\ep \cdot V' \left( \nabla^\ep w^\ep \right)$: the identity~\eqref{eq:15592107}.} Building upon the identity~\eqref{formulawL.2sc} and the bound~\eqref{eq:11471208}, we establish the identity
\begin{align} \label{eq:15592107}
    \nabla^\ep \cdot V' \left( \nabla^\ep w^\ep \right)
    & = \nabla^\ep \cdot \sum_{y \in \mathcal{Y}_\kappa}  \vec{\chi}_{y} V' \left( \nabla^\ep v_{y} \right)   + \nabla^\ep \cdot \vec{\mathcal{E}}_2,
\end{align}
where the error term $\vec{\mathcal{E}}_2$ satisfies the estimate
\begin{equation} \label{eq:15602107}
    \| \vec{\mathcal{E}}_2 \|_{L^2(Q^\ep)} \leq \mathcal{O}_2 \left( C \ep^{\frac12} \bigl( 1 + | \log \ep|^{\frac12} \indc_{\{ d = 2\}} \bigr) \right). 
\end{equation}
To establish~\eqref{eq:15592107} and~\eqref{eq:15602107}, we first decompose the term $\vec{\mathcal{E}}_2$ as follows
\begin{equation} \label{eq:12001208}
    \vec{\mathcal{E}}_2
     = \underset{\eqref{eq:12001208}-(i)}{\underbrace{V' \left( \nabla^\ep w^\ep \right) -  V' \left( \sum_{y \in \mathcal{Y}_\kappa} \vec{\chi}_{y} \nabla^\ep v_{y} \right)}}
     +  \underset{\eqref{eq:12001208}-(ii)}{\underbrace{V' \left( \sum_{y \in \mathcal{Y}_\kappa} \vec{\chi}_{y} \nabla^\ep v_{y} \right)  - \sum_{y \in \mathcal{Y}_\kappa} \vec{\chi}_{y} V' \left( \nabla^\ep v_{y} \right)}}
\end{equation}
and estimate the two terms~\eqref{eq:12001208}-(i) and~\eqref{eq:12001208}-(ii) in two distinct substeps.

\medskip

\textit{Substep 3.1. Estimating the term~\eqref{eq:12001208}-(i).} We use that the map $V'$ is Lipschitz and the identity~\eqref{formulawL.2sc}. We obtain
\begin{equation*}
    \left| V' \left( \nabla^\ep w^\ep \right) -  V' \left( \sum_{y \in \mathcal{Y}_\kappa} \vec{\chi}_{y} \nabla^\ep v_{y} \right) \right| \leq C | \vec{\mathcal{E}}_1 |.
\end{equation*}
Applying the estimate~\eqref{eq:11471208}, we deduce that
\begin{equation} \label{eq:15161508}
    \left\| \eqref{eq:12001208}-(i) \right\|_{L^2(Q^\ep)} \leq \mathcal{O}_2 \left( C \ep^{\frac12} \bigl( 1 + | \log \ep |^{\frac12} \indc_{\{ d = 2\}} \bigr) \right).
\end{equation}

\textit{Substep 3.2. Estimating the term~\eqref{eq:12001208}-(ii).} We use that the map $V'$ is Lipschitz, together with the inequality~\eqref{eq:1805110888}. We obtain, for any $z = (t ,y) \in \mathcal{Z}_\kappa$,
\begin{align*}
    \biggl\|  V' \biggl( \ \sum_{y' \in \mathcal{Y}_\kappa} \vec{\chi}_{y'} \nabla^\ep v_{y'} \biggr)  - V' \left( \nabla^\ep v_y \right)  \biggr\|_{\underline{L}^2(z + Q_\kappa^\ep)} & \leq C \biggl\| \biggl( \ \sum_{y' \in \mathcal{Y}_\kappa} \vec{\chi}_{y'} \nabla^\ep v_{y'} \biggr) - \nabla^\ep v_{y}  \biggr\|_{\underline{L}^2(z + Q_\kappa^\ep)} 
    \leq C E_{z}.
\end{align*}
A similar argument yields
\begin{align*}
    \biggl\|  \sum_{y' \in \mathcal{Y}_\kappa} \vec{\chi}_{y'} V' \left( \nabla^\ep v_{y'} \right)  - V' \left( \nabla^\ep v_{y} \right)  \biggr\|_{\underline{L}^2(z + Q_\kappa^\ep)} 
    & = \biggl\| \sum_{y' \in \mathcal{Y}_\kappa} \vec{\chi}_{y'} \left( V'(\nabla^\ep v_{y'})  -V'(\nabla^\ep v_{y}) \right)  \biggr\|_{\underline{L}^2(z + Q_\kappa^\ep)} \\ 
    & \leq C \biggl\| \sum_{y' \in \mathcal{Y}_\kappa} \vec{\chi}_{y'} \left( \nabla^\ep v_{y'}  -\nabla^\ep v_{y} \right)  \biggr\|_{\underline{L}^2(z + Q_\kappa^\ep)} \\ & \leq C E_{z}.
\end{align*}
Combining the two previous displays and summing over the vertices $z \in \mathcal{Z}_\kappa$, we deduce that
\begin{equation} \label{eq:15171508}
    \left\|  \eqref{eq:12001208}-(ii) \right\|_{L^2( Q^\ep)} \leq \mathcal{O}_2 \left( C \ep^{\frac12} \bigl( 1 + | \log \ep |^{\frac12} \indc_{\{ d = 2\}} \bigr) \right).
\end{equation}
Combining the inequalities~\eqref{eq:15161508} and~\eqref{eq:15171508} completes the proof of~\eqref{eq:15602107}.

\medskip

\textit{Step 4. Computing the value of $\nabla^\ep \cdot V' \left( \nabla^\ep w^\ep \right)$: the identity~\eqref{eq:16391508}.} In this step, we prove the identity
\begin{equation} \label{eq:16391508}
     \nabla^\ep \cdot \sum_{y \in \mathcal{Y}_\kappa}  \vec{\chi}_{y} V' \left( \nabla^\ep v_{y} \right)  =  \sum_{y \in \mathcal{Y}_\kappa}  \chi_{y} \nabla^\ep \cdot V' ( \nabla^\ep v_{y} )  +  \vec{\nabla}^\ep \cdot D_p \bar \sigma \left( \vec{\nabla}^\ep \bar u^\ep \right) + \mathcal{E} + \nabla^\ep \cdot \vec{\mathcal{E}}_3, 
\end{equation}
where the two error terms $\mathcal{E}$ and $\vec{\mathcal{E}}_3$ satisfy
\begin{equation} \label{eq:09191608}
    \left\| \mathcal{E} \right\|_{\underline{H}^{-1}_{\mathrm{par}} \left( Q^\ep \right)} + \| \vec{\mathcal{E}}_3 \|_{L^2(Q^\ep)} \leq \mathcal{O}_2 \left( C \ep^{\frac12} \bigl( 1 + | \log \ep |^{\frac12} \indc_{\{ d = 2\}} \bigr) \right).
\end{equation}
We introduce the following time-dependent vector field
\begin{equation*}
    \sum_{y \in \mathcal{Y}_\kappa}\vec{\chi}_{y} D_p \bar \sigma \left( \xi_{y} \right) : (t , e) \mapsto  \sum_{y \in \mathcal{Y}_\kappa}\vec{\chi}_{y}(e) D_p \bar \sigma \left( \xi_{y}(t) \right) \cdot e.
\end{equation*}
To prove the identity~\eqref{eq:16391508} and the upper bound~\eqref{eq:09191608}, we will prove the two identities
\begin{equation} \label{eq:09501608}
    \vec{\nabla}^\ep \cdot D_p \bar \sigma \left( \vec{\nabla}^\ep \bar u^\ep \right) = \vec{\nabla}^\ep \cdot \biggl( \ \sum_{y \in \mathcal{Y}_\kappa}\vec{\chi}_{y} D_p\bar \sigma \left( \xi_{y} \right) \biggr) + \nabla^\ep \cdot \vec{\mathcal{E}_3}
\end{equation}
and
\begin{equation} \label{eq:09511608}
    \nabla^\ep \cdot \biggl( \ \sum_{y \in \mathcal{Y}_\kappa}  \vec{\chi}_{y} \left( V' \left( \nabla^\ep v_{y} \right) -  D_p\bar \sigma \left( \xi_{y} \right)\right) \biggr)=  \sum_{y \in \mathcal{Y}_\kappa}  \chi_{y} \nabla^\ep \cdot V' \left( \nabla^\ep v_{y} \right)  + \mathcal{E}
\end{equation}
together with the estimate~\eqref{eq:09191608}. Each of these identities is proved in a specific substep.

\medskip

\textit{Substep 4.1. Proof of the identity~\eqref{eq:09501608}.} In this step, we focus on the identity~\eqref{eq:09501608} and the estimate on the error term $\vec{\mathcal{E}}_3$. Using the definition~\eqref{eq:discellop} and the identity $\sum_{y \in \mathcal{Y}_\kappa} \vec{\chi}_{y}=1$, we may write, for any $t \in (-1,0)$ and $e = (x_0 , x_1) \in \vec{E}(D^\ep)$,
\begin{align*}
    \vec{\mathcal{E}}_3 (t , e) & := D_p \bar \sigma \left( \vec{\nabla}^\ep \bar u^\ep(t , x_0) \right) \cdot e - \sum_{y \in \mathcal{Y}_\kappa}\vec{\chi}_{y}(t , e)  D_p \bar \sigma \left( \xi_{y} \right) \cdot e \\
    & \; = \sum_{y \in \mathcal{Y}_\kappa}\vec{\chi}_{y}(t , e) \left( D_p \bar \sigma \left( \vec{\nabla}^\ep \bar u ^\ep(t , x_0) \right) \cdot e -  D_p \bar \sigma \left( \xi_{y} \right) \cdot e \right).
\end{align*}
For any $z = (t , y) \in \mathcal{Z}_\kappa$, we have
\begin{equation*}
     \| \vec{\mathcal{E}}_3 \|_{\underline{L}^{2} \left( z +  Q_\kappa^\ep \right)} \leq C \sum_{y' \sim y} \bigl\| \vec{\nabla}^\ep \bar u ^\ep - \xi_{y'} \bigr\|_{\underline{L}^2(z +  Q_\kappa^\ep)}  \leq C E_z.
\end{equation*}
Summing over the vertices $z \in \mathcal{Z}_\kappa$, we obtain
\begin{equation*}
     \| \vec{\mathcal{E}}_3 \|_{L^2(Q^\ep)} \leq \mathcal{O}_2 \left( C \ep^{\frac12} \bigl( 1 + | \log \ep |^{\frac12} \indc_{\{ d = 2\}} \bigr) \right).
\end{equation*}

\medskip

\textit{Substep 4.2. Proof of the identity~\eqref{eq:09511608}.} In this substep, we focus on the identity~\eqref{eq:09511608} and prove the error estimate~\eqref{eq:09191608} on the term $\mathcal{E}_2$. Expanding the discrete divergence, we may write
\begin{align} \label{eq:11111608}
     \nabla^\ep \cdot \biggl( \ \sum_{y \in \mathcal{Y}_\kappa}  \vec{\chi}_{y} \left( V' \left( \nabla^\ep v_{y} \right) - D_p \bar \sigma \left( \xi_{y} \right)\right) \biggr)& =   \sum_{y \in \mathcal{Y}_\kappa}  \chi_{y} \nabla^\ep \cdot \left( V' \left( \nabla^\ep v_{y} \right) - D_p  \bar \sigma \left( \xi_{y} \right) \right) \\ & \qquad + \sum_{y \in \mathcal{Y}_\kappa} \nabla^\ep \chi_{y} \cdot \left( V' \left( \nabla^\ep v_{y} \right) - D_p \bar \sigma \left( \xi_{y} \right) \right), \notag
\end{align}
where we used the notation, for any $(t , x) \in (0 , \infty) \times D^\ep$,
\begin{align} \label{eq:16350512}
\lefteqn{ 
\sum_{y \in \mathcal{Y}_\kappa} \nabla^\ep \chi_{y} \cdot \left( V' \left( \nabla^\ep v_{y} \right) -  D_p \bar \sigma \left( \xi_{y} \right) \right) (t , x) }
\qquad & \notag \\
& 
=  \sum_{y \in \mathcal{Y}_\kappa}  \sum_{\substack{e \in \vec{E}\left( D^\ep\right) \\ e = (x , y')}} \nabla^\ep \chi_{y}(t , e) \left( V' \left( \nabla^\ep v_{y}(t , e) \right) -  D_p \bar \sigma \left( \xi_{y}  \right) \cdot e \right).
\end{align}
We next simplify the previous display as follows: using that the terms $ D_p \bar \sigma \left( \xi_{y} \right)$ are spatially constant, we write
\begin{equation} \label{eq:11112608}
    \nabla^\ep \cdot \left( V' \left( \nabla^\ep v_{y} \right) - D_p \bar \sigma \left( \xi_{y} \right) \right) = \nabla^\ep \cdot  V' \left( \nabla^\ep v_{y} \right).
\end{equation}
From~\eqref{eq:11111608} and~\eqref{eq:11112608}, we obtain the following formula for the error term $\mathcal{E}$,
\begin{equation*}
    \mathcal{E} := \sum_{y \in \mathcal{Y}_\kappa} \nabla^\ep \chi_{y} \cdot \left( V' \left( \nabla^\ep v_{y} \right) - D_p \bar \sigma \left( \xi_{y} \right) \right).
\end{equation*}
The $\underline{H}^{-1}_{\mathrm{par}}(Q^\ep)$-norm of the error term $\mathcal{E}$ is then estimated by Proposition~\ref{prop4.3fluxest}. The proofs of the identity~\eqref{eq:09511608} and the inequality~\eqref{eq:09191608} are complete.

\medskip

\textit{Step 5. Estimating the norm of the difference $v := u^\ep - w^\ep$.} We next combine the three identities~\eqref{eq:16002107.sec4},~\eqref{eq:15592107} and~\eqref{eq:16391508} and use the definition~\eqref{def.discequep} of the map $\bar u^\ep$ and the one of the Langevin dynamic~\eqref{eq:defuLthmhydro}. We obtain the identity~\eqref{eq:16012107} with the error term
\begin{equation*}
    \vec{\mathcal{E}} := \vec{\mathcal{E}}_2 + \vec{\mathcal{E}}_3.
\end{equation*}
The upper bounds~\eqref{eq:15602107} and~\eqref{eq:09191608} ensure that the error terms $\mathcal{E}$ and $\vec{\mathcal{E}}$ satisfy the estimate~\eqref{eq:10451909}. We then take the difference between the equations~\eqref{eq:defuLthmhydro} and~\eqref{eq:16012107} and obtain that the map $v := u^\ep - w^\ep$ solves the linear parabolic equation
\begin{equation} \label{eq:v2scexp}
    \left\{ \begin{aligned}
    \partial_t \left( v - \mathcal{E}^{\mathrm{mean}} \right) - \nabla^\ep \cdot \a \nabla^\ep v & = \nabla^\ep \cdot \vec{\mathcal{E}} + \mathcal{E} &~\mbox{in}~&Q^\ep, \\
    v & = \phi^\ep &~\mbox{in}~ &\partial_{\mathrm{par}} Q^\ep.
    \end{aligned} \right.
\end{equation}
where the environment $\a$ is given by the formula
\begin{equation*}
    \a(t , e) := \int_0^1 V''(s \nabla^\ep \bar u^\ep (t , e) + (1-s)\nabla^\ep w^\ep (t , e) ) ds.
\end{equation*}
Using the linearity of the equation~\eqref{eq:v2scexp}, we may decompose the map $v$ according to the formula $v = v_0 + v_1 + v_2$, where the functions $v_0$, $v_1$ and $v_2$ are the solutions of the parabolic equations
\begin{equation} \label{eq:v0scexp}
    \left\{ \begin{aligned}
    \partial_t v_0 - \nabla^\ep \cdot \a \nabla^\ep v_0 & = 0 &~\mbox{in}~&Q^\ep, \\
    v_0 & = \phi^\ep &~\mbox{in}~& \partial_{\mathrm{par}}  Q^\ep,
    \end{aligned} \right.
\end{equation}
and
\begin{equation} \label{eq:v1scexp}
    \left\{ \begin{aligned}
    \partial_t v_1 - \nabla^\ep \cdot \a \nabla^\ep v_1 & = \nabla^\ep \cdot \vec{\mathcal{E}} + \mathcal{E} &~\mbox{in}~&Q^\ep, \\
    v_1 & = 0 &~\mbox{in}~& \partial_{\mathrm{par}}  Q^\ep,
    \end{aligned} \right.
\end{equation}
and
\begin{equation} \label{eq:v22scexp}
    \left\{ \begin{aligned}
    \partial_t \left( v_2 - \mathcal{E}^{\mathrm{mean}} \right) - \nabla^\ep \cdot \a \nabla^\ep v_2 & = 0  &~\mbox{in}~&Q^\ep, \\
    v_2 & = 0 &~\mbox{in}~& \partial_{\mathrm{par}}  Q^\ep.
    \end{aligned} \right.
\end{equation}
We then estimate the $L^2$-norm of the gradients of the three maps $v_0$, $v_1$ and $v_2$ in three separate substeps.

\medskip

\textit{Substep 5.1. Estimating the term $v_0$.} Regarding the term $v_0$, we will prove the estimate
\begin{equation} \label{eq:17592009}
    \left\| \nabla v_0 \right\|_{L^2 ( Q^\ep)} \leq  \mathcal{O}_2 \left( C \ep^{\frac12} \bigl( 1 + | \log \! \ep |^{\frac12} \indc_{\{ d = 2\}} \bigr) \right).
\end{equation}
The discrete version of Stokes' theorem (suitably scaled to take into account the parameter~$\ep$) gives
\begin{equation*}
    \sum_{x \in D^\ep} \left(- \nabla^\ep \cdot \a \nabla^\ep v_0 (t , x) \right) v_0(t , x) = \sum_{e \in E(D^\ep)}  \a(t , e) (\nabla^\ep v_0 (t , x))^2 + \frac1\ep \sum_{x \in \partial D^\ep} \phi^\ep(t , x) \mathbf{n} \cdot \a \nabla^\ep v_0 (t , x),
\end{equation*}
where the notation $\mathbf{n} \cdot \a \nabla^\ep v_0$ refers to the discrete version of the normal component of the vector field $\a \nabla^\ep v_0$, formally defined by, for $t \in (-1 ,0)$ and $x \in \partial D^\ep$,
\begin{equation*}
    \mathbf{n} \cdot \a \nabla^\ep v_0 (t , x) := \sum_{\substack{e = (y , x) \in \vec{E}(\Zd) \\ y \in D^\ep}}  \a(t , e) \nabla^\ep v_0(t , e).
\end{equation*}
Multiplying the equation~\eqref{eq:v0scexp} by $v_0$, integrating over the parabolic cylinder $Q^\ep$ and applying the Cauchy-Schwarz inequality, we obtain (after multiplication by the volume factor $\ep^d$ to normalize the $L^2(Q^\ep)$-norms)
\begin{align} \label{eq:12500712}
    \lefteqn{\frac 12 \left\| v_0 (0 , \cdot) \right\|_{L^2(D^\ep)}^2 - \frac 12 \left\| \phi^\ep (-1 , \cdot) \right\|_{L^2(D^\ep)}^2 + \left\| \nabla^\ep v_0 \right\|_{L^2(Q^\ep)}^2} \qquad & \\ &
    \leq C \ep^{-1} \Biggl( \ep^d \int_I \sum_{x \in \partial D^\ep} |\phi^\ep(t , x)|^2 dt \Biggr)^{\frac12} \Biggl( \ep^d \int_I \sum_{\substack{e = (x , y) \in E(D^\ep) \\ x \in D^\ep, y \notin D^\ep}} |\nabla^\ep v_0(t , e)|^2 dt \Biggr)^{\frac12}. \notag
\end{align}
We next use the inequality
\begin{equation*}
    \ep^d \int_I \sum_{\substack{e = (x , y) \in E(\Zd) \\ x \in D^\ep, y \notin D^\ep}} |\nabla^\ep v_0(t , e)|^2 dt \leq \ep^d \int_I \sum_{e\in E(D^\ep) } |\nabla^\ep v_0(t , e)|^2 dt = \left\| \nabla^\ep v_0 \right\|_{L^2(Q^\ep)}^2.
\end{equation*}
Combining the two previous displays and using that the first term in the right-hand side of~\eqref{eq:12500712} is nonnegative, we obtain the upper bound
\begin{equation*}
     \left\| \nabla^\ep v_0 \right\|_{L^2(Q^\ep)}^2 \leq C \ep^{-1} \Biggl( \ep^d \int_I \sum_{x \in \partial D^\ep} |\phi^\ep(t , x)|^2 dt \Biggr)^{\frac12} \left\| \nabla^\ep v_0 \right\|_{L^2(Q^\ep)} +  \frac 12 \left\| \phi^\ep(0 , \cdot) \right\|_{L^2(D^\ep)}^2,
\end{equation*}
which implies
\begin{equation*}
    \left\| \nabla^\ep v_0 \right\|_{L^2(Q^\ep)}^2 \leq C \ep^{d-2} \int_I \sum_{x \in \partial D^\ep} |\phi^\ep(t , x)|^2 dt + C \left\| \phi^\ep(0 , \cdot) \right\|_{L^2(D^\ep)}^2.
\end{equation*}
Using the pointwise bound on the corrector stated in Proposition~\ref{prop.2.2BL}, the fact that the cardinality of the boundary $\partial D^\ep$ is of order $\ep^{1-d}$ (since $D$ is assumed to be Lipschitz) and the property~\eqref{sum.Onotation} of the $\mathcal{O}$-notation, we obtain
\begin{equation} \label{eq:18060710}
    \ep^{d-2} \int_I \sum_{x \in \partial D^\ep} |\phi^\ep(t , x)|^2 dt  \leq \mathcal{O}_1 \left(  C\ep  \left( 1 + \left| \log \! \ep \right| \indc_{\{ d = 2\}} \right) \right)
\end{equation}
and, by a similar argument,
\begin{equation*}
    \left\| \phi^\ep(0 , \cdot) \right\|_{L^2(Q^\ep)}^2 \leq \mathcal{O}_1 \left(  C\ep^2\left( 1 + \left| \log \! \ep \right| \indc_{\{ d = 2\}} \right) \right).
\end{equation*}
A combination of the previous displays gives
\begin{equation} \label{eq:estv_0goodone}
    \left\| \nabla^\ep v_0 \right\|_{L^2(Q^\ep)} \leq \mathcal{O}_2 \left( C \ep^{\frac12} \bigl( 1 + | \log \! \ep |^{\frac12} \indc_{\{ d = 2\}} \bigr) \right).
\end{equation}

\medskip

\textit{Substep 5.2. Estimating the term $v_1$.} In this substep, we estimate the term $\underline{L}^2 ( Q^\ep)$-norm of the map $v_1$. Applying Lemma~\ref{lem.idenH-1par} together with a rescaling argument, we obtain the existence of a pair of functions~$(w , w^*) \in L^2(I ; H^1_0(D^\ep)) \times L^2(I ; H^{-1}(D^\ep))$ such that
\begin{equation} \label{eq:decompE}
    \mathcal{E} = \partial_t w  + w^*,
\end{equation}
and $t \mapsto w(t , x)$ is continuous, $w(-1 , x) = w(0 , x) = 0$, and
\begin{equation*}
    \left\| w \right\|_{\underline{L}^2(I ; \underline{H}^1(D^\ep))} + \left\| w^* \right\|_{\underline{L}^2(I ; \underline{H}^{-1}(D^\ep))}  \leq C \left\| \mathcal{E} \right\|_{\underline{H}^{-1}_{\mathrm{par}}(Q^\ep)}.
\end{equation*}
Combining the equations \eqref{eq:v1scexp} and~\eqref{eq:decompE}, we deduce that
\begin{equation*}
    \partial_t ( v_1 - w) = \nabla^\ep \cdot \a \nabla^\ep v_1 + \nabla^\ep \cdot \vec{\mathcal{E}} + w^* ~\mbox{in}~ Q^\ep.
\end{equation*}
Multiplying the equation by $(v_1 - w)$, integrating over the parabolic cylinder $Q^\ep$, and using that $w$ is equal to $0$ on the boundary of the time interval, we obtain the estimate
\begin{align*}
    \frac12\left\| v_1(0 , \cdot)\right\|_{L^2(D^\ep)}^2 + \left\| \nabla^\ep v_1 \right\|_{L^2(Q^\ep)}^2 & \leq C \left\| w \right\|_{\underline{L}^2(I ; \underline{H}^1(D^\ep))}^2 + C \left\| w^* \right\|_{\underline{L}^2(I ; \underline{H}^{-1}(D^\ep))}^2  + C \| \vec{\mathcal{E}} \|_{L^2(Q^\ep)}^2 \\
    & \leq C \| \vec{\mathcal{E}} \|_{L^2(Q^\ep)}^2 + C \left\| \mathcal{E} \right\|_{\underline{H}^{-1}_{\mathrm{par}}(Q^\ep)}^2.
\end{align*}
We deduce that
\begin{equation} \label{bound nabla v1ep}
    \left\| \nabla^\ep v_1 \right\|_{L^2 (Q^\ep)} \leq \mathcal{O}_2 \left( C \ep^{\frac12} \bigl( 1 + | \log \ep |^{\frac12} \indc_{\{ d = 2\}} \bigr) \right).
\end{equation}

\medskip

\textit{Substep 5.3. Estimating the term $v_2$.} We first let $Y$ be the solution of the stochastic differential equation
\begin{equation*}
    \left\{ \begin{aligned}
    d Y_t(x) &= - \frac{1}{\kappa^2} Y_t(x) \, dt + d \mathcal{E}^{\mathrm{mean}}(t , x) &~\mbox{for}~ (t, x) \in (-1 , 0) \times D^\ep,  \\
    Y_{-1}(x) & = 0 &~\mbox{for}~ x \in D^\ep.
    \end{aligned}
   \right.
\end{equation*}
Since the term $\mathcal{E}^{\mathrm{mean}}$ is Gaussian (as a sum of independent Brownian motions), the process $Y$ is Gaussian. Moreover, it can be written as the stochastic integral
\begin{equation*}
    Y_t(x) = \int_{-1}^t e^{-\frac{(t-s)}{\kappa^2}} d \mathcal{E}^{\mathrm{mean}}(s , x)
\end{equation*}
Combining the previous identity with the definition of $\mathcal{E}_{\mathrm{mean}}$ stated in~\eqref{def.Emean}, we see that the process $Y$ satisfies the following properties: for each $t , s\in (-1,0)$, and each $x, y \in D^\ep$,
\begin{equation} \label{eq:16481312}
    \left\{ \begin{aligned}
    \mathrm{var} \left[ Y_t(x) \right] & \leq C \ep^d \kappa^{2-d}, \\
    \mathrm{cov} \left[Y_t(x) , Y_s(y) \right] & \leq C  \ep^d \kappa^{2-d} e^{-\frac{|t - s|}{\kappa^2}}, \\
     \mathrm{cov} \left[ Y_t(x) , Y_s(y)\right] &= 0 &\mbox{if}~ |x - y| \geq 4 \kappa.
    \end{aligned} \right.
\end{equation}
Combining these properties with the multiscale Poincar\'e inequality (after suitable rescaling and omitting some of the technical details), yields the bound on the $\underline{H}^{-1}_{\mathrm{par}}(Q^\ep)$-norm of the process $Y$,
\begin{equation} \label{eq:16491312}
    \left\| Y \right\|_{\underline{H}^{-1}_{\mathrm{par}}(Q^\ep)} \leq C \kappa \times \ep^{\frac{d}{2}} \kappa^{1-\frac{d}{2}}.
\end{equation}
Intuitively, the estimate~\eqref{eq:16491312} is the product of two terms: the first one corresponds to the scale above which the process $Y$ behaves like an uncorrelated, space-time, white noise and the second one corresponds to the typical size of $Y$ (as obtained in the first line of~\eqref{eq:16481312}).
Using the process $Y$, we may rewrite the definition of the map $v_2$ as follows
    \begin{equation*}
    \left\{ \begin{aligned}
    \partial_t \left( v_2 - Y \right) - \nabla^\ep \cdot \a \nabla^\ep v_2 & = \frac{1}{\kappa^2} Y &~\mbox{in}~&Q^\ep, \\
    v_2 & = 0 &~\mbox{in}~& \partial_{\mathrm{par}}  Q^\ep.
    \end{aligned} \right.
\end{equation*}
In this form, the $L^2(Q^\ep)$-norm of the gradient of the map $v_2$ using energy estimates, the bounds~\eqref{eq:16481312} and~\eqref{eq:16491312}, and the identification of the space $\underline{H}^{-1}_{\mathrm{par}}(Q^\ep)$ similar to the one used to estimate the function~$v_1$. We obtain (skipping the technical details as the techniques are essentially the same as the ones used to estimate the $L^2$-norms of the gradient of the functions $v_0$ and $v_1$ above)
\begin{equation} \label{eq:17011312}
    \left\| \nabla^\ep v_2 \right\|_{L^2(Q^\ep)} \leq \mathcal{O}_2 \left(  \frac{C \ep}{\kappa} + C \kappa \right) \leq \mathcal{O}_2 \left( C \ep^{\frac12} \left( 1 + \left| \ln \! \ep \right|^{\frac 12} \right) \right).
\end{equation}

\medskip

\textit{Step 6. The conclusion.} Combining~\eqref{eq:estv_0goodone},~\eqref{bound nabla v1ep} and~\eqref{eq:17011312}, we obtain
\begin{equation*}
    \left\| \nabla^\ep  u^\ep - \nabla^\ep w^\ep \right\|_{L^2(Q^\ep)} \leq \mathcal{O}_2 \left( C \ep^{\frac 12} \left( 1 + \left| \ln \! \ep \right|^{\frac 12} \right) \right),
\end{equation*}
which completes the proof of the estimate stated in Remark~\ref{remark1.2}. Combining the previous inequality with the bound~\eqref{eq:18060710} and the Poincar\'e inequality (with a trace term) gives the upper bound
\begin{equation*}
    \left\| u^\ep - w^\ep \right\|_{L^2(Q^\ep)} \leq \mathcal{O}_2 \left( C \ep^{\frac 12} \left( 1 + \left| \ln \! \ep \right|^{\frac 12} \right) \right).
\end{equation*}
Using the definition of the two-scale expansion $w^\ep$ (stated in~\eqref{def.wL}) and Proposition~\ref{prop.2.2BL}, we obtain the estimate
\begin{equation*}
    \left\| \bar u - w^\ep \right\|_{L^2(Q^\ep)} \leq \mathcal{O}_2 \left( C \ep \left( 1 + \left| \ln \! \ep \right|^{\frac 12} \right) \right).
\end{equation*}
A combination of the two previous displays together with the triangle inequality completes the proof of Theorem~\ref{Th.quantitativehydr}.
\end{proof}

\section{$C^2$-regularity of the surface tension} \label{sec.section4}
This section is devoted to the proof of the $C^2$-regularity of the surface tension $\bar \sigma$ following the outline presented in Section~\ref{OutlineC2reg}. The section is mostly independent of Sections~\ref{sec.section2} and~\ref{eq:quanthydrolim2sc} (except that the bound~\eqref{eq:20291910grad} of Proposition~\ref{prop.2.2BL} is used), and is structured as follows. In Section~\ref{section5.1}, we prove a general property of stochastic processes. In Section~\ref{section.linearizationC2reg}, we establish a linearization estimate for the Langevin dynamic, following the argument of~\cite[Lemma 2.4]{AFK1} in the case of nonlinear elliptic equation. Finally in Section~\ref{sectionC2regfinal}, we introduce a finite-volume approximation $(\tau_L)_{L \geq 1}$ of the surface tension and establish, building upon the results of the Sections~\ref{section5.1} and~\ref{section.linearizationC2reg}, that its gradient is equicontinuous. Taking the limit $L \to \infty$ then implies that the second derivative of the surface tension is continuous by the Arzel\`a-Ascoli Theorem. 

In the rest of this section, and contrary to the two previous sections, we only consider Langevin dynamics with slopes which are constant in time. For $z \in \R \times \Zd$, we denote by $\phi_{z , L}(\cdot ; p)$ (resp. $\phi_{L}(\cdot ; p)$) the solution of the dynamic introduced in Definition~\ref{def.firstordercorrfinvol} in the parabolic cylinder $z + Q_L$ (resp. $Q_L$) with a constant slope $p\in \Rd$.

\subsection{A general property for the solutions of stochastic differential equations} \label{section5.1}

In this section, we prove the following general property for solutions of stochastic differential equations of the form $d X_t = h_t dt + dB_t$.

\begin{proposition} \label{prop:diffOrnstein-Uhlenbeck}
    Fix a constant $K > 0$ and a filtration $\left(\mathcal{F}_t\right)_{t \geq 0}$. Let $\left( B_t \right)_{t \geq 0}$ be a Brownian motion for the filtration $(\mathcal{F}_t)_{t \geq 0}$ and let $(h_t)_{t \geq 0}$ square-integrable process adapted to the filtration $(\mathcal{F}_t)_{t \geq 0}$. Assume that $(h_t)_{t \geq 0}$ satisfies the integrability estimate
    \begin{equation} \label{eq:asshstochbounded}
        \left\| h \right\|_{L^2((0,1))} \leq \mathcal{O}_2(K).
    \end{equation}
    Let $X_t$ be a solution of the stochastic differential equation
    \begin{equation} \label{eq:15292009}
        dX_t = h_t dt + dB_t \hspace{5mm} \mbox{in} \hspace{5mm} (0, 1).
    \end{equation}
    Then there exists a constant $C := C(K) < \infty$ such that, for any $\ep > 0$,
    \begin{equation*}
        \int_0^1 \indc_{ \left\{ \left| X_t \right|\leq \ep \right\} } \, dt \leq \mathcal{O}_2 \left( C \ep \right).
    \end{equation*}
\end{proposition}

\begin{proof}
Fix $\ep > 0$, and let $f_\ep : \R \to \R$ be defined by the formula
\begin{equation*}
    f_\ep (x) := \max( 2 - \ep^{-1} |x| , 0 ).
\end{equation*}
We note that this function is continuous, compactly supported in the interval $[-2 \ep , 2\ep]$ and satisfies $f_\ep \geq \indc_{ [-\ep , \ep]}$. We next denote by $F_\ep : \R \to \R$ the function satisfying
\begin{equation*}
    \left( F_\ep \right)'' = f_\ep \hspace{5mm} \mbox{and} \hspace{5mm} F_\ep = 0 ~\mbox{on}~(-\infty, -2\ep).
\end{equation*}
The explicit formula of $F_\ep$ could be easily obtained but is not relevant for our purposes; instead we record four of its elementary properties: the map~$F_\ep$ is twice continuously differentiable, is nonnegative, satisfies~$F_\ep(x) \leq 2\ep \max(x + 2\ep, 0)$ and $F_\ep'(x) \leq 2\ep \indc_{[-2\ep , \infty)}$. 
Applying It\^{o}'s formula to the process $(F_\ep(X_t))_{t \geq 0}$ yields
\begin{align*}
    F_{\ep}(X_1) & =F_{\ep}( X_0) + \int_0^1 F'_\ep (X_t) \, dX_t + \frac12 \int_0^1 F_{\ep}''(X_t) \, dt \\
    & = F_{\ep}(X_0) + \int_0^1 F'_\ep (X_t) \, dB_t + \int_0^1 F'_\ep (X_t) h_t \, dt + \frac12 \int_0^1 F_{\ep}''(X_t) \, dt.
\end{align*}
Using the properties of the function $F_\ep$, we deduce that
\begin{align} \label{eq:09502109}
   \int_0^1 \indc_{ \left\{ \left| X_t \right|\leq \ep \right\} } \, dt 
    & \leq \int_0^1 F_{\ep}''(X_t) \, dt \\
    & \leq 2 \left| F_\ep(X_1) - F_\ep(X_0)\right| + 2 \biggl| \int_0^1 F'_\ep (X_t) \, dB_t \biggr| + 4\ep \!\int_0^1 \left| h_t \right| \, dt. \notag
\end{align}
We then estimate the three terms in the right-hand side. For the first one, we use that the map $F'_\ep$ is $(2\ep)$-Lipschitz together with the definition~\eqref{eq:15292009} of the stochastic process $X_t$. We obtain
\begin{equation*}
\left| F_\ep(X_1) - F_\ep(X_0)\right|
\leq 2\ep \left| X_1 - X_0 \right| 
\leq 2\ep \left| B_1 + \int_0^1 h_t \, dt \right| 
\leq 2\ep \left| B_1 \right| +2\ep \left\| h \right\|_{L^2(0,1)}.
\end{equation*}
Using that the law of $B_1$ is a centered normal distribution of variance $1$ together with the assumption~\eqref{eq:asshstochbounded}, we deduce that
\begin{equation} \label{eq:12109}
    \left| F_\ep(X_1) - F_\ep(X_0)\right| \leq \mathcal{O}_2(C\ep).
\end{equation}
To estimate the second term, we use that $t \mapsto \int_0^t F_\ep'(X_s) dB_s$ is a martingale whose quadratic variation is the function $t \mapsto \int_0^t F_\ep'(X_s)^2 \, ds$. By the Dubins-Schwarz theorem (extending the function $s \mapsto F_\ep'(X_s)$ by $1$ for $s \geq 1$ if necessary), there exists a Brownian motion $(\beta_t)_{t \geq 0}$ such that $\int_0^t F_\ep'(X_s) dB_s = \beta_{\int_0^t F_\ep'(X_s)^2 \, ds}$. Using that the map $F_\ep'$ is bounded by $2\ep$ and standard estimates on the supremum of a Brownian motion, we deduce that
\begin{equation} \label{eq:22109}
   \left|  \int_0^1 F_\ep'(X_s) dB_s \right| =  \left| \beta_{\int_0^1 F_\ep'(X_s)^2 \, ds} \right| \leq \sup_{s \in [0 , 4\ep^2]} \left| \beta_s \right| \leq \mathcal{O}_2 \left( C \ep \right).
\end{equation}
Combining~\eqref{eq:09502109} with~\eqref{eq:12109} and~\eqref{eq:22109} and the definition of the map $F_\ep$, we deduce the existence of a finite constant $C := C(K)<\infty$ such that
\begin{equation*}
\int_0^1 \indc_{ \left\{ \left| X_t \right|\leq \ep \right\} } \, dt  \leq \left| F_\ep(X_1) - F_\ep(X_0)\right| + 2 \left| \int_0^1 F'_\ep (X_t) \, dB_t \right| + 4\ep \left\| h \right\|_{L^2(0,1)} 
\leq \mathcal{O}_2(C \ep)\,.
\end{equation*}
The proof of Proposition~\ref{prop:diffOrnstein-Uhlenbeck} is complete.
\end{proof}

As a consequence of Proposition~\ref{prop:diffOrnstein-Uhlenbeck}, we record the following corollary which estimates the amount of time spent by the process $X_t$ in a measurable set of prescribed Lebesgue measure.

\begin{corollary} \label{cor:diffOrnstein-Uhlenbeck}
    Fix a constant $K > 0$ and a filtration $\left(\mathcal{F}_t\right)_{t \geq 0}$. Let $\left( B_t \right)_{t \geq 0}$ be a Brownian motion for the filtration $(\mathcal{F}_t)_{t \geq 0}$ and let $(h_t)_{t \geq 0}$ square-integrable process adapted to the filtration $(\mathcal{F}_t)_{t \geq 0}$. Assume that $(h_t)_{t \geq 0}$ satisfies the integrability estimate
    \begin{equation*}
        \left\| h \right\|_{L^2(0,1)} \leq \mathcal{O}_2(K).
    \end{equation*}
    Let $(X_t)_{t \in [0 , 1]}$ be a solution of the stochastic differential equation
    \begin{equation} \label{eq:1529200922}
        dX_t = h_t dt + dB_t \hspace{5mm} \mbox{in} \hspace{5mm} (0, 1).
    \end{equation}
    Then there exists a constant $C := C(K) < \infty$ such that for any measurable set $A \in \mathcal{B}(\R)$,
    \begin{equation*}
        \int_0^1 \indc_{ \left\{ X_t \in A \right\} } \, dt \leq \mathcal{O}_2 \left( C |A| \right).
    \end{equation*}
\end{corollary}

\begin{proof}
We first note that, for any $a \in \R$, the process $X_t + a$ is a solution of the stochastic differential equation~\eqref{eq:1529200922}. Applying Proposition~\ref{prop:diffOrnstein-Uhlenbeck}, we obtain that, for every~$a \in \R$ and~$\ep > 0$,
\begin{equation} \label{eq:10312109}
    \int_0^1 \indc_{ \left\{ \left| X_t +a \right|\leq \ep \right\} } \, dt \leq \mathcal{O}_2 \left( C \ep \right).
\end{equation}
The stochastic integrability estimate~\eqref{eq:10312109} can be equivalently reformulated as follows: for any interval $I \subseteq \R$,
\begin{equation} \label{eq:10332109}
    \int_0^1 \indc_{ \left\{ X_t \in I \right\} } \, dt \leq \mathcal{O}_2 \left( C |I| \right).
\end{equation}
Let us consider a measurable set $A \subseteq \R$ satisfying $|A| \leq 1$ (otherwise the statement is trivially satisfied). Using the outer regularity of the Lebesgue measure on $\R$, we deduce that there exists a countable collection of disjoint intervals $(I_n)_{n \in \mathbb{N}}$ satisfying $A \subseteq \cup_{n \in  \mathbb{N}} I_n$ and $\sum_{n \in \mathbb{N}} |I_n| \leq 2|A|$. We thus obtain
\begin{equation*}
    \int_0^1 \indc_{ \left\{ X_t \in A \right\} } \, dt \leq \sum_{n \in \N} \int_0^1 \indc_{ \left\{ X_t \in I_n \right\}  } \, dt.
\end{equation*}
Using Proposition~\ref{prop:diffOrnstein-Uhlenbeck}, we deduce that
\begin{equation*}
     \int_0^1 \indc_{ \left\{ X_t \in A \right\} } \, dt \leq \mathcal{O}_2 \left(C \sum_{n \in \mathbb{N}} |I_n|\right)  \leq \mathcal{O}_2(C|A|).
\end{equation*}
The proof of Corollary~\ref{cor:diffOrnstein-Uhlenbeck} is complete.
\end{proof}

\subsection{Linearization estimate} \label{section.linearizationC2reg}

In this section, and following~\cite[Lemma 2.4]{AFK2}, we identify the first-order term in the asymptotic development of the difference $\phi_L (\cdot, \cdot ; q)-\phi_L (\cdot, \cdot ; p)$ with $p , q \in \Rd$ and $|p - q| \ll 1$. This term is characterized to be the solution of the linearized parabolic equation and is introduced in Definition~\ref{def.linearizedcorr} below. In Lemma~\ref{lemma3.92209}, we estimate quantitatively the $L^2$-norm of the map $\phi_L (\cdot, \cdot ; q)-\phi_L (\cdot, \cdot ; p) -  \nabla w_{L,p , q-p}$ and obtain a rate of convergence depending on information which is intrinsic to the potential $V$ (and specifically, on how precisely the second derivative $V''$ can be approximated by Lipschitz functions).

\begin{definition} \label{def.linearizedcorr}
    Fix $L \in \N$ and $p , \xi \in \Rd$, and denote by $\a := V''(\nabla \phi_L(\cdot, \cdot; p))$. We let $w_{L,p , \xi} : Q_L \to \R$ be the solution of the parabolic equation with periodic boundary conditions
\begin{equation} \label{eq:11252109}
    \left\{ \begin{aligned}
    \partial_t w_{L,p , \xi} - \nabla \cdot \a \nabla w_{L,p , \xi} & = \nabla \cdot \a \xi & ~\mbox{in} &~ Q_L, \\
    w_{L,p , \xi}(-L^2 , \cdot) & = 0 & ~\mbox{in} &~  \Lambda_L, \\
    w_{L,p , \xi} (t , \cdot) &\in \Omega_{\Lambda_L , \mathrm{per}} &~\mbox{for}&~t \in I_L.
    \end{aligned} \right.
\end{equation}
\end{definition}

\begin{lemma}\label{lemma3.92209}
Fix $L \in \N$ and $R\geq 1$. Then there exists a continuous function~$\chi_R : \R_+ \to \R_+$ depending on $c_+ , c_- ,R$ and the rate of the limit in~\eqref{conv.regularityV''} over all parameters $(R,\ep) \in [1,\infty)\times (0,1]$, 
which satisfies $\chi_R(0) = 0$, such that, for every $p,q\in B_R$,
\begin{equation} \label{eq:09362309}
    \left\| \nabla \phi_L (\cdot, \cdot ; q) - \nabla \phi_L (\cdot, \cdot ; p) -  \nabla w_{L,p , q-p} \right\|_{\underline{L}^2(Q_L)} \leq \mathcal{O}_2 \left( C \chi_R(|p - q|) |p-q|\right).
\end{equation}
\end{lemma}

\begin{proof} We split the proof into several steps.

\medskip

\textit{Step 1. Setup and preliminary observations.} Fix $L \in \N$, $R\geq1$ and $p, q \in B_R$. We may assume that $|p - q| \leq 1$. We introduce the notation $v_{L} (t , x ; p) := p \cdot x + \phi_L( t , x ; p)$ and $v_{L} (t , x ; q) := q \cdot x + \phi_L(t,x ; q)$. We first observe that the difference $\tilde w = \phi_L(\cdot, \cdot ; q ) - \phi_L(\cdot, \cdot ; p)$ solves the parabolic equation
\begin{equation} \label{eq:11262109}
    \left\{ \begin{aligned}
    \partial_t \tilde w - \nabla \cdot \tilde \a \nabla \tilde w & = \nabla \cdot \tilde \a (q-p) & ~\mbox{in} &~ Q_L, \\
    \tilde w(-L^2, \cdot)& = 0 & ~\mbox{in} &~ \Lambda_L, \\
    \tilde w (t , \cdot) &\in \Omega_{\Lambda_L , \mathrm{per}} &~\mbox{for}&~t \in I_L,
    \end{aligned} \right.
\end{equation}
where the environment $\tilde \a$ is given by the formula
\begin{equation} \label{eq:15492309}
    \tilde \a (t , e) := \int_0^1 V''(s \nabla v_L(t , e ; q) + (1-s) \nabla v_L(t , e ; p)) \, ds.
\end{equation}
Using the uniform ellipticity of the environment $\tilde \a$, we have the estimate
\begin{equation} \label{eq:energyesttildew}
   \left\| \nabla \tilde w \right\|_{\underline{L}^2(Q_L)} \leq C |p-q|.
\end{equation}
For later use, we note that the Meyers estimate (Proposition~\ref{interiorparabolicMEyers}) applied to the function $\tilde w$ yields the following result: there exists an exponent $\gamma_{0}:= \gamma_0(d , c_+ , c_-) > 0$ such that
\begin{equation} \label{eq:Meyersfortildew}
    \left\| \nabla \tilde w \right\|_{\underline{L}^{2 + \gamma_{0}} (Q_{L})} \leq  C |p-q|.
\end{equation}
Taking the difference between the equations~\eqref{eq:11252109} and~\eqref{eq:11262109}, we obtain that the map $v := \tilde w - w_{L,p,q-p}$ solves the parabolic equation
\begin{equation} \label{def.vlinearizationest}
     \left\{ \begin{aligned}
    \partial_t v - \nabla \cdot \a \nabla v & = \nabla \cdot \left( (\tilde \a - \a) (q-p + \nabla \tilde w) \right)& ~\mbox{in} &~ Q_L, \\
    v(-L^2 , \cdot) & = 0 & ~\mbox{in} &~\Lambda_L, \\
    v (t , \cdot) &\in \Omega_{\Lambda_L , \mathrm{per}} &~\mbox{for all} &~t \in I_L.
    \end{aligned} \right.
\end{equation}
The definition~\eqref{def.vlinearizationest} leads to the energy estimate
\begin{equation} \label{eq:16022209}
    \left\| \nabla v(t , \cdot) \right\|_{\underline{L}^2(Q_L)}  \leq C \left\| (\tilde \a - \a) (q-p + \nabla \tilde w) \right\|_{\underline{L}^2(Q_L)}.
\end{equation}
We next estimate the term in the right-hand side of~\eqref{eq:16022209}, and split the argument into two steps.

\medskip

\textit{Step 2. Comparing the environment $\a$ and $\tilde \a$.} We first prove that the term $\tilde \a - \a$ is small when the slopes $p , q$ are close to each other. Formally, we prove the inequality: for any $\ep >0$, there exist $\delta:= \delta(d , R , V, \ep) > 0$ and $C:= C(d , R , V) < \infty $ such that, for every $p,q\in B_R$ satisfying $|p - q| \leq \delta$,
\begin{equation} \label{eq:13462009}
    \left\| \tilde \a - \a \right\|_{\underline{L}^1(Q_L)} \leq \mathcal{O}_2( C \ep).
\end{equation}
Here and throughout, dependence of constants on $V$ is restricted to dependence on the rate of the limit in~\eqref{conv.regularityV''} over all possible parameters $(R,\ep) \in [1,\infty)\times (0,1]$. To establish~\eqref{eq:13462009}, we fix $\ep > 0$ and use Proposition~\ref{prop.2.2BL} to obtain the following tail estimate on the gradient of the field: we claim that there exists a constant $C := C(d , c_+ , c_- , R) < \infty$ such that for any $(t,e) \in I_{L} \times E \left(\Lambda_{L} \right)$, any slopes $p,q \in B_R$ with $|p - q| \leq 1$, and any $S > 0$,
\begin{equation*}
    \mathbb{P} \bigl[ \left|\nabla v_L(t,e ; q) \right| \geq S \bigr] \leq C \exp \left(- \frac{S^2}{C}\right).
\end{equation*}
We may thus select~$S_\ep := C \ep^{-1}>0$ for some constant $C$
depending only on~$d , c_+ , c_-,R$ such that, for every~$q \in B_R$ with $|p - q| \leq 1$, and every $(t,e) \in I_{L} \times E\left(\Lambda_{L} \right)$,
\begin{equation} \label{eq:15072209}
     \indc_{\left\{ \left|\nabla v_L(t,e ; q) \right| \geq S_\ep \right\}} \leq \mathcal{O}_2(\ep).
\end{equation}
Using~\eqref{conv.regularityV''}, for every~$\ep>0$, we may select a parameter $\kappa_\ep >0$, depending on~$d,R$ and the rate of the limit in~\eqref{conv.regularityV''}, in addition to~$\ep$, such that
\begin{equation} \label{eq:Arepkappaep}
    \left| A_{S_\ep , \kappa_\ep}(\ep)  \right| \leq \ep.
\end{equation}
We will now prove the estimate~\eqref{eq:13462009} with the value $\delta := \kappa_\ep \ep >0$. We introduce the notation
\begin{equation*}
    \a_{\kappa_\ep} := V_{\kappa_\ep}''(\nabla v_L(\cdot, \cdot ; p))
\end{equation*}
and
\begin{equation*}
    \tilde \a_{\kappa_\ep} := \int_0^1 V_{\kappa_\ep}''(s \nabla v_L(\cdot ; q) + (1-s) \nabla v_L(\cdot ; p)) \, ds.
\end{equation*}
By the triangle inequality, we can write
\begin{equation} \label{eq:14562209}
    \left\| \tilde \a - \a \right\|_{\underline{L}^1(Q_L)} \leq \underset{\eqref{eq:14562209}-(i)}{\underbrace{\left\| \tilde \a - \tilde \a_{\kappa_\ep} \right\|_{\underline{L}^1(Q_L)}}} + \underset{\eqref{eq:14562209}-(ii)}{\underbrace{\left\| \tilde \a_{\kappa_\ep} - \a_{\kappa_{\ep}} \right\|_{\underline{L}^1(Q_L)}}} + \underset{\eqref{eq:14562209}-(iii)}{\underbrace{\left\| \a_{\kappa_{\ep}} - \a \right\|_{\underline{L}^1(Q_L)}}},
\end{equation}
and estimate the three terms in the right-hand side separately.

\medskip

\textit{Substep 2.1. Estimating the term~\eqref{eq:14562209}-(i).} We first use the inequality: for any $x \in \R$,
\begin{equation*}
    \left| V_{\kappa_\ep}''(x) - V''(x) \right| \leq \ep + C \indc_{\{ x \in  A_{S_\ep , \kappa_\ep}(\ep)\}} + C \indc_{\{ |x| \geq S_\ep \}}.
\end{equation*}
We thus obtain
\begin{align} \label{eq:13602209}
     \left\| \tilde \a - \tilde \a_{\kappa_\ep} \right\|_{\underline{L}^1(Q_L)} & \leq \ep + \frac{C}{L^{d+2}}  \int_{I_L} \sum_{e \in \vec{E}(\Lambda_L)} \int_0^1 \indc_{\{ s \nabla v_L(t , e ; q) + (1-s) \nabla v_L(t , e ; p)  \in A_{S_\ep , \kappa_\ep}(\ep)\}} \, ds \, dt \\ & 
     \qquad + \frac{C}{L^{d+2}}  \int_{I_L} \sum_{e \in \vec{E}(\Lambda_L)} \int_0^1 \indc_{\{ \left| s \nabla v_L(t , e ; q) + (1-s) \nabla v_L(t , e ; p) \right|  \geq S_\ep \}} \, ds \, dt. \notag
\end{align}
We next estimate the two terms in the right-hand side. To treat the first term, we prove the estimate: for any $s \in [0 , 1 ]$, any time $T \in (-L^2, -1)$, and any edge $e \in \vec{E}(\Lambda_L)$,
\begin{equation} \label{eq:13592209}
    \int_T^{T+1} \indc_{\{ s \nabla v_L(t , e ; q) + (1-s) \nabla v_L(t , e ; p)  \in A_{S_\ep , \kappa_\ep}(\ep)\}} \, dt \leq \mathcal{O}_2 (C \ep).
\end{equation}
The proof relies on an application of Corollary~\ref{cor:diffOrnstein-Uhlenbeck}. We fix $s \in (0,1)$, a time $T \in (-L^2, -1)$, an edge $e = (x ,y) \in \vec{E}(\Lambda_L)$ and denote by $(X_t)_{t  \in [0 , 1]}$ the stochastic process
\begin{equation*}
    X_t := s \nabla v_L(T + t , e ; q) + (1-s) \nabla v_L(T + t , e ; p).
\end{equation*}
Using the definition of the first-order correctors $v_L(\cdot, \cdot ; p)$ and $v_L(\cdot, \cdot ; q)$, we see that the process $X_t$ solves a stochastic differential equation of the form
\begin{equation*}
    dX_t = h_t dt + 2 d \tilde B_t,
\end{equation*}
where $\tilde B_t = \left( B_{T+t}(y) -  B_{T+t}(x)\right)/\sqrt{2}$ is a Brownian motion for the filtration $\mathcal{F}_t := \sigma \left( \{ B_{T + s}(w) \, : \, s \leq t, w \in \Zd\}\right)$, and the adapted process $h_t$ is defined by the formula
\begin{align*}
    h_t & := s \nabla \cdot V' (\nabla  v_L( \cdot; q)) (T + t , y ) + (1 - s) \nabla \cdot V' (\nabla  v_L( \cdot; p)) (T + t , y ) \\
    & \quad - s \nabla \cdot V' (\nabla  v_L( \cdot; q)) (T + t , x ) + (1 - s) \nabla \cdot V' (\nabla v_L( \cdot; p)) (T + t , x ).
\end{align*}
Using the pointwise bound on the gradient of the dynamic stated in Proposition~\ref{prop.2.2BL}, we deduce that, for any $t \in (0 , 1)$,
\begin{equation*}
    \left| h_t \right| \leq \mathcal{O}_2 (C).
\end{equation*}
Using the property~\eqref{int.Onotation} of the $\mathcal{O}$-notation, we can integrate over the times $t$ in the interval~$[0,1]$ and deduce that
\begin{equation*}
    \left\| h \right\|_{L^2((0 , 1))} \leq \mathcal{O}_2 (C).
\end{equation*}
We apply Corollary~\ref{cor:diffOrnstein-Uhlenbeck} to the process $X_t$ with the set~$A_{S_\ep , \kappa_\ep}(\ep)$, and use the inequality~\eqref{eq:Arepkappaep} to complete the proof of the estimate~\eqref{eq:13592209}. Noting that the inequality~\eqref{eq:13592209} is valid for any $s \in (0,1)$, integrating over this variable and using the property~\eqref{int.Onotation} of the $\mathcal{O}$-notation, we deduce that: for
any time $T \in (-L^2, -1)$, and any edge $e \in E(\Lambda_L)$,
\begin{equation*}
     \int_0^1 \int_T^{T+1} \indc_{\left\{ \left| s \nabla v_L(t , e ; p) + (1-s) \nabla v_L(t , e ; q) \right|  \geq S_\ep \right\}} \, dt  \, ds \leq \mathcal{O}_2(C \ep).
\end{equation*}
We next sum over the times $T \in \{-L^2 , \ldots, -1\}$, over the edges of the box $\Lambda_L$, and use the property~\eqref{sum.Onotation} of the $\mathcal{O}$-notation. We obtain
\begin{align} \label{eq:15172209}
     \lefteqn{\frac{C}{L^{d+2}}  \int_{I_L} \sum_{e \in \vec{E}(\Lambda_L)} \int_0^1 \indc_{\left\{ s \nabla v_L(t , e ; q) + (1-s) \nabla v_L(t , e ; p)  \in A_{S_\ep , \kappa_\ep}(\ep)\right\}} \, ds \, dt} \qquad & \\ & =  \frac{C}{L^{d+2}}  \sum_{T \in \{ -L^2 , \ldots, -1\}} \sum_{e \in \vec{E}(\Lambda_L)} \int_0^1 \int_T^{T+1} \indc_{\{ s \nabla v_L(t , e ; q) + (1-s) \nabla v_L(t , e ;p)  \in A_{S_\ep , \kappa_\ep}(\ep)\}} \, ds \, dt \notag \\
     & \leq \mathcal{O}_2(C \ep). \notag
\end{align}
We next estimate the third term in the right-hand side of~\eqref{eq:13602209}. Using the estimate~\eqref{eq:15072209}, we see that, for any $s \in (0 , 1)$, any time $t \in (-L^2 , 0)$, and any edge $e \in E \left( \Lambda_L\right)$,
\begin{align*}
     \indc_{\{ \left| s \nabla v_L(t , e ;q) + (1-s) \nabla v_L(t , e ; p) \right|  \geq S_\ep \}} & \leq \indc_{\{ \left| \nabla v_L(t , e ; q) \right| \geq S_\ep \}} + \indc_{\{ \left| \nabla v_L(t , e ; p) \right|  \geq S_\ep \}}
     \leq \mathcal{O}_2(C \ep).
\end{align*}
Integrating over $s$ in the interval $(0,1)$, over the times $t \in (-L^2 , 0)$, summing over the edges $e \in E \left( \Lambda_L\right)$, and using the property~\eqref{sum.Onotation} of the $\mathcal{O}$-notation, we deduce that
\begin{equation} \label{eq:15182209}
    \frac{C}{L^{d+2}}  \int_{I_L} \sum_{e \in \vec{E}(\Lambda_L)} \int_0^1 \indc_{\{ \left| s \nabla v_L(t , e ; q) + (1-s) \nabla v_L(t , e ; p) \right|  \geq S_\ep \}} \, ds \, dt \leq  \mathcal{O}_2(C \ep).
\end{equation}
Combining the estimates~\eqref{eq:13602209},~\eqref{eq:15172209} and~\eqref{eq:15182209}, we deduce that
\begin{equation*}
     \left\| \tilde \a - \tilde \a_{\kappa_\ep} \right\|_{\underline{L}^1(Q_L)} \leq \mathcal{O}_2(C \ep).
\end{equation*}

\medskip

\textit{Substep 2.2. Estimating the term ~\eqref{eq:14562209}-(ii).} We use the inequality~\eqref{eq:VrgLIP} and obtain, for any $(t , e) \in I_L \times E \left(\Lambda_L\right)$,
\begin{equation*}
   | \tilde \a_{\kappa_\ep}(t , e) - \a_{\kappa_{\ep}}(t , e)| \leq \frac{C}{\kappa_\ep}  \left| \nabla \tilde w(t , e) \right|.
\end{equation*}
Taking the $L^1$-norm over the parabolic cylinder $Q_{L}$, using Jensen's inequality and the bound~\eqref{eq:energyesttildew}, we deduce that
\begin{equation}
    \left\| \tilde \a - \a_{\kappa_\ep} \right\|_{\underline{L}^1(Q_L)} \leq \frac{C}{\kappa_\ep} \left\| \nabla \tilde w \right\|_{\underline{L}^1 \left( Q_{L}\right)} \leq \frac{C}{\kappa_\ep} \left\| \nabla \tilde w \right\|_{\underline{L}^2 \left( Q_{L}\right)} \leq \frac{C |p - q|}{\kappa_\ep} \leq C\ep.
\end{equation}

\medskip

\textit{Substep 2.3. Estimating the term~\eqref{eq:14562209}-(iii).} We use the same technique as the one used to estimate the term~\eqref{eq:14562209}-(i). We first note that the following bound holds
\begin{align*}
     \left\| \a - \a_{\kappa_\ep} \right\|_{\underline{L}^1(Q_L)} & \leq \ep + \frac{C}{L^{d+2}}  \int_{I_L} \sum_{e \in \vec{E}(\Lambda_L)} \int_0^1 \indc_{\{\nabla v_L(t , e ; p)  \in A_{S_\ep , \kappa_\ep}(\ep)\}} \, ds \, dt \\ & 
     \qquad + \frac{C}{L^{d+2}}  \int_{I_L} \sum_{e \in \vec{E}(\Lambda_L)} \int_0^1 \indc_{\{ \left| \nabla v_L(t , e ; p) \right|  \geq S_\ep \}} \, ds \, dt. \notag
\end{align*}
We apply the estimate~\eqref{eq:13592209} with $s = 1$ and obtain: for any time $T \in (-L^2 , -1)$ and any edge $e \in \vec{E} \left( \Lambda_L\right)$,
\begin{equation*}
    \int_T^{T+1} \indc_{\{ \nabla v_L(t , e ; p)  \in A_{S_\ep , \kappa_\ep}(\ep)\}} \, dt \leq \mathcal{O}_2 (C \ep).
\end{equation*}
Summing over the times $T \in \{-L^2 , \ldots, -1\}$ and over the edges of the box $\Lambda_L$, we obtain
\begin{equation*} 
    \frac{1}{L^{d+2}} \int_{I_L} \sum_{e \in \vec{E}\left( \Lambda_L \right)} \indc_{\{ \nabla v_L(t , e ; p)  \in A_{S_\ep , \kappa_\ep}(\ep)\}} \leq \mathcal{O}_2 \left( C \ep \right).
\end{equation*}
A combination of the previous estimates with~\eqref{eq:15072209} yields
\begin{equation*}
     \left\| \a_{\kappa_{\ep}} - \a \right\|_{\underline{L}^1(Q_L)} \leq \mathcal{O}_2(C \ep).
\end{equation*}

Combining the results of the Substeps 1.1, 1.2 and 1.3 completes the proof of the estimate~\eqref{eq:13462009}.

\medskip

\textit{Step 3. Estimating the right-hand side of~\eqref{eq:16022209}.}
We now complete the proof of Lemma~\ref{lemma3.92209} using the estimate~\eqref{eq:13462009}. By H\"{o}lder's inequality, the upper bound~\eqref{eq:Meyersfortildew}, and denoting by $p_0 := (2+ \gamma_0)/\gamma_0$, we may rewrite~\eqref{eq:16022209} as 
\begin{align} \label{eq:16202209}
   \left\| \nabla v(t , \cdot) \right\|_{\underline{L}^2(Q_L)} & \leq C \left\| (\tilde \a - \a) (q-p + \nabla \tilde w) \right\|_{\underline{L}^2(Q_L)}  \\
   & \leq C \left\| \tilde \a - \a \right\|_{\underline{L}^{p_0}(Q_L)} \left\| q-p + \nabla \tilde w \right\|_{\underline{L}^{2+\gamma_0}(Q_L)} \notag \\
   & \leq C \left\| \tilde \a - \a \right\|_{\underline{L}^{p_0}(Q_L)} |p-q|. \notag
\end{align}
Using that the maps $\tilde \a$ and $\a$ are bounded by the constant $c_+$, and interpolating the space $L^{p_0}(Q_L)$ between the spaces $L^1(Q_L)$ and $L^\infty(Q_L)$, we deduce that
\begin{equation} \label{eq:1621209}
    \left\| \tilde \a - \a \right\|_{\underline{L}^{p_0}(Q_L)} \leq \left\| \tilde \a - \a \right\|_{\underline{L}^{1}(Q_L)}^{1/p_0} \left\| \tilde \a - \a \right\|_{L^{\infty}(Q_L)}^{(p_0 - 1)/p_0}  \leq C \left\| \tilde \a - \a \right\|_{\underline{L}^{1}(Q_L)}^{1/p_0}.
\end{equation}
A combination of~\eqref{eq:16202209} and~\eqref{eq:1621209} shows
\begin{equation*}
    \left\| \nabla v(t , \cdot) \right\|_{\underline{L}^2(Q_L)} \leq C \left\| \tilde \a - \a \right\|_{\underline{L}^{1}(Q_L)}^{1/p_0} |p-q|.
\end{equation*}
We then set
\begin{equation} \label{def.chipr}
    \chi_R (r) := \inf \Bigl\{ \ep > 0 \, : \, \forall p,q \in B_R~\mbox{with}~ |p - q| \leq r, ~ \left\| \tilde \a - \a \right\|_{\underline{L}^{1}(Q_L)}^{1/p_0} \leq \mathcal{O}_2 \left( \ep \right)  \Bigr\},
\end{equation}
so that we have the inequality
\begin{equation*}
    \left\| \nabla v(t , \cdot) \right\|_{\underline{L}^2(Q_L)} \leq \mathcal{O}_2 \left(C \chi_R (|p - q|) |p-q|\right).
\end{equation*}
The inequality~\eqref{eq:13462009} ensures that $\chi_R (r) \to 0$ as~$r \to 0$ which yields the lemma. 
\end{proof}

\begin{remark} \label{Rk:gentildew}
We record the following generalization of Lemma~\ref{lemma3.92209} whose proof can be deduced from the same argument. For every~$L\in\N$, $p , q \in \Rd$ and $\xi \in \Rd$, we denote by $\tilde w_{L,p , q , \xi} : Q_L \to \R$ the solution of the parabolic equation
\begin{equation*}
    \left\{ \begin{aligned}
    \partial_t \tilde w_{L,p , q , \xi} - \nabla \cdot \tilde \a \nabla \tilde w_{L,p , q , \xi} & = \nabla \cdot \tilde \a \xi & ~\mbox{in} ~& Q_L, \\
    \tilde w_{L,p , q , \xi}(-L^2, \cdot)& = 0 & ~\mbox{in}~& \Lambda_L, \\
     \tilde w_{L,p , q , \xi} (t , \cdot) &\in \Omega_{\Lambda_L , \mathrm{per}} &~\mbox{for}~&t \in I_L,
    \end{aligned} \right.
\end{equation*}
where $\tilde \a$ is the environment defined in~\eqref{eq:15492309}, then we have the inequality
\begin{equation} \label{eq:gentildew}
    \left\| w_{p , \xi} - \tilde w_{p , q , \xi} \right\|_{\underline{L}^2(Q_L)} \leq \mathcal{O}_2 \left( C \chi_R(|p-q|) |\xi| \right).
\end{equation}
\end{remark}

\begin{remark}
An investigation of the proof shows that if the second derivative $V^{\prime\prime}$ is $C^{0,\alpha}$ then $\chi_R$ is a H\"older modulus (with a regularity exponent $\beta \ll \alpha$).
\end{remark}

\subsection{$C^2$-regularity of the surface tension} \label{sectionC2regfinal}

This section is devoted to the proof of Theorem~\ref{t.regsurfacetension}, building upon the result of Lemma~\ref{lemma3.92209}.

\begin{proof}[Proof of Theorem~\ref{t.regsurfacetension}]
Let us recall the notation $v_L (t , x ; p) : = p \cdot x + \phi_L (t , x ; p)$. We first set, for $p \in \Rd$,
\begin{equation*}
   \tau_L (p) := \mathbb{E} \left[ (V'(\nabla v_{L}(\cdot , \cdot ; p) ))_{Q_{L/2}} \right].
\end{equation*}
Using Proposition~\ref{prop3.9} (where $q$ is chosen to be the constant slope equal to $p$), we have that
\begin{equation}
\label{e.tauL.barsigma}
   \left| \tau_L (p) - D_p \bar \sigma(p) \right| \leq 
   CL^{-1} \Bigl( 1 + (\log L)^{\frac12}  \indc_{\{ d = 2\}} \Bigr).
\end{equation}
Using the result of Lemma~\ref{lemma3.92209}, we will prove the following two properties:
\begin{enumerate}
    \item[(i)] The map $\tau_L$ is differentiable and its gradient is given by
    \begin{equation} \label{eq:08412309}
        \partial_{i} \tau_L(p) = \mathbb{E} \left[ (V''(\nabla v_{L}(\cdot , \cdot ; p)) (e_i + \nabla w_{L,p, e_i}) )_{Q_{L/2}} \right], \quad p\in\Rd, \ i\in \{1,\ldots,d\}\,;
    \end{equation}
    \item[(ii)] The partial derivatives of $\tau_L$ are continuous and their moduli of continuity satisfy
    \begin{equation} \label{eq:14062309}
        \left|  \partial_{i} \tau_L(p) - \partial_{i} \tau_L(q)  \right| \leq C \chi_R (|p - q|),\quad p,q\in B_R\,,
    \end{equation}
where the moduli~$\{\chi_R\}_{R\geq 1}$ are given in Lemma~\ref{lemma3.92209}.
In particular, the family~$\left\{ \partial_{i} \tau_L\right\}_{L \geq 0}$ is locally equicontinuous on~$\Rd$. 
\end{enumerate}
Since the map $\xi \mapsto w_{L, p , \xi}$ is linear, the differentiability of~$\tau_L$ and the identity~\eqref{eq:08412309} are equivalent to the statement that
\begin{equation} \label{eq:09492309}
    \tau_L (p+ \xi) = \tau_L (p) + \mathbb{E} \left[ (V''(\nabla v_{L}(\cdot , \cdot ; p)) (\xi + \nabla w_{L,p, \xi}) )_{Q_{L/2}} \right] + o(|\xi|)
    \quad \mbox{as} \ |\xi|\to 0\,.
\end{equation}
We now prove~\eqref{eq:09492309}. Taking the expectation in the inequality~\eqref{eq:09362309} of Lemma~\ref{lemma3.92209}, we obtain, for every $|\xi|\leq 1$,
\begin{equation*}
    \E \left[ \left\| \nabla \phi_{L} \left( \cdot , \cdot ; p + \xi\right) - \nabla \phi_{L} \left( \cdot , \cdot ; p \right) - \nabla w_{L,p, \xi} \right\|_{\underline{L}^2(Q_L)} \right] \leq C \chi_{|p|+1}(|\xi|) |\xi|.
\end{equation*}
Using that the map $V'$ is Lipschitz together with Lemma~\ref{lemma3.92209} and Jensen's inequality, we obtain
\begin{equation*}
    \left| \tau_L(p + \xi) - \E \left[ \left( V'(\nabla v_{L} \left( \cdot , \cdot ; p \right) + \xi +  \nabla w_{L,p, \xi}) \right)_{Q_{L/2}}   \right] \right| \leq C \chi_{|p|+1}(|\xi|) |\xi|.
\end{equation*}
Consequently,~\eqref{eq:09492309} is equivalent to
\begin{align} 
\label{eq:13522309}
\lefteqn{
 \E \left[ \left( V'(\nabla v_{L} \left( \cdot ; p \right) + \xi +  \nabla w_{L,p, \xi}) \right)_{Q_{L/2}}  \right] 
 } \qquad & \notag \\
&
=  \tau_L (p) + \mathbb{E} \left[ (V''(\nabla v_{L}(\cdot ; p)) (\xi + \nabla w_{L,p, \xi}) )_{Q_{L}/2} \right] + o(\xi) 
\quad \mbox{as} \ |\xi|\to 0\,.
\end{align}
Since the map $V'$ is Lipschitz, it is differentiable almost everywhere. We denote by $\mathrm{Diff}_{V'} \subseteq \R$ the set where it is differentiable. We thus have, for any $(t , e) \in I_L \times E(\Lambda_L)$ such that $\nabla v_{L} ( t, e ; p ) \in \mathrm{Diff}_{V'}$,
\begin{align} 
\label{eq:13232309}
&
V'(\nabla v_{L} ( t, e ; p ) {+} \xi {+} \nabla w_{L,p, \xi}) = V'(\nabla v_{L} ( t, e ; p ) ) +  V''(\nabla v_{L} \left(t,e ; p \right)) ( \xi {+} \nabla w_{L,p, \xi}) + o(\xi)
\end{align}
Additionally, since the map $V'$ is Lipschitz and $V''$ is bounded, and since the map $\xi \mapsto \nabla w_{L,p, \xi}$ is linear, we have the following upper bound: for any $|\xi|\leq 1$,
\begin{multline} \label{eq:13242309}
|\xi|^{-1}  \left| V'(\nabla v_{L} \left( \cdot, \cdot  ; p \right) + \xi + \nabla w_{L,p, \xi}) - V'(\nabla v_{L} \left( \cdot, \cdot ; p \right) ) -  V''(\nabla v_{L} \left( \cdot, \cdot ; p \right)) (\xi + \nabla w_{L,p, \xi} )\right|
\\ 
\leq C( 1 +  \left| \nabla w_{L,p, \xi/|\xi|} \right| )
\leq C + C \sum_{i = 1}^d  \left| \nabla w_{L,p, e_i} \right|. 
\end{multline}
Using the definition of $w_{L, p , \xi}$ stated in~\eqref{eq:11252109}, we have, for every~$i \in \{1 , \ldots, d\}$, 
\begin{equation} \label{eq:13252309}
    \E \left[ \left\| \nabla w_{L,p, e_i}  \right\|_{\underline{L}^2(Q_L)} \right] \leq C.
\end{equation}
Finally, since the set $\mathrm{Diff}_{V'}$ has full Lebesgue measure, we can apply~\eqref{lemma3.92209} to obtain
\begin{equation} \label{eq:13262309}
    \E \Biggl[ \int_{I_L} \sum_{e \in E \left( \Lambda_L \right)} \indc_{\{ \nabla v_{L} \left( t, e ; p \right) \notin \mathrm{Diff}_{V'}\}} \, dt \Biggr] = 0.
\end{equation}
Combining~\eqref{eq:13232309},~\eqref{eq:13242309},~\eqref{eq:13252309} and~\eqref{eq:13262309} with the dominated convergence theorem completes the proof of~\eqref{eq:13522309}. The proof of~\eqref{eq:09492309}, and thus of~\eqref{eq:08412309}, is complete.

\smallskip

We now prove the continuity estimate~\eqref{eq:14062309}. Pick $R\geq 1$ and $p,q\in B_R$. Using the identity~\eqref{eq:08412309} and the triangle inequality, we first write
\begin{align} \label{eq:16082309}
    \lefteqn{\left| \partial_{i} \tau_L(q) - \partial_{i} \tau_L(p) \right| }\qquad & \\ & = \left| \mathbb{E} \left[ (V''(\nabla v_{L}(\cdot, \cdot ; p)) (e_i + \nabla w_{L,p, e_i}) )_{Q_{L/2}} \right] - \mathbb{E} \left[ (V''(\nabla v_{L}(\cdot, \cdot ; q)) (e_i + \nabla w_{L,q, e_i}) )_{Q_{L/2}} \right] \right| \notag \\
    & \leq \left| \mathbb{E} \left[ (V''(\nabla v_{L}(\cdot, \cdot ; p)) (e_i + \nabla w_{L,p, e_i}) )_{Q_{L/2}} \right] - \mathbb{E} \left[ (V''(\nabla v_{L}(\cdot, \cdot ; p))(e_i +  \nabla w_{L,q, e_i}) )_{Q_{L/2}} \right] \right| \notag \\
    & \quad + \left| \mathbb{E} \left[ (V''(\nabla v_{L}(\cdot, \cdot ; p)) (e_i +\nabla w_{L,q, e_i}) )_{Q_{L/2}} \right] - \mathbb{E} \left[ (V''(\nabla v_{L}(\cdot, \cdot ; q)) (e_i + \nabla w_{L,q, e_i}) )_{Q_{L/2}} \right] \right|. \notag
\end{align}
We next estimate the two terms in the right-hand side. For the first one, we use that~$V''$ is bounded together with Jensen's inequality to obtain
\begin{multline*}
    \left| \mathbb{E} \left[ (V''(\nabla v_{L}(\cdot, \cdot ; p)) (e_i + \nabla w_{L,p, e_i} )_{Q_{L/2}} \right] - \mathbb{E} \left[ (V''( \nabla v_{L}(\cdot, \cdot ; p)) (e_i + \nabla w_{L,q, e_i}) )_{Q_{L/2}} \right] \right| \\
    \leq \mathbb{E} \left[ \left\| \nabla w_{L,p, e_i} -   \nabla w_{L,q, e_i} \right\|_{\underline{L}^2(Q_L)} \right]
    \,.
\end{multline*}
We estimate the term in the right-hand side using the triangle inequality and the upper bound~\eqref{eq:gentildew} stated in Remark~\ref{Rk:gentildew} with $\xi = e_i$. We obtain
\begin{align*}
\lefteqn{
    \mathbb{E} \left[ \left\| \nabla w_{L,p, e_i} - \nabla w_{L, q, e_i} \right\|_{\underline{L}^2(Q_L)} \right]
    } \qquad & \\ 
    & \leq  \mathbb{E} \left[ \left\| \nabla w_{L,p, e_i} -   \nabla \tilde w_{L,p , q, e_i} \right\|_{\underline{L}^2(Q_L)} \right] +  \mathbb{E} \left[ \left\|   \nabla \tilde w_{L,p , q, e_i} - \nabla w_{L, q, e_i} \right\|_{\underline{L}^2(Q_L)} \right] 
    \leq  C \chi_R(|p-q|).
\end{align*}
A combination of the two previous displays yields
\begin{multline} \label{eq:17172309}
    \left| \mathbb{E} \left[ (V''(\nabla v_{L}(\cdot, \cdot ; p)) (e_i + \nabla w_{L,p, e_i}) )_{Q_{L/2}} \right] - \mathbb{E} \left[ (V''( \nabla v_{L}(\cdot, \cdot ; p)) (e_i + \nabla w_{L,q, e_i}) )_{Q_{L/2}} \right] \right| \\ \leq  C \chi_R(|p-q|).
\end{multline}
To estimate the second term in the right-hand side of~\eqref{eq:16082309}, we first note that, by the Cauchy-Schwarz inequality,
\begin{align} \label{eq:17132309}
    \lefteqn{\left| \mathbb{E} \left[ (V''(\nabla v_{L}(\cdot, \cdot ; p)) (e_i + \nabla w_{L,q, e_i}) )_{Q_{L/2}} \right] - \mathbb{E} \left[ (V''( \nabla v_{L}(\cdot, \cdot ; q)) (e_i + \nabla w_{L,q, e_i}) )_{Q_{L/2}} \right] \right|}\qquad \notag & \\ & \leq C \E \left[ \left\| V''(\nabla v_{L}(\cdot, \cdot ; p)) -  V''(\nabla v_{L}(\cdot, \cdot ; q) )\right\|_{\underline{L}^2 (Q_L)} \right] \E \left[ \left\| \nabla w_{L,q, e_i} \right\|_{\underline{L}^2 (Q_L)} \right] \notag \\
    & \leq C \E \left[ \left\| V''(\nabla v_{L}(\cdot, \cdot ; p)) -  V''(\nabla v_{L}(\cdot, \cdot ; q)) \right\|_{\underline{L}^2 (Q_L)} \right]. 
\end{align}
Recalling the definition of the coefficient $\tilde \a$ stated in~\eqref{eq:15492309}, and the definition of the modulus of continuity~\eqref{def.chipr}, we deduce that
\begin{align} 
\label{eq:17142309}
     \lefteqn{\E \left[ \left\| V''(\nabla v_{L}(\cdot, \cdot ; p)) -  V''(\nabla v_{L}(\cdot, \cdot ; q)) \right\|_{\underline{L}^2 (Q_L)} \right]} \qquad & \notag \\ & \leq \E \left[ \left\| V''( \nabla v_{L}(\cdot, \cdot ; p)) -  \tilde \a \right\|_{\underline{L}^2 (Q_L)} \right] +  \E \left[ \left\| \tilde \a -  V''( \nabla v_{L}(\cdot, \cdot ; q) \right\|_{\underline{L}^2 (Q_L)} \right] \\
     & \leq C \chi_R(|p-q|). \notag
\end{align}
A combination of~\eqref{eq:17132309} and~\eqref{eq:17142309} yields the upper bound
\begin{multline} \label{eq:17162309}
    \left| \mathbb{E} \left[ (V''(\nabla v_{L}(\cdot, \cdot ; p))(e_i +  \nabla w_{L,q, e_i}) )_{Q_{L/2}} \right] - \mathbb{E} \left[ (V''(\nabla v_{L}(\cdot, \cdot ; p)) (e_i + \nabla w_{L,q, e_i}) )_{Q_{L/2}} \right] \right| \\
    \leq C \chi_R(|p-q|).
\end{multline}
A combination of the three inequalities~\eqref{eq:16082309},~\eqref{eq:17172309} and~\eqref{eq:17162309} completes the proof~\eqref{eq:14062309}.

To conclude, we observe~\eqref{eq:14062309} that the family~$\left\{ p \mapsto D_p\tau_L(p) \,:\, L\in\N \right\}$ is locally bounded and equicontinuous on~$\Rd$. We may therefore extract a subsequence converging locally uniformly to a continuous function. Since the map~$\tau_L$ converges locally uniformly as~$L\to\infty$ to~$D_p\bar \sigma(p)$ by~\eqref{e.tauL.barsigma}, we deduce that the function~$D_p \bar \sigma$ belongs the space $C^1(\R^d)$ and the whole sequence~$D_p\tau_L$ converges locally uniformly to $D^2_p\bar\sigma$ as $L\to \infty$. In particular, $\bar \sigma \in C^2$ and its second derivative~$D^2_p \bar \sigma$ satisfies, in view of~\eqref{eq:14062309},
\begin{equation*}
    \left|  D^2_p \bar\sigma(p) - D^2_p \bar\sigma(q)  \right| \leq C \chi_R(|p-q|),\quad p,q\in B_R\,.
\end{equation*}
The proof of Theorem~\ref{t.regsurfacetension} is complete.
\end{proof}

\begin{remark}
In view of~\eqref{eq:08412309}, the proof of Theorem~\ref{t.regsurfacetension} yields the following formula for~$D^2\bar\sigma$:
\begin{equation} \label{eq:14062309.blargh}
  \partial_{i} D_p\bar\sigma (p) = 
  \lim_{L\to \infty}
  \mathbb{E} \left[ \bigl(V''(\nabla v_{L}(\cdot ; p)) (e_i + \nabla w_{L,p, e_i}) \bigr)_{Q_{L/2}} \right], \quad p\in\Rd, \ i \in \{1,\ldots,d\}\,.
\end{equation}
\end{remark}

\section{Quantitative hydrodynamic limit with optimized stochastic error} \label{Optimizingstochint}

This section and Section~\ref{sec.Section5} are devoted to the proof of Theorem~\ref{theoremlargescale}. As mentioned in Section~\ref{outlinelargescale}, we first prove in Proposition~\ref{prop4.4} a version of Theorem~\ref{Th.quantitativehydr} more adapted to our purposes. The (technical) reason is twofold:
\begin{itemize}
    \item The statement of Theorem~\ref{Th.quantitativehydr} requires that the boundary conditions belongs to the Sobolev space $H^2(Q)$. This assumption is used in order to optimize the rate of convergence in Theorem~\ref{Th.quantitativehydr} but is too strong for Theorem~\ref{theoremlargescale}.
    \item Using the result of Theorem~\ref{Th.quantitativehydr} directly yields a suboptimal stochastic integrability for the minimal scale.
\end{itemize}
In order to overcome these two difficulties, we establish Proposition~\ref{prop4.4}, which exhibits the following features:
\begin{itemize}
    \item The boundary condition is required to belong to the Sobolev space $W^{1 , 2+ \gamma_0}_{\mathrm{par}}(Q)$, where $\gamma_0$ is the exponent appearing in the Meyers estimate (Proposition~\ref{propMeyers}). Additionally, we prove a result which is uniform over the boundary conditions whose $W^{1 , 2+ \gamma_0}_{\mathrm{par}}(Q)$-norm are bounded by a constant $M$.
    \item We decompose the error term into two parts: a deterministic error of the form $\ep^\beta$ for some small exponent $\beta > 0$, and a random error of the form $\mathcal{O}_2 \left( \ep^s \right)$ for a fixed exponent $s \in (0 , d/2)$ (see~\eqref{eq:17391412}). This way of writing the error optimizes the stochastic term at the cost of a deterministic error term.
\end{itemize}

\begin{proposition}[Quantitative hydrodynamic limit with optimized stochastic error] \label{prop4.4}
Fix $s \in (0 , d/2)$, $\ep \in (0,1)$ and let $Q := [-1 , 0] \times [-1 , 1]^d$. Let $\gamma_0 > 0$ be the exponent of Proposition~\ref{propMeyers}. For $f \in W^{1 ,2+ \gamma_0}_{\mathrm{par}}(Q)$, let $u^\ep_f : Q^\ep \to \R$ be the solution of the system of stochastic differential equations
\begin{equation} \label{eq:defuLthmhydro6.1}
    \left\{ \begin{aligned}
    d u^\ep_f(t , x) & = \nabla^\ep \cdot V'(\nabla^\ep u^\ep_f) (t , x) dt + \sqrt{2} \ep dB_{\frac{t}{\ep^2}}\left( \frac{x}{\ep} \right) &~\mbox{for}~& (t,x) \in  Q^\ep,   \\
    u^\ep_f & = \tilde f_\ep &~\mbox{on}~& \partial_{\mathrm{par}} Q^\ep,
    \end{aligned} \right.
\end{equation}
and let $\bar u : Q \to \R$ be the solution of the parabolic equation
\begin{equation} \label{eq:defubarthmhydro6.2}
    \left\{ \begin{aligned}
    \partial_t \bar u_f - \nabla \cdot D_p \bar \sigma(\nabla \bar u_f) & = 0 &~\mbox{in}~& Q, \\
    \bar u_f &= f &~\mbox{on}~& \partial_{\mathrm{par}} Q. \\
    \end{aligned} \right.
\end{equation}
Then, there exists an exponent $\beta_{s} >0 $ and constant $C < \infty$ depending on $d ,s, c_+ , c_-$ such that, for any $M > 0$,
\begin{equation} \label{eq:17391412}
    \sup_{f \, : \, \left\| f \right\|_{W^{1 ,2+ \gamma_0}_{\mathrm{par}}}(Q) \leq M}\left\| u_f - \bar u_f \right\|_{L^2 \left( Q^\ep \right)} \leq C (M + 1 ) \ep^{\beta_{s}} + \mathcal{O}_2 \left( C \ep^{s} \right).
\end{equation}
\end{proposition}

The proof of Proposition~\ref{prop4.4} follows a strategy which is similar to the one used to establish Theorem~\ref{Th.quantitativehydr} and is closely related to the argument of~\cite[Section 11.4]{AKMbook}. It is decomposed into different sections and is structured as follows. In Section~\ref{section.stochheatequation}, we obtain some suitable bounds on the finite-volume corrector and its flux. These estimates will be optimized with respect to the stochastic term at the cost of a suboptimal, but algebraically small deterministic term. The proof can be found in Section~\ref{Sec:concentrationineq} and makes use of a comparison with the Langevin dynamic for the Gaussian free field which is introduced in Section~\ref{langevinfreefield}. Section~\ref{sectionproof6.2} contains the proof of Proposition~\ref{prop4.4} and is decomposed into two subsections. In Section~\ref{section.meyersetinternalreg}, we record the regularity properties on the solution of the limiting equation~\eqref{eq:defubarthmhydro6.2} which are used in the proof. Section~\ref{section6.2.2.2} contains the proof of Proposition~\ref{prop4.4}. As in the proof of Theorem~\ref{Th.quantitativehydr}, it is based on a two-scale expansion with three main differences: we choose the mesoscopic scale $\kappa$ to be almost as large as the macroscopic scale, add a (large) boundary layer where the finite-volume corrector is assigned the constant slope equal to $0$, and use the estimates on the $\underline{L}^2$-norm of the corrector and the $\underline{H}^{-1}_{\mathrm{par}}$-norm of its flux established in Section~\ref{section.stochheatequation}.

\subsection{Concentration inequalities for the corrector and its flux} \label{section.stochheatequation}

In this subsection, we establish some estimates for the finite-volume corrector and its flux.

\subsubsection{The dynamic of the Gaussian free field} \label{langevinfreefield}
In this section, we introduce the Langevin dynamics of the Gaussian free field and record some of its properties. We denote by $I := [s_- , s_+] \subseteq \R$, and let $\Lambda \subseteq \Zd$ be a box of sidelength $L$.

\begin{definition}[Dynamic of the Gaussian free field] \label{def.DynGFF}
    Let $\psi$ be a Gaussian free field in the box $\Lambda$ with average value equal to $0$ (i.e., a random surface distributed according to the Gibbs measure~\eqref{def.Gibbs} with $V(x) = x^2/2$ and $p = 0$) and independent of the Brownian motions $(B_t(x))_{t \in \R , x \in \Zd}$. We denote by $\psi_{{I \times \Lambda}} : I \times \Lambda \to \R$ be the solution of the system of stochastic differential equations
\begin{equation*}
    \left\{ \begin{aligned}
    d \psi_{{I \times \Lambda}}(t , x) & = \Delta \psi_{{I \times \Lambda}} (t ,x) dt + \sqrt{2} dB_t(x) - \frac{\sqrt{2}}{\left| \Lambda \right|} \sum_{y \in \Lambda}  dB_t(y) &~\mbox{for}&~(t,x) \in I \times \Lambda,  \\
    \psi_{ Q}( s_- , \cdot) &= \psi &~\mbox{for}&~ x \in \Lambda, \\
    \psi_{{I \times \Lambda}} (t , \cdot ) &\in \Omega_{\Lambda, \mathrm{per}} & \mbox{for}&~ t \in I.
    \end{aligned} \right.
\end{equation*}
\end{definition}
Since the dynamic is stationary with respect to the Gaussian free field, we have that, for any time $t \in I$, the random surface $\psi_{{I \times \Lambda}}(t, \cdot)$ is distributed according to a Gaussian free field with average value $0$ in the box $\Lambda$. We record below three properties of the dynamic $\psi_{{I \times \Lambda}}$. The proof of these results can be obtained by explicit computations, essentially diagonalizing the Laplacian on a discrete box to reduce the problem to concentration estimates for a sum of independent random variables. A detailed sketch of the argument is given below.

\begin{proposition}[Concentration for the dynamic of the free field]
There exists a constant $C := C(d) < \infty$ such that, for any $z \in I \times \Lambda$, and any $\ell \in \N$ such that $z + Q_\ell \subset I \times \Lambda$,
\begin{equation} \label{eq:16102611}
    \bigl| \bigl( \nabla \psi_{{I \times \Lambda}}  \bigr)_{z + Q_\ell} \bigr| 
    \leq \mathcal{O}_2 \bigl( C\ell^{-\frac d2} \bigl) 
    \qquad \mbox{and}\qquad 
    \left\| \nabla \psi_{{I \times \Lambda}} \right\|_{\underline{L}^2(Q)} 
    \leq C 
    + \mathcal{O}_2\bigl( C L^{-\frac d2} \bigr),
\end{equation}
and
\begin{equation} \label{eq:14102611}
    \left\| \psi_{{I \times \Lambda}} \right\|_{\underline{L}^2(Q)} \leq 
    \left\{
    \begin{aligned} 
    &  C \ln L +  \mathcal{O}_2\bigl( C \bigr) & \mbox{if} & \ d = 2\,, \\
    & C + \mathcal{O}_2\left( C L^{1- \frac{d}{2}} \right) & \mbox{if} & \ d \geq 3.
    \end{aligned}
    \right.
\end{equation}
\end{proposition}

\begin{proof}
Without loss of generality, we may assume that $z = 0$ and denote by $L$ the sidelength of the box $\Lambda$. 

We now prove first inequality in~\eqref{eq:16102611}. We first fix a time $t \in I$ and integer $i \in \{ 1 , \ldots, d\}$, and note that the random variable $\bigl( \nabla \psi_{{I \times \Lambda}} (t , \cdot) \bigr)_{\Lambda_\ell, i}$ is Gaussian and that its variance is given by
\begin{equation*}
    \mathrm{Var} \left[ \bigl( \nabla \psi_{{I \times \Lambda}} (t , \cdot) \bigr)_{\Lambda_\ell , i} \right] = \frac{1}{|\Lambda_\ell|} \sum_{x \in \Lambda_\ell} \nabla_i u(x)
\end{equation*}
where $u$ is the periodic solution to the equation $- \Delta u = - \nabla \cdot ( \indc_{\Lambda_\ell} e_i) $ in the box $\Lambda_L$. In particular, the map $u$ satisfies the identity
\begin{equation*}
    \sum_{e \in \vec{E} (\Lambda_L)} (\nabla u(e))^2 = \frac{1}{|\Lambda_\ell|} \sum_{x \in \Lambda_\ell} \nabla_i u(x).
\end{equation*}
Additionally, the Cauchy-Schwarz inequality implies that
\begin{equation*}
    \frac{1}{|\Lambda_\ell|} \sum_{x \in \Lambda_\ell} \nabla_i u(x) \leq \left( \frac{1}{|\Lambda_\ell|} \sum_{x \in \Lambda_\ell} (\nabla_i u(x))^2  \right)^{\frac 12} \leq \left( \frac{1}{|\Lambda_\ell|} \sum_{e \in \vec{E} (\Lambda_L)} (\nabla u(e))^2  \right)^{\frac 12}.
\end{equation*}
A combination of the two previous displays implies that
\begin{equation*}
    \frac{1}{|\Lambda_\ell|} \sum_{x \in \Lambda_\ell} \nabla_i u(x) \leq \frac{1}{\left| \Lambda_\ell \right|} \leq \frac{C}{\ell^d}.
\end{equation*}
As the previous estimates are valid for any $i \in \{ 1 , \ldots, d \}$ and the random variable $\bigl( \nabla \psi_{{I \times \Lambda}} (t , \cdot) \bigr)_{\Lambda_\ell} $ is Gaussian, we obtain the upper bound
\begin{equation*}
    \bigl| \bigl( \nabla \psi_{{I \times \Lambda}}  (t , \cdot)\bigr)_{\Lambda_\ell} \bigr| 
    \leq \mathcal{O}_2 \bigl( C\ell^{-\frac d2} \bigl).
\end{equation*}
We then integrate over time and use the property~\eqref{int.Onotation} to complete the proof of~\eqref{eq:16102611}.

We next prove the remaining estimates by diagonalizing the Laplacian. For any $k \in \{ -L , \ldots, L \}^d$, we denote by 
\begin{equation} \label{eq:orthonormLap}
    e_{\mathbf{k}}(x) := \frac{1}{(2L+1)^{d/2}} \exp \left( \frac{2 i \pi \mathbf{k} \cdot x}{2L+1} \right).
\end{equation}
This collection of function is a complex orthonormal basis of $\Omega_{\Lambda, \mathrm{per}}$, i.e., 
\begin{equation*}
    \sum_{x \in \Lambda} e_{\mathbf{k}}(x)\overline{e_{\mathbf{k'}}(x)} = \indc_{\{ \mathbf{k} = \mathbf{k'}\}},
\end{equation*}
and that it diagonalizes the Laplacian,
\begin{equation*}
    - \Delta e_{\mathbf{k}} = \lambda_{\mathbf{k}}  e_{\mathbf{k}} ~~\mbox{with}~~\lambda_{\mathbf{k}} = \sum_{i = 1}^d \left( 2 - 2 \cos \left( \frac{2 \pi \mathbf{k}_i}{2L+1} \right) \right).
\end{equation*}
Since for any time $t$, the law of $\psi_{I \times \Lambda}(t , \cdot)$, i.e., a Gaussian multivariate distribution whose covariance matrix is given by the Green's function, we have the identity in law
\begin{equation*}
    \psi_{I \times \Lambda}(t , \cdot) \overset{\mathrm{(law)}}{=} \sum_{\mathbf{k} \in \{ -L , \ldots, L\}^d \setminus \{0 \}} \frac{X_{\mathbf{k}} e_{\mathbf{k}}}{\sqrt{\lambda_{\mathbf{k}}}}  ,
\end{equation*}
where $(X_{\mathbf{k}})_{\mathbf{k}}$ is a collection of standard complex Gaussian random variables satisfying the following two conditions: (i) $X_{\mathbf{k}}= \bar X_{\mathbf{-k}}$ and (ii) $X_{\mathbf{k}}$ and $X_{\mathbf{k'}}$ are independent if $\mathbf{k'} \notin \{ \mathbf{k},-\mathbf{k} \}$.

Using that the family of functions $(e_{\mathbf{k}})_{k \in \{ -L , \ldots, L \}^d}$ is orthonormal and diagonalizes the Laplacian, we may compute the law of the $L^2$-norm of the map $\psi_{I \times \Lambda}(t , \cdot)$ and its gradient as follows
\begin{equation} \label{eq:20140510}
     \left\| \psi_{I \times \Lambda}(t , \cdot) \right\|_{\underline{L}^2(\Lambda) } \overset{\mathrm{(law)}}{=} \left( \frac{1}{\left| \Lambda \right|}\sum_{\mathbf{k} \in \{ -L , \ldots, L\}^d \setminus \{0 \}} \frac{|X_{\mathbf{k}}|^2}{\lambda_{\mathbf{k}}} \right)^{\sfrac12}
\end{equation}
and
\begin{equation} \label{eq:20130510}
    \left\| \nabla \psi_{I \times \Lambda}(t , \cdot) \right\|_{\underline{L}^2(\Lambda) }\overset{\mathrm{(law)}}{=} \left( \frac{1}{\left| \Lambda \right|} \sum_{\mathbf{k} \in \{ -L , \ldots, L\}^d \setminus \{ 0 \}} |X_{\mathbf{k}}|^2 \right)^{\sfrac 12}.
\end{equation}
To treat the (easier) term~\eqref{eq:20130510}, we note that the map, defined on $\mathbb{C}^{(2L+1)^d-1}$ and valued in $\R$,
\begin{equation*}
(x_{\mathbf{k}})  \mapsto  \left( \frac{1}{\left| \Lambda \right|} \sum_{\mathbf{k} \in \{ -L , \ldots, L\}^d \setminus \{0 \}} |x_{\mathbf{k}}|^2 \right)^{\sfrac 12}
\end{equation*}
is $\sqrt{1/|\Lambda|}$-Lipschitz (where we used the Euclidean metric as in Proposition~\ref{propEfronstein}). We may thus apply the Gaussian concentration inequality (or to be precise a slight modification of it to take into account that the Gaussian random variables $X_{\mathbf{k}}$ are complex and are not exactly i.i.d.), and note that the expectation of any sides of~\eqref{eq:20130510} is bounded uniformly in $L$, to obtain that
\begin{equation*}
    \left\| \nabla \psi_{{I \times \Lambda}}(t , \cdot) \right\|_{\underline{L}^2(\Lambda)} \leq C + O_2 \left( C L^{-d/2}\right).
\end{equation*}
After (suitable) integration with respect to the time variable (making use of the properties listed in Section~\ref{section.stochint}), we obtain the second estimate of~\eqref{eq:16102611}.

We next treat the term~\eqref{eq:20140510}. To this end, we note that the non-zero eigenvalues $\lambda_{\mathbf{k}}$ are always larger than $c/L^2$. We deduce that the right-hand side of~\eqref{eq:20140510} is a $C L/ \sqrt{|\Lambda|}$-Lipschitz function of Gaussian random variables. The Gaussian concentration inequality implies this time that
\begin{equation*}
    \left\| \psi_{{I \times \Lambda}}(t , \cdot) \right\|_{\underline{L}^2(\Lambda)} \leq \E \left[ \left\| \psi_{{I \times \Lambda}}(t , \cdot) \right\|_{\underline{L}^2(\Lambda)} \right] + O_2 \left( C L^{1 - d/2} \right).
\end{equation*}
The expectation in the right-hand side is of order $C ( 1 + \sqrt{\ln L} \indc_{d = 2})$. The proof of~\eqref{eq:14102611} is complete.

\end{proof}

\subsubsection{Concentration inequalities for the finite-volume corrector and its flux} \label{Sec:concentrationineq}

In this subsection, we prove concentration estimates for the $\underline{L}^2$-norm of the finite-volume corrector~$\psi_{{I \times \Lambda}}$, introduced in Definition~\ref{def.firstordercorrfinvol}, as well as the $\underline{H}^{-1}_{\mathrm{par}}$-norm of its flux. In the following statement, we assume that the length of the time interval $I$ is close to $L^2$ (the sidelength of the box $\Lambda$). This assumption simplifies the proof and will be satisfied when we apply the two-scale expansion in Section~\ref{section6.2.2.2}.

\begin{proposition}[Concentration inequality for the corrector and its flux] \label{prop.concentration} For any exponent $s \in (0 , d/2)$, there exist a constant $C : = C(s , d , c_+, c_-) < \infty$ and two exponents $\delta_s := \delta_s(d, s) > 0$ and $\zeta_s := \zeta_s(d, s) > 0$ such that, if $L^2 \leq |I| \leq L^{2 + \zeta_s }$, then for any time-dependent slope $q : I \to \Rd$, and any $z \in I \times \Lambda$ satisfying $z + Q_L \subseteq I \times \Lambda$,
\begin{equation} \label{eq:12312611}
     \left\| \phi_{{I \times \Lambda}}(\cdot ; q) \right\|_{\underline{L}^2 \left( z + Q_{L} \right)}  \leq C  L^{ 1 -\delta_s }+ \mathcal{O}_2 \left( C  L^{1 -s}\right),
\end{equation}
and
\begin{equation} \label{eq:12322611}
     \left\|  V'(p + \nabla \phi_{{I \times \Lambda}}(\cdot ;q)) - \E \left[ V'(p + \nabla \phi_{{I \times \Lambda}}(\cdot ;q)) \right] \right\|_{\underline{H}^{-1}_{\mathrm{par}} (z  + Q_{L})} \\
    \leq C L^{ 1 -\delta_s }+ \mathcal{O}_2 \left( C   L^{1 -s}\right).
\end{equation}
\end{proposition}

\begin{remark}
The proof below gives the explicit values $\delta_s = \frac 12 (1 - \frac{2s}d)$, $\zeta_s := (1 - \frac{2s}{d})$ (though they are not used in the rest of the argument).
\end{remark}

\begin{proof}
We assume, without loss of generality, that $\Lambda := \Lambda_L$ (i.e., the center of the box is the vertex $0$), that $z = 0$ and $q = 0$. We first prove the concentration inequality on the $\underline{L}^2(I \times \Lambda_L)$-norm of the gradient of the map $\phi_{{I \times \Lambda_L}}$: there exists a constant $C := C(d , c_+ , c_-) < \infty$ such that
\begin{equation} \label{eq:15442510}
    \left\| \nabla \phi_{{I \times \Lambda_L}} \right\|_{\underline{L}^2(I \times \Lambda_L)} \leq C + \mathcal{O}_2 \left( C L^{-d/2} \right).
\end{equation}
Let $\psi_{{I \times \Lambda_L}}$ be the solution of the stationary dynamic introduced in Definition~\ref{def.DynGFF}. The map $w :=  \phi_{{I \times \Lambda_L}} - \psi_{{I \times \Lambda_L}}$ solves the equation
\begin{equation} \label{def.w16202510}
    \partial_t w = \nabla \cdot \left( V'(\nabla w + \nabla \psi_{{I \times \Lambda_L}})- \nabla \psi_{{I \times \Lambda_L}}\right)  \hspace{5mm}\mbox{in} \hspace{5mm} I \times \Lambda_L.
\end{equation}
An energy estimate then yields the inequality
\begin{equation} \label{eq:15412510}
    \left\| \nabla w \right\|_{\underline{L}^2\left(I \times \Lambda_L \right)}^2 \leq C \left\| \nabla \psi_{{I \times \Lambda_L}} \right\|_{\underline{L}^2\left(I \times \Lambda_L\right)}^2 + \frac{C}{|I|} \left\| \psi_{{I \times \Lambda_L}}(s_- , \cdot) \right\|_{\underline{L}^2\left(\Lambda_L\right)}^2,
\end{equation}
and consequently
\begin{equation} \label{eq:15411510}
    \left\| \nabla \phi_{{I \times \Lambda_L}} \right\|_{\underline{L}^2\left(I \times \Lambda_L\right)}^2 \leq C \left\| \nabla \psi_{{I \times \Lambda_L}} \right\|_{\underline{L}^2\left(I \times \Lambda_L\right)}^2 + \frac{C}{L^2} \left\| \psi_{{I \times \Lambda_L}}(s_- , \cdot) \right\|_{\underline{L}^2\left(\Lambda_L\right)}^2.
\end{equation}
The inequality~\eqref{eq:15442510} is then a consequence of~\eqref{eq:15411510} and the properties~\eqref{eq:14102611} and~\eqref{eq:16102611} of the map $\psi_{{I \times \Lambda}}$. We then deduce that
\begin{align*}
    \left\| \nabla \phi_{{I \times \Lambda}} \right\|_{\underline{L}^2\left(Q_L\right)}^2 \leq \frac{|Q|}{|Q_L|} \left\| \nabla \phi_{{I \times \Lambda}} \right\|_{\underline{L}^2\left(Q \right)}^2 & \leq \frac{|I|}{L^2} \left( C + \mathcal{O}_1 \left( C L^{-d} \right) \right) \\
    & \leq C L^{\zeta_s} \left( 1  + \mathcal{O}_1 \left( L^{-d} \right)\right).
\end{align*}
We next prove the inequality~\eqref{eq:12322611}. To ease the presentation of the argument, we assume that $L = 3^m$ for some integer $m \in \N$. We next set $\delta_s = \frac 12 (1 - \frac{2s}d)$, $\zeta_s := (1 - \frac{2s}{d})$ and $m_{s} = \lfloor 2s m / d \rfloor$. By the multiscale Poincar\'e inequality (Proposition~\ref{prop:multscPoinc}), it is sufficient to show that
\begin{multline} \label{eq:16372510}
    \sum_{k = 0}^m 3^{k}  \left( \frac{1}{|\mathcal{Z}_{k ,m}|} \sum_{z \in \mathcal{Z}_{k , m}} \left| \left(V'(\nabla \phi_{{I \times \Lambda_L}}) - \E \left[ V'(\nabla \phi_{{I \times \Lambda_L}}) \right] \right)_{z + Q_{3^k}} \right|^2 \right)^{\sfrac 12} \\ \leq C 3^{m(1-\delta_s)} + \mathcal{O}_2(C 3^{  m - d m_{s}/2}).
\end{multline}
To prove~\eqref{eq:16372510}, we truncate the sum in the left-hand side at the value $m_{s}$ as follows
\begin{multline} \label{eq:1134101010}
    \sum_{k = 0}^m 3^{k}  \left( \frac{1}{|\mathcal{Z}_{k ,m}|} \sum_{z \in \mathcal{Z}_{k , m}} \left| \left(V'(\nabla \phi_{{I \times \Lambda_L}}) - \E \left[ V'(\nabla \phi_{{I \times \Lambda_L}}) \right] \right)_{z + Q_{3^k}} \right|^2 \right)^{\sfrac 12}  \\ = \underset{\eqref{eq:1134101010}-(i)}{\underbrace{\sum_{k = 0}^{m_s} 3^{k}  \left( \frac{1}{|\mathcal{Z}_{k ,m}|} \sum_{z \in \mathcal{Z}_{k , m}} \left| \left(V'(\nabla \phi_{{I \times \Lambda_L}}) - \E \left[ V'(\nabla \phi_{{I \times \Lambda_L}}) \right] \right)_{z + Q_{3^k}} \right|^2 \right)^{\sfrac 12}}} \\
    \underset{\eqref{eq:1134101010}-(ii)}{\underbrace{\sum_{k = m_s}^m 3^{k}  \left( \frac{1}{|\mathcal{Z}_{k ,m}|} \sum_{z \in \mathcal{Z}_{k , m}} \left| \left(V'(\nabla \phi_{{I \times \Lambda_L}}) - \E \left[ V'(\nabla \phi_{{I \times \Lambda_L}}) \right] \right)_{z + Q_{3^k}} \right|^2 \right)^{\sfrac 12}}}.
\end{multline}
We then estimate the two terms separately by writing
\begin{align*}
    (\eqref{eq:1134101010}-(i))  \leq C 3^{m_{s}} \left( \left\| \nabla \phi_{{I \times \Lambda_L}} \right\|_{\underline{L}(Q_{3^m})} +1 \right) 
     & \leq C 3^{m_s + \zeta_s m/2}  (1 + \mathcal{O}_2( 3^{- dm/2})) \\
     & \leq C 3^{m (1 - \delta_s)}  (1 + \mathcal{O}_2( 3^{- dm/2}))
\end{align*}
and
\begin{equation*}
     (\eqref{eq:1134101010}-(ii))  \leq C\sum_{k = m_{s}}^{m} 3^{k} \mathcal{O}_2 \left( 3^{- dk/2} \right) 
      \leq \mathcal{O}_2 \left( C 3^{m - d m_s/2} \right).
\end{equation*}
Combining the two previous displays completes the proof of~\eqref{eq:16372510}.

We next prove the inequality~\eqref{eq:12312611}. The proof follows a similar outline. We consider the map $w = \phi_{{I \times \Lambda_L}} - \psi_{{I \times \Lambda_L}}$ and note that the equation~\eqref{def.w16202510} can be rewritten as follows (using that the divergence of a spatially constant term is equal to $0$)
\begin{equation*}
    \partial_t w = \nabla \cdot \left( V'(\nabla \phi_{{I \times \Lambda_L}}) - \E \left[ V'(\nabla \phi_{{I \times \Lambda_L}})\right] - \nabla \psi_{{I \times \Lambda_L}} \right)   \hspace{5mm}\mbox{in} \hspace{5mm} I \times \Lambda_L.
\end{equation*}
Defining $\g :=  V'(\nabla \phi_{{I \times \Lambda_L}}) - \E \left[  V'(\nabla \phi_{{I \times \Lambda_L}}) \right] - \nabla \psi_{{I \times \Lambda_L}}$, we can use Proposition~\ref{multscalepoincrealpar} and write
\begin{equation} \label{eq:16512510}
    \left\| w \right\|_{\underline{L}^2(Q_{3^m})} \leq C \left\| \nabla w \right\|_{\underline{H}^{-1}_{\mathrm{par}}(Q_{3^m})} + C \left\| \g \right\|_{\underline{H}^{-1}_{\mathrm{par}}(Q_{3^m})}.
\end{equation}
We next estimate the two terms in the right-hand side using the multiscale Poincar\'e inequality. We first write
\begin{equation*}
    \left\| \nabla w \right\|_{\underline{H}^{-1}_{\mathrm{par}}(Q_{3^m})} \leq C \left\| \nabla w \right\|_{\underline{L}^{2}(Q_{3^m})} + C \sum_{k = 0}^m 3^{k} \left( \frac{1}{|\mathcal{Z}_{k , m}|}  \sum_{z \in \mathcal{Z}_{k , m}} 
    \bigl| (\nabla w)_{z + Q_{3^k}} \bigr|^2 \right)^{\sfrac12}.
\end{equation*}
Using the same argument as in the proof of~\eqref{eq:16372510}, and making use of~\eqref{eq:15000112} instead of~\eqref{eq:14590112}, we obtain
\begin{equation*}
     \left\| \nabla w \right\|_{\underline{H}^{-1}_{\mathrm{par}}(Q_{3^m})} \leq C 3^{ (1 -\delta_s)m }+ \mathcal{O}_2 \left( C 3^{(1 -s)m}\right).
\end{equation*}
There only remains to estimate the second term in the right-hand side of~\eqref{eq:16512510}. To this end, we first write
\begin{equation*}
     \left\| \g \right\|_{\underline{H}^{-1}_{\mathrm{par}}(Q_{3^m})} \leq \left\| V'(\nabla \phi_{{I \times \Lambda_L}}) - \E \left[  V'(\nabla \phi_{{I \times \Lambda_L}}) \right] \right\|_{\underline{H}^{-1}_{\mathrm{par}}(Q_{3^m})} + \left\| \nabla \psi_{{I \times \Lambda_L}} \right\|_{\underline{H}^{-1}_{\mathrm{par}}(Q_{3^m})}.
\end{equation*}
The first term in the right-hand side is estimated by the inequality~\eqref{eq:12322611}. The second term is estimated by applying the multiscale Poincar\'e and is identical to the proof of~\eqref{eq:12322611}. We obtain
\begin{equation} \label{eq:16192611}
    \left\| w \right\|_{\underline{L}^2(Q_{3^{m-1}})} \leq C 3^{ (1 -\delta_s)m }+ \mathcal{O}_2 \left( C 3^{(1 -s)m}\right).
\end{equation}
To complete the proof, we write
\begin{equation*}
    \left\| \phi_{{I \times \Lambda_L}} \right\|_{\underline{L}^2(Q_{3^{m-1}})} \leq \left\| w \right\|_{\underline{L}^2(Q_{3^m})} + \left\| \psi_{{I \times \Lambda_L}} \right\|_{\underline{L}^2(Q_{3^m})}
\end{equation*}
and use the inequalities~\eqref{eq:14102611} and~\eqref{eq:16192611}.
\end{proof}

We next establish a technical refinement of Proposition~\ref{prop.concentration} which will be used in the proof of Proposition~\ref{prop4.4}. Specifically, we prove an estimate which is uniform over the time-dependent slopes $q : I \to \Rd$ satisfying the two following criteria:
\begin{itemize}
    \item[(i)] They are constant on the time intervals of the form $[-(n+1) L^2 , n L^2]$ for $n \in \N$;
    \item[(ii)] They are bounded by a constant $M$.
\end{itemize}
We will denote this set by $\mathcal{S}_{I,M}$, i.e.,
\begin{equation} \label{set.sql}
    \mathcal{S}_{I, M} := \left\{ q : I \to \Rd \, : \, q ~\mbox{is constant on}~ [-(n+1) L^2 , n L^2] ~\mbox{for}~ n \in \N ~\mbox{and}~ \left\| q \right\|_{L^\infty(I)} \leq M \right\}.
\end{equation} 

\begin{proposition} \label{prop.concentrationunif}
Let~$s \in (0 , d/2)$. There exist~$C : = C(s , d , c_+, c_-) < \infty$ and two exponents $\delta_s , \zeta_s > 0$ depending on $d,s$, such that, if $L^2 \leq |I| \leq L^{2 + \zeta_s}$, then for every $M \in [1,\infty)$ and every $z \in I \times \Lambda$ such that $z + Q_L \subseteq I \times \Lambda$,
\begin{equation*}
    \sup_{q \in \mathcal{S}_{I, M}} \left\| \phi_{{I \times \Lambda}}(\cdot ; q) \right\|_{\underline{L}^2 ( z + Q_L )}  \leq C (M+1)  L^{ 1 -\delta_s }+ \mathcal{O}_2 \bigl( C L^{1 -s}\bigr),
\end{equation*}
and
\begin{equation*}
     \sup_{q \in \mathcal{S}_{I, M}} \left\|  V'(p + \nabla \phi_{{I \times \Lambda}}(\cdot ;q)) - \E \left[ V'(p + \nabla \phi_{{I \times \Lambda}}(\cdot ;q)) \right] \right\|_{\underline{H}^{-1}_{\mathrm{par}} (z  + Q_{L})}
    \leq C (M+1) L^{ 1 -\delta_s }+ \mathcal{O}_2 \bigl( C   L^{1 -s}\bigr).
\end{equation*}
\end{proposition}

\begin{proof}
The proof relies on a combination of Proposition~\ref{prop.concentration} and a union bound. We first set $s_0 := (s + d) / 2 $ so that $s < s_0  <d$. By choosing the exponent $\zeta_s > 0$ small enough we may:
\begin{itemize}
    \item[(i)] Apply the result of Proposition~\ref{prop.concentration} with the exponent $s_0$ (instead of $s$);
    \item[(ii)] Find a finite collection of slopes $\mathcal{Q}
\subseteq \mathcal{S}_{I , M}$ satisfying the following properties:
\begin{equation} \label{eq:11531712}
    \forall q \in \mathcal{S}_{I , M}, ~\exists q_0 \in \mathcal{Q},~\left\| q - q_0 \right\|_{\underline{L}^2(I)} \leq M L^{-1},
\end{equation}
and
\begin{equation} \label{eq:16540801}
    \left| \mathcal{Q} \right| \leq  \exp \left(L^{(s_0 - s)/2} \right).
\end{equation}
\end{itemize}
Indeed, each slope of the set $\mathcal{S}_{I , M}$ can only take $L^{\zeta_s}$-different values, all of them belonging to the interval $[- M , M]$, we may thus consider a mesh of size $M/L^2$ for each of these intervals, and obtain an upper bound on the cardinality of $\mathcal{Q}$ of the form $(L^2)^{L^{\zeta_s}}$. We can then select the parameter $\zeta_s$ sufficiently small so as to have the upper bound~\eqref{eq:16540801}.

Applying the property~\eqref{sum.Omaximum} of the $\mathcal{O}$-notation, we have that
\begin{align} \label{eq:16091712}
   \sup_{q_0 \in \mathcal{Q}} \left\| \phi_{{I \times \Lambda}}(\cdot ; q_0) \right\|_{\underline{L}^2 ( z + Q_L )}  & \leq  C L^{ 1 - \delta_{s_0}} + \mathcal{O}_2 \bigl( C L^{1 - s_0}  (\log  |\mathcal{Q}|)^{\frac12} \bigr) \\
    & \leq  C L^{ 1 - \delta_{s_0}} + \mathcal{O}_2 \bigl( C L^{1- s}\bigr). \notag
\end{align}
We then fix $q \in \mathcal{S}_{I , M}$ and $q_0 \in \mathcal{Q}$ such that~\eqref{eq:11531712} is verified. We note that the difference $\phi_{{I \times \Lambda}}(\cdot ; q) - \phi_{{I \times \Lambda}}(\cdot ; q_0)$ solves a parabolic equation with uniformly elliptic coefficient in the cylinder $I \times \Lambda$. An energy estimate thus gives the upper bound
\begin{equation*}
    \left\| \nabla \phi_{{I \times \Lambda}}(\cdot ; q) -  \nabla \phi_{{I \times \Lambda}}(\cdot ; q_0) \right\|_{\underline{L}^2(I \times \Lambda)} \leq C \left\| q - q_0 \right\|_{\underline{L}^2 \left( I \right)} \leq \frac{CM}{L}.
\end{equation*}
Consequently, by the Poincar\'e inequality for functions with spatial average equal to $0$ in the box $\Lambda$,
\begin{align} \label{eq:16081712}
    \left\|  \phi_{{I \times \Lambda}}(\cdot ; q) - \phi_{{I \times \Lambda}}(\cdot ; q_0) \right\|_{\underline{L}^2(z + Q_L)} & \leq L  \left\| \nabla  \phi_{{I \times \Lambda}}(\cdot ; q) - \nabla \phi_{{I \times \Lambda}}(\cdot ; q_0) \right\|_{\underline{L}^2(z + Q_L)}  \\
    & \leq L^{1 + \zeta_s}  \left\| \nabla  \phi_{{I \times \Lambda}}(\cdot ; q) - \nabla \phi_{{I \times \Lambda}}(\cdot ; q_0) \right\|_{\underline{L}^2(I \times \Lambda)} \notag  \\
    & \leq C M L^{\zeta_s}. \notag
\end{align}
Combining~\eqref{eq:16091712} and~\eqref{eq:16081712} (and assuming that $\zeta_s$ is small), we obtain
\begin{align*}
    \sup_{q \in \mathcal{S}_{I , M}} \left\| \phi_{{I \times \Lambda}}(\cdot ; q) \right\|_{\underline{L}^2 \left( z + Q_L \right)} & \leq  C L^{ 1 - \delta_{s_0}} + \mathcal{O}_2 \bigl( C L^{1- s}\bigr) + C M L^{\zeta_s} \\
    & \leq C (M+1) L^{ 1 - \delta_{s_0}} + \mathcal{O}_2 \bigl( C L^{1- s}\bigr).
\end{align*}
The proof of Proposition~\ref{prop.concentrationunif} is complete.
\end{proof}

\subsection{Proof of Proposition~\ref{prop4.4}} \label{sectionproof6.2}

This section is devoted to the proof of Proposition~\ref{prop4.4}. In Section~\ref{section.meyersetinternalreg}, we recall some standard regularity estimates for solutions of nonlinear parabolic equations, and, in Section~\ref{section6.2.2.2}, we implement the two-scale expansion using the results of Proposition~\ref{prop.concentrationunif}.

\subsubsection{Regularity for the solutions of the homogenized equation and discretization scheme} \label{section.meyersetinternalreg}

We recall the following notation: for each $r > 0$, we denote by $Q(r)$ the parabolic cylinder $Q = [-1 , 0] \times [-1 , 1]^d$ to which a boundary layer of size $r$ has been removed, i.e,
\begin{equation} \label{eq:14341801}
Q(r) := \left\{ (t , x) \in Q \, : \, t \geq -1 + r^2 ~\mbox{and}~ \dist (x , \partial [-1 , 1]^d) \geq r  \right\} ~\mbox{and}~ Q^\ep(r) := Q(r) \cap Q^\ep.
\end{equation}
We then collect some regularity properties satisfied by the solution of $\bar u_f$ of the equation~\eqref{eq:defubarthmhydro6.2}: the global Meyers estimate and the interior regularity. The proof of the global Meyers estimate essentially follows from~\cite[Appendix B]{ABM} (written in the linear setting), or from the one of~\cite{parviainen2009global} (written for more general, nonlinear and degenerate parabolic equations but with different assumptions regarding the regularity of the boundary conditions). The interior regularity is obtained by noting that the functions $v := \nabla \bar u_f$ and $w = \partial_t \bar u_f$ are solutions of linear parabolic equations with uniformly elliptic coefficient field and by applying either the Caccioppoli inequality (Proposition~\ref{prop.CAccioppolipara}), or the regularity estimates of Proposition~\ref{paraNash}.

\begin{proposition}[Meyers estimate and interior regularity for the homogenized equation] \label{prop4.5Meyersetreg}
There exist an exponent $\gamma_0 := \gamma_0(d , c_+ , c_-) > 0$ and a constant $C:= C(d , c_+ , c_-) < \infty$ such that
\begin{equation*}
    \left\| \nabla \bar u_f \right\|_{L^{(2 + \gamma_0)} (Q)}  \leq C (\left\| f \right\|_{W^{1,2 + \gamma_0}_{\mathrm{par}}(Q)} +1), 
\end{equation*}
and, for any $r \in (0 , 1)$,
\begin{equation*}
     r^{(2 + d)/2} \left\| \nabla \bar u_f \right\|_{L^\infty(Q(r))}  + r \left\| \nabla^{2} \bar u_f \right\|_{L^{2} (Q(r))} +  r^2 \left\| \partial_t \nabla \bar u_f \right\|_{L^2(Q(r))}  \leq C (\left\| f \right\|_{H^{1}_{\mathrm{par}}(Q)} +1). 
\end{equation*}
\end{proposition}

We additionally record the quantitative estimate on the discretization scheme that we will use in the proof. We note that the right-hand side depends on the $W^{1 , 2 + \gamma_0}_{\mathrm{par}}(Q)$-norm of the boundary conditions instead of its $H^2(Q)$-norm. This causes a deterioration of the rate of convergence, but is more suitable for our purposes.

\begin{proposition}[Discretization scheme with $W^{1 , 2 + \gamma_0}_{\mathrm{par}}$ boundary conditions] \label{prop6.8discschemeloc}
Let $\bar u_f^\ep$ be the solution of the discretized equation~\eqref{def.discequep}. There exist a constant $C:= C(d, c_+ , c_-) < \infty$ and an exponent $\beta := \beta(d , c_+ , c_-) > 0$ such that, for any $\ep > 0$,
\begin{equation*} 
    \left\| \bar u^\ep_f - \bar u_f  \right\|_{L^2(Q)} + \| \vec{\nabla}^\ep \bar u^\ep_f - \nabla \bar u_f  \|_{L^2(Q)} \leq C \ep^{\beta}  (\left\| f \right\|_{W^{1,2 + \gamma_0}_{\mathrm{par}}(Q)} +1).
\end{equation*}
\end{proposition}
The proof of this proposition can be found in Appendix~\ref{app.appendixA}.

\subsubsection{Two-scale expansion and proof of Proposition~\ref{prop4.4}}  \label{section6.2.2.2}

In this section, we implement the two-scale expansion to prove Proposition~\ref{prop4.4}, the argument is similar to the proof of Theorem~\ref{Th.quantitativehydr} and to the argument of~\cite[Section 11.4]{AKMbook} (it can in fact be seen as a combination of these two proofs), we thus only provide a detailed sketch of the argument, making use of the notation introduced in Section~\ref{eq:quanthydrolim2sc}.

\begin{proof}[Proof of Proposition~\ref{prop4.4}]
By Proposition~\ref{prop6.8discschemeloc}, it is sufficient to prove the result with the solution $\bar u^\ep$ of the discretized equation~\eqref{def.discequep} instead of the function $\bar u$. We first choose a boundary layer $r$ and a mesoscopic scale $\kappa$ of the form $r = \ep^{\theta_0}$ and $\kappa = \ep^{\theta_1}$ for some small exponents $0 < \theta_0 \ll \theta_1 \ll 1$ to be selected later in the argument, and define $L := \kappa / \ep = \ep^{ \theta_1-1}$. We then introduce the notation
\begin{equation*}
    \mathcal{Y}_\kappa := \kappa \Zd \cap [-1 , 1]^d \hspace{5mm} \mbox{and} \hspace{5mm} \mathcal{Z}_\kappa := \left(\kappa^2 \N_* \times \kappa \Zd \right) \cap Q . 
\end{equation*}
For $f \in W^{1 , 2 + \gamma_0}_{\mathrm{par}}(Q)$ and $z\in \mathcal{Z}_{\kappa}$, we denote by 
\begin{equation*}
\xi_z := 
\left\{
\begin{aligned}
& \left( \nabla^\ep  \bar u^\ep_f \right)_{z + Q_{2\kappa}^\ep} 
& \mbox{if} & \ z \in Q(r), \\
& 0 & \mbox{if} & \ z \notin Q(r).
\end{aligned} \right.
\end{equation*}
We note that, by Proposition~\ref{prop4.5Meyersetreg}, one has the upper bound
\begin{equation*}
    \sup_{z \in \mathcal{Z}_\kappa} \left| \xi_z \right| \leq r^{-(2 + d)/2} (M+1).
\end{equation*}
As in Section~\ref{eq:quanthydrolim2sc} (and specifically~\eqref{def.xi_y}), we define the map $\xi_y : [-1 , 0] \to \Rd$ according to the formula
    \begin{equation*}
    \xi_y(t) := \sum_{z \in \mathcal{Z}_{\kappa}} \xi_z \indc_{\{ (t , y) \in z + Q_{\kappa}^\ep\}}.
\end{equation*}
We next introduce the set
\begin{multline*}
     \mathcal{S}_\ep := \left\{ q : [-1 , 0] \to \Rd \, : \, q ~\mbox{is constant on}~ [-(n+1) \kappa^2 , n \kappa^2] ~\mbox{for}~ n \in \N \right. \\ \left. ~\mbox{and}~ \left\| q \right\|_{L^\infty([-1 , 0])} \leq r^{-(2 + d)/2} (M+1) \right\}.
\end{multline*}
This set is a suitably rescaled version of the sets $\mathcal{S}_{I , M}$ introduced in~\eqref{set.sql}; it is defined so that $\xi_y \in \mathcal{S}_\ep$. We then introduce the function
\begin{equation*}
     \phi^\ep(t , x) :=
     \ep \sum_{y \in \mathcal{Y}_{\kappa}} \chi_{y}(x) \phi_{Q_{y}} \left( \frac{t}{\ep^2} , \frac{x}{\ep} ; \xi_y(t) \right),
\end{equation*}
as well as the two-scale expansion
\begin{equation*}
    w^\ep := \bar u^\ep_f + \phi^\ep.
\end{equation*}
In the present setting, the error terms $(E_z)_{z \in \mathcal{Z}_\kappa}$ take the following form, distinguishing whether the point $z$ belongs to the interior of the cylinder or lies in the boundary layer:
\begin{itemize}
\item For any $z = (t,y) \in Q(r)$, we let
\begin{align*}
    E_z :=  \sum_{z' \sim z}\left\| \vec{\nabla}^\ep \bar u^\ep_f - \xi_z  \right\|_{\underline{L}^2 \left( z + Q_\kappa^\ep \right)}  + \sum_{z' \sim z} \left| \xi_z - \xi_{z'}\right|  +\ep \kappa^{-1}  \sum_{ y' \sim y} \left\|  \phi_{Q_{y'}} \left(\cdot, \cdot ; \xi_{y'} \right)\right\|_{\underline{L}^2 (z/\ep + Q_{L})}.
\end{align*}
\item For any $z = (t,y) \in Q \setminus Q(r)$, we let
\begin{align*}
    E_z & :=  \| \nabla^{\ep} \bar u^\ep_f \|_{\underline{L}^2( z + Q_{2\kappa}^\ep)} +   \sum_{ y' \sim y} \left\| \nabla \phi_{Q_{y'}} \left(\cdot, \cdot ; 0 \right)\right\|_{\underline{L}^2 (z/\ep + Q_{L})}.
\end{align*}
\end{itemize}
Using Proposition~\ref{prop4.5Meyersetreg} and Proposition~\ref{prop6.8discschemeloc} (and a computation similar to the one of~\eqref{eq:estL2normEz}), we deduce that
\begin{align*}
\lefteqn{ \frac{1}{|\mathcal{Z}_\kappa|}\sum_{z \in \mathcal{Z}_\kappa \cap Q(r)}  \sum_{z' \sim z} \left( \left\| \vec{\nabla}^\ep \bar u^\ep - \xi_z  \right\|_{\underline{L}^2 \left( z + Q_\kappa^\ep \right)}^2  + \left| \xi_z - \xi_{z'}\right|^2 \right) } \qquad & \\ & \leq C \| \vec{\nabla}^\ep \bar u^\ep - \nabla \bar u  \|_{L^2(Q)}^2 + C \kappa \| \nabla^2 \bar u  \|_{L^2(Q(r))}^2 + C \kappa^2 \| \partial_t \nabla \bar u  \|_{L^2(Q(r))}^2 \\
& \leq  C \left( \ep^{2\beta} + \kappa^2 r^{-2} + \kappa^4 r^{-4} \right)  (\left\| f \right\|_{W^{1,2 + \gamma_0}_{\mathrm{par}}(Q)}^2 +1) \\
& \leq C \left( \ep^{2\beta} + \ep^{2\theta_1 - 2\theta_0} \right) (\left\| f \right\|_{W^{1,2 + \gamma_0}_{\mathrm{par}}(Q)}^2 +1).
\end{align*}
If the exponent $\theta_0$ defining the size of the mesoscopic scale is chosen small enough, then one can apply (the suitably rescaled version of) Proposition~\ref{prop.concentrationunif} with the exponent $s_0 := (s + d)/2$ and obtain, for any $z = (t , y) \in \mathcal{Z}_\kappa$,
\begin{equation*}
\sup_{q \in \mathcal{S}_\ep} \left\| \phi_{Q_y}(\cdot, \cdot ; q) \right\|_{\underline{L}^2 \left( z + Q_L \right)}
\leq 
C (M +1 ) r^{-(2 + d)/2}   L^{ 1 -\delta_{s_0} }
+ \mathcal{O}_2 \left( C L^{1 -s_0}\right).
\end{equation*}
Using the identity $L := \lfloor \kappa/\ep \rfloor$, we deduce that 
\begin{equation*}
\sup_{q \in \mathcal{S}_\ep} 
\ep \kappa^{-1} 
\left\| \phi_{Q_y}(\cdot, \cdot ; q) \right\|_{\underline{L}^2 \left( z + Q_L \right)} 
\leq 
C (M+1)r^{-(2 + d)/2} \left( \frac \ep \kappa \right)^{\delta_{s_0}} 
+ \mathcal{O}_2 \left( C \left( \frac \ep \kappa \right)^{s_0} \right).
\end{equation*}
Using that $s_0 > s$, choosing the exponents $\theta_0$ and $\theta_1$ small enough, we deduce that there exists $\beta_s := \beta_s(d , s) > 0$ such that
\begin{equation*}
     \sup_{q \in \mathcal{S}_\ep} \ep \kappa^{-1} \left\| \phi_{Q_y}(\cdot ; q) \right\|_{\underline{L}^2 \left( z + Q_L \right)} \leq C (M+1) \ep^{\beta_s} + \mathcal{O}_2 \left( C \ep^{s} \right).
\end{equation*}
Summing over the points $z \in \mathcal{Z}_\kappa \cap Q(r)$ and using that, for any map $f \in W^{1 , 2 + \gamma_0}_{\mathrm{par}}(Q)$ satisfying $\left\| f \right\|_{W^{1 , 2+\gamma_0}_{\mathrm{par}}(Q)} \leq M$, the slope $\xi_y$ belongs to the set $q \in \mathcal{S}_\ep$, we deduce that there exists an exponent $\beta_s := \beta_s(d , s) > 0$ such that the error terms $(E_z)_{z \in \mathcal{Z}_\kappa}$ satisfy: 
\begin{equation} \label{eq:13000912}
    \sup_{f \, : \, \left\| f \right\|_{W^{1 , 2 + \gamma_0}_{\mathrm{par}}}(Q) \leq M} \frac{1}{|\mathcal{Z}_\kappa|}\sum_{z \in \mathcal{Z}_\kappa \cap Q(r)} E_z^2  \leq C (M+1) \ep^{2\beta_{s}} + \mathcal{O}_1 \left( C \ep^{2s} \right).
\end{equation}
We next estimate the sum of error terms in the boundary layer. By Proposition~\ref{prop4.5Meyersetreg} and H\"{o}lder's inequality, we have
\begin{align*}
    \left\| \nabla \bar u_f \right\|_{L^2( Q \setminus Q(r))} & \leq \left( \left|Q \setminus Q(r) \right| \right)^{\frac{(2 + \gamma_0) }{ 2 - (2 + \gamma_0) }} \left\| \nabla \bar u_f \right\|_{L^{(2 + \gamma_0) }(Q \setminus Q(r))} \\
    & \leq \left( \left|Q \setminus Q(r) \right| \right)^{\frac{(2 + \gamma_0) }{ 2 - (2 + \gamma_0) }} (\left\| f \right\|_{W^{1,2 + \gamma_0}_{\mathrm{par}}(Q)}+1) \\
    & \leq r^{\frac{(2 + \gamma_0) }{ 2 - (2 + \gamma_0) }} (\left\| f \right\|_{W^{1,2 + \gamma_0}_{\mathrm{par}}(Q)} +1).
\end{align*}
By Proposition~\ref{prop6.8discschemeloc} and using the definition $r = \ep^{\theta_0}$, we deduce that
\begin{align*}
     \left\| \nabla^\ep \bar u^\ep_f \right\|_{L^2( Q^\ep \setminus Q^\ep(r))} & \leq \| \vec{\nabla}^\ep \bar u^\ep_f - \nabla \bar u_f  \|_{L^2(Q \setminus Q(r))} + \left\| \nabla \bar u_f \right\|_{L^2( Q \setminus Q(r))} \\
     & \leq \| \vec{\nabla}^\ep \bar u^\ep_f - \nabla \bar u_f  \|_{L^2(Q)} + \left\| \nabla \bar u_f \right\|_{L^2( Q \setminus Q(r))} \\
     & \leq C \left( \ep^{\beta} + \ep^{ \theta_0 \frac{(2 + \gamma_0) }{ 2 - (2 + \gamma_0) }} \right) (\left\| f \right\|_{W^{1,2 + \gamma_0}_{\mathrm{par}}(Q)} +1).
\end{align*}
Combining the previous display with the bound~\eqref{eq:15442510} and choosing the exponents $\theta_0$ and $\theta_1$ small enough (depending on $d , s, c_+ , c_- $), we obtain that there exists an exponent $\beta_{s} > 0$ such that
\begin{equation} \label{eq:13010912}
    \sup_{f \, : \, \left\| f \right\|_{W^{1 , 2 + \gamma_0}_{\mathrm{par}}}(Q) \leq M} \frac{1}{|\mathcal{Z}_\kappa|}\sum_{z \in \mathcal{Z}_\kappa \cap ( Q \setminus Q(r))} E_z^2  \leq C (M + 1) \ep^{2\beta_{s}} + \mathcal{O}_1 \left( C \ep^{2s} \right).
\end{equation}
Combining~\eqref{eq:13000912} and~\eqref{eq:13010912} yields
\begin{equation} \label{eq:13421912}
    \sup_{f \, : \, \left\| f \right\|_{W^{1 , 2 + \gamma_0}_{\mathrm{par}}}(Q) \leq M} \frac{1}{|\mathcal{Z}_\kappa|}\sum_{z \in \mathcal{Z}_\kappa} E_z^2 \leq C (M+1)  \ep^{2 \beta_{s}} + \mathcal{O}_1 \left( C \ep^{2s} \right) .
\end{equation}
Equipped with the estimate~\eqref{eq:13421912}, one can essentially rewrite the two-scale expansion of Section~\ref{Th.quantitativehydr} once the term arising from the flux (Proposition~\ref{prop4.3fluxest}) has been estimated (by optimizing the stochastic term). To this end, we may combine the definition of the $\underline{H}^{-1}_{\mathrm{par}} (Q^\ep)$-norm with the properties~\eqref{prop.partofunity2sc} and~\eqref{prop.partofunity2scbis} of the cutoff functions $\chi_{y}$. We obtain
\begin{align*}
     \lefteqn{\biggl\| \sum_{y \in \mathcal{Y}_\kappa} \nabla^\ep \chi_{y} \cdot \left( V' \left( \nabla^\ep v_{y} \right) - D_p \bar \sigma \left( \xi_{y} \right) \right) \biggr\|_{\underline{H}^{-1}_\mathrm{par} (Q^\ep)} } \qquad & \\ &
     \leq \sum_{y \in \mathcal{Y}_\kappa} \biggl\| \nabla^\ep \chi_{y} \cdot \left( V' \left( \nabla^\ep v_{y} \right) - D_p \bar \sigma \left( \xi_{y} \right) \right) \biggr\|_{\underline{H}^{-1}_\mathrm{par} (Q^\ep)} 
     \\ & 
    \leq C \kappa^{-2} \sum_{\substack{z \in \mathcal{Z}_\kappa \\ z = (t , y)}} \left\| V' \left( \nabla^\ep v_{y} \right) - D_p \bar \sigma \left( \xi_{y} \right) \right\|_{\underline{H}^{-1}_\mathrm{par} (z + Q^\ep_{2\kappa})}.
\end{align*}
The term in the right-hand side can then be decomposed as follows
\begin{align*}
    \lefteqn{\left\| V' \left( \nabla^\ep v_{y} \right) - D_p \bar \sigma \left( \xi_{y} \right) \right\|_{\underline{H}^{-1}_\mathrm{par} (z + Q^\ep_{2\kappa})} } \qquad & \\ &
    \leq  \left\| V' \left( \xi_y + \nabla \phi_{Q_y} (\cdot , \cdot ; \xi_y) \right) - \E \left[ V' \left(  \xi_y + \nabla \phi_{Q_y} (\cdot , \cdot ; \xi_y) \right) \right] \right\|_{\underline{H}^{-1}_\mathrm{par} (z/\ep + Q_{2L})} \\ & \qquad + \left\| \E \left[ V' \left( \nabla v_y \right) \right] - D_p \bar \sigma \left( \xi_{y} \right) \right\|_{\underline{H}^{-1}_\mathrm{par} (z + Q^\ep_{2\kappa})}.
\end{align*}
The first term in the right-hand side is estimated thanks to the concentration inequality stated in Proposition~\ref{prop.concentrationunif}. The second term in the right-hand side is estimated thanks to Proposition~\ref{prop3.9}. We obtain the upper bound
\begin{equation} \label{eq:1301091222}
\sup_{f \, : \, \left\| f \right\|_{W^{1 , 2 + \gamma_0}_{\mathrm{par}}}(Q) \leq M} 
\biggl\| \sum_{y \in \mathcal{Y}_\kappa} \nabla^\ep \chi_{y} \cdot \left( V' \left( \nabla^\ep v_{y} \right) - D_p \bar \sigma \left( \xi_{y} \right) \right) \biggr\|_{\underline{H}^{-1}_\mathrm{par} (Q^\ep)} \leq C (M+1) \ep^{\beta_{s}} + \mathcal{O}_2 \left( C \ep^{s} \right).
\end{equation}
One can then rewrite the proof of Theorem~\ref{Th.quantitativehydr} of 
Section~\ref{Th.quantitativehydr}, using the estimates~\eqref{eq:13000912} and~\eqref{eq:1301091222} instead of~\eqref{eq:estL2normEz} and~\eqref{eq:21000212} to complete the proof of Proposition~\ref{prop4.4}.
\end{proof}

\section{Large-scale regularity for the Langevin dynamic} \label{sec.Section5}

This section is devoted to the proof of the large-scale regularity stated in Theorem~\ref{theoremlargescale}. The strategy of the argument follows the one initially introduced in~\cite{AS} in the context of stochastic homogenization, and his reminiscent of the Schauder regularity theory: using the quantitative homogenization theorem (or rather Proposition~\ref{prop4.4} which optimizes the stochastic term), we are able to iterate the homogenization estimate over different scales, use the $C^{1,\alpha}$ regularity of the solution of the homogenized equation (see Proposition~\ref{C11minusregforthehydrodyn}) and transfer it to the Langevin dynamic.

The proof is decomposed in three subsections and is structured as follows. In Section~\ref{sec:reghomogeq}, we record, mostly without proof, some standard regularity estimates for the solutions of the homogenized equation. Sections~\ref{sectionapprox} and~\ref{section.proofoftheoremlargescale} constitute the core of the proof of Theorem~\ref{theoremlargescale}. In Section~\ref{sectionapprox}, we use the results of the previous sections to establish that any solution of the Langevin dynamic is well-approximated, over large scales, by a solution of the homogenized equation. In Section~\ref{section.proofoftheoremlargescale}, we iterate the previous results down the scales and make use of the regularity of the homogenized equation to complete the proof of Theorem~\ref{theoremlargescale}.

\subsection{$C^{1,\alpha}$-regularity for the hydrodynamic limit} \label{sec:reghomogeq}

In this section, we collect some regularity properties of the solution of the equation~\eqref{eq:11582309}. We recall the definition of the H\"{o}lder seminorms stated in Section~\ref{sectionholdernorms}. Before stating the result, we introduce the set $\mathcal{P}_1$ of affine functions in $\Rd$,
\begin{equation*}
    \mathcal{P}_1 := \left\{ \ell : \Rd \to \R \, : \, \exists p \in \Rd, \, c \in \R, ~ \ell(x) = p \cdot x + c \right\}.
\end{equation*}
In the following statement, we will denote by $\bar Q_L$ the continuous cylinder $[-L^2 , 0] \times [-L , L]^d \subseteq \R \times \Rd$; the $\underline{L}^2(Q_L)$ and H\"{o}lder norms are considered in the discrete setting (but similar statements hold in the continuous setting). They can essentially be obtained by differentiating the equation and applying the De Giorgi-Nash-Moser regularity.

\begin{proposition}[Interior $C^{1 , \alpha}$-regularity for solutions of the homogenized equation] \label{C11minusregforthehydrodyn}
Fix $L > 0$, and let $\bar u : \bar Q_L \to \R$ be a solution of the nonlinear parabolic equation
\begin{equation} \label{eq:11582309}
    \partial_t \bar u - \nabla \cdot D_p \bar \sigma(\nabla \bar u) = 0 ~\mbox{in}~ \bar Q_L,
\end{equation}
then one has the regularity estimate: there exists an exponent $\alpha  > 0$ and a constant $C_{\alpha} < \infty$ depending on $d , c_+ , c_-$ such that, for any $ L \geq 0$,
    \begin{equation*}
        \left[\bar u \right]_{C^{1 , \alpha} (Q_{L/2})} \leq  \frac{C_{\alpha} }{L^{1+\alpha}}  \left\| \bar u  - (\bar u)_{Q_L}\right\|_{\underline{L}^2 \left( Q_L\right)}.
    \end{equation*}
    Consequently, for any $l \in (0 , L)$,
\begin{equation} \label{eq:13132309}
    \inf_{\ell \in \mathcal{P}_1} \left\| \bar u - \ell \right\|_{\underline{L}^2 \left( Q_l \right)} \leq C_{\alpha} \left( \frac{l}{L} \right)^{\!\! 1+\alpha}   \inf_{\ell \in \mathcal{P}_1} \left\| u - \ell \right\|_{\underline{L}^2 \left( Q_L \right)}.
\end{equation}
\end{proposition}

\subsection{Approximation by solutions of the homogenized equation} \label{sectionapprox}

In this section, we prove that every solution $u : Q_L \to \R$ of the Langevin dynamic~\eqref{def.ulargescale} is well-approximated on sufficiently large scales by a solution of the nonlinear parabolic equation~\eqref{eq:11582309}.

\begin{proposition} \label{lemma3.17}
Fix~$s \in (0 , d)$ and~$M \in[1,\infty)$. There exist~$C_{\mathrm{hom}} := C_{\mathrm{hom}} (s, M,d,c_+,c_-) < \infty$, $C := C (s, d,c_+,c_-) < \infty$, an exponent $\beta := \beta (s,d , c_+ , c_-) > 0$, and a nonnegative random variable $\mathcal{M}_{\mathrm{hom}}^s$ satisfying
\begin{equation} \label{14263009}
    \mathcal{M}_{\mathrm{hom}}^s \leq \mathcal{O}_{s} (C_{\mathrm{hom}}),
\end{equation}
such that the following statement holds. For any $L \geq  \mathcal{M}_{\mathrm{hom}}^s$ and any solution $u : Q_L \to \R$ of the Langevin dynamic
\begin{equation*}
    \left\{ \begin{aligned}
    d u(t , x) = \nabla \cdot V'(\nabla u) &(t , x) dt + \sqrt{2} dB_{t}\left( x \right)~\mbox{for}~(t,x) \in  Q_L, \\
    \frac{1}{L} \left\| u - (u)_{Q_L} \right\|_{\underline{L}^2(Q_L)} &\leq M,
    \end{aligned} \right.
\end{equation*}
there exists a function $\bar u : Q_{L/2} \to \R$ solution of the equation
\begin{equation} \label{eqhomog.2612}
     \partial_t \bar u - \nabla \cdot D_p \bar \sigma(\nabla \bar u) = 0  \hspace{5mm} \mbox{in}~ \bar Q_{L/2}
\end{equation}
such that
\begin{equation} \label{eq:19181912}
    \left\| u - \bar u \right\|_{\underline{L}^2 (Q_{L/2})} \leq C L^{- \beta} \| u - \left( u \right)_{Q_L} \|_{\underline{L}^2(Q_L)} + C L^{1- \beta}.
\end{equation}
\end{proposition}

\begin{proof}
The strategy of the proof relies on the observation that the map $u$ solves the system of stochastic differential equations
\begin{equation} \label{eq:deflanglargescael}
    \left\{ \begin{aligned}
    d u(t , x) &= \nabla \cdot V'(\nabla u) (t , x) dt + \sqrt{2} dB_{t}\left( x \right) &~\mbox{for} &~(t,x) \in  Q_{L/2}, \\
    u & = u  & ~\mbox{on} &~ \partial_{\mathrm{par}} Q_{L/2}.
    \end{aligned} \right.
\end{equation}
and apply (the suitably rescaled version of) Proposition~\ref{prop4.4} to the system~\eqref{eq:deflanglargescael}. A technical problem is caused by the fact that the boundary condition $u$ does not belong to the Sobolev space $W^{1,2+\gamma_0}_{\mathrm{par}}$ (essentially due to the roughness in the time variable caused by the Brownian motions). We correct this lack of regularity by using a (simpler) solution of the Langevin dynamic denoted by $\varphi_y$ below together with a partition of unity.

\medskip

\textit{Step 1. Declaration of exponents, mesoscopic scales and partition of unity.} We fix a small exponent $0 < \nu \ll 1$ whose value will be selected later in the proof. We define a (large) mesoscopic scale by setting $\ell = L^{1-\nu}$. We then introduce the sets
\begin{equation*}
    \mathcal{Y} := \ell \Zd \cap \Lambda_{L/2} \hspace{5mm} \mbox{and} \hspace{5mm} \partial \mathcal{Y} := \left\{  y \in \mathcal{Y} \, : \, \dist(y , \partial \Lambda_{L/2}) \leq \ell \right\}.
\end{equation*}
For each $y \in \partial \mathcal{Y}$, we denote by $Q_y$ the parabolic cylinder $I_{L/2} \times (y + \Lambda_{2 \ell})$, and let $\varphi_y$ be the solution of the Langevin dynamic introduced in~\eqref{Langevin.def.corrQ} of Definition~\ref{def.firstordercorrfinvol} in the cylinder $Q_y$ with slope $q = 0$. In particular, we emphasize that the dynamic $\varphi$ is not assumed to have spatial average equal to $0$ in this proof. We then introduce a partition of unity $(\chi_y)_{y \in \partial \mathcal{Y}}$ satisfying
\begin{equation} \label{prop.partofunity2screpeat}
    0 \leq \chi_y \leq \indc_{\{y + \Lambda_{2 \ell} \}}, \hspace{3mm} \sum_{y\in  \partial \mathcal{Y}} \chi_y = 1 \hspace{3mm} \mbox{on} ~\partial \Lambda_{L/2}, ~\mbox{and}~ \| \nabla \chi_y \|_{L^\infty(\Lambda_{L/2})} \leq \frac{C}{\ell}.
\end{equation}
We denote by
\begin{equation*}
    \varphi := \sum_{y \in \partial \mathcal{Y}} \chi_y \varphi_y,
\end{equation*}
and note that $\varphi$ is supported in a boundary layer of size $\ell$ around the set $I_{L/2} \times \partial \Lambda_{L/2}$.

As we will have to regularize the function $\varphi$ with respect to the time variable, we let $\eta : \R \to [0,1]$ be a smooth nonnegative cutoff function supported in the interval $[-1 ,1]$ and satisfying $\int_\R \eta = 1$. We then denote by $\tilde \varphi_y$ the time convolution of $\varphi_y$ and $\eta$,
\begin{equation} \label{eq:18422212}
    \tilde \varphi_y(t  , x) = \varphi \star \eta (t , x) = \int_\R \varphi_y (s  , x) \eta(t - s) \, ds,
\end{equation}
and by 
$$\tilde \varphi = \sum_{y \in \partial \mathcal{Y}} \chi_y \tilde \varphi_y.$$

We let $\gamma_0 > 0$ be the exponent which appears in the Meyers estimate (Proposition~\ref{interiorparabolicMEyers}) and set $s_0 := (s + d) / 2$. We then let $2 \beta_0$ be the minimum of the exponent $\beta_{s_0 , 2 + \gamma_0}$ appearing in Proposition~\ref{prop4.4}, the exponent $\delta_{s_0}$ appearing in the right-hand side of Proposition~\ref{prop.concentration} and the exponent $(s_0 - s)/4$. We additionally let $C_0/4$ be the maximum of the constants which appear in the right-hand sides of Proposition~\ref{prop.concentration}, the inequality~\eqref{eq:15442510}, the Meyers estimate (Proposition~\ref{interiorparabolicMEyers}), the multiscale Poincar\'e inequality (Proposition~\ref{prop:multscPoinc}) and Proposition~\ref{prop4.4} (with the parameter $s := s_0/2$). 

\medskip

\textit{Step 2. Definition of the minimal scales.} Equipped with these constants and exponents, we define three minimal scales $\mathcal{M}_0, \mathcal{M}_1$ and $\mathcal{M}_2$ as follows. The first one provides a minimal scale above which homogenization (in the form of Proposition~\ref{prop4.4}) occurs uniformly over the boundary condition:
\begin{equation*}
    \mathcal{M}_0 := \sup \Biggl\{ L \in \N \, : \, \sup_{ \substack{f : Q_L \to \R \\ K_f \leq C L^{\beta_0} (M+1)} }  \frac{\left\| u_f - \bar u_f \right\|_{\underline{L}^2 \left( Q_{L/2} \right)}}{ K_f + 1 } \geq C_0 L^{1-\beta_0} \Biggr\},
\end{equation*}
where the constant $C$ in the supremum is the one which appears in~\eqref{eq:11322012} and where we set 
$$K_f := \left\| f \right\|_{\underline{W}^{1, 2 + \gamma_0}_{\mathrm{par}}(Q_{L/2})}.$$
The variable $\mathcal{M}_1$ provides a minimal scale above which the $L^2$-norms of the corrector $\varphi_y$ are smaller than $C_0 \ell^{1- \beta_0}$ and the $L^2$-norm of its gradient is bounded:
\begin{equation*}
    \mathcal{M}_1 := \sup \left\{ L \in \N \, : \, \sup_{y \in \partial \mathcal{Y}} \left\| \varphi_y \right\|_{\underline{L}^2 (Q_y)} \geq C_0 \ell^{1- \beta_0} ~\mbox{or}~ \sup_{y \in \partial \mathcal{Y}} \left\| \nabla \varphi_y \right\|_{\underline{L}^2(Q_y)} \geq C_0 \right\}.
\end{equation*}
Finally, the minimal scale $\mathcal{M}_2$ is defined only in terms of the Brownian motions, and will be used to estimate two technical terms (specifically~\eqref{eq:15052312} and~\eqref{eq:16572312} below). It is defined as follows: we let $\theta : \R \to \R$ be the compactly supported (in the interval $(-1 , 1)$) function defined by the formula $\theta(t) = \int_{-\infty}^t \eta(s) ds - \indc_{\{t \geq 0\}}$, and define the stochastic integrals 
\begin{equation*}
    \left(\eta \cdot B\right) (t , x) := \int_{t-1}^{t+1} \eta(s) d B_s(x) ~~\mbox{and} ~~\left(\theta \cdot B\right) (t , x) := \int_{t-1}^{t+1} \theta(s) d B_s(x).
\end{equation*}
We then set
\begin{multline*}
    \mathcal{M}_2 := \sup \left\{ L \in \N \, : \, \frac{1}{L} \left\| \theta \cdot B \right\|_{\underline{L}^2 \left( I_{L/2} \times \partial \Lambda_{L/2} \right)}^2 \geq C_0 L^{-\nu} \right. \\ \left.  ~\mbox{and}~ \left\| \left( \sum_{y \in \partial \mathcal{Y}} \chi_y \right)  \left( \eta \cdot B \right) \right\|_{\underline{H}^{-1}_{\mathrm{par}} \left( Q_{L/2} \right)}^2 \geq C_0 L^{-\nu} \right\}.
\end{multline*}
We define the minimal scale $\mathcal{M}_{\mathrm{hom}}^s$ as follows
\begin{equation} \label{def.minscalehomog}
    \mathcal{M}_{\mathrm{hom}}^s  := \mathcal{M}_0 \vee \mathcal{M}_1 \vee \mathcal{M}_2.
\end{equation}

\medskip

\textit{Step 3. Quantifying the stochastic integrability of the minimal scale $\mathcal{M}_{\mathrm{hom}}^s$.} We first verify that the minimal scale $\mathcal{M}_{\mathrm{hom}}^s$ satisfies the stochastic integrability estimate~\eqref{14263009}. By Proposition~\ref{prop4.4} and the definitions of the constant $C_0$ and the exponent $\beta_0$, we have the estimate
\begin{align} \label{eq:15252012}
    \mathbb{P} \left( \sup_{ \substack{f : Q_L \to \R \\ K_f \leq C L^{\beta_0}(M+1)} }  \frac{\left\| u_f - \bar u_f \right\|_{\underline{L}^2 \left(Q_{L/2} \right)}}{ K_f + 1 } \geq C_0 L^{1-\beta_0} \right) & \leq \exp \left( - \frac{L^{s_0- 2\beta_0}}{C} \right) \\
    & \leq \exp \left( - \frac{L^{s}}{C} \right). \notag
\end{align}
A similar argument, using Proposition~\ref{prop.concentration}, the identity $\ell = L^{1-\nu}$ and a concentration estimate on the averaged sum of Brownian motions, yields the upper bound
\begin{equation*}
    \mathbb{P} \left( \left\| \varphi_{y} \right\|_{\underline{L}^2(Q_y)} \geq C_0 L^{1- \beta_0} \right) \leq \exp \left( - \frac{L^{(1 -\nu)(s_0 - 2 \beta_0)}}{C} \right).
\end{equation*}
Using that the cardinality of the set $\partial \mathcal{Y}$ is smaller than $C L^{(d-1)\nu}$, a union bound, and choosing the exponent $\nu$ small enough, we deduce that
\begin{equation} \label{eq:15502112}
    \mathbb{P} \left( \sup_{y \in \partial \mathcal{Y}} \left\| \varphi_{y} \right\|_{\underline{L}^2(Q_y)} \geq C_0 L^{1- \beta_0} \right) \leq C L^{(d-1)\nu} \exp \left( - \frac{L^{(1 -\nu)(s_0 - 2 \beta_0)}}{C} \right)  \leq  C \exp \left( - \frac{L^{s}}{C} \right).
\end{equation}
The same argument, using this time the inequality~\eqref{eq:15442510}, yields
\begin{equation} \label{eq:11442412}
     \mathbb{P} \left( \sup_{y \in \partial \mathcal{Y}} \left\| \nabla \varphi_{y} \right\|_{\underline{L}^2(Q_y)} \geq C_0 \right)  \leq  C \exp \left( - \frac{L^{s}}{C} \right).
\end{equation}
Finally, using that the map $\theta$ is supported in the interval $[-1 , 1]$, we can use standard concentration estimates for sum of independent random variables to obtain that
\begin{equation} \label{eq:11542412}
    \mathbb{P} \biggl[   \frac{1}{L} \left\| \theta \cdot B \right\|_{\underline{L}^2 \left( I_{L/2} \times \partial \Lambda_{L/2} \right)}^2 \geq \frac{C_0}{L} \biggr] \leq  C \exp \left( - \frac{L^{s}}{C} \right),
\end{equation}
and similarly, using this time the multiscale Poincar\'e inequality stated in Proposition~\ref{multscalepoincrealpar},
\begin{equation} \label{eq:11532412}
\mathbb{P} \Biggl[
\biggl\| \left( \sum_{y \in \partial \mathcal{Y}} \chi_y \right)  \left( \eta \cdot B \right) \biggr\|_{\underline{H}^{-1}_{\mathrm{par}} \left( Q_{L/2} \right)}^2
\geq C_0 L^{-\nu}
\Biggr] 
\leq  
C \exp \left( - \frac{L^{s}}{C} \right).
\end{equation}
Combining the fives estimates~\eqref{eq:15252012},~\eqref{eq:15502112},~\eqref{eq:11442412},~\eqref{eq:11542412} and~\eqref{eq:11532412} with a union bound shows, for any $L > 0$,
\begin{equation} \label{eq:14592909}
    \mathbb{P} \left[ \mathcal{M}_{\mathrm{hom}}^s > L \right] \leq \sum_{L' = L}^\infty C \exp \left( - \frac{(L')^{s}}{C} \right)  \leq C \exp \left( - \frac{L^{s}}{C} \right).
\end{equation}
The inequality~\eqref{eq:14592909} implies the stochastic integrability estimate~\eqref{14263009}.

\medskip

\textit{Step 4. Proving the approximation estimate~\eqref{eq:19181912}: Defining the function $\bar u$.} We next prove the inequality~\eqref{eq:19181912}. To this end, we fix a sidelength $L \geq 2 \mathcal{M}_{\mathrm{hom}}^s$. For each $y \in \partial \mathcal{Y}$, we denote by $w_y := u - \varphi_y$ and set $w := \sum_{y} \chi_y w_y$. Let us note that $w$ is equal to $u - \varphi$ on the boundary $\partial_{\mathrm{par}} Q_L$. We next note that, for each $y \in \partial \mathcal{Y}$, the map $w_y$ solves a linear parabolic equation with uniformly elliptic coefficients in the cylinder $Q_y$. We may thus apply the Meyers estimate (Proposition~\ref{interiorparabolicMEyers}) and the Caccioppoli inequality (Proposition~\ref{prop.CAccioppolipara}) and obtain: for any $z \in Q_y$ such that $z + Q_{2\ell} \subseteq Q_y$,
\begin{align*}
    \left\| \nabla w_y \right\|_{\underline{L}^{2 + \gamma_0} (z + Q_{\ell})} &\leq \frac{C}{\ell} \left\| w_y - (w_y)_{z + Q_{2\ell}} \right\|_{\underline{L}^2(z + Q_{2\ell})} \\ 
    & \leq \frac{C}{\ell} \left\| u - (u)_{z + Q_{2\ell}} \right\|_{\underline{L}^2(z + Q_{2\ell})}  + \frac{C}{\ell}\left\| \varphi_y \right\|_{\underline{L}^2(z + Q_{2\ell})}.
\end{align*}
Covering the cylinder $Q_{L/2}$ with cylinders of the form $z + Q_\ell$, we may choose the exponent $\nu$ small enough (depending on $\beta_0$) so that
\begin{equation*}
   \sum_{y \in \partial \mathcal{Y}} \left\| \chi_y \nabla w_y \right\|_{\underline{L}^{2 + \gamma_0} (Q_{y})}  \leq C L^{\beta_0/2} \left( \frac{1}{L}\| u - \left( u \right)_{Q_L} \|_{\underline{L}^2(Q_L)} + 1 \right).
\end{equation*}
A similar (and simpler) computation, involving regularity estimates for solution of parabolic equation (for instance Proposition~\ref{paraNash} provides $L^\infty$-regularity for solutions of parabolic equations) yields the bound
\begin{equation*}
    \sum_{y \in \partial \mathcal{Y}} \left\| \nabla \chi_y w_y \right\|_{\underline{L}^{2 + \gamma_0} (Q_{y})}  \leq C L^{\beta_0/2} \left( \frac{1}{L}\| u - \left( u \right)_{Q_L} \|_{\underline{L}^2(Q_L)} + 1 \right).
\end{equation*}
A combination of the two previous displays shows
\begin{equation*}
    \left\| \nabla w \right\|_{\underline{L}^{2+ \gamma_0} (Q_{L/2})} \leq C L^{\beta_0/2} \left( \frac{1}{L}\| u - \left( u \right)_{Q_L} \|_{\underline{L}^2(Q_L)} + 1 \right).
\end{equation*}
Since $\partial_t w = \sum_{y \in \partial \mathcal{Y}} \chi_y \nabla \cdot \a_y \nabla w_y$, we may use the properties of the maps $\left( \chi_y \right)_{y \in \mathcal{Y}_\kappa}$ and choose the exponent $\nu$ small enough so as to have
\begin{equation}  \label{eq:11322012}
    \left\| w \right\|_{\underline{W}^{1, 2 + \gamma_0}_{\mathrm{par}} (Q_{L/2})} \leq  C L^{\beta_0/2} \left( \frac{1}{L}\| u - \left( u \right)_{Q_L} \|_{\underline{L}^2(Q_L)} + 1 \right).
\end{equation}
We then let $\psi$ be the solution of the system of stochastic equations
\begin{equation*}
    \left\{ \begin{aligned}
    d \psi (t , x) & = \nabla \cdot V'(\nabla \psi) (t , x) dt + \sqrt{2} dB_{t}\left( x \right) & ~\mbox{for}~ & (t,x) \in  Q_{L/2}, \\
    \psi  & = w  & ~\mbox{on}~ & \partial_{\mathrm{par}} Q_{L/2}.
    \end{aligned} \right.
\end{equation*}
Using~\eqref{eq:11322012} and applying Proposition~\ref{prop4.4} (after suitable rescaling), we obtain that the solution $\bar u$ of the equation
\begin{equation*}
    \left\{ \begin{aligned}
    \partial_t \bar u - \nabla \cdot D_p \bar \sigma(\nabla \bar u) & = 0 &~\mbox{in} &~ \bar Q_{L/2}, \\
    \bar u &= w &~\mbox{on} &~ \partial_{\mathrm{par}} \bar Q_{L/2} \\
    \end{aligned} \right.
\end{equation*}
satisfies
\begin{equation} \label{eq:12292012}
    \left\| \psi - \bar u \right\|_{\underline{L}^2 \left( Q_{L/2} \right)} \leq C L^{- \beta_0/2} \| u - \left( u \right)_{Q_{L}} \|_{\underline{L}^2(Q_L)} + C L^{1- \beta_0/2}.
\end{equation}
From~\eqref{eq:12292012} and the triangle inequality, we see that to prove~\eqref{eq:19181912}, it is sufficient to show, for some exponent $\beta > 0$,
\begin{equation} \label{eq:12312012}
    \left\| \psi - u \right\|_{\underline{L}^2 \left( Q_{L/2} \right)} \leq C L^{1- \beta}.
\end{equation}

\medskip

\textit{Step 5. Proving the approximation estimate~\eqref{eq:19181912}: Proof of~\eqref{eq:12312012}.}
We note that the difference $h = u - \psi$ solves a parabolic equation of the form
\begin{equation} \label{eq:12362012}
    \left\{ \begin{aligned}
    \partial_t h - \nabla \cdot \a \nabla h &= 0 &~\mbox{in}~&Q_{L/2}, \\
    h &= \varphi &~\mbox{on}~&\partial_{\mathrm{par}} Q_{L/2},
    \end{aligned} \right.
\end{equation}
where $\a$ is a uniformly elliptic coefficient field. To prove~\eqref{eq:12312012}, we introduce a good test function which will then be tested in the parabolic equation: we let $h_0$ be the map defined by the formula
\begin{equation*}
    h_0 := \tilde \varphi \indc_{Q_{L/2}}+  \varphi \indc_{\partial_{\mathrm{par}} Q_{L/2}},
\end{equation*}
so that the map $H := h - h_0$ solves the parabolic equation
\begin{equation} \label{eq:17422212}
    \left\{ \begin{aligned}
     \partial_t H - \nabla \cdot \a \nabla H &= \partial_t h_0 - \nabla \cdot \a \nabla h_0  &~\mbox{in} &~Q_{L/2}, \\
    H &= 0 &~\mbox{on} &~\partial_{\mathrm{par}} Q_{L/2}.
    \end{aligned} \right.
\end{equation}
We next prove that there exists an exponent $\gamma := \gamma(d , s , c_+ , c_-)  > 0$ such that 
\begin{equation} \label{eq:18202212}
    \left\|  \nabla H \right\|_{\underline{L}^2 (Q_{L/2})} \leq C L^{-\gamma}.
\end{equation}
The inequality~\eqref{eq:18202212} is sufficient to conclude the proof of Proposition~\ref{lemma3.17}: indeed, using Poincar\'e's inequality (applied to the function $H$), the definition of the minimal scale $\mathcal{M}_1$ and the inequality $\ell \leq L$, we deduce that, if we set $\beta := \min (\gamma , \beta_0)$,
\begin{equation*}
    \left\| h \right\|_{\underline{L}^2(Q_{L/2})} \leq C L \left\|  \nabla H \right\|_{\underline{L}^2(Q_{L/2})} + \| h_0 \|_{\underline{L}^2 \left( Q_{L/2} \right)} \leq C L^{1 - \gamma} + C \ell^{1 - \beta_0} \leq C L^{1-\beta}.
\end{equation*}
We split the proof of the inequality~\eqref{eq:17422212} into two substeps: we first estimate the norm of the gradient of $h_0$ and then the norm of the time derivative of $h_0$ and prove that they are both small (in a suitable sense). Since $H$ solves the parabolic equation~\eqref{eq:17422212} this is enough to obtain~\eqref{eq:18202212} (using an energy estimate).

\medskip

\textit{Substep 5.1. Estimating the norm of the gradient of $h_0$.} To establish~\eqref{eq:18202212}, we estimate the $L^2$-norm of the gradient of the map $h_0$. Specifically, we prove the estimate
\begin{equation} \label{eq:15112312}
    \left\| \nabla h_0 \right\|_{\underline{L}^2(Q_{L/2})} \leq C L^{-\nu}.
\end{equation}
For any edge $e = (x , y)$ with $x , y \in \Lambda_L$, we have the identity
\begin{equation*}
     \nabla h_0(t, e) = \nabla \tilde \varphi(t,e),
\end{equation*}
and for any edge $e = (x , y)$ with $x \in \Lambda_L$ and $y \notin \Lambda_L$, we have the identity
\begin{equation} \label{eq:15062312}
    \nabla h_0(t,e) = \nabla \tilde \varphi(t , e) + \left( \tilde \varphi(t , y) - \varphi(t , y) \right).
\end{equation}
Consequently
\begin{equation} \label{eq:15072312}
    \left\| \nabla h_0 \right\|_{\underline{L}^2\left( Q_{L/2} \right)}^2 \leq \| \nabla \tilde \varphi \|_{\underline{L}^2\left( Q_{L/2} \right)}^2 +  \frac{C}{L} \| \varphi - \tilde \varphi \|_{\underline{L}^2\left( I_{L/2} \times \partial \Lambda_{L/2} \right)}^2.
\end{equation}
Using the definition of the minimal scale $\mathcal{M}_1$ and the fact that the map $\tilde \varphi$ is supported in a boundary layer of size $L^{(1-\nu)}$ around $I_{L/2} \times \partial \Lambda_{L/2}$, we have the upper bound
\begin{equation} \label{eq:09251412}
   \| \nabla \tilde \varphi \|_{\underline{L}^2\left( Q_{L/2} \right)}^2 \leq C L^{-\nu}.
\end{equation}
Additionally, if we let $\theta : \R \to \R$ be the map with bounded variation and compact support defined by the formula $\theta(t) = \int_{-\infty}^t \eta(s) ds - \indc_{\{t \geq 0\}}$, then we have, for any $x \in \partial \Lambda_{L/2}
$,
\begin{align} \label{eq:14482312}
    (\varphi - \tilde \varphi)(t,x) & = \varphi \star (\delta_0 - \eta) (t,x)  \\
    & = \underset{\eqref{eq:14482312}-(i)}{\underbrace{\sum_{y \in \partial \mathcal{Y}} \chi_y(x) \nabla \cdot \left( V'(\nabla \varphi_y) \star \theta\right)(t , x)}} + \underset{\eqref{eq:14482312}-(ii)}{\underbrace{ \int_{t-1}^{t+1} \theta(s) dB_s(x)}}. \notag
\end{align}
The term~\eqref{eq:14482312}-(i) can be estimated by ignoring the discrete divergence, using that the map $V'$ is Lipschitz and the definition of the minimal scale $\mathcal{M}_1$. We obtain
\begin{equation} \label{eq:15042312}
    \frac{1}{L} \left\|  \eqref{eq:14482312}-(i) \right\|_{\underline{L}^2(I_{L/2} \times \partial \Lambda_{L/2})}^2 \leq \frac{C}{L^{d+2}}  \sum_{y \in \partial \mathcal{Y}}  \left\|\nabla \varphi_y  \right\|_{L^2(Q_y)}^2 \leq C L^{-\nu}.
\end{equation}
The term~\eqref{eq:14482312}-(ii) is estimated by the minimal scale $\mathcal{M}_2$ and we have
\begin{equation} \label{eq:15052312}
    \frac{1}{L} \left\|  \eqref{eq:14482312}-(ii) \right\|_{\underline{L}^2(I_{L/2} \times \partial \Lambda_{L/2})}^2 \leq \frac{C}{L}.
\end{equation}
Combining the estimates~\eqref{eq:15062312},~\eqref{eq:15072312},\eqref{eq:09251412},~\eqref{eq:14482312},~\eqref{eq:15042312} and~\eqref{eq:15052312} completes the proof of~\eqref{eq:15112312}.

\medskip

\textit{Substep 5.2. Control of the time derivative $\partial_t h_0$.} We have the identity
\begin{equation} \label{eq:18502212}
    \partial_t h_0 (t , x) = \underset{\eqref{eq:18502212}-(i)}{\underbrace{\sum_{y\in \partial \mathcal{Y}} \chi_y(x) \nabla \cdot \left( V'(\nabla \varphi_y) \star \eta \right) (t , x)}}  +  \underset{\eqref{eq:18502212}-(ii)}{\underbrace{\left( \sum_{y \in \partial \mathcal{Y}} \chi_y(x) \right)  \int _{t-1}^{t+1} \eta(s) dB_s(x)}}.
\end{equation}
The $\underline{L}^{2}(I_{L/2}, \underline{H}^{-1} \left(  \Lambda_{L/2}\right))$-norm of the term~\eqref{eq:18502212}-(i) can be estimated as follows
\begin{equation*}
    \left\| \eqref{eq:18502212}-(i) \right\|_{\underline{L}^{2}\left(I_{L/2}, \underline{H}^{-1} \left(  \Lambda_{L/2}\right) \right)}^2 \leq  \frac{C}{L^{d+2}} \sum_{y \in \partial \mathcal{Y}} \left\| \nabla \varphi_y \right\|_{L^2 \left( Q_{y} \right)}^2 \leq C L^{-\nu}.
\end{equation*}
Using the definition of the minimal scale $\mathcal{M}_2$, we have the estimate
\begin{equation} \label{eq:16572312}
    \left\|  \eqref{eq:18502212}-(ii)\right\|_{\underline{H}^{-1}_{\mathrm{par}} \left( Q_{L/2} \right)}^2 \leq C L^{- \nu}.
\end{equation}
Combining the three previous displays with~\eqref{eq:15112312}, we have the decomposition
\begin{equation*}
    \partial_t h_0 - \nabla \cdot \a \nabla h_0 = \mathcal{E}_1 + \mathcal{E}_2,
\end{equation*}
with
\begin{equation*}
    \mathcal{E}_1 = \sum_{y\in \partial \mathcal{Y}} \chi_y \nabla \cdot \left( V'(\nabla \varphi_y) \star \eta \right) - \nabla \cdot \a \nabla h_0,
\end{equation*}
and, for any $(t , x) \in Q_{L/2}$,
\begin{equation*}
    \mathcal{E}_2(t , x) =  \left( \sum_{y \in \partial \mathcal{Y}} \chi_y(x) \right)  \int _{t-1}^{t+1} \eta(s) dB_s(x),
\end{equation*}
so that
\begin{equation*}
    \left\| \mathcal{E}_1  \right\|_{\underline{L}^2(I_{L/2} , \underline{H}^{-1} \left( \Lambda_{L/2}\right))}^2 \leq C L^{-\nu} \hspace{5mm} \mbox{and} \hspace{5mm} \left\| \mathcal{E}_2 \right\|_{\underline{H}^{-1}_{\mathrm{par}}(Q_{L/2})}^2 \leq C L^{-\nu}
\end{equation*}

\medskip

\textit{Substep 5.2. The conclusion.} The proof of the estimate~\eqref{eq:18202212} then follows from an energy estimate, using the identification of the space $\underline{H}^{-1}_{\mathrm{par}}$ stated in Lemma~\ref{lem.idenH-1par}. The proof is identical to the one of the term $v_1$ in~\eqref{eq:v1scexp} of Section~\ref{secTh.quantitativehydr}, we thus omit the technical details.
\end{proof}

\subsection{Proof of Theorem~\ref{theoremlargescale}} \label{section.proofoftheoremlargescale}
This section is devoted to the proof of Theorem~\ref{theoremlargescale}, building upon the results established in Sections~\ref{sec:reghomogeq} and~\ref{sectionapprox}.

\begin{proof}[Proof of Theorem~\ref{theoremlargescale}]
Fix $s \in (0 ,d)$, $L , M < \infty$, and let $u$ be a solution of the Langevin dynamic satisfying
\begin{equation*}
    \left\{ \begin{aligned}
    d u(t , x) = \nabla \cdot V'(\nabla u) &(t , x) dt + \sqrt{2} dB_{t}\left( x \right)~\mbox{for}~(t,x) \in  Q_L, \\
    \frac{1}{L} \left\| u - (u)_{Q_L} \right\|_{\underline{L}^2(Q_L)} &\leq M.
    \end{aligned} \right.
\end{equation*}
We let $C_0 := C_0 (d ,s, c_- , c_+ ) < \infty$ and $L_0 := L_0 (d ,s, c_- , c_+ ) < \infty$ be large constants to be selected later in the argument (see~\eqref{def.c_0induchyp} and~\eqref{eq:19502412}). We consider the minimal scale $\mathcal{M}_{\mathrm{hom}}^s$ introduced in Proposition~\ref{lemma3.17} with the constant $C_0 (M+1)$ (instead of $M$). We denote by
\begin{equation*}
    M_u := \frac{1}{L} \left\| u - (u)_{Q_L} \right\|_{\underline{L}^2(Q_L)}.
\end{equation*}
We then set $\mathcal{X} := \mathcal{M}_{\mathrm{hom}}^s \vee L_0$, and prove the two inequalities: for any $L \geq \mathcal{X}$,
\begin{equation} \label{eq:1217021000}
    \sup_{l \in [\mathcal{X} , L]} \frac 1l \left\| u - (u)_{Q_l} \right\|_{\underline{L}^2(Q_l)} \leq  C_0 (M_u +1),
\end{equation}
and, for any $l \in [\mathcal{X} , L]$, 
\begin{equation} \label{eq:1217021000bis}
    \inf_{\ell \in \mathcal{P}_1}\left\| u - \ell \right\|_{\underline{L}^2(Q_l)} \leq C \left( \frac{l}{L} \right)^{1+\alpha} \inf_{\ell \in \mathcal{P}_1}\left\| u - \ell \right\|_{\underline{L}^2(Q_L)} + C l^{-\beta} (M+1).
\end{equation} 
To prove the inequality~\eqref{eq:1217021000}, we first introduce a few parameters and definitions. We let $C_\alpha$ be the constant which appears in Proposition~\ref{C11minusregforthehydrodyn} and let 
\begin{equation*}
    \theta :=   \left( 4^{d/2 + 1}C_{\alpha} \right)^{- 1/(2\alpha)} \in (0 , 1).
\end{equation*}
We denote by $l_j := \theta^j L$ and let $J$ be the largest integer such that $\theta^J L \geq \mathcal{X}$.
We introduce the excess decay $E_1$ by the formula, for any $l > 0$,
\begin{equation} \label{eq:12170210}
    E_1(l) := \frac 1l \inf_{\ell \in \mathcal{P}_1} \left\| u - \ell \right\|_{\underline{L}^2(Q_l)}.
\end{equation}
For $j \in \{ 1 , \ldots , J\}$, denote by $\ell_j : x \mapsto p_j \cdot x + (u)_{Q_{lj}}$ the minimizing affine function in the definition of~$E_1(l_j)$.
We next record two preliminary estimates pertaining to the slopes $(p_j)_{j \in\{ 0 , \ldots , J\} }$. First, the slope $p_0$ can be bounded in terms of the constant $M_u$: we have
\begin{equation*}
    \left| p_0 \right| \leq C E_1 (L) + C \left\| u - (u)_{Q_L} \right\|_{\underline{L}^2(Q_L)} \leq C \left\| u - (u)_{Q_L} \right\|_{\underline{L}^2(Q_L)} \leq CM_u.
\end{equation*}
Second, the variation of the slope between two scales can be measured in terms of the excess decay: we have, for any $j \in \{ 0 , \ldots, J-1 \}$,
\begin{align*}
    \left| p_{j+1} - p_j \right| & \leq C \left( \frac{1}{l_{j+1}}\left\| u - \ell_{j+1} \right\|_{\underline{L}^2 \left(  Q_{l_{j+1}}\right)} + \frac{1}{l_{j+1}} \left\| u - \ell_j \right\|_{\underline{L}^2 \left(  Q_{l_{j+1}}\right)} \right) \\
    & \leq C \left( E_1(l_{j+1}) + \theta^{-d/2 -2} E_1( l_{j}) \right).
\end{align*}
Using the definition of the excess decay $E_1$, we additionally have
\begin{equation*}
    E_1 (l_{j+1}) = \frac{1}{l_{j+1}} \inf_{\ell \in \mathcal{P}_1}\left\| u - \ell \right\|_{\underline{L}^2 \left(  Q_{l_{j+1}}\right)} \leq \frac{\theta^{-d/2 -2}}{l_j} \inf_{\ell \in \mathcal{P}_1}\left\| u - \ell \right\|_{\underline{L}^2 \left(  Q_{l_{j}}\right)} \leq \theta^{-d/2 -2} E_1 (l_{j}).
\end{equation*}
As a consequence of the two previous displays, we obtain the inequality
\begin{equation*}
    \left| p_{j+1} - p_j \right| \leq C  \theta^{-d/2 -2} E_1( l_{j}).
\end{equation*}
We thus have, for any $j \in \{ 0 , \ldots, J-1 \}$,
\begin{align} \label{eq:14230210}
    \frac{1}{l_j}\left\| u - (u)_{Q_{l_j}} \right\|_{\underline{L}^2 \left( Q_{l_j}\right)} \leq  E_1(l_j) + C |p_j|
    & \leq  E_1(l_j) + C |p_{0}| + C \sum_{k = 0}^{j-1} |p_{k-1} - p_k|  \\
    & \leq   C \theta^{-d/2 -2} \sum_{k = 0}^{j-1} E_1(l_k) + C M_u. \notag
\end{align}
We next let $\beta$ be the exponent which appears in Proposition~\ref{lemma3.17} and define the constant $C_0$ according to the formula
\begin{equation} \label{def.c_0induchyp}
    C_0 := C  \theta^{-d -2} \left( \sum_{j = 0}^{\infty} \theta^{\alpha j/2} + \frac{\sum_{j = 0}^{\infty} \theta^{\beta j} }{1 - \theta^\beta} \right) + C 
    = C   \left( \frac{\theta^{-d -2}}{1-\theta^{\alpha/2}} + \frac{\theta^{-d -2}}{(1 - \theta^\beta)^2} +1\right), 
\end{equation}
where $C$ is the constant which appears in the right-hand side of~\eqref{eq:14230210}. The constant $C_0$ only depends on the parameters $d , s, c_+ , c_-$. 

We next prove, by an inductive argument, the following upper bound on the excess decay~$E_1$: for each $j \in \{0 , \ldots, J\}$,
\begin{equation} \label{eq:13590210}
    E_1 \left( l_{j} \right) \leq \theta^{\alpha j/2}  E_1 \left( L \right) +  \frac{\theta^{\beta (J-j+1)}}{1 - \theta^\beta} (M_u+1).
\end{equation}
and
\begin{equation} \label{eq:15040210}
    \frac{1}{l_j}\left\| u - (u)_{Q_{l_j}} \right\|_{\underline{L}^2 \left( Q_{l_j}\right)} \leq \theta^{d/2 + 2} C_0 (M_u+1).
\end{equation}

\textit{Initialization:} In the case $j = 0$, the inequalities~\eqref{eq:13590210} and~\eqref{eq:15040210} are clearly satisfied.

\vspace{2mm}

\textit{Induction:} We assume that the inequalities~\eqref{eq:13590210} and~\eqref{eq:15040210} are valid for the integers between $0$ and~$j$, and prove that they hold with the integer $j+1$. The strategy is to apply Proposition~\ref{lemma3.17} combined with the $C^{1 , \alpha}$-regularity for the solutions of the homogenized equation stated in Proposition~\ref{C11minusregforthehydrodyn}. By~\eqref{eq:15040210} and the inequality $\theta < 1$, we have the upper bound
\begin{equation} \label{eq14180210}
    \frac{1}{l_j}\left\| u - (u)_{Q_{l_j}} \right\|_{\underline{L}^2 ( Q_{l_j})} \leq C_0 (M_u+1).
\end{equation}
We may thus apply Proposition~\ref{lemma3.17} and obtain that there exists a function $\bar u_j : Q_{l_j /2} \to \R$ solution of the equation~\eqref{eqhomog.2612} such that
\begin{equation*}
    \frac{1}{l_j}\left\| \bar u_{j} - u \right\|_{\underline{L}^2(Q_{ l_{j/4}})} \leq C (M_u+1) l_j^{-\beta}.
\end{equation*}
Assuming, without loss of generality, that $\theta$ is smaller than $1/4$ and using the $C^{1,\alpha}$-regularity of the map~$\bar u,$ we can write
\begin{align} \label{eq:21472409}
    l_{j+1} E_1(l_{j+1}) & = \inf_{\ell \in \mathcal{P}_1} \left\| u - \ell \right\|_{\underline{L}^2(Q_{ l_{j+1}})} \\
    & \leq \inf_{\ell \in \mathcal{P}_1} \left\| \bar u_{j} - \ell \right\|_{\underline{L}^2(Q_{ l_{j+1}})} + \left\| \bar u_{j} - u \right\|_{\underline{L}^2(Q_{ l_{j+1}})} \notag \\
    & \leq  C_\alpha \theta^{1+\alpha}  \inf_{\ell \in \mathcal{P}_1}\left\| \bar u_j - \ell \right\|_{\underline{L}^2(Q_{l_j/4})} + \theta^{-\frac d2 - 1} C (M_u+1) l_j^{1-\beta} \notag \\
    & \leq   C_\alpha \theta^{1+\alpha} \inf_{\ell \in \mathcal{P}_1}\left\| u - \ell \right\|_{\underline{L}^2(Q_{l_j/4})} + C_\alpha \theta^{1+\alpha} \left\| \bar u_j - u \right\|_{\underline{L}^2(Q_{l_j/4})}  + C (M_u+1) l_j^{1-\beta} \notag \\
    & \leq 4^{d/2 + 1} C_\alpha \theta^{1+\alpha}  (l_j E_1(l_j)) + C \left( C_\alpha \theta^{1+\alpha}  + \theta^{-\frac d2 - 1} \right)  (M_u+1) l_j^{1-\beta}. \notag
\end{align}
Using the definition of $\theta$, we have that $4^{d/2 + 1} C_\alpha \theta^{\alpha} \leq \theta^{\alpha/2}.$ We additionally recall that $ l_j \geq  \theta^{j - J} L_0$, and select the constant $L_0$ large enough so that
\begin{equation} \label{eq:19502412}
      C \left( C_\alpha \theta^{\alpha}  + \theta^{-\frac d2 - 2} \right) L^{-\beta}_0 \leq 1.
\end{equation}
We note that the constant $L_0$ depends only on the parameters $d , s, c_+ , c_-$. Using the identity $l_{j+1} = \theta l_j$, the computation~\eqref{eq:21472409} can thus be simplified and becomes
\begin{equation*}
    E_1(l_{j+1}) \leq \theta^{\alpha/2} E_1(l_j) + \theta^{\beta (J-j)} (M_u+1).
\end{equation*}
Applying the induction hypothesis yields
\begin{align*}
    E_1(l_{j+1}) & \leq \theta^{\alpha/2} \left( \theta^{\alpha j/2}  E_1 \left( L \right) +  \frac{\theta^{\beta (J-j+1)}}{1 - \theta^\beta} (M_u+1) \right) + \theta^{\beta (J-j)} (M_u+1) \\
    & \leq \theta^{\alpha (j+1)/2}  E_1 \left( L \right) +  \frac{\theta^{\beta (J -j)}}{1 - \theta^\beta} (M_u+1).
\end{align*}
The proof of the inequality~\eqref{eq:13590210} is complete. The inequality~\eqref{eq:15040210} can be deduced from the inequality~\eqref{eq:14230210} with the induction hypothesis~\eqref{eq:13590210}, the inequality $E_1(L) \leq M_u$ and the definition of the constant $C_0$ stated in~\eqref{def.c_0induchyp}. We obtain
\begin{align*}
    \frac{1}{l_{j+1}}\left\| u - (u)_{Q_{l_{j+1}}} \right\|_{\underline{L}^2 \left( Q_{l_{j+1}}\right)} & \leq  C \theta^{-d/2 -2} \sum_{k = 0}^{j} E_1(l_j) + C M_u \\
    & \leq  C \theta^{-d/2 -2} \sum_{k = 0}^{j} \left( \theta^{\alpha k/2} M_u + \frac{\theta^{\beta k}}{1 - \theta^\beta}  (M_u+1) \right) + CM_u \\
    & \leq \theta^{d/2 + 2}  C_0 (M_u+1).
\end{align*}
To conclude the proof of~\eqref{eq:1217021000}, we note that, for any $l \in [\mathcal{X} , L]$, there exists an integer $j \in \{ 0 , \ldots , J\}$ such that $\theta l_j \leq l \leq l_j$, and consequently
\begin{equation*}
     \frac 1l \left\| u - (u)_{Q_l} \right\|_{\underline{L}^2(Q_l)} \leq \theta^{-d/2 - 2} \frac 1{l_j} \left\| u - (u)_{Q_{l_j}} \right\|_{\underline{L}^2(Q_{l_j})} \leq  C_0 (M_u+1).
\end{equation*}
The $C^{1,\alpha}$-large scale regularity estimate stated in~\eqref{eq:1217021000bis} can be deduced from~\eqref{eq:13590210} by a similar argument.
\end{proof}

\appendix

\section{Quantitative approximation scheme for nonlinear parabolic equation} \label{app.appendixA}

This appendix is devoted to the proofs of Proposition~\ref{prop.approx} and Proposition~\ref{prop6.8discschemeloc} and provides a quantitative approximation scheme for solutions of nonlinear parabolic equations.

\begin{proof}[Proof of Proposition~\ref{prop.approx}]
Fix $\ep > 0$, a boundary conditions $f \in H^2(\R \times \Rd)$, let $\bar u $ be the solution of the parabolic equation~\eqref{eq:defubarthmhydro} and extend it by the value $f$ outside of the set $Q$. We recall that the map $\bar u$ satisfies the following estimates:
\begin{itemize}
    \item The $H^2$-regularity estimate
    \begin{equation} \label{eq:19050501}
        \left\| \bar u \right\|_{L^2 \left( I , H^2(D) \right)} \leq C \left\| f \right\|_{ H^2(Q)};
    \end{equation}
    \item The $H^1$-regularity for the time derivative
    \begin{equation} \label{eq:19060501}
        \left\| \partial_t \bar u \right\|_{L^2 \left( I , H^1(D) \right)} \leq C \left\| f \right\|_{H^2(Q)}.
    \end{equation}
\end{itemize}
For $r > 0$, we denote by $D(r)$ and $D^\ep(r)$ the sets
\begin{equation*}
    D(r) := \left\{ x \in D \, : \, \dist \left( x , \partial D \right) \geq r \right\}, ~~ D^\ep(r) := D^\ep \cap  D(r),
\end{equation*}
In the rest of the proof, we set $r = \sqrt{\ep}$.
We let $\chi : \Rd \to \R$ be a smooth nonnegative cutoff function supported in $[-1 , 1]^d$ and satisfying $\int_{\Rd} \chi(x) \, dx = 1$. We then rescale the map $\chi$ by setting $\chi_{r} = r^{-d} \chi(\cdot/r)$ and define, for $(t , x) \in I \times D(r)$,
\begin{equation*} 
    \bar u \star \chi_{r} ( t , x) =  \int_{\Rd} \bar u(t , x - y) \chi_{r} (y) \, dy.
\end{equation*}
Let $\xi : \Rd \to \R$ be a cutoff function satisfying
\begin{equation*}
    \indc_{ D(2r)} \leq \xi \leq  \indc_{ D(r)}, ~\left| \nabla \xi \right|  \leq C r^{-1}.
\end{equation*}
We then denote by
\begin{equation} \label{def.mapU1}
    U = \xi \left( \bar u \star \chi_{r}\right) + (1 - \xi) (\bar u \star \chi_{\ep} ),
\end{equation}
that is, we mollify the function $\bar u$ on a scale $r = \sqrt{\ep}$ inside the cylinder $I \times D(2r)$ and on a scale $\ep$ in the boundary layer $Q \setminus (I \times D(r))$ (and use a smooth cutoff function to transition between $I \times D(2r)$ and $Q \setminus (I \times D(r))$). We then note that, with this definition, we have the identity $U = \tilde f_\ep ~\mbox{on}~ I \times \partial D^\ep$ (as $\partial D^\ep$ is defined to be the external vertex boundary) and, on the set $\{ 0 \} \times D^\ep$,
\begin{equation} \label{eq:13481801}
   \| U(0 , \cdot) - \tilde f_\ep(0 , \cdot) \|_{L^2 \left( D^\ep \right)} \leq C r \left\| \nabla f(0 , \cdot) \right\|_{L^2 \left( D \right)} \leq C r \left\| f \right\|_{H^2 \left( Q \right)},
\end{equation}
where the first inequality is a consequence of the definition of the function $U$ and the second one is a consequence of the existence of a trace for the gradient of a function in the Sobolev space $H^2(Q)$.

We will denote by~$\vec{\nabla}^\ep U$ the piecewise constant discrete gradient field defined by the formula
\begin{equation} \label{def.mapnablaUU}
    \vec{\nabla}^\ep U(t , x) := \sum_{y \in \ep \Zd} \vec{\nabla}^\ep U(t , y) \indc_{\{ y +  [-\ep, \ep]^d \}}(x).
\end{equation}
Using the $H^2$-regularity of the function $\bar u$, we have
\begin{equation*}
    \left\| \bar u - U  \right\|_{L^2(Q)} + \| \nabla \bar u - \vec{\nabla}^\ep U \|_{L^2(Q)} \leq C r \left\| f \right\|_{H^2(Q)}.
\end{equation*}
We next compute the time derivative of the map $U$ inside the set $I \times D(r)$. Using that the map $\bar u$ solves the parabolic equation~\eqref{eq:defubarthmhydro}, we have the identity
\begin{equation*}
    \partial_t \left( \bar u \star \chi_{r}\right) (t , x) = - \int_{\Rd} D_p \bar \sigma \left( \nabla \bar u(t , x - y)\right) \cdot \nabla \chi_{r} (y) \, dy.
\end{equation*}
For each $i \in \{ 1 , \ldots, d\}$, we let $\eta_{\ep, i}: \Rd \to \R$ be the indicator function of the straight line joining the vertices $0$ and $\ep e_i$ divided by $\ep$ so that we have, for any (sufficiently regular) function $v : \Rd \to \R$,
\begin{equation} \label{eq:11030601}
    \nabla_i^\ep v (x) = \frac{v (x + \ep e_i) - v (x)}{\ep} = \int_{0}^\ep \frac{\nabla_i v (x + t e_i)}{\ep} \, dt = \nabla_i v \star \eta_{ \ep, i}(x).
\end{equation}
Using that the map $D_p \bar \sigma$ is Lipschitz and the regularity estimate~\eqref{eq:19050501} on the map $\bar u$, we further obtain that
\begin{equation*}
    \left\| D_p \bar \sigma \left( \nabla \bar u\right) - D_p \bar \sigma  \left( \nabla \bar u\right)  \star \eta_{\ep} \right\|_{L^2(I \times D^\ep(r) )} \leq C \ep \left\| f \right\|_{H^2(Q)},
\end{equation*}
where we used the notation
\begin{equation*}
    D_p \bar \sigma  \left( \nabla \bar u\right)  \star \eta_{\ep} = \left( D_p \bar \sigma  \left( \nabla \bar u\right)_1 \star \eta_{\ep, 1} , \ldots, D_p \bar \sigma  \left( \nabla \bar u\right)_d \star \eta_{\ep, d} \right).
\end{equation*}
Consequently, we may write
\begin{equation} \label{eq:10590601}
    \partial_t \left( \bar u \star \chi_{r}\right) (t , x)  = - \int_{\Rd} D_p \bar \sigma \left( \nabla \bar u(t , x - y)\right) \star \eta_\ep \cdot \nabla \chi_{r} (y) \, dy + \mathcal{E}_0(t , x),
\end{equation}
where, by the identity $r = \ep^{1/2}$, $\mathcal{E}_0$ is an error term satisfying
\begin{align} \label{eqtE000.jan}
    \left\| \mathcal{E}_0 \right\|_{L^2 \left( I \times D^\ep(r) \right)}  \leq Cr^{-1} \left\| D_p \bar \sigma \left( \nabla \bar u\right) - D_p \bar \sigma  \left( \nabla \bar u\right)  \star \eta_{ \ep} \right\|_{L^2(I \times D^\ep(r))}  & \leq C \ep r^{-1} \left\| f \right\|_{H^2(Q)} \\
    & \leq C \ep^{1/2} \left\| f \right\|_{H^2(Q)} . \notag
\end{align}
The first term in the right-hand side of~\eqref{eq:10590601} can be rewritten using the properties of the map $\eta_\ep$ as follows
\begin{align*}
    \int_{\Rd} D_p \bar \sigma \left( \nabla \bar u(t , x - y)\right) \star \eta_\ep \cdot \nabla \chi_{r} (y) \, dy & = \int_{\Rd} D_p \bar \sigma \left( \nabla \bar u(t , x - y)\right) \cdot \nabla \chi_{r} (y) \star \eta_\ep \, dy \\
    & = \int_{\Rd} D_p \bar \sigma \left( \nabla \bar u(t , x - y)\right) \cdot \vec{\nabla}^\ep \chi_{r} (y) \, dy \\
    & = - \vec{\nabla}^\ep \cdot \left( D_p \bar \sigma \left( \nabla \bar u\right) \star \chi_{r} \right) (t , x).
\end{align*}
Using once again the regularity estimate~\eqref{eq:19050501} on the function $\bar u$, we have
\begin{equation*}
    \left\| D_p \bar \sigma \left( \nabla \bar u\right) \star \chi_{r} - D_p \bar \sigma ( \vec{\nabla}^\ep U )  \right\|_{L^2(I \times D^\ep(r))} \leq C r \left\| f \right\|_{H^2(Q)}.
\end{equation*}
A combination of the two previous displays shows the identity: for any $(t , x) \in I \times D(r)$,
\begin{equation*}
    \int_{\Rd} D_p \bar \sigma \left( \nabla \bar u(t , x - y)\right) \star \eta_\ep \cdot \nabla \chi_{r} (y) \, dy = \vec{\nabla}^\ep \cdot  D_p \bar \sigma ( \vec{\nabla}^\ep U ) + \vec{\nabla}^\ep \cdot \mathcal{E}_1, 
\end{equation*}
where $\mathcal{E}_1$ is an error term satisfying
\begin{equation} \label{eqtE001.jan}
    \left\| \mathcal{E}_1 \right\|_{L^2 \left(I \times D^\ep(r)\right)} \leq  C \ep^{1/2} \left\| f \right\|_{H^2(Q)}.
\end{equation}
We have thus proved the identity (inside the set $I \times D(r)$)
\begin{equation*}
    \partial_t \left( \bar u \star \chi_{r}\right) = \vec{\nabla}^\ep \cdot  D_p \bar \sigma ( \vec{\nabla}^\ep U ) + \mathcal{E}_0 + \vec{\nabla}^\ep \cdot \mathcal{E}_1.
\end{equation*}
Using the definition of the map $U$ stated in~\eqref{def.mapU1}, we see that it solves the discrete parabolic equation 
\begin{equation*}
\partial_t U - \vec{\nabla}^\ep \cdot  D_p \bar \sigma ( \vec{\nabla}^\ep U ) =  \mathcal{E} ~\mbox{in}~Q^\ep,
\end{equation*}
where the error term $\mathcal{E}$ takes the following form
\begin{equation*}
    \mathcal{E} :=   \xi \mathcal{E}_0 + \xi  \vec{\nabla}^\ep \cdot \mathcal{E}_1 + \mathcal{E}_2,
\end{equation*}
where the terms $\mathcal{E}_2$ take the following form
\begin{equation} \label{eq:09292801}
    \mathcal{E}_2 :=  (1 - \xi) \left(\left(\partial_t \bar u \right) \star \chi_{\ep} \right) + (1 - \xi) \vec{\nabla}^\ep \cdot  D_p \bar \sigma ( \vec{\nabla}^\ep U ).
\end{equation}
Let us note that the error term $\mathcal{E}_2$ is supported in the boundary layer $I \times (D \setminus D(2r))$, that, by the regularity estimates~\eqref{eq:19050501} and~\eqref{eq:19060501}, its $L^2(I \times (D \setminus D(2r)))$-norm is bounded by the $H^2(Q)$-norm of $f$. We thus obtain
\begin{equation} \label{est.termE3final}
\left\| \mathcal{E}_2 \right\|_{L^2(I \times (D^\ep \setminus D^\ep(2r)))} \leq  C \left\| f \right\|_{H^2(Q)}.
\end{equation}
We next note that the map $w = U - \bar u^\ep$ solves a linear parabolic equation of the form, for some uniformly elliptic environment $\a$
\begin{equation*}
\left\{ \begin{aligned}
 \partial_t w - \vec{\nabla}^\ep \cdot \a  \vec{\nabla}^\ep w & = \mathcal{E} &~\mbox{in} &~Q^\ep, \\
w & = 0 &~\mbox{on} &~ I \times \partial D^\ep, \\
w & = U - \tilde f  &~\mbox{on} &~ \{0 \} \times D^\ep.
\end{aligned} \right.
\end{equation*}
Using the properties of the cutoff function $\xi$ and an energy estimate, we obtain the upper bound
\begin{align*}
  \left\| \nabla w \right\|^2_{L^2(Q^\ep)} & \leq C \left\| \mathcal{E}_0 \right\|_{L^2(Q^\ep(r))} \left\| w \right\|_{L^2(Q^\ep)}  \\
   & \qquad + C \left\| \mathcal{E}_1 \right\|_{L^2(Q^\ep(r))} \left(  \left\| \nabla w \right\|_{L^2(I \times D^\ep(r))} + r^{-1} \left\| w \right\|_{L^2(I \times (D \setminus D(2r)))} \right) \\
   & \qquad + C \left\| \mathcal{E}_2 \right\|_{L^2(I \times (D^\ep \setminus D^\ep(2r)))}  \left\| w \right\|_{L^2(I \times (D^\ep \setminus D^\ep(2r)))} \\
   & \qquad + C \|U(0 , \cdot) - \tilde f_\ep (0 , \cdot)  \|_{L^2 (D^\ep)}^2.
\end{align*}
The Poincar\'e inequality applied to the map $w$ (which is equal to $0$ on the boundary of the set~$D^\ep$) then shows
\begin{equation*}
    \left\| w \right\|_{L^2(I \times (D^\ep \setminus D^\ep(2r)))} \leq C r \left\| \nabla w \right\|_{L^2(I \times (D^\ep \setminus D^\ep(2r)))} ~\mbox{and}~ \left\| w \right\|_{L^2(Q^\ep)} \leq C \left\| \nabla w \right\|_{L^2(Q^\ep)}.
\end{equation*}
A combination of the two previous displays gives
\begin{align*}
    \left\| \nabla w \right\|_{L^2(Q^\ep)} & \leq C \left\| \mathcal{E}_0 \right\|_{L^2(I \times D^\ep(r))}  + C \left\| \mathcal{E}_1 \right\|_{L^2(I \times D^\ep(r))} + C r \left\| \mathcal{E}_3 \right\|_{L^2(I \times (D^\ep \setminus D^\ep(2r)))}   \\
    & \qquad + C \|U(0 , \cdot) - \tilde f_\ep (0 , \cdot)  \|_{L^2 (D^\ep)}.
\end{align*}
The first two terms can be estimated thanks to~\eqref{eq:13481801},~\eqref{eqtE000.jan},~\eqref{eqtE001.jan} and~\eqref{est.termE3final}. The proof of Proposition~\ref{prop.approx} is complete.
\end{proof}

The proof of Proposition~\ref{prop6.8discschemeloc} follows similar lines with the following differences: we consider a boundary layer with respect to the space and time variables, and consider a mesoscopic scale and a boundary layer of different sizes (the parameters $r$ and $\kappa$ in the proof below). We thus only provide a detailed sketch of the argument.

\begin{proof}[Proof of Proposition~\ref{prop6.8discschemeloc}]
Fix $\ep > 0$, a boundary condition $f \in W^{1 , 2+\gamma_0}_{\mathrm{par}} (Q)$, let $\bar u $ be the solution of the parabolic equation~\eqref{eq:defubarthmhydro} and extend it by the value $f$ outside of the set $Q$. We recall the notation $Q(r)$ introduced in~\eqref{eq:14341801}. In this case, the map $\bar u$ satisfies the following estimates:
\begin{itemize}
\item The global Meyers estimate: there exists an exponent $\gamma_0 := \gamma_0(d , c_+ , c_-) > 0$ such that
\begin{equation*}
    \left\| \nabla \bar u \right\|_{L^{2 + \gamma_0} (Q)}  \leq C (\left\| f \right\|_{W^{1,2 + \gamma_0}_{\mathrm{par}}(Q)} +1).
\end{equation*}
\item The interior regularity: for any $r \in (0 , 1)$,
\begin{equation*}
     r^{(d+2)/2}\left\| \nabla \bar u \right\|_{L^\infty(Q(r))}  + r \left\| \nabla^{2} \bar u \right\|_{L^{2} (Q(r))} +  r^2 \left\| \partial_t \nabla \bar u \right\|_{L^2(Q(r))}  \leq C (\left\| f \right\|_{H^{1}_{\mathrm{par}}(Q)} +1). 
\end{equation*}
\end{itemize}
We fix a boundary layer $r = \ep^{\theta_0}$ and a mesoscopic scale $\kappa = \ep^{\theta_1}$ with two exponents $\theta_0, \theta_1$ satisfying $0 < \theta_0 \ll \theta_1 \ll 1$ whose values will be decided later in the argument. We let $\chi : \Rd \to \R$ be a smooth nonnegative cutoff function supported in $[-1 , 1]^d$ and satisfying $\int_{\Rd} \chi = 1$. We then rescale the map $\chi$ by setting $\chi_{\kappa} = \kappa^{-d} \chi(\cdot/\kappa)$ and $\chi_{\ep} = \ep^{-d} \chi(\cdot/\ep)$ and define, for $(t , x) \in Q(r)$,
\begin{equation*}
    \bar u \star \chi_{\kappa} ( t , x) =  \int_{\Rd} \bar u(t , x - y) \chi_{\kappa} (y) \, dy.
\end{equation*}
We then let $\xi : \Rd \to \R$ be a space-time cutoff function satisfying
\begin{equation*}
    \indc_{Q(2r)} \leq \xi \leq  \indc_{Q(r)}, ~ r \left| \nabla \xi \right| + r^2  \left| \partial_t \xi \right|  \leq C.
\end{equation*}
We denote by
\begin{equation} \label{def.mapU}
    U = \xi \left( \bar u \star \chi_{\kappa}\right) + (1 - \xi) (\bar u \star \chi_{\ep} ).
\end{equation}
Using the Meyers estimate, the interior regularity of the map $\bar u$ and H\"{o}lder's inequality, we obtain that there exists an exponent $\beta >0$ such that
\begin{align*}
    \left\| \bar u - U  \right\|_{L^2(Q)} + \| \nabla \bar u - \vec{\nabla}^\ep U \|_{L^2(Q)} & \leq C \kappa \left\| \nabla^{2} \bar u \right\|_{L^{2} (Q(r))} + \left\| \nabla \bar u \right\|_{L^{2} (Q\setminus Q(2r))}  \\
    & \leq C \left( \kappa r^{-1} + \left|Q\setminus Q(2r) \right|^{\frac{2}{2 + \gamma_0}} \right) (\left\| f \right\|_{W^{1,2 + \gamma_0}_{\mathrm{par}}(Q)} +1)\\
    & \leq C \ep^{\beta} (\left\| f \right\|_{W^{1,2 + \gamma_0}_{\mathrm{par}}(Q)} +1) .
\end{align*}
Using the same argument as in the proof of Proposition~\ref{prop.approx}, we obtain the identity, inside the cylinder $Q(r)$,
\begin{equation*}
    \partial_t \left( \bar u \star \chi_{\kappa}\right) = \vec{\nabla}^\ep \cdot  D_p \bar \sigma ( \vec{\nabla}^\ep U ) + \mathcal{E}_0 + \vec{\nabla}^\ep \cdot \mathcal{E}_1,
\end{equation*}
where the two error terms $\mathcal{E}_0$ and $\mathcal{E}_1$ satisfy the estimate, for some exponent $\beta > 0$,
\begin{equation} \label{eq:17572601}
    \left\| \mathcal{E}_0 \right\|_{L^2 \left( Q^\ep(r)\right)} + \left\| \mathcal{E}_1 \right\|_{L^2 \left( Q^\ep(r) \right)}  \leq  C \ep^{\beta} (\left\| f \right\|_{H^1_\mathrm{par}(Q)} +1).
\end{equation}
Using the definition of the map $U$ stated in~\eqref{def.mapU}, we see that it solves the discrete parabolic equation 
\begin{equation*}
\begin{aligned}
\partial_t U - \vec{\nabla}^\ep \cdot  D_p \bar \sigma ( \vec{\nabla}^\ep U ) =  \mathcal{E} ~\mbox{in}~Q^\ep
\end{aligned}
\end{equation*}
where the error term $\mathcal{E}$ takes the following form
\begin{equation*}
    \mathcal{E} :=   \xi \mathcal{E}_0 + \xi  \vec{\nabla}^\ep \cdot \mathcal{E}_1 + \mathcal{E}_2 +  \mathcal{E}_3,
\end{equation*}
with
\begin{equation*}
    \left\{ 
    \begin{aligned}
    \mathcal{E}_2 & :=  (1 - \xi) \left(\left(\partial_t \bar u \right) \star \chi_{\ep} \right) + (1 - \xi) \vec{\nabla}^\ep \cdot  D_p \bar \sigma ( \vec{\nabla}^\ep U ), \\
   \mathcal{E}_3 & := \partial_t \xi \left( \bar u \star \chi_\kappa - \bar u \star \chi_\ep  \right).
    \end{aligned}
    \right.
\end{equation*}
The $L^{ 2 + \gamma_0} \left( I , W^{-1 , 2 + \gamma_0}(\Lambda^\ep_1) \right)$-norm of the term $ \mathcal{E}_2$ can be estimated as follows (omitting some of the technical details)
\begin{equation} \label{17582601}
    \left\| \mathcal{E}_2 \right\|_{L^{ 2 + \gamma_0} \left( I , W^{-1 , 2 + \gamma_0}(\Lambda^\ep_1) \right)} \leq C  ( \left\| f \right\|_{W^{1 , 2 + \gamma_0}_\mathrm{par}(Q)}+1).
\end{equation}
Additionally, the $L^2(Q^\ep \setminus Q^\ep(2r))$-norm of the term $\mathcal{E}_3$ can be estimated as follows
\begin{equation*} 
     \left\|  \mathcal{E}_3  \right\|_{L^2(Q^\ep \setminus Q^\ep(2r))} \leq C r^{-2} \kappa \left\| \nabla \bar u \right\|_{L^2 \left(  Q \setminus Q(2r) \right)} \leq C \ep^{\theta_1 - 2 \theta_0}  (\left\| f \right\|_{W^{1 , 2 + \gamma_0}_\mathrm{par}(Q)} +1).
\end{equation*}
We then choose the exponents $\theta_0$ and $\theta_1$ so that $2\theta_0 < \theta_1$ and deduce that
\begin{equation} \label{17592601}
    \left\|  \mathcal{E}_3  \right\|_{L^2(Q^\ep \setminus Q^\ep(2r))}  \leq C \ep^\beta (\left\| f \right\|_{W^{1 , 2 + \gamma_0}_\mathrm{par}(Q)} +1).
\end{equation}
Thus, the map $w = U - \bar u^\ep$ solves a linear parabolic equation of the form, for some uniformly elliptic environment $\a$,
\begin{equation*}
\left\{ \begin{aligned}
 \partial_t w - \vec{\nabla}^\ep \cdot \a  \vec{\nabla}^\ep w & = \mathcal{E} &~\mbox{in} &~Q^\ep, \\
w & = 0 &~\mbox{on} &~ \partial_{\mathrm{par}} Q^\ep.
\end{aligned} \right.
\end{equation*}
Denoting by $q_0 := (2 + \gamma_0) / (1 + \gamma_0) < 2$ the conjugate exponent of $(2 + \gamma_0)$, using the properties of the cutoff function $\xi$ and an energy estimate, we obtain the upper bound
\begin{align*}
  \left\| \nabla w \right\|^2_{L^2(Q^\ep)} & \leq C \left\| \mathcal{E}_0 \right\|_{L^2(Q^\ep(r))} \left\| w \right\|_{L^2(Q^\ep)}  \\
   & \qquad + C \left\| \mathcal{E}_1 \right\|_{L^2(Q^\ep(r))} \left(  \left\| \nabla w \right\|_{L^2(Q^\ep)} + r^{-1} \left\| w \right\|_{L^2(Q^\ep \setminus Q^\ep(2r))} \right) \\
   & \qquad + C \left\| \mathcal{E}_2 \right\|_{L^{ 2 + \gamma_0} \left( I , W^{-1 , (2 + \gamma_0)}(\Lambda^\ep_1) \right)}  \left\| \nabla w \right\|_{L^{q_0}(Q^\ep \setminus Q^\ep(2r))} \\
   & \qquad + C \left\|  \mathcal{E}_3  \right\|_{L^2(Q^\ep \setminus Q^\ep(2r))}  \left\| w \right\|_{L^2(Q^\ep \setminus Q^\ep(2r))}.
\end{align*}
Combining the inequalities~\eqref{eq:17572601},~\eqref{17582601} and~\eqref{17592601} (and using H\"{o}lder's inequality to estimate the third term in the right-hand side), we deduce that there exists an exponent $\beta > 0$ such that
\begin{equation*}
    \left\| \nabla w \right\|_{L^2(Q^\ep)}  \leq C \ep^{\beta} (\left\| f \right\|_{W^{1 , 2 + \gamma_0}_\mathrm{par}(Q)} +1).
\end{equation*}
\end{proof}

{\small
\bibliographystyle{abbrv}
\bibliography{Langevin.bib}

\newcommand{\noop}[1]{} \def\cprime{$'$}
\begin{thebibliography}{10}

\bibitem{adams2019cauchy}
S.~Adams, S.~Buchholz, R.~Koteck{\`y}, and S.~M{\"u}ller.
\newblock Cauchy-born rule from microscopic models with non-convex potentials.
\newblock {\em arXiv preprint arXiv:1910.13564}, 2019.

\bibitem{AKM16}
S.~Adams, R.~Koteck{\'y}, and S.~M{\"u}ller.
\newblock Strict convexity of the surface tension for non-convex potentials.
\newblock {\em arXiv preprint arXiv:1606.09541}, 2016.

\bibitem{ACDS18}
S.~Andres, A.~Chiarini, J.-D. Deuschel, and M.~Slowik.
\newblock Quenched invariance principle for random walks with time-dependent
  ergodic degenerate weights.
\newblock {\em Ann. Probab.}, 46(1):302--336, 2018.

\bibitem{andres2021quenched}
S.~Andres, A.~Chiarini, and M.~Slowik.
\newblock Quenched local limit theorem for random walks among time-dependent
  ergodic degenerate weights.
\newblock {\em Prob. Theory Related Fields}, 179(3):1145--1181, 2021.

\bibitem{AT21}
S.~Andres and P.~A. Taylor.
\newblock Local limit theorems for the random conductance model and
  applications to the {G}inzburg-{L}andau {$\nabla\phi$} interface model.
\newblock {\em J. Stat. Phys.}, 182(2):Paper No. 35, 35, 2021.

\bibitem{ABM}
S.~Armstrong, A.~Bordas, and J.-C. Mourrat.
\newblock Quantitative stochastic homogenization and regularity theory of
  parabolic equations.
\newblock {\em Anal. PDE}, 11(8):1945--2014, 2018.

\bibitem{AD}
S.~Armstrong and P.~Dario.
\newblock Elliptic regularity and quantitative homogenization on percolation
  clusters.
\newblock {\em Comm. Pure Appl. Math.}, 71(9):1717--1849, 2018.

\bibitem{AFK2}
S.~Armstrong, S.~J. Ferguson, and T.~Kuusi.
\newblock Higher-order linearization and regularity in nonlinear
  homogenization.
\newblock {\em Arch. Ration. Mech. Anal.}, 237(2):631--741, 2020.

\bibitem{AFK1}
S.~Armstrong, S.~J. Ferguson, and T.~Kuusi.
\newblock Homogenization, linearization, and large-scale regularity for
  nonlinear elliptic equations.
\newblock {\em Comm. Pure Appl. Math.}, 74(2):286--365, 2021.

\bibitem{AKM1}
S.~Armstrong, T.~Kuusi, and J.-C. Mourrat.
\newblock Mesoscopic higher regularity and subadditivity in elliptic
  homogenization.
\newblock {\em Comm. Math. Phys.}, 347(2):315--361, 2016.

\bibitem{AKM}
S.~Armstrong, T.~Kuusi, and J.-C. Mourrat.
\newblock The additive structure of elliptic homogenization.
\newblock {\em Invent. Math.}, 208(3):999--1154, 2017.

\bibitem{AKMbook}
S.~Armstrong, T.~Kuusi, and J.-C. Mourrat.
\newblock {\em Quantitative stochastic homogenization and large-scale
  regularity}, volume 352 of {\em Grundlehren der Mathematischen
  Wissenschaften}.
\newblock Springer-Nature, 2019.

\bibitem{AS}
S.~Armstrong and C.~Smart.
\newblock Quantitative stochastic homogenization of convex integral
  functionals.
\newblock {\em Ann. Sci. \'Ec. Norm. Sup\'er. (4)}, 49(2):423--481, 2016.

\bibitem{AW}
S.~Armstrong and W.~Wu.
\newblock {$C^2$} regularity of the surface tension for the {$\nabla\phi$}
  interface model.
\newblock {\em Comm. Pure Appl. Math.}, 75(2):349--421, 2022.

\bibitem{Ar}
D.~G. Aronson.
\newblock Bounds for the fundamental solution of a parabolic equation.
\newblock {\em Bull. Amer. Math. Soc.}, 73:890--896, 1967.

\bibitem{AL1}
M.~Avellaneda and F.-H. Lin.
\newblock Compactness methods in the theory of homogenization.
\newblock {\em Comm. Pure Appl. Math.}, 40(6):803--847, 1987.

\bibitem{AL2}
M.~Avellaneda and F.-H. Lin.
\newblock {$L^p$} bounds on singular integrals in homogenization.
\newblock {\em Comm. Pure Appl. Math.}, 44(8-9):897--910, 1991.

\bibitem{BK07}
M.~Biskup and R.~Koteck\'{y}.
\newblock Phase coexistence of gradient {G}ibbs states.
\newblock {\em Probab. Theory Related Fields}, 139(1-2):1--39, 2007.

\bibitem{biskup2018limit}
M.~Biskup and P.-F. Rodriguez.
\newblock Limit theory for random walks in degenerate time-dependent random
  environments.
\newblock {\em J. Funct. Anal.}, 274(4):985--1046, 2018.

\bibitem{BS11}
M.~Biskup and H.~Spohn.
\newblock Scaling limit for a class of gradient fields with nonconvex
  potentials.
\newblock {\em Ann. Probab.}, 39(1):224--251, 2011.

\bibitem{borell1975brunn}
C.~Borell.
\newblock The {B}runn-{M}inkowski inequality in {G}auss space.
\newblock {\em Invent. Math.}, 30(2):207--216, 1975.

\bibitem{boucheron2013concentration}
S.~Boucheron, G.~Lugosi, and P.~Massart.
\newblock {\em Concentration inequalities: A nonasymptotic theory of
  independence}.
\newblock Oxford university press, 2013.

\bibitem{BL75}
H.~Brascamp and E.~Lieb.
\newblock Some inequalities for {G}aussian measures and the long-range order of
  one-dimensional plasma, functional integration and its applications, ed. by
  {A}.{M}. {A}rthurs, 1975.

\bibitem{BL76}
H.~Brascamp and E.~Lieb.
\newblock On extensions of the {B}runn-{M}inkowski and {P}r\'{e}kopa-{L}eindler
  theorems, including inequalities for log concave functions, and with an
  application to the diffusion equation.
\newblock {\em J. Functional Analysis}, 22(4):366--389, 1976.

\bibitem{BLL75}
H.~J. Brascamp, E.~H. Lieb, and J.~L. Lebowitz.
\newblock The statistical mechanics of anharmonic lattices.
\newblock {\em Bull. Inst. Internat. Statist.}, 46(1):393--404 (1976), 1975.

\bibitem{BY}
D.~Brydges and H.-T. Yau.
\newblock Grad {$\phi$} perturbations of massless {G}aussian fields.
\newblock {\em Comm. Math. Phys.}, 129(2):351--392, 1990.

\bibitem{cardaliaguet2022scaling}
P.~Cardaliaguet, N.~Dirr, and P.~E. Souganidis.
\newblock Scaling limits and stochastic homogenization for some nonlinear
  parabolic equations.
\newblock {\em J. Differential Equations}, 307:389--443, 2022.

\bibitem{CG21}
N.~Clozeau and A.~Gloria.
\newblock Quantitative nonlinear homogenization: control of oscillations.
\newblock {\em arXiv preprint arXiv:2104.04263}, 2021.

\bibitem{CD12}
C.~Cotar and J.-D. Deuschel.
\newblock Decay of covariances, uniqueness of ergodic component and scaling
  limit for a class of {$\nabla\phi$} systems with non-convex potential.
\newblock {\em Ann. Inst. Henri Poincar\'{e} Probab. Stat.}, 48(3):819--853,
  2012.

\bibitem{CDM09}
C.~Cotar, J.-D. Deuschel, and S.~M\"{u}ller.
\newblock Strict convexity of the free energy for a class of non-convex
  gradient models.
\newblock {\em Comm. Math. Phys.}, 286(1):359--376, 2009.

\bibitem{CK12}
C.~Cotar and C.~K{\"u}lske.
\newblock Existence of random gradient states.
\newblock {\em Ann. Appl. Probab.}, 22(4):1650--1692, 2012.

\bibitem{CK15}
C.~Cotar and C.~K{\"u}lske.
\newblock Uniqueness of gradient {G}ibbs measures with disorder.
\newblock {\em Probab. Theory Related Fields}, 162(3-4):587--635, 2015.

\bibitem{DM1}
G.~Dal~Maso and L.~Modica.
\newblock Nonlinear stochastic homogenization.
\newblock {\em Ann. Mat. Pura Appl. (4)}, 144:347--389, 1986.

\bibitem{DM2}
G.~Dal~Maso and L.~Modica.
\newblock Nonlinear stochastic homogenization and ergodic theory.
\newblock {\em J. Reine Angew. Math.}, 368:28--42, 1986.

\bibitem{dario2021convergence}
P.~Dario.
\newblock Convergence of the thermodynamic limit for random-field random
  surfaces.
\newblock {\em Ann. App. Probab., to appear}, 2021.

\bibitem{DHP1}
P.~Dario, M.~Harel, and R.~Peled.
\newblock Random-field random surfaces.
\newblock {\em arXiv preprint arXiv:2101.02199}, 2021.

\bibitem{DG1}
E.~De~Giorgi.
\newblock Sulla differenziabilit\`a e l'analiticit\`a delle estremali degli
  integrali multipli regolari.
\newblock {\em Mem. Accad. Sci. Torino. Cl. Sci. Fis. Mat. Nat. (3)}, 3:25--43,
  1957.

\bibitem{DD05}
T.~Delmotte and J.-D. Deuschel.
\newblock On estimating the derivatives of symmetric diffusions in stationary
  random environment, with applications to {$\nabla\phi$} interface model.
\newblock {\em Probab. Theory Related Fields}, 133(3):358--390, 2005.

\bibitem{DGI00}
J.-D. Deuschel, G.~Giacomin, and D.~Ioffe.
\newblock Large deviations and concentration properties for {$\nabla\phi$}
  interface models.
\newblock {\em Probab. Theory Related Fields}, 117(1):49--111, 2000.

\bibitem{DNV19}
J.-D. Deuschel, T.~Nishikawa, and Y.~Vignaud.
\newblock Hydrodynamic limit for the {G}inzburg-{L}andau {$\nabla\phi$}
  interface model with non-convex potential.
\newblock {\em Stochastic Process. Appl.}, 129(3):924--953, 2019.

\bibitem{deuschel2022isomorphism}
J.-D. Deuschel and P.-F. Rodriguez.
\newblock An isomorphism theorem for {G}inzburg-{L}andau interface models and
  scaling limits.
\newblock {\em arXiv preprint arXiv:2206.14805}, 2022.

\bibitem{FN19}
J.~Fischer and S.~Neukamm.
\newblock Optimal homogenization rates in stochastic homogenization of
  nonlinear uniformly elliptic equations and systems.
\newblock {\em Arch. Ration. Mech. Anal.}, 242(1):343--452, 2021.

\bibitem{fontaine1983non}
J.-R. Fontaine.
\newblock Non-perturbative methods for the study of massless models.
\newblock In {\em Scaling and Self-Similarity in Physics}, pages 203--226.
  Springer, 1983.

\bibitem{FSS}
J.~Fr\"{o}hlich, B.~Simon, and T.~Spencer.
\newblock Infrared bounds, phase transitions and continuous symmetry breaking.
\newblock {\em Comm. Math. Phys.}, 50(1):79--95, 1976.

\bibitem{F05}
T.~Funaki.
\newblock Stochastic interface models.
\newblock In {\em Lectures on {P}robability {T}heory and {S}tatistics}, volume
  1869 of {\em Lecture Notes in Math.}, pages 103--274. Springer, Berlin, 2005.

\bibitem{FS}
T.~Funaki and H.~Spohn.
\newblock Motion by mean curvature from the {G}inzburg-{L}andau {$\nabla \phi$}
  interface model.
\newblock {\em Comm. Math. Phys.}, 185(1):1--36, 1997.

\bibitem{GOS}
G.~Giacomin, S.~Olla, and H.~Spohn.
\newblock Equilibrium fluctuations for {$\nabla\phi$} interface model.
\newblock {\em Ann. Probab.}, 29(3):1138--1172, 2001.

\bibitem{giaquinta1982partial}
M.~Giaquinta and M.~Struwe.
\newblock On the partial regularity of weak solutions of nonlinear parabolic
  systems.
\newblock {\em Math. Z.}, 179(4):437--451, 1982.

\bibitem{GT01}
D.~Gilbarg and N.~S. Trudinger.
\newblock {\em Elliptic partial differential equations of second order}.
\newblock Classics in Mathematics. Springer-Verlag, Berlin, 2001.
\newblock Reprint of the 1998 edition.

\bibitem{giusti2003direct}
E.~Giusti.
\newblock {\em Direct methods in the calculus of variations}.
\newblock World Scientific, 2003.

\bibitem{GNO}
A.~Gloria, S.~Neukamm, and F.~Otto.
\newblock Quantification of ergodicity in stochastic homogenization: optimal
  bounds via spectral gap on {G}lauber dynamics.
\newblock {\em Invent. Math.}, 199(2):455--515, 2015.

\bibitem{GNO14}
A.~Gloria, S.~Neukamm, and F.~Otto.
\newblock A regularity theory for random elliptic operators.
\newblock {\em Milan J. Math.}, 88(1):99--170, 2020.

\bibitem{GO1}
A.~Gloria and F.~Otto.
\newblock An optimal variance estimate in stochastic homogenization of discrete
  elliptic equations.
\newblock {\em Ann. Probab.}, 39(3):779--856, 2011.

\bibitem{GO2}
A.~Gloria and F.~Otto.
\newblock An optimal error estimate in stochastic homogenization of discrete
  elliptic equations.
\newblock {\em Ann. Appl. Probab.}, 22(1):1--28, 2012.

\bibitem{GO15}
A.~Gloria and F.~Otto.
\newblock The corrector in stochastic homogenization: optimal rates, stochastic
  integrability, and fluctuations.
\newblock {\em arXiv preprint arXiv:1510.08290}, 2015.

\bibitem{GO115}
A.~{Gloria} and F.~{Otto}.
\newblock {Quantitative results on the corrector equation in stochastic
  homogenization.}
\newblock {\em {J. Eur. Math. Soc. (JEMS)}}, 19(11):3489--3548, 2017.

\bibitem{guo1988nonlinear}
M.~Z. Guo, G.~C. Papanicolaou, and S.~R.~S. Varadhan.
\newblock Nonlinear diffusion limit for a system with nearest neighbor
  interactions.
\newblock {\em Comm. Math. Phys.}, 118(1):31--59, 1988.

\bibitem{han2011elliptic}
Q.~Han and F.~Lin.
\newblock {\em Elliptic partial differential equations}, volume~1.
\newblock American Mathematical Soc., 2011.

\bibitem{HS}
B.~Helffer and J.~Sj\"ostrand.
\newblock On the correlation for {K}ac-like models in the convex case.
\newblock {\em J. Statist. Phys.}, 74(1-2):349--409, 1994.

\bibitem{JKO}
V.~V. Jikov, S.~M. Kozlov, and O.~A. Ole\u{\i}nik.
\newblock {\em Homogenization of differential operators and integral
  functionals}.
\newblock Springer-Verlag, Berlin, 1994.
\newblock Translated from the Russian by G. A. Yosifian [G. A. Iosif\cprime
  yan].

\bibitem{K1}
S.~M. Kozlov.
\newblock Averaging of differential operators with almost periodic rapidly
  oscillating coefficients.
\newblock {\em Mat. Sb. (N.S.)}, 107(149)(2):199--217, 317, 1978.

\bibitem{KO06}
C.~K\"{u}lske and E.~Orlandi.
\newblock A simple fluctuation lower bound for a disordered massless random
  continuous spin model in {$D=2$}.
\newblock {\em Electron. Comm. Probab.}, 11:200--205, 2006.

\bibitem{KO08}
C.~K\"{u}lske and E.~Orlandi.
\newblock Continuous interfaces with disorder: {E}ven strong pinning is too
  weak in two dimensions.
\newblock {\em Stochastic Process. Appl.}, 118(11):1973--1981, 2008.

\bibitem{ladyvzenskaja1988linear}
O.~A. Lady{\v{z}}enskaja, V.~A. Solonnikov, and N.~N. Ural'ceva.
\newblock {\em Linear and quasi-linear equations of parabolic type}, volume~23.
\newblock American Mathematical Soc., 1988.

\bibitem{magazinov2020concentration}
A.~Magazinov and R.~Peled.
\newblock Concentration inequalities for log-concave distributions with
  applications to random surface fluctuations.
\newblock {\em To appear in Ann. Probab., arXiv preprint arXiv:2006.05393},
  2020.

\bibitem{MM94}
K.~Messaoudi and G.~Michaille.
\newblock Stochastic homogenization of nonconvex integral functionals.
\newblock {\em RAIRO Mod\'{e}l. Math. Anal. Num\'{e}r.}, 28(3):329--356, 1994.

\bibitem{meyers1963p}
N.~G. Meyers.
\newblock An {$L^{p}$}-estimate for the gradient of solutions of second order
  elliptic divergence equations.
\newblock {\em Ann. Scuola Norm. Sup. Pisa Cl. Sci. (3)}, 17:189--206, 1963.

\bibitem{Mi}
J.~Miller.
\newblock Fluctuations for the {G}inzburg-{L}andau {$\nabla\phi$} interface
  model on a bounded domain.
\newblock {\em Comm. Math. Phys.}, 308(3):591--639, 2011.

\bibitem{moser1961harnack}
J.~Moser.
\newblock On {H}arnack's theorem for elliptic differential equations.
\newblock {\em Comm. Pure and Appl. Math.}, 14(3):577--591, 1961.

\bibitem{mourrat2016anchored}
J.-C. Mourrat and F.~Otto.
\newblock Anchored {N}ash inequalities and heat kernel bounds for static and
  dynamic degenerate environments.
\newblock {\em J. Funct. Anal.}, 270(1):201--228, 2016.

\bibitem{NS}
A.~Naddaf and T.~Spencer.
\newblock On homogenization and scaling limit of some gradient perturbations of
  a massless free field.
\newblock {\em Comm. Math. Phys.}, 183(1):55--84, 1997.

\bibitem{NS1}
A.~Naddaf and T.~Spencer.
\newblock Estimates on the variance of some homogenization problems.
\newblock {\em preprint}, 871(10.1007), 1998.

\bibitem{nash1958continuity}
J.~Nash.
\newblock Continuity of solutions of parabolic and elliptic equations.
\newblock {\em Amer. J. Math.}, 80:931--954, 1958.

\bibitem{nishikawa2003hydrodynamic}
T.~Nishikawa.
\newblock Hydrodynamic limit for the {G}inzburg-{L}andau {$\nabla\phi$}
  interface model with boundary conditions.
\newblock {\em Probab. Theory Related Fields}, 127(2):205--227, 2003.

\bibitem{PV1}
G.~C. Papanicolaou and S.~R.~S. Varadhan.
\newblock Boundary value problems with rapidly oscillating random coefficients.
\newblock In {\em Random fields, {V}ol. {I}, {II} ({E}sztergom, 1979)},
  volume~27 of {\em Colloq. Math. Soc. J\'anos Bolyai}, pages 835--873.
  North-Holland, Amsterdam, 1981.

\bibitem{parviainen2009global}
M.~Parviainen.
\newblock Global gradient estimates for degenerate parabolic equations in
  nonsmooth domains.
\newblock {\em Ann. Mat. Pura Appl. (4)}, 188(2):333--358, 2009.

\bibitem{Sh}
S.~Sheffield.
\newblock Random surfaces.
\newblock {\em Ast\'{e}risque}, (304):vi+175, 2005.

\bibitem{Sj}
J.~Sj{\"o}strand.
\newblock Correlation asymptotics and {W}itten {L}aplacians.
\newblock {\em Algebra i Analiz}, 8(1):160--191, 1996.

\bibitem{sudakov1978extremal}
V.~N. Sudakov and B.~S. Tsirel'son.
\newblock Extremal properties of half-spaces for spherically invariant
  measures.
\newblock {\em Zap. Nau\v{c}n. Sem. Leningrad. Otdel. Mat. Inst. Steklov.
  (LOMI)}, 41:14--24, 165, 1974.

\bibitem{VK08}
A.~C. van Enter and C.~K{\"u}lske.
\newblock Nonexistence of random gradient {G}ibbs measures in continuous
  interface models in {$d=2$}.
\newblock {\em Ann. Appl. Probab.}, 18(1):109--119, 2008.

\bibitem{V06}
Y.~Velenik.
\newblock Localization and delocalization of random interfaces.
\newblock {\em Probab. Surv.}, 3:112--169, 2006.

\bibitem{wu2022local}
W.~Wu.
\newblock Local central limit theorem for gradient field models.
\newblock {\em arXiv preprint arXiv:2202.13578}, 2022.

\bibitem{Y1}
V.~V. Yurinski{\u\i}.
\newblock Averaging of symmetric diffusion in a random medium.
\newblock {\em Sibirsk. Mat. Zh.}, 27(4):167--180, 215, 1986.

\bibitem{ZKO82}
V.~V. Zhikov, S.~M. Kozlov, and O.~A. Ole\u{\i}nik.
\newblock Averaging of parabolic operators.
\newblock {\em Trudy Moskov. Mat. Obshch.}, 45:182--236, 1982.

\end{thebibliography}
}

\end{document}